\documentclass[a4paper,11pt]{book}

\usepackage[a4paper]{geometry}
\usepackage[utf8]{inputenc}
\usepackage[T1]{fontenc}

\usepackage{lmodern}
\rmfamily
\DeclareFontShape{T1}{lmr}{b}{sc}{<->ssub*cmr/bx/sc}{}
\DeclareFontShape{T1}{lmr}{bx}{sc}{<->ssub*cmr/bx/sc}{}

\usepackage{amssymb,amsmath,amsthm,amscd,amstext,amsfonts}
\usepackage{mathrsfs}

\usepackage{subfig}

\usepackage{tabularx} 
\usepackage{calc} 
\usepackage{graphicx} 

\usepackage{color}
\usepackage[colorlinks]{hyperref}

\usepackage[nottoc]{tocbibind}

\usepackage{makeidx}
\makeindex

\usepackage{dsfont}
\usepackage{epigraph}

\usepackage{tikz}
\usepackage{tkz-graph}
\usepackage{tkz-berge}
\usetikzlibrary{matrix,arrows}

\usepackage{fancyhdr}
\makeatletter
\let\ps@plain=\ps@empty
\makeatother

\let\origdoublepage\cleardoublepage
\newcommand{\clearemptydoublepage}{%
  \clearpage
  {\pagestyle{empty}\origdoublepage}%
}
\let\cleardoublepage\clearemptydoublepage

\theoremstyle{plain} 

\newenvironment{dedicace}{%
  \newpage\thispagestyle{empty}
  \hfill\begin{minipage}{100mm}\begin{flushright}\it}{%
  \end{flushright}\end{minipage}\vfill}

\title{Mod\`ele de Littelmann pour cristaux g\'eom\'etriques, 
        fonctions de Whittaker sur des groupes de Lie semi-simples et mouvement brownien}
\author{Reda CHHAIBI}

\begin{document}

\def\half{\frac{1}{2}}

\def\l{\lambda}
\def\t{\theta}
\def\T{\Theta}
\def\m{\mu}
\def\a{\alpha}
\def\b{\beta}
\def\g{\gamma}
\def\o{\omega}
\def\p{\varphi}
\def\D{\Delta}
\def\O{\Omega}

\def\A{{\mathbb A}}
\def\N{{\mathbb N}}
\def\Z{{\mathbb Z}}
\def\Q{{\mathbb Q}}
\def\R{{\mathbb R}}
\def\C{{\mathbb C}}

\def\P{{\mathbb P}}
\def\E{{\mathbb E}}

\def\Ac{{\mathcal A}}
\def\Bc{{\mathcal B}}
\def\Cc{{\mathcal C}}
\def\Dc{{\mathcal D}}
\def\Ec{{\mathcal E}}
\def\Fc{{\mathcal F}}
\def\Gc{{\mathcal G}}
\def\Hc{{\mathcal H}}
\def\Ic{{\mathcal I}}
\def\Kc{{\mathcal K}}
\def\Lc{{\mathcal L}}
\def\qhat{{\mathcal R}}
\def\Oc{{\mathcal O}}
\def\Pc{{\mathcal P}}
\def\Qc{{\mathcal Q}}
\def\Rc{{\mathcal R}}
\def\Sc{{\mathcal S}}
\def\Tc{{\mathcal T}}
\def\Uc{{\mathcal U}}
\def\Zc{{\mathcal Z}}

\def\Acc{{\mathscr A}}
\def\Bcc{{\mathscr B}}
\def\Ccc{{\mathscr C}}

\def\Bfrak{{\mathfrak B}}

\def\afrak{{\mathfrak a}}
\def\bfrak{{\mathfrak b}}
\def\gfrak{{\mathfrak g}}
\def\hfrak{{\mathfrak h}}
\def\kfrak{{\mathfrak k}}
\def\nfrak{{\mathfrak n}}
\def\pfrak{{\mathfrak p}}
\def\ufrak{{\mathfrak u}}

\newtheorem{theorem}{Theorem}[section]

\newtheorem{definition}[theorem]{Definition}
\newtheorem{example}[theorem]{Example}
\newtheorem{properties}[theorem]{Properties}
\newtheorem{rmk}[theorem]{Remark}
\newtheorem{notation}[theorem]{Notation}

\newtheorem{lemma}[theorem]{Lemma}
\newtheorem{proposition}[theorem]{Proposition}
\newtheorem{thm}[theorem]{Theorem}
\newtheorem{corollary}[theorem]{Corollary}

\newtheorem{conjecture}{Conjecture}

\pdfbookmark[0]{Page de garde}{garde}

\thispagestyle{empty}

\begin{center}
  \begin{tabularx}{\textwidth}{m{1.5cm}Xm{3cm}X}
	& \large Universit\'e Paris VI - Pierre et Marie Curie
	& \includegraphics[width=7cm]{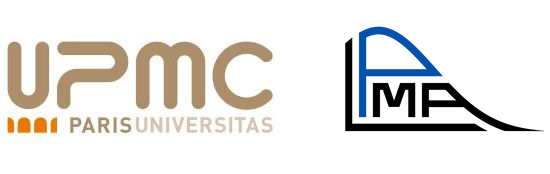}
        & 
  \end{tabularx}
\end{center}

\begin{center}
\vspace{\stretch{1}}
{\Large \textbf{Ecole Doctorale Paris Centre}}

\vspace{\stretch{2}}

{\Huge \textsc{Th\`ese de doctorat}}

\vspace{\stretch{1}}

{\LARGE Discipline : Math\'ematiques}

\vspace{\stretch{3}}

{\large pr\'esent\'ee par}
\vspace{\stretch{1}}

\textbf{{\LARGE Reda \textsc{CHHAIBI}}}

\vspace{\stretch{2}}
\hrule
\vspace{\stretch{1}}
{\LARGE \textbf{Mod\`ele de Littelmann pour cristaux g\'eom\'etriques, fonctions de Whittaker sur des groupes de Lie et mouvement brownien}}
\vspace{\stretch{1}}
\hrule
\vspace{\stretch{1}}

{\Large 
dirig\'ee par Philippe \textsc{Bougerol}}

\vspace{\stretch{5}}

{\Large Soutenue le 24 janvier 2013 devant le jury compos\'e de :}

\vspace{\stretch{2}}
{\Large
\begin{tabular}{lll}
M. Philippe \textsc{Biane}      & CNRS, Universit\'e Paris-Est       & examinateur\\
M. Alexei   \textsc{Borodin}    & MIT                                & rapporteur\\
M. Philippe \textsc{Bougerol}   & Universit\'e Paris VI              & directeur\\
M. Emmanuel \textsc{Breuillard} & Universit\'e Paris-Sud             & examinateur\\
M. Yves     \textsc{Lejan}      & Universit\'e Paris-Sud             & examinateur\\
M. Peter    \textsc{Littelmann} & Universit\"at zu K\"oln            & rapporteur\\
M. Marc     \textsc{Yor}        & Universit\'e Paris VI              & examinateur\\
\end{tabular}
}

\end{center}


\newpage

\vspace*{\fill}

\noindent\begin{center}
\begin{minipage}[t]{(\textwidth-2cm)/2}
Universit\'e Paris VI.\\
LPMA. Laboratoire de Probabilit\'es et Mod\`eles Al\'eatoires.\\
4, Place Jussieu. 16-26.
75005, Paris, France.
\end{minipage}
\hspace{1.5cm}
\begin{minipage}[t]{(\textwidth-2cm)/2}
\'Ecole doctorale Paris centre Case 188 \\ 
4 place Jussieu \\
75 252 Paris cedex 05
\end{minipage}
\end{center}


\begin{dedicace}
	Cette th\`ese est d\'edi\'ee \`a mes parents.
\end{dedicace}

\epigraph{``Les probabilit\'es? Un simple exercice d'int\'egration!''}{Attribu\'e \`a un membre du groupe Bourbaki}
\epigraph{``All deterministic mathematics are a particular case of Probability Theory.''}{ D.B., a good friend }
\epigraph{``Les math\'ematiques ne sont qu'une histoire de groupes.'' (All mathematics are just a tale about groups) }{ Henri Poincar\'e }



\pdfbookmark[0]{Remerciements}{remerciements}
\chapter*{Remerciements}

Ces remerciements sont \`a tous les gens ayant contribu\'e \`a cette th\`ese, parfois sans en avoir conscience. Je l'\'ecris en fran\c cais car c'est pour moi la langue du coeur.

Je remercie d'abord et avant tout Philippe Bougerol, mon directeur de th\`ese. Merci de m'avoir offert un tr\`es beau sujet de th\`ese, et d'avoir guid\'e dans mes premiers pas en recherche. 

Aussi, je remercie mes parents qui ont toujours attach\'e une tr\`es grande valeur \`a l'\'education, sous toute ses formes. Cette th\`ese leur est d\'edi\'ee.

Le tout aurait \'et\'e impossible sans le travail de la communaut\'e math\'ematique. Je suis \`a ce niveau redevable \`a tous les math\'ematiciens cit\'es dans la bibliographie. Avec Philippe Biane et Neil O'Connell, j'ai aussi partag\'e des conversations tr\`es enrichissantes.

Je suis reconnaissant aux rapporteurs, messieurs Littelmann et Borodin, pour avoir pris la peine de la lire, et d'avoir jug\'e le travail accompli. De m\^eme, je suis content que le jury prenne la peine de se d\'eplacer. 

Une mention sp\'eciale va \`a l'attention de mes professeurs: Emmanuel Roblet pour avoir su transmettre une v\'eritable passion des math\'ematiques, m\^eme si je ne l'ai pas toujours assum\'e. Saab Abou-jaoud\'e pour avoir parfaitement illustr\'e que les math\'ematiques sont un domaine trop important pour \^etre trop pris au s\'erieux. Nizar Touzi pour m'avoir donn\'e envie de connaitre plus en d\'etail le mouvement brownien. Valdo Durrleman pour m'avoir encourag\'e \`a poursuivre mes choix de cours hybrides \`a l'Ecole Polytechnique, m\^elant formule de Black-Scholes et cohomologie de De Rham. Nicole El Karoui pour sa bienveillance alors que je quittais le monde des math\'ematiques financi\`eres - ainsi que sa tr\`es belle lettre de recommandation qui m'a permis d'avoir ma bourse de th\`ese.

Le mot de la fin est adress\'e \`a mes amis, qui m'ont soutenu les 3 ans durant. Paul Bourgade pour m'avoir donn\'e envie de faire une th\`ese en parlant de la sienne dans un couloir. Amel Bentata, Karim Bounebache, Andreea Minca et Gilles Wainrib mes compagnons gal\'eriens - de bureau \`a Chevaleret. Les ins\'eparables Sophies (D\'ed\'e et Laruelle). Cyril Labb\'e pour avoir fourni les pauses cookie pendant la r\'edaction. Ceux qui ont commenc\'e une th\`ese avec moi: Yacine Barhoumi, Antoine Dahlqvist, Cl\'ement Foucart, Eric Lu\c con, Pascal Maillard, Salim Noreddine. Bonne chance \`a ceux qui finiront leur th\`ese un peu plus tard: Guillaume C\'ebron, Jean-Paul Daniel, Malik Drici, Xan Duhalde, Pablo Lessa, Bastien Mallein, Nelo Molter. Omer Adelman et Amaury Lambert, merci de m'avoir fait confiance pour les travaux dirig\'es d'int\'egration. Damien Simon pour les longues discussions math\'ematiques et physiques, toujours int\'eressantes. A Fran\c  cois Jaulin et la famille Jaulin, un grand merci de m'avoir adopt\'e.
Ceux qui ont suivi l'affaire d'un peu plus loin, mais dont j'ai toujours senti la pr\'esence bienveillante: Anouar Araamouch, Romain Balp, Alexandre Boritchev, Sally Lanar, Anouar Kiassi, Amine Naciri.

\pdfbookmark[0]{R\'esum\'e}{resume}
    
\chapter*{R\'esum\'e}

\section*{R\'esum\'e}
De fa\c con g\'en\'erale, cette th\`ese s'int\'eresse aux liens entre th\'eorie des repr\'esentations et probabilit\'es. Elle se subdivise en principalement trois parties.\\

Dans un premier volet plut\^ot alg\'ebrique, nous construisons un mod\`ele de chemins pour les cristaux g\'eom\'etriques de Berenstein et Kazhdan, pour un groupe de Lie complexe semi-simple. Il s'agira pour l'essentiel de d\'ecrire la structure alg\'ebrique, ses morphismes naturels et ses param\'etrisations. La th\'eorie de la totale positivit\'e y jouera un role particuli\`erement important.

Ensuite, nous avons choisi d'anticiper sur les r\'esultats probabilistes et d'exhiber une mesure canonique sur les cristaux g\'eom\'etriques. Celle-ci utilise comme ingr\'edients le superpotentiel de vari\'et\'e drapeau, et une mesure invariante sous les actions cristallines. La mesure image par l'application poids joue le role de mesure de Duistermaat-Heckman. Sa transform\'ee de Laplace d\'efinit les fonctions de Whittaker, fournissant une formule int\'egrale particuli\`erement int\'eressante pour tous les groupes de Lie. Il apparait alors clairement que les fonctions de Whittaker sont aux cristaux g\'eom\'etriques, ce que les caract\`eres sont aux cristaux combinatoires classiques. La r\`egle de Littlewood-Richardson est aussi expos\'ee.

Enfin nous pr\'esentons l'approche probabiliste permettant de trouver la mesure canonique. Elle repose sur l'id\'ee fondamentale que la mesure de Wiener induira la bonne mesure sur les structures alg\'ebriques du mod\`ele de chemins.

Dans une derni\`ere partie, nous d\'emontrons comment notre mod\`ele g\'eom\'etrique d\'eg\'en\`ere en le mod\`ele de Littelmann continu classique, pour retrouver des r\'esultats connus. Par exemple, la mesure canonique sur un cristal g\'eom\'etrique de plus haut poids d\'eg\'en\`ere en une mesure uniforme sur un polytope, et retrouve les param\'etrisations des cristaux continus.

\subsection*{Mots-clefs}
Cristaux g\'eom\'etriques, Mod\`ele de chemins de Littelmann g\'eom\'etrique, Mesure de Duistermaat-Heckman g\'eom\'etrique, Mouvement brownien, Th\'eor\`eme de Pitman 2M-X, Transform\'ees de Pitman, Hamiltonien de Toda, Fonctions de Whittaker, Identit\'es en loi Beta-Gamma, Totale positivit\'e.

\vspace{1cm}

\begin{center} \rule{\textwidth/3}{1pt} \end{center}

\vspace{1cm}

\begin{center} \Large \bf Littelmann path model for geometric crystals, Whittaker functions on Lie groups and Brownian motion \end{center}
\section*{Abstract}
Generally speaking, this thesis focuses on the interplay between the representations of Lie groups and probability theory. It subdivides into essentially three parts.\\

In a first rather algebraic part, we construct a path model for geometric crystals in the sense of Berenstein and Kazhdan, for complex semi-simple Lie groups. We will mainly describe the algebraic structure, its natural morphisms and parameterizations. The theory of total positivity will play a particularly important role.

Then, we anticipate on the probabilistic part by exhibiting a canonical measure on geometric crystals. It uses as ingredients the superpotential for the flag manifold and a measure invariant under the crystal actions. The image measure under the weight map plays the role of Duistermaat-Heckman measure. Its Laplace transform defines Whittaker functions, providing an interesting formula for all Lie groups. Then it appears clearly that Whittaker functions are to geometric crystals, what characters are to combinatorial crystals. The Littlewood-Richardson rule is also exposed.

Finally we present the probabilistic approach that allows to find the canonical measure. It is based on the fundamental idea that the Wiener measure will induce the adequate measure on the algebraic structures through the path model.

In the last chapter, we show how our geometric model degenerates to the continuous classical Littelmann path model and thus recover known results. For example, the canonical measure on a geometric crystal of highest weight degenerates into a uniform measure on a polytope, and recovers the parameterizations of continuous crystals.

\subsection*{Keywords}
Geometric crystals, Geometric Littelmann path model, Geometric Duistermaat-Heckman measure, Brownian motion, Pitman 2M-X theorem, Pitman transforms, Quantum Toda Hamiltonian, Whittaker functions, Givental-type integral representation of Toda eigenfunctions, Beta-Gamma algebra identities, Total positivity.

\setcounter{tocdepth}{2} 
\pdfbookmark[0]{Table of contents}{tableofcontents} 

\tableofcontents



\chapter{Introduction}
Si l'on s'int\'eresse aux liens entre th\'eorie des repr\'esentations et probabilit\'es, les mod\`eles combinatoires pour la th\'eorie des repr\'esentations de groupes de Lie tels que le mod\`ele de chemins de Littelmann constituent un pont naturel entre ces deux domaines. Car, apr\`es tout, les chemins al\'eatoires ou plut\^ot les marches al\'eatoires sont parmi les objets d'\'etude pr\'ef\'er\'es des probabilistes.

Dans le cas du groupe $SL_{n+1}$, il est connu que la combinatoire de la th\'eorie des repr\'esentations se retrouve dans celle des tableaux de Young (voir Fulton \cite{bib:Fulton97}). Et la correspondance de Robinson-Schensted-Knuth (RSK) met en bijection ces m\^emes chemins, appel\'es alors 'mots', avec des paires de tableaux de Young $(P,Q)$ de m\^eme forme. Finalement, cette correspondance fournit le lien le plus direct entre chemins et combinatoire de la th\'eorie des repr\'esentations. Plus d'informations dans ce sens sont donn\'ees au chapitre suivant. L'exemple phare qui nous a guid\'e et qui illustre parfaitement le genre de math\'ematiques qui nous int\'eresse est un r\'esultat de O'Connell (\cite{bib:OC03}). Il d\'emontre la propri\'et\'e de Markov pour le tableau $Q$ obtenu par RSK, si la variable d'entr\'ee est une marche al\'eatoire.

Dans les ann\'ees 1990, Littelmann a d\'ecrit un mod\`ele combinatoire o\`u il s'agit de compter des chemins discrets sur le r\'eseau des poids d'un groupe de Lie $G$ (\cite{bib:Littelmann}, \cite{bib:Littelmann95}, \cite{bib:Littelmann97}). Il g\'en\'eralise ainsi la combinatoire des tableaux de Young. Ces chemins sont une r\'ealisation d'objets alg\'ebriques, les cristaux de Kashiwara. Ils permettent d'obtenir un grand nombre d'informations sur les repr\'esentations de ce groupe: une formule des caract\`eres effective, une r\`egle de Littlewood-Richardson pour la d\'ecomposition de produits tensoriel en irr\'eductibles, une r\`egle de branchement etc... Les travaux de Biane et al. (\cite{bib:BBO}, \cite{bib:BBO2}) ont consist\'e, entre autres, \`a construire des cristaux continus qui peuvent \^etre vus comme la limite continue du mod\`ele de Littelmann. En consid\'erant des chemins browniens, les cristaux al\'eatoires engendr\'es sont d\'ecrits et munis de mesures canoniques. L'analogue de l'exemple phare cit\'e pr\'ec\'edemment est la propri\'et\'e de Markov du processus de plus haut poids $\Pc_{w_0} W$, qui s'interpr\`ete comme une g\'en\'eralisation du th\'eor\`eme de Pitman. Ici $W$ est un mouvement brownien standard multi-dimensionnel et $\Pc_{w_0}$ est une transformation de chemins qui g\'en\'eralise la transformation de Pitman:
$$ \Pc: W \mapsto \left( W_t - 2 \inf_{0 \leq s \leq t} W_t ; t \geq 0 \right)$$

Suite \`a cel\`a, dans les ann\'ees 2000, Berenstein et Kazhdan (\cite{bib:BK00}, \cite{bib:BK04}, \cite{bib:BK06}) ont introduit les cristaux g\'eom\'etriques, qui sont des ``relev\'es g\'eom\'etriques'' des cristaux de Kashiwara. La d\'enomination est d\^ue au fait que les cristaux de Berenstein et Kazhdan capturent des informations de nature g\'eom\'etrique sur la partie totalement positive du groupe et d\'eg\'en\`erent en les cristaux de Kashiwara, gr\^ace \`a une proc\'edure de tropicalisation. Au sein de la communaut\'e probabiliste commen\c cait aussi le relev\'e g\'eom\'etrique des r\'esultats cit\'es au paragraphe pr\'ec\'edent: Matsumoto et Yor (\cite{bib:MY00-1}, \cite{bib:MY00-2}) d\'emontrent une d\'eformation du th\'eor\`eme de Pitman \`a base de fonctionnelles exponentielles. Puis, O'Connell (\cite{bib:OConnell}) prouve la propri\'et\'e de Markov d'un processus de plus haut poids ``g\'eom\'etrique'', dans le cas du groupe lin\'eaire $GL_n$. Ses travaux \'etaient motiv\'es par une application aux polym\`eres dirig\'es et plus globalement, l'\'equation de KPZ.

Dans ce contexte, mon directeur m'a propos\'e de m'int\'eresser aux processus de plus haut poids li\'es \`a d'autres groupes. Les premi\`eres bribes de r\'esultats arrivant, nous nous sommes rendus compte que nous examinions des cristaux g\'eom\'etriques al\'eatoires constitu\'es de chemins browniens, et qu'il \'etait possible de d\'egager une th\'eorie g\'en\'erale pour tous les groupes de Lie complexes. Nous avons choisi de subdiviser cette th\`ese en trois parties principales et une derni\`ere qui explique comment on retrouve certains r\'esultats d\'ej\`a connus.

\paragraph{Mod\`ele de Littelmann pour cristaux g\'eom\'etriques:}

Ce premier volet est certainement le plus alg\'ebrique. Nous commen\c cons par d\'efinir axiomatiquement une notion de cristal g\'eom\'etrique, et une op\'eration alg\'ebrique de produit tensoriel. L'exemple typique, d\'ej\`a trait\'e par Berenstein et Kazhdan, est la vari\'et\'e totalement positive $\Bc = B_{>0}$, $B$ \'etant le sous-groupe de Borel inf\'erieur. Dans le cas du groupe $G = SL_{n+1}$, $B$ est constitu\'e des matrices triangulaires inf\'erieures et $\Bc \subset B$ de celles dont tous les mineurs sont strictement positifs. Un invariant essentiel fix\'e par les actions cristallines est le plus haut poids:
$$ hw: \Bc \longrightarrow \afrak $$
L'application $hw$ fournit un feuilletage de $\Bc$ en cristaux de plus haut poids:
$$ \Bc = \bigsqcup_{\lambda \in \afrak} \Bc(\lambda)$$

Une de nos valeurs ajout\'ees \`a ce niveau est le fait d'exhiber des param\'etrisations de $\Bc$ qui sont les relev\'es g\'eom\'etriques des param\'etrisations usuelles des cristaux classiques. Elles se r\'ev\`eleront ensuite non seulement utiles pour les calculs mais aussi naturelles, car compatibles avec des relev\'es g\'eom\'etriques au niveau du mod\`ele de chemin, qui se 'projette' sur $\Bc$.

Ensuite, nous passons \`a la description du mod\`ele de chemins proprement dit. Notons $\afrak \approx \R^n$ la partie r\'eelle de la sous-alg\`ebre de Cartan associ\'ee au groupe $G$. Toujours dans le cas $G = SL_{n+1}$, il s'agit simplement de l'ensemble des matrices diagonales de trace nulle. Pour un horizon $T>0$, nous consid\'erons $C_0\left( [0,T], \afrak \right)$ l'ensemble des chemins continus sur un segment $[0,T]$, nuls en $0$ et \`a valeurs dans $\afrak$. Nous munissons $C_0\left( [0,T], \afrak \right)$ d'une structure de cristal g\'eom\'etrique. Cette derni\`ere porte toutes les caract\'eristiques du mod\`ele de Littelmann classique dans le sens o\`u le poids d'un chemin est d\'efini comme son extr\'emit\'e, et le produit tensoriel de deux cristaux est isomorphe au cristal obtenu par la concat\'enation de leurs chemins.

La 'projection' $p(\pi)$ d'un chemin $\pi \in C_0\left( [0,T], \afrak \right)$ sur $\Bc$ a lieu en r\'esolvant une \'equation diff\'erentielle ordinaire sur le groupe r\'esoluble $B$, pilot\'ee par le chemin $\pi$. Moralement, il faut y penser comme le m\'ecanisme d'insertion du 'mot' $\pi$ dans $\Bc$, de la m\^eme fa\c on que la correspondance de Robinson-Schensted-Knuth ins\`ere un mot dans un tableau de Young semi-standard. Le plus haut poids $\lambda$ de $\pi \in C_0\left( [0,T], \afrak \right)$ est obtenu comme l'extr\'emit\'e du chemin $\Tc_{w_0}\pi$, o\`u $\Tc_{w_0}$ est un relev\'e g\'eom\'etrique de la transform\'ee de Pitman $\Pc_{w_0}$.

Nous d\'emontrons que l'application de 'projection' $p$ est un isomorphisme de cristaux entre un cristal de chemins connexe et un cristal de plus haut poids $\Bc(\lambda)$. Il en d\'ecoule un analogue du th\'eor\`eme d'ind\'ependance de Littelmann qui dit que la structure d'un cristal connexe ne d\'epend que de son plus haut poids. Enfin, nous interpr\'etons la bijection suivante comme une correspondance de Robinson-Schensted-Knuth g\'eom\'etrique:
$$ \begin{array}{cccc}
  RSK: & C_0\left( [0, T], \afrak \right) & \longrightarrow & \left\{ (x, \eta) \in \Bc \times C\left( ]0,T], \afrak \right) \ | \ hw(x) = \eta(T) \right\}\\
       &         \pi                      &   \mapsto       & \left( p(\pi), (\Tc_{w_0} \pi_t; 0 < t \leq T ) \right)
   \end{array}
 $$

\paragraph{Mesure canonique sur les cristaux g\'eom\'etriques:}

La deuxi\`eme partie anticipe un peu sur les r\'esultats probabilistes de la suivante. En quelques mots, en consid\'erant le cristal al\'eatoire engendr\'e par un chemin brownien, une mesure canonique sur le cristal g\'eom\'etrique $\Bc(\lambda)$ appara\^it par calcul. Il est question dans cette partie d'en tirer les cons\'equences.

Les ingr\'edients essentiels sont une mesure de r\'ef\'erence torique $\omega$ sur $\Bc(\lambda)$ ainsi que le superpotentiel $f_B: \Bc \rightarrow \R_{>0}$. En fait, il s'agissait d'objets introduits par Rietsch (\cite{bib:Rietsch11}) et $f_B$ quant \`a lui a \'et\'e utilis\'e par Berenstein et Kazhdan (\cite{bib:BK06}) pour tropicaliser leurs cristaux g\'eom\'etriques en cristaux de Kashiwara. La force de notre approche consiste \`a montrer qu'il s'agit en effet d'objets naturels, qui apparaissent dans une mesure canonique sur $\Bc(\lambda)$. Dans l'\'etude de $f_B$, nous r\'epondons aussi \`a une question ouverte de Rietsch au sujet de l'existence et l'unicit\'e d'un minimum sur chaque $\Bc(\lambda)$. 

La mesure image de cette mesure canonique par l'application de poids incarne naturellement la multiplicit\'e des poids sur le cristal $\Bc(\lambda)$. Elle joue le role de mesure de Duistermaat-Heckman. Sa transform\'ee de Laplace d\'efinit les fonctions de Whittaker, qui sont par cons\'equent l'\'equivalent des caract\`eres. La phrase qui consiste \`a dire que ``Les fonctions de Whittaker sont aux cristaux g\'eom\'etriques, ce que les caract\`eres sont aux cristaux discrets'' prend alors tout son sens. Pour ces fonctions importantes en th\'eorie des repr\'esentations et en th\'eorie des nombres, cel\`a fournit des formules int\'egrales int\'eressantes pour tous les groupes de Lie complexes semi-simples. Nous prendrons le temps de relier ces fonctions de Whittaker \`a celles introduites \`a la base par Jacquet dans \cite{bib:Jacquet67}, et de pr\'esenter un th\'eor\`eme de Plancherel qui implique, dans un certain sens, une orthogonalit\'e de ces caract\`eres g\'eom\'etriques.

La r\`egle de Littlewood-Richardson est aussi expos\'ee. La mesure induite sur les composantes connexes d'un produit tensoriel fait appara\^itre naturellement la charge centrale introduite par Berenstein et Kazhdan.

\paragraph{Cristaux al\'eatoires et mouvement brownien hypoelliptique sur le groupe r\'esoluble $B$:}

Au sein de cette th\`ese, ce chapitre incarne au mieux la philosophie probabiliste qui consiste caricaturalement \`a ``lancer une pi\`ece en l'air et voir jusqu'o\`u elle nous m\`ene''. Puisque la mesure de Wiener est la mesure naturelle sur l'ensemble des chemins continus, il s'agit de consid\'erer le cristal al\'eatoire engendr\'e par un chemin brownien et d'examiner les mesures qui en d\'ecoulent. Il s'agira d'une \'etude selon deux aspects.

\begin{itemize}
 \item D\'ecrire le processus de plus haut poids. On verra qu'il s'agit d'une diffusion reli\'ee au hamiltonien de Toda.
 \item D\'ecrire la mesure induite sur le cristal conditionnellement \`a son plus haut poids. Cette mesure sera la mesure canonique exploit\'ee dans la deuxi\`eme partie.
\end{itemize}

A cause de la correspondance de RSK g\'eom\'etrique, cette analyse tourne principalement autour de l'\'etude d'un mouvement brownien hypoelliptique $\left( B_t( W^{(\mu)} ), t \geq 0 \right)$ sur le groupe r\'esoluble $B$, pilot\'e par un mouvement brownien euclidien $W^{(\mu)}$ avec drift $\mu$. Lorsque ce drift est dans la chambre de Weyl, la partie $N$ de ce mouvement brownien hypoelliptique converge et donne une mesure invariante que l'on pourra explicitement calculer. 

Les fonctions de Whittaker sont des fonctions harmoniques pour ce processus. Plus pr\'ecis\'ement, elles sont induites par un caract\`ere du sous-groupe unipotent $N \subset B$. Ici $N$ jouera le role de 'fronti\`ere' au sens de Furstenberg. La mesure invariante sur $N$ donnera alors une int\'egrale de Poisson pour les fonctions de Whittaker. 

\paragraph{D\'eg\'en\'erescences:}

Il est possible de d\'eformer contin\^ument notre mod\`ele g\'eom\'etrique de chemins gr\^ace \`a un param\`etre $q>0$, qui peut s'interpr\'eter comme une temp\'erature. La limite lorsque $q$ tend vers z\'ero retrouve le mod\`ele de Littelmann continu pr\'esent\'e dans \cite{bib:BBO2}. Dans cette partie, nous d\'ecrivons le proc\'ed\'e de crystallisation, les structures restantes \`a temp\'erature nulle ainsi que les mesures naturelles correspondantes. Nous retrouvons alors la plupart des r\'esultats de \cite{bib:BBO} et \cite{bib:BBO2}.

\vspace{1cm}

\begin{center} \rule{\textwidth/3}{1pt} \end{center}

\vspace{1cm}

If one is interested in the interplay between representation theory and probability theory, combinatorial tools such as the Littelmann path model give a natural bridge. Afterall, random paths or more accurately random walks are among the probabilist's favorite objects.

In the case of $SL_{n+1}$, it is well-known that the combinatorics of representations are within the combinatorics of Young tableaux (see Fulton \cite{bib:Fulton97}). And the Robinson-Schensted-Knuth (RSK) correspondence gives a bijection between paths, that are called in this context 'words', and pairs $(P,Q)$ of Young tableaux with the same shape. In the end, this correspondence yields the most direct link between paths and the combinatorics of representation theory. More informations are presented in the next chapter. The key example that has guided us and that illustrates perfectly the kind of mathematics we are interested in, is a theorem by O'Connell (\cite{bib:OC03}). He proved the Markov property for the $Q$ tableau obtained through RSK, if the input variable is a random walk.

In the nineties, Littelmann described a path model based on the enumeration of certain discrete paths on the weight lattice of a Lie group $G$ (\cite{bib:Littelmann}, \cite{bib:Littelmann95}, \cite{bib:Littelmann97}). It generalizes the combinatorics of Young tableaux. These paths are a realization of algebraic objects, Kashiwara crystals. They encode a great deal of information on the group's representations: they yield an effective character formula, a Littlewood-Richardson rule for the decomposition of tensor products, a branching rule etc... In \cite{bib:BBO} and \cite{bib:BBO2}, Biane et al. constructed continuous crystals that can be seen as the continuous limit of Littelmann's path model. Then, by considering Brownian paths, the generated random crystals are described and endowed with canonical measures. There is a theorem analogous to the previous key example, which is the Markov property for the highest weight process $\Pc_{w_0} W$. It is interpreted as a generalization of Pitman's theorem. Here, $W$ is a multidimensional Brownian motion and $\Pc_{w_0}$ is a path transform generalizing the following simple Pitman transform:
$$ \Pc: W \mapsto \left( W_t - 2 \inf_{0 \leq s \leq t} W_t ; t \geq 0 \right)$$

Then, in the 2000s, Berenstein and Kazhdan (\cite{bib:BK00}, \cite{bib:BK04}, \cite{bib:BK06}) introduced geometric crystals, which are ``geometric liftings'' of Kashiwara crystals. Such a denomination is due to the fact that these objects use geometric information on the totally positive varieties in the group $G$ and degenerate into Kashiwara crystals after a tropicalisation procedure. Probabilists also started the geometric lifting of the results cited in the previous paragraph: Matsumoto and Yor (\cite{bib:MY00-1}, \cite{bib:MY00-2}) proved a deformation of Pitman's theorem based on exponential functionals of Brownian motion. And O'Connell (\cite{bib:OConnell}) proved the Markov property for a ``geometric'' highest weight process, in the case of the general linear group $GL_n$. His work was motivated by directed polymers and more generally the KPZ equation.

In this context, my advisor suggested I should investigate highest weight processes for other Lie groups. With the first bits of results, we realized that we were examining random geometric crystals made of Brownian paths and that it is possible to build a theory for all complex Lie groups. This thesis has three main parts and a final chapter that explains how to recover known results.

\paragraph{Littelmann path model for geometric crystals:}

This first part is certainly the most algebraic. We start by defining axiomatically a notion of geometric crystal and an algebraic tensor product operation. The typical example, already known to Berenstein and Kazhdan, is the totally positive variety $\Bc = B_{>0}$, where $B$ is the lower Borel subgroup. In the case of $G = SL_{n+1}$, $B$ is simply the set of lower triangular matrices and $\Bc \subset B$ is made of the matrices with positive minors. An essential invariant fixed by the crystal actions is the highest weight:
$$ hw: \Bc \longrightarrow \afrak $$
The map $hw$ gives rise to a foliation of $\Bc$ into highest weight crystals:
$$ \Bc = \bigsqcup_{\lambda \in \afrak} \Bc(\lambda)$$

At this level, one of our added values is to expose parametrizations for $\Bc$ that are geometric liftings of the parametrizations for usual crystals. These liftings are not only handy for computations but also natural, as they are commpatible with geometric liftings in the path model, which 'projects' onto $\Bc$.

Next, we deal with the path model itself. Let $\afrak \approx \R^n$ be the real part of the Cartan subalgebra associated to the group $G$. Still in the case of $G = SL_{n+1}$, it is simply the set of diagonal matrices with zero trace. For a fixed time horizon $T>0$, we consider $C_0\left( [0,T], \afrak \right)$, the set of $\afrak$-valued continuous paths on the interval $[0,T]$ and that vanish at $0$. We endow $C_0\left( [0,T], \afrak \right)$ with a geometric crystal structure. This structure has all the characteristics of the classical Littelmann path model as the weight of a path is given by its endpoint, and the tensor product of two crystals is isomorphic to the crystal obtained by the concatenation of their elements.

The 'projection' $p(\pi)$ of a path $\pi \in C_0\left( [0,T], \afrak \right)$ on $\Bc$ is the result obtained by solving on $B$ an ordinary differential equation driven by the path $\pi$. Morally speaking, one can think about it as an insertion procedure of the 'word' $\pi$ inside of $\Bc$, in the same fashion as the Robinson-Schensted-Knuth correspondence inserts a word inside a semi-standard Young tableau. The highest weight $\lambda$ for $\pi \in C_0\left( [0,T], \afrak \right)$ is obtained as the endpoint of $\Tc_{w_0}\pi$, where $\Tc_{w_0}$ is a geometric lifting of the Pitman transform $\Pc_{w_0}$.

We prove that the 'projection' map $p$ is a crystal isomorphism between a connected path crystal and a highest weight crystal $\Bc(\lambda)$. An analogue of Littelmann's independence theorem follows: the structure of a connected crystal depends only on its highest weight. Finally, we interpret the following bijection as the geometric counterpart of the Robinson-Schensted-Knuth correspondence:
$$ \begin{array}{cccc}
  RSK: & C_0\left( [0, T], \afrak \right) & \longrightarrow & \left\{ (x, \eta) \in \Bc \times C\left( ]0,T], \afrak \right) \ | \ hw(x) = \eta(T) \right\}\\
       &         \pi                      &   \mapsto       & \left( p(\pi), (\Tc_{w_0} \pi_t; 0 < t \leq T ) \right)
   \end{array}
 $$

\paragraph{Canonical measure on geometric crystals:}

The second part anticipates a little and uses the probabilistic results from the following chapter. In short, by considering the random crystal generated by a Brownian motion, a canonical measure on the geometric crystal $\Bc(\lambda)$ is computed. We think it is better to draw the full implications of that result before diving into stochastic analysis.

The essential ingredients are a toric reference measure $\omega$ on $\Bc(\lambda)$ and the superpotential $f_B: \Bc \rightarrow \R_{>0}$. In fact, both objects were introduced by Rietsch (\cite{bib:Rietsch11}) and $f_B$ was used by Berenstein et Kazhdan (\cite{bib:BK06}) in order to 'cut' tropicalized geometric crystals. The strength of our approach consists in showing that these are natural objects that appear in the canonical measure on $\Bc(\lambda)$. In the study of $f_B$, we also answer a question by Rietsch concerning the existence and uniqueness of a minimum on each $\Bc(\lambda)$. 

The image measure of our canonical measure through the weight map naturally embodies weight multiplicities in the geometric crystal $\Bc(\lambda)$. It plays the role of Duistermaat-Heckman measure. Its Laplace transform defines Whittaker functions, that are therefore analogous to characters. Thus, the sentence ``Whittaker functions are to geometric crystals, what characters are to combinatorial crystals'' takes on its full meaning. In the case of all complex semi-simple groups, we obtain integral formulae for these functions that are important in representation theory and number theory. We take the time of explaining how our Whittaker functions are related to those introduced originally by Jacquet in \cite{bib:Jacquet67}. Moreover, we present a Plancherel theorem, which in a way, implies the orthogonality of these geometric characters.

Furthermore, the Littlewood-Richardson rule is exposed. The central charge introduced by Berenstein and Kazhdan appears naturally in the measure induced by Brownian motion on the connected components of a tensor product, in the same way the superpotential showed up in the measure induced on geometric crystals.

\paragraph{Random crystals and hypoelliptic Brownian motion on the solvable group $B$:}

In the scope of this thesis, this chapter incarnates best the philosophy underlying probability theory, which can be very broadly summarized as ``let's toss a coin and see where it leads us''. Since the Wiener measure is the natural measure on continuous paths, we will consider a random crystal generated by a Brownian path and examine the induced measures. This study will have two aspects.

\begin{itemize}
 \item Describing the highest weight process. We will see it is a diffusion process related to the quantum Toda Hamiltonian.
 \item Describing the measure induced on a geometric crystal conditionally to its highest weight. This measure will be the canonical measure used in the second chapter.
\end{itemize}

Because of the geometric RSK correspondence, our analysis revolves mainly around the study of a hypoelliptic Brownian motion $\left( B_t( W^{(\mu)} ), t \geq 0 \right)$ on the solvable group $B$, driven by an Euclidian Brownian motion $W^{(\mu)}$ with drift $\mu$. When this drift is in the Weyl chamber, the $N$ part of this hypoelliptic Brownian motion converges and gives an invariant measure we will be able to compute.

Whittaker functions are harmonic for this process. More exactly, they are induced by a character for the unipotent subgroup $N \subset B$. Here, $N$ will play the role of boundary in Furstenberg's sense. And the invariant measure on $N$ will give rise to a Poisson integral for Whittaker functions.

\paragraph{Degenerations:}

It is possible to continuously deform our geometric path model thanks to a parameter $q>0$. This parameter can be interpreted as a temperature. In the limit as $q$ goes to zero, one recovers the continuous Littelmann path model presented in \cite{bib:BBO2}. In this final chapter, we describe the crystallization procedure, the remaining structures at zero temperature and the corresponding natural measures. This allows us to recover most of the results in \cite{bib:BBO} et \cite{bib:BBO2}.



\chapter{Classical related works: An informal panorama}
In this chapter, we informally depict known results that illustrate the interplay between probability theory and combinatorial representation theory, which is the kind of mathematics we will be dealing with. 

\section{The Robinson-Schensted-Knuth correspondence}
\label{section:classical_rsk}
This correspondence is probably the simplest example as it gives a correspondence between paths, in this case called 'words', and the representation theoretic objects that are Young tableaux. A standard reference is Fulton \cite{bib:Fulton97}.

An integer partition $\lambda$ is a tuple $\lambda_1 \geq \lambda_2 \geq \dots \geq 0$ such that
$$ |\lambda| := \sum_i \lambda_i < \infty $$
Every partition $\lambda$ can be visually represented as a Young diagram, a collection of left-justified cells: $\lambda_1$ cells on the first row, $\lambda_2$ on the second etc... 

\begin{example}
The Young diagram associated to the partition $\lambda = (5,3,2)$ is:
$$
\def\lr#1{\multicolumn{1}{|@{\hspace{.6ex}}c@{\hspace{.6ex}}|}{\raisebox{-.3ex}{$#1$}}}
\raisebox{-.6ex}{\begin{array}[b]{ccccc}
\cline{1-1}\cline{2-2}\cline{3-3}\cline{4-4}\cline{5-5}
\lr{ \ }&\lr{ \ }&\lr{ \ }&\lr{ \ }&\lr{ \ }\\
\cline{1-1}\cline{2-2}\cline{3-3}\cline{4-4}\cline{5-5}
\lr{ \ }&\lr{ \ }&\lr{ \ }\\
\cline{1-1}\cline{2-2}\cline{3-3}
\lr{ \ }&\lr{ \ }\\
\cline{1-1}\cline{2-2}
\end{array}} $$ 
\end{example}

Consider an alphabet of $n$ letters $\A = \left\{ 1, \dots, n \right\}$. A semi-standard (resp. standard) Young tableau is a filling of a Young diagram using the alphabet, such that the entries are weakly (resp. strictly) increasing from left to right and strictly increasing down the columns. We will use the abbreviations SST (resp. ST) for ``semi-standard tableau'' (resp. ``standard tableau''). 

In any case, the shape of a tableau $P$ is the integer partition obtained by erasing its entries and is denoted $sh(P)$.

\begin{example}
 The following $P$ and $Q$ are respectively a SST and a ST:
$$ P = 
\def\lr#1{\multicolumn{1}{|@{\hspace{.6ex}}c@{\hspace{.6ex}}|}{\raisebox{-.3ex}{$#1$}}}
\raisebox{-.6ex}{ $\begin{array}[b]{cccccc}
\cline{1-1}\cline{2-2}\cline{3-3}\cline{4-4}\cline{5-5}\cline{6-6}
\lr{1}&\lr{1}&\lr{1}&\lr{1}&\lr{2}&\lr{2}\\
\cline{1-1}\cline{2-2}\cline{3-3}\cline{4-4}\cline{5-5}\cline{6-6}
\lr{2}&\lr{2}\\
\cline{1-1}\cline{2-2}
\end{array}$} $$
$$ Q = 
\def\lr#1{\multicolumn{1}{|@{\hspace{.6ex}}c@{\hspace{.6ex}}|}{\raisebox{-.3ex}{$#1$}}}
\raisebox{-.6ex}{ $\begin{array}[b]{cccccc}
\cline{1-1}\cline{2-2}\cline{3-3}\cline{4-4}\cline{5-5}\cline{6-6}
\lr{1}&\lr{2}&\lr{4}&\lr{5}&\lr{7}&\lr{8}\\
\cline{1-1}\cline{2-2}\cline{3-3}\cline{4-4}\cline{5-5}\cline{6-6}
\lr{3}&\lr{6}\\
\cline{1-1}\cline{2-2}
\end{array}$} $$
\end{example}

Now let us describe an operation called row insertion. For a tableau $T$ and a letter $x \in \A$, one forms a new tableau that has one more entry labelled by $x$. Start with the first row and insert $x$ at the leftmost position that is strictly larger than $x$, in order to preserve the weakly increasing property from left to right. If that position is taken by a letter $y > x$, replace it by $x$ and continue the same procedure with $y$ on the next row. We say that $y$ has been bumped. If $x$ is at least as large as all the entries of the current row, $x$ is appended at the end. The procedure then stops.

The Robinson-Schensted-Knuth correspondence is given applying an algorithm. It has the following specifications:
\begin{itemize}
\item Input: A word $w \in \A^{ (\N) }$
\item Algorithm: Do row insertions of letters to obtain a tableau $P$ and record the growth in $Q$.
\item Output: A pair $(P, Q)$ of tableaux where $P$ is a SST and $Q$ is a ST.
\end{itemize}

\begin{thm}
The RSK correspondence is bijection between words and pairs $(P, Q)$ of tableaux, where $P$ is SST, $Q$ is ST and 
$sh(P) = sh(Q)$
\end{thm}

\begin{example}
\label{example:rsk_algorithm}
Applying the RSK algorithm to the word $w = \left( 1,2,1,1,2,1,2,2 \right)$, one obtains the following sequence of tableaux:
\begin{center}
{\def\lr#1{\multicolumn{1}{|@{\hspace{.6ex}}c@{\hspace{.6ex}}|}{\raisebox{-.3ex}{$#1$}}}
\raisebox{-.6ex}{$\begin{array}[b]{c}
\cline{1-1}
\lr{1}\\
\cline{1-1}
\end{array}$}
}
$\longrightarrow$
{\def\lr#1{\multicolumn{1}{|@{\hspace{.6ex}}c@{\hspace{.6ex}}|}{\raisebox{-.3ex}{$#1$}}}
\raisebox{-.6ex}{$\begin{array}[b]{cc}
\cline{1-1}\cline{2-2}
\lr{1}&\lr{2}\\
\cline{1-1}\cline{2-2}
\end{array}$}
}
$\longrightarrow$
{\def\lr#1{\multicolumn{1}{|@{\hspace{.6ex}}c@{\hspace{.6ex}}|}{\raisebox{-.3ex}{$#1$}}}
\raisebox{-.6ex}{$\begin{array}[b]{cc}
\cline{1-1}\cline{2-2}
\lr{1}&\lr{1}\\
\cline{1-1}\cline{2-2}
\lr{2}\\
\cline{1-1}
\end{array}$}
}
$\longrightarrow$
{\def\lr#1{\multicolumn{1}{|@{\hspace{.6ex}}c@{\hspace{.6ex}}|}{\raisebox{-.3ex}{$#1$}}}
\raisebox{-.6ex}{$\begin{array}[b]{ccc}
\cline{1-1}\cline{2-2}\cline{3-3}
\lr{1}&\lr{1}&\lr{1}\\
\cline{1-1}\cline{2-2}\cline{3-3}
\lr{2}\\
\cline{1-1}
\end{array}$}
}
$\longrightarrow$
{\def\lr#1{\multicolumn{1}{|@{\hspace{.6ex}}c@{\hspace{.6ex}}|}{\raisebox{-.3ex}{$#1$}}}
\raisebox{-.6ex}{$\begin{array}[b]{cccc}
\cline{1-1}\cline{2-2}\cline{3-3}\cline{4-4}
\lr{1}&\lr{1}&\lr{1}&\lr{2}\\
\cline{1-1}\cline{2-2}\cline{3-3}\cline{4-4}
\lr{2}\\
\cline{1-1}
\end{array}$}
}
$\longrightarrow$
{\def\lr#1{\multicolumn{1}{|@{\hspace{.6ex}}c@{\hspace{.6ex}}|}{\raisebox{-.3ex}{$#1$}}}
\raisebox{-.6ex}{$\begin{array}[b]{cccc}
\cline{1-1}\cline{2-2}\cline{3-3}\cline{4-4}
\lr{1}&\lr{1}&\lr{1}&\lr{1}\\
\cline{1-1}\cline{2-2}\cline{3-3}\cline{4-4}
\lr{2}&\lr{2}\\
\cline{1-1}\cline{2-2}
\end{array}$}
}
$\longrightarrow$
{\def\lr#1{\multicolumn{1}{|@{\hspace{.6ex}}c@{\hspace{.6ex}}|}{\raisebox{-.3ex}{$#1$}}}
\raisebox{-.6ex}{$\begin{array}[b]{ccccc}
\cline{1-1}\cline{2-2}\cline{3-3}\cline{4-4}\cline{5-5}
\lr{1}&\lr{1}&\lr{1}&\lr{1}&\lr{2}\\
\cline{1-1}\cline{2-2}\cline{3-3}\cline{4-4}\cline{5-5}
\lr{2}&\lr{2}\\
\cline{1-1}\cline{2-2}
\end{array}$}
}
$\longrightarrow$
{\def\lr#1{\multicolumn{1}{|@{\hspace{.6ex}}c@{\hspace{.6ex}}|}{\raisebox{-.3ex}{$#1$}}}
\raisebox{-.6ex}{$\begin{array}[b]{cccccc}
\cline{1-1}\cline{2-2}\cline{3-3}\cline{4-4}\cline{5-5}\cline{6-6}
\lr{1}&\lr{1}&\lr{1}&\lr{1}&\lr{2}&\lr{2}\\
\cline{1-1}\cline{2-2}\cline{3-3}\cline{4-4}\cline{5-5}\cline{6-6}
\lr{2}&\lr{2}\\
\cline{1-1}\cline{2-2}
\end{array} = P$}
}

\ \\

{\def\lr#1{\multicolumn{1}{|@{\hspace{.6ex}}c@{\hspace{.6ex}}|}{\raisebox{-.3ex}{$#1$}}}
\raisebox{-.6ex}{$\begin{array}[b]{c}
\cline{1-1}
\lr{1}\\
\cline{1-1}
\end{array}$}
}
$\longrightarrow$
{\def\lr#1{\multicolumn{1}{|@{\hspace{.6ex}}c@{\hspace{.6ex}}|}{\raisebox{-.3ex}{$#1$}}}
\raisebox{-.6ex}{$\begin{array}[b]{cc}
\cline{1-1}\cline{2-2}
\lr{1}&\lr{2}\\
\cline{1-1}\cline{2-2}
\end{array}$}
}
$\longrightarrow$
{\def\lr#1{\multicolumn{1}{|@{\hspace{.6ex}}c@{\hspace{.6ex}}|}{\raisebox{-.3ex}{$#1$}}}
\raisebox{-.6ex}{$\begin{array}[b]{cc}
\cline{1-1}\cline{2-2}
\lr{1}&\lr{2}\\
\cline{1-1}\cline{2-2}
\lr{3}\\
\cline{1-1}
\end{array}$}
}
$\longrightarrow$
{\def\lr#1{\multicolumn{1}{|@{\hspace{.6ex}}c@{\hspace{.6ex}}|}{\raisebox{-.3ex}{$#1$}}}
\raisebox{-.6ex}{$\begin{array}[b]{ccc}
\cline{1-1}\cline{2-2}\cline{3-3}
\lr{1}&\lr{2}&\lr{4}\\
\cline{1-1}\cline{2-2}\cline{3-3}
\lr{3}\\
\cline{1-1}
\end{array}$}
}
$\longrightarrow$
{\def\lr#1{\multicolumn{1}{|@{\hspace{.6ex}}c@{\hspace{.6ex}}|}{\raisebox{-.3ex}{$#1$}}}
\raisebox{-.6ex}{$\begin{array}[b]{cccc}
\cline{1-1}\cline{2-2}\cline{3-3}\cline{4-4}
\lr{1}&\lr{2}&\lr{4}&\lr{5}\\
\cline{1-1}\cline{2-2}\cline{3-3}\cline{4-4}
\lr{3}\\
\cline{1-1}
\end{array}$}
}
$\longrightarrow$
{\def\lr#1{\multicolumn{1}{|@{\hspace{.6ex}}c@{\hspace{.6ex}}|}{\raisebox{-.3ex}{$#1$}}}
\raisebox{-.6ex}{$\begin{array}[b]{cccc}
\cline{1-1}\cline{2-2}\cline{3-3}\cline{4-4}
\lr{1}&\lr{2}&\lr{4}&\lr{5}\\
\cline{1-1}\cline{2-2}\cline{3-3}\cline{4-4}
\lr{3}&\lr{6}\\
\cline{1-1}\cline{2-2}
\end{array}$}
}
$\longrightarrow$
{\def\lr#1{\multicolumn{1}{|@{\hspace{.6ex}}c@{\hspace{.6ex}}|}{\raisebox{-.3ex}{$#1$}}}
\raisebox{-.6ex}{$\begin{array}[b]{ccccc}
\cline{1-1}\cline{2-2}\cline{3-3}\cline{4-4}\cline{5-5}
\lr{1}&\lr{2}&\lr{4}&\lr{5}&\lr{7}\\
\cline{1-1}\cline{2-2}\cline{3-3}\cline{4-4}\cline{5-5}
\lr{3}&\lr{6}\\
\cline{1-1}\cline{2-2}
\end{array}$}
}
$\longrightarrow$
{\def\lr#1{\multicolumn{1}{|@{\hspace{.6ex}}c@{\hspace{.6ex}}|}{\raisebox{-.3ex}{$#1$}}}
\raisebox{-.6ex}{$\begin{array}[b]{cccccc}
\cline{1-1}\cline{2-2}\cline{3-3}\cline{4-4}\cline{5-5}\cline{6-6}
\lr{1}&\lr{2}&\lr{4}&\lr{5}&\lr{7}&\lr{8}\\
\cline{1-1}\cline{2-2}\cline{3-3}\cline{4-4}\cline{5-5}\cline{6-6}
\lr{3}&\lr{6}\\
\cline{1-1}\cline{2-2}
\end{array} = Q$}
}

\end{center}
\end{example}

\begin{rmk}
Strictly speaking, this correspondence is called the Robinson-Schensted correspondence. RSK usually refers to the generalization of Knuth when an integer matrix is taken as input.
\end{rmk}

Notice that the RSK correspondence gives a dynamically growing shape. If one takes in input a random walk, by discarding the $P$ tableau and looking only at the $Q$ tableau, one observes a shape process evolving in time.
\begin{thm}[O'Connell \cite{bib:OC03}]
\label{thm:shape_is_markov}
The shape is a Markov process.
\end{thm}
The transition kernel is written in term of Schur functions, which are the characters of the group $GL_n$.

\section{Pitman's theorem}

\begin{thm}[Pitman (1975)]
If $B$ is a standard Brownian motion, then
$$ 2 \sup_{0 \leq s \leq t}\left(B_s\right) - B_t$$
is a $3$-dimensional Bessel process, meaning that it is a Markovian diffusion process with infinitesimal generator
$$ \frac{1}{2} \frac{ d^2 }{dx^2} + \frac{ d }{dx}\left( \log(x) \right) \frac{ d }{dx}$$
\end{thm}

The coefficient $2$ is essential and there is a very strong rigidity with respect to that coefficient. Indeed, further investigations by Rogers and Pitman prove that $k=0, k=1, k=2$ are the only cases where the Markov property holds among the processes of the form:
$$ k \sup_{0 \leq s \leq t}\left(B_s\right) - B_t$$

At first glance, this theorem seems like a puzzling oddity, a singularity. Indeed the Markov property is a very rare feature in general processes, and it is very surprizing that a process tailored with extrema of Brownian motion, ends up being Markovian. Afterall, $\sup_{ 0 \leq s \leq t}( B_s )$ is a typical example of non-Markovian behavior.

By replacing the Brownian motion by its opposite, one might prefer the use of the path transform, which we call the Pitman transform:
$$ \Pc(\pi)(t) = \pi(t) - 2 \inf_{0 \leq s \leq t}( \pi(s) ) $$

Here, the miracle at play is of algebraic nature and is more simply explained by going back to a discrete setting. In fact, Pitman proved his theorem in the discrete setting first, from which he deduced the continuous case.

\begin{thm}[Discrete Pitman's theorem (1975)]
\label{thm:discrete_pitman}
Let $B_n$ a standard random walk on $\mathbb{Z}$. Then:
    $$X_n = B_n - 2 \inf_{0 \leq k \leq n} B_k$$
is Markov with transition kernel:
$$Q(x, x+1) = \frac{1}{2}\frac{x+2}{x+1} \quad Q(x, x-1) = \frac{1}{2}\frac{x}{x+1}$$
Moreover (intertwining measure):
$$ \P\left( B_n = b | X_n = x \right) = \mathcal{U}\left( \{-x, \dots, x-4, x-2, x\} \right)$$
\end{thm}

Although it seems quite unrelated, Pitman's theorem is in fact a particular case of theorem \ref{thm:shape_is_markov} when $\A = \{1, 2\}$ is an alphabet of two letters. Indeed, the input word can be seen as random walk by reading $1$ as a 'down' move and $2$ as an 'up' move. Then the integer partition $\lambda = \left\{ \lambda_1 \geq \lambda_2 \geq 0\right\}$ can be simply recorded by the quantity $\lambda_1-\lambda_2$. On figure \ref{fig:rw_and_its_pitman_transform}, by comparing to example \ref{example:rsk_algorithm}, one can see that shape is exactly given by the random walk's Pitman transform. 

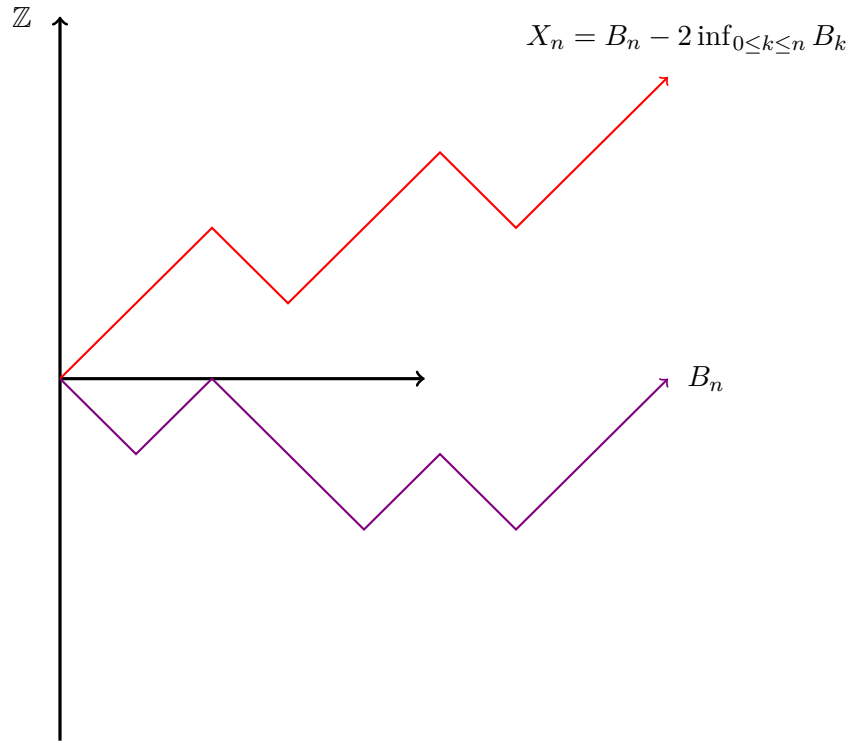
\begin{figure}[ht!]
\centering
\caption{An instance of random walk corresponding to the word $w = \left( 1,2,1,1,2,1,2,2 \right)$, and its Pitman transform}
\label{fig:rw_and_its_pitman_transform}

\begin{tikzpicture}[auto, bend right, scale=1]
\path [draw, very thick, color = black, ->] (0,0) -- (4.80000000000000 , 0 );
\node at ( -0.5, 4.80000000000000 ) {$\Z$};
\path [draw, very thick, color = black, ->] ( 0, -4.80000000000000 )
-- ( 0, 4.80000000000000 );

\definecolor{color_1}{rgb}{ 1.0,0.0,0.0 }
\path [draw, thick, color = color_1, ->] (0,0) -- (2,2) -- (3,1) -- (5,3) -- (6,2) -- (8,4);
\node at ( 8.25, 4.5 ) {$X_n = B_n - 2 \inf_{0 \leq k \leq n} B_k$};

\definecolor{color_0}{rgb}{ 0.5,0.0,0.5 }
\path [draw, thick, color = color_0, ->] (0,0) -- (1,-1) -- (2,0) -- (4,-2) -- (5,-1) -- (6,-2) -- (8,0);
\node at ( 8.5 , 0 ) {$B_n$};
\end{tikzpicture}

\end{figure}

The intertwining measure represents the missing information from the filtration generated by the $\left( X_n; n\geq 0 \right)$ only. A more representation-theoretic way of seeing the transition matrix $Q$ as:
$$Q(x, x+1) = \frac{1}{2}\frac{\dim V(x+1)}{\dim V(x)} \quad Q(x, x-1) = \frac{1}{2}\frac{\dim V(x-1)}{\dim V(x)}$$
where $V(x)$ is the representation with highest weight $x$ for $SL_2$ (more details in \cite{bib:Bia06}). Hence the idea that Pitman-type theorems should find their sources in the combinatorics of representation theory.

\section{Combinatorial representation theory}
This section plays the role of an informal introduction to the field of combinatorial representation theory, taking as a starting point the combinatorics of random walks and Young tableaux. Let $G$ be a complex semi-simple group. After adding the center, we will sometimes use $G=GL_n(\C)$ instead of the simple group $SL_n(\C)$.

\subsection{Combinatorial representation theory}

\begin{definition}
 A branch of mathematics that aims to extract information about irreductible representations of $G$ from combinatorial objects (diagrams, tableaux, paths)
\end{definition}

Let $P$ be weight lattice in $G$ and $P^+$ the dominant weights. In the case of $GL_{n}$:
$$ P = \left\{ \lambda \in \R^{n} \ | \lambda_i-\lambda_{i+1} \in \Z \right\}$$ 
$$ P^+ = \left\{ \lambda \in P \ | \ \lambda_1 \geq \dots \geq \lambda_{n} \right\}$$ 
A fundamental theorem tells us that irreductible representations are indexed by dominant weights. For every $\lambda \in P^+$, $V(\lambda)$ will denote the representation with highest weight $\lambda$, unique up to isomorphism.

A quantity for interest in combinatorial representation theory is the character $ch(V)$ of a representation $V$. It is a useful function on $\hfrak$ that packages efficiently weight multiplicities:
$$ ch(V)(x) = \sum_{\mu \in P} \dim V_\mu e^{\langle \mu, x \rangle}$$
For $GL_{n}$, the characters of irreductible representations are exactly the Schur functions.

\subsubsection{Weight diagrams}
Weight diagrams are a graphical presentation of characters, where larger multiplicities are represented by larger red dots (see figure \ref{fig:A2_WeightDiag} for the weight diagram of the $SL_3$ representation $V(\rho)$ ).

\begin{figure}[ht!]
\centering
\caption{Weight diagram for the representation in type $A_2$ with highest weight $\lambda = \rho$}
\label{fig:A2_WeightDiag}
\includegraphics{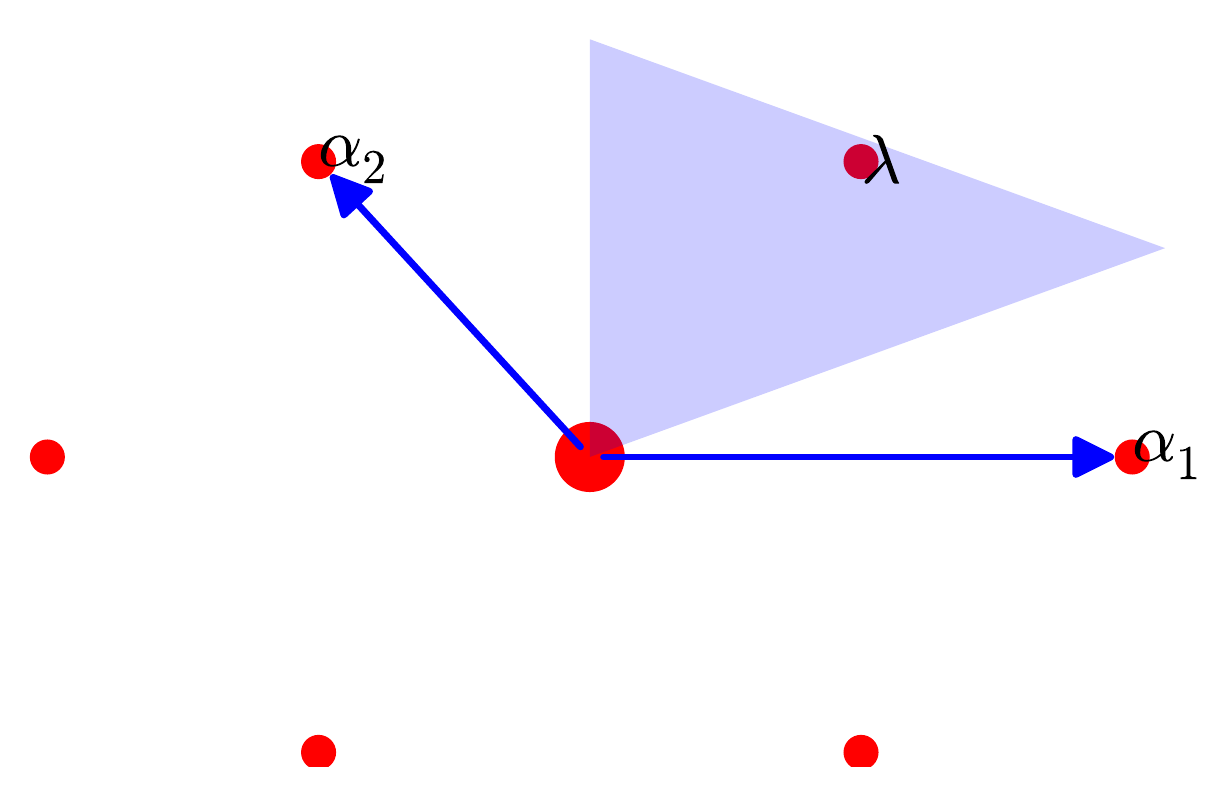}
\end{figure}

\subsubsection{Crystals of tableaux}
Here $G=GL_{n}(\C)$. Given a dominant weight $\lambda \in P^+$, an old theorem due to Littelwood relates semi-standard tableaux to the dimension of $V(\lambda)$.

\begin{thm}[Littlewood]
Let $\lambda = \{ \lambda_1 \geq \dots \geq \lambda_{n+1} \} \in P^+$. Semi-standard tableaux with $n+1$ letters and shape $\lambda$ give the dimension of $V(\lambda)$.
\end{thm}

In fact, there is way to arrange these semi-standard tableaux on the weight lattice $P$ and get the weight diagram (cf figures \ref{label:A1_CrystalOfTableaux} and \ref{label:A2_CrystalOfTableaux_Vrho}). Here, the weight of a semistandard tableau $P$ is defined as a vector $wt(P) \in \R^{n}$ with $wt(P)_k$ being the number of letters $k$ minus the number of letters $k+1$ in $P$. For instance, on figure \ref{label:A2_CrystalOfTableaux_Vrho}, we see that there are two tableaux of weight zero, giving a multiplicity of $2$.

\begin{figure}[ht!]
\centering
\caption{$A_1$-type crystal of tableaux for highest weight $\lambda = 2 \alpha$}
\label{label:A1_CrystalOfTableaux}

\begin{tikzpicture}[>=latex,line join=bevel,]
\node (1+1+1+2) at (24bp,238bp) [draw,draw=none] {${\def\lr#1{\multicolumn{1}{|@{\hspace{.6ex}}c@{\hspace{.6ex}}|}{\raisebox{-.3ex}{$#1$}}}\raisebox{-.6ex}{$\begin{array}[b]{cccc}\cline{1-1}\cline{2-2}\cline{3-3}\cline{4-4}\lr{1}&\lr{1}&\lr{1}&\lr{2}\\\cline{1-1}\cline{2-2}\cline{3-3}\cline{4-4}\end{array}$}}$};
  \node (1+2+2+2) at (24bp,86bp) [draw,draw=none] {${\def\lr#1{\multicolumn{1}{|@{\hspace{.6ex}}c@{\hspace{.6ex}}|}{\raisebox{-.3ex}{$#1$}}}\raisebox{-.6ex}{$\begin{array}[b]{cccc}\cline{1-1}\cline{2-2}\cline{3-3}\cline{4-4}\lr{1}&\lr{2}&\lr{2}&\lr{2}\\\cline{1-1}\cline{2-2}\cline{3-3}\cline{4-4}\end{array}$}}$};
  \node (2+2+2+2) at (24bp,10bp) [draw,draw=none] {${\def\lr#1{\multicolumn{1}{|@{\hspace{.6ex}}c@{\hspace{.6ex}}|}{\raisebox{-.3ex}{$#1$}}}\raisebox{-.6ex}{$\begin{array}[b]{cccc}\cline{1-1}\cline{2-2}\cline{3-3}\cline{4-4}\lr{2}&\lr{2}&\lr{2}&\lr{2}\\\cline{1-1}\cline{2-2}\cline{3-3}\cline{4-4}\end{array}$}}$};
  \node (1+1+1+1) at (24bp,314bp) [draw,draw=none] {${\def\lr#1{\multicolumn{1}{|@{\hspace{.6ex}}c@{\hspace{.6ex}}|}{\raisebox{-.3ex}{$#1$}}}\raisebox{-.6ex}{$\begin{array}[b]{cccc}\cline{1-1}\cline{2-2}\cline{3-3}\cline{4-4}\lr{1}&\lr{1}&\lr{1}&\lr{1}\\\cline{1-1}\cline{2-2}\cline{3-3}\cline{4-4}\end{array}$}}$};
  \node (1+1+2+2) at (24bp,162bp) [draw,draw=none] {${\def\lr#1{\multicolumn{1}{|@{\hspace{.6ex}}c@{\hspace{.6ex}}|}{\raisebox{-.3ex}{$#1$}}}\raisebox{-.6ex}{$\begin{array}[b]{cccc}\cline{1-1}\cline{2-2}\cline{3-3}\cline{4-4}\lr{1}&\lr{1}&\lr{2}&\lr{2}\\\cline{1-1}\cline{2-2}\cline{3-3}\cline{4-4}\end{array}$}}$};
  \draw [blue,->] (1+1+1+1) ..controls (24bp,292.79bp) and (24bp,273.03bp)  .. (1+1+1+2);
  \definecolor{strokecol}{rgb}{0.0,0.0,0.0};
  \pgfsetstrokecolor{strokecol}
  \draw (33bp,276bp) node {$1$};
  \draw [blue,->] (1+1+2+2) ..controls (24bp,140.79bp) and (24bp,121.03bp)  .. (1+2+2+2);
  \draw (33bp,124bp) node {$1$};
  \draw [blue,->] (1+1+1+2) ..controls (24bp,216.79bp) and (24bp,197.03bp)  .. (1+1+2+2);
  \draw (33bp,200bp) node {$1$};
  \draw [blue,->] (1+2+2+2) ..controls (24bp,64.789bp) and (24bp,45.027bp)  .. (2+2+2+2);
  \draw (33bp,48bp) node {$1$};
\end{tikzpicture}

\end{figure}
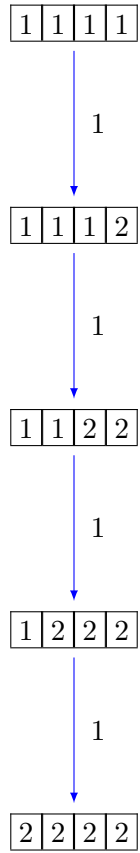
\begin{figure}[ht!]
\centering
\caption{$A_2$-type crystal of tableaux for highest weight $\lambda = \rho = \alpha_1 + \alpha_2$}
\label{label:A2_CrystalOfTableaux_Vrho}

\begin{tikzpicture}[auto, scale=2.0]
\begin{scope}[xshift=0, yshift=0]
\node (-2--1--1-) at (0.0,1.732) [draw,draw=none] 
{${\def\lr#1{\multicolumn{1}{|@{\hspace{.6ex}}c@{\hspace{.6ex}}|}{\raisebox{-.3ex}{$#1$}}}
\raisebox{-.6ex}{$\begin{array}[b]{cc}
\cline{1-1}\cline{2-2}
\lr{1}&\lr{1}\\
\cline{1-1}\cline{2-2}
\lr{2}\\
\cline{1-1}
\end{array}$}
}$};
\end{scope}
\begin{scope}[xshift=0, yshift=0]
\node (-2--1--2-) at (-1.5,0.866) [draw,draw=none]
{${\def\lr#1{\multicolumn{1}{|@{\hspace{.6ex}}c@{\hspace{.6ex}}|}{\raisebox{-.3ex}{$#1$}}}
\raisebox{-.6ex}{$\begin{array}[b]{cc}
\cline{1-1}\cline{2-2}
\lr{1}&\lr{2}\\
\cline{1-1}\cline{2-2}
\lr{2}\\
\cline{1-1}
\end{array}$}
}$};
\end{scope}
\begin{scope}[xshift=0, yshift=0]
\node (-2--1--3-) at (0.0,0.0) [draw,draw=none]
{${\def\lr#1{\multicolumn{1}{|@{\hspace{.6ex}}c@{\hspace{.6ex}}|}{\raisebox{-.3ex}{$#1$}}}
\raisebox{-.6ex}{$\begin{array}[b]{cc}
\cline{1-1}\cline{2-2}
\lr{1}&\lr{3}\\
\cline{1-1}\cline{2-2}
\lr{2}\\
\cline{1-1}
\end{array}$}
}$};
\end{scope}
\begin{scope}[xshift=0, yshift=0]
\node (-3--1--3-) at (1.5,-0.866) [draw,draw=none]
{${\def\lr#1{\multicolumn{1}{|@{\hspace{.6ex}}c@{\hspace{.6ex}}|}{\raisebox{-.3ex}{$#1$}}}
\raisebox{-.6ex}{$\begin{array}[b]{cc}
\cline{1-1}\cline{2-2}
\lr{1}&\lr{3}\\
\cline{1-1}\cline{2-2}
\lr{3}\\
\cline{1-1}
\end{array}$}
}$};
\end{scope}
\begin{scope}[xshift=0, yshift=0]
\node (-3--2--3-) at (0.0,-1.732) [draw,draw=none]
{${\def\lr#1{\multicolumn{1}{|@{\hspace{.6ex}}c@{\hspace{.6ex}}|}{\raisebox{-.3ex}{$#1$}}}
\raisebox{-.6ex}{$\begin{array}[b]{cc}
\cline{1-1}\cline{2-2}
\lr{2}&\lr{3}\\
\cline{1-1}\cline{2-2}
\lr{3}\\
\cline{1-1}
\end{array}$}
}$};
\end{scope}
\begin{scope}[xshift=0, yshift=0]
\node (-3--1--1-) at (1.5,0.866) [draw,draw=none]
{${\def\lr#1{\multicolumn{1}{|@{\hspace{.6ex}}c@{\hspace{.6ex}}|}{\raisebox{-.3ex}{$#1$}}}
\raisebox{-.6ex}{$\begin{array}[b]{cc}
\cline{1-1}\cline{2-2}
\lr{1}&\lr{1}\\
\cline{1-1}\cline{2-2}
\lr{3}\\
\cline{1-1}
\end{array}$}
}$};
\end{scope}
\begin{scope}[xshift=10, yshift=-10]
\node (-3--1--2-) at (0.0,0.0) [draw,draw=none]
{${\def\lr#1{\multicolumn{1}{|@{\hspace{.6ex}}c@{\hspace{.6ex}}|}{\raisebox{-.3ex}{$#1$}}}
\raisebox{-.6ex}{$\begin{array}[b]{cc}
\cline{1-1}\cline{2-2}
\lr{1}&\lr{2}\\
\cline{1-1}\cline{2-2}
\lr{3}\\
\cline{1-1}
\end{array}$}
}$};
\end{scope}
\begin{scope}[xshift=0, yshift=0]
\node (-3--2--2-) at (-1.5,-0.866) [draw,draw=none] {${\def\lr#1{\multicolumn{1}{|@{\hspace{.6ex}}c@{\hspace{.6ex}}|}{\raisebox{-.3ex}{$#1$}}}\raisebox{-.6ex}{$\begin{array}[b]{cc}\cline{1-1}\cline{2-2}\lr{2}&\lr{2}\\
\cline{1-1}\cline{2-2}\lr{3}\\\cline{1-1}\end{array}$}}$};
\end{scope}
\draw [blue,->] (-2--1--1-) to node {$1$} (-2--1--2-);
\draw [red,->] (-2--1--1-) to node {$2$} (-3--1--1-);
\draw [red,->] (-2--1--2-) to node {$2$} (-2--1--3-);
\draw [red,->] (-2--1--3-) to node {$2$} (-3--1--3-);
\draw [blue,->] (-3--1--1-) to node {$1$} (-3--1--2-);
\draw [blue,->] (-3--1--2-) to node {$1$} (-3--2--2-);
\draw [blue,->] (-3--1--3-) to node {$1$} (-3--2--3-);
\draw [red,->] (-3--2--2-) to node {$2$} (-3--2--3-);
\end{tikzpicture}

\end{figure}
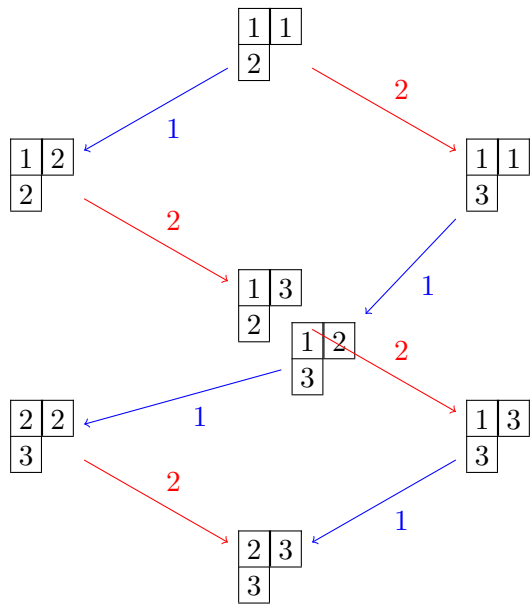 

\subsubsection{Kashiwara crystals and parametrizations}
Kashiwara constructed operators that $\left( f_1, f_2, \dots, f_{n}\right)$ that give all tableaux from the highest weight one, given a certain shape. Their action is represented by the letters $1, \dots, n$ and give a graph called the crystal graph (see figures \ref{label:A1_CrystalOfTableaux} and \ref{label:A2_CrystalOfTableaux_Vrho}).

Then the idea is to get rid of tableaux and use only coordinates (see fig. \ref{label:A2_CrystalGraph_Vrho}) called string coordinates. For every ${\bf i} = (i_1, i_2, \dots, i_{\ell(w_0)})$ a reduced expressions of the longuest element $w_0 = \left(n \dots 3 2 1 \right)$, one can associate a coordinate chart $[x_{-\bf i}]$. Reduced expressions are factorizations into a minimal number of simple reflections. For $SL_{n+1}$, the length of $w_0$ is $\ell(w_0) = \frac{n(n+1)}{2}$

\begin{example}
 \begin{itemize}
  \item $n=1$: $W = S_2$ has the unique reflection as longest element.
  \item $n=2$: $W = S_3$ $w_0 = \left( 3 2 1 \right) = (12)(23)(12) = (23)(12)(23)$. 
 \end{itemize}
\end{example}

The string coordinates $[x_{-\bf i}](b)$ of an element $b$ in the crystal graph is given by specifying $\ell(w_0)$ integers. They correspond to the number of steps necessary to climb the crystal graph and reach the highest weight element, by going all the way successively along the directions in ${\bf i}$. For instance, in the crystal graph of figure \ref{label:A2_CrystalOfTableaux_Vrho}, pick ${\bf i} = (1,2,1)$. Then:
$$ [x_{-\bf i}]\left( {\def\lr#1{\multicolumn{1}{|@{\hspace{.6ex}}c@{\hspace{.6ex}}|}{\raisebox{-.3ex}{$#1$}}}
\raisebox{-.6ex}{$\begin{array}[b]{cc}
\cline{1-1}\cline{2-2}
\lr{1}&\lr{1}\\
\cline{1-1}\cline{2-2}
\lr{2}\\
\cline{1-1}
\end{array}$}
} \right)
= \left( 0, 0, 0 \right)$$

$$ [x_{-\bf i}]\left( {\def\lr#1{\multicolumn{1}{|@{\hspace{.6ex}}c@{\hspace{.6ex}}|}{\raisebox{-.3ex}{$#1$}}}\raisebox{-.6ex}{$\begin{array}[b]{cc}\cline{1-1}\cline{2-2}\lr{2}&\lr{2}\\
\cline{1-1}\cline{2-2}\lr{3}\\\cline{1-1}\end{array}$}} \right)
= \left( 2, 1, 0\right)$$

$$ [x_{-\bf i}]\left( {\def\lr#1{\multicolumn{1}{|@{\hspace{.6ex}}c@{\hspace{.6ex}}|}{\raisebox{-.3ex}{$#1$}}}
\raisebox{-.6ex}{$\begin{array}[b]{cc}
\cline{1-1}\cline{2-2}
\lr{2}&\lr{3}\\
\cline{1-1}\cline{2-2}
\lr{3}\\
\cline{1-1}
\end{array}$}
} \right)
= \left( 1, 2, 1\right)$$

In the end, one can discard the tableaux and keep only the coordinates as in figure \ref{label:A2_CrystalGraph_Vrho}. This gives a realization of Kashiwara crystals and such a construction can be generalized to other types in the Cartan-Killing classification.
\begin{figure}[ht!]
\centering
\caption{Crystal graph for a crystal of type $A_2$ with highest weight $\lambda = \rho$}
\label{label:A2_CrystalGraph_Vrho}
\includegraphics{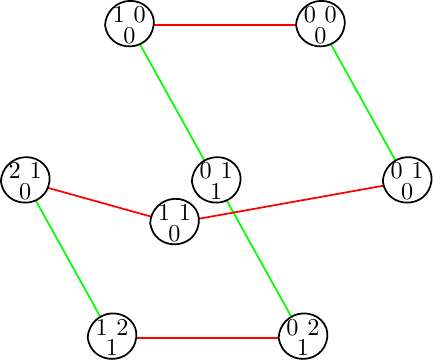}
\end{figure}

\subsubsection{Path crystals: Littelmann's path model}
Littelmann replaced tableaux by piece-wise linear paths in $\afrak$, the real Cartan subalgebra, changing directions at rational points. The weight of a path is simply its endpoint. Paths are combinatorial objects very appreciated by probabilists, considering them as random walks. Acting on a path $\pi$ indexed by the interval $[0,T]$, the Kashiwara operators have the expression:
\begin{align*}
 \forall 0 \leq t \leq T, & e^n_\alpha \pi(t) = \pi(t) + \inf_{0 \leq s \leq T} \alpha^\vee (\pi(s)) \alpha - \\
                          & \ \min\left( \inf_{0 \leq s \leq t} \alpha^\vee (\pi(s)) - n, \inf_{t \leq s \leq T} \alpha^\vee (\pi(s))\right) \alpha
\end{align*}
for an integer $n \in \Z$ such that:
$$ -\alpha^\vee (\pi(T)) + \inf_{0 \leq s \leq T} \alpha^\vee (\pi(s)) \leq n \leq -\inf_{0 \leq s \leq T} \alpha^\vee (\pi(s))$$

A theorem known as Littelmann's independence theorem, states that the crystal structure depends only on the dominant path's weight. Such fact can be observed by comparing figures \ref{label:A1_PathCrystal_LS} and \ref{label:A1_PathCrystal_Long}, which represent two isomorphic path crystals with different dominant paths.

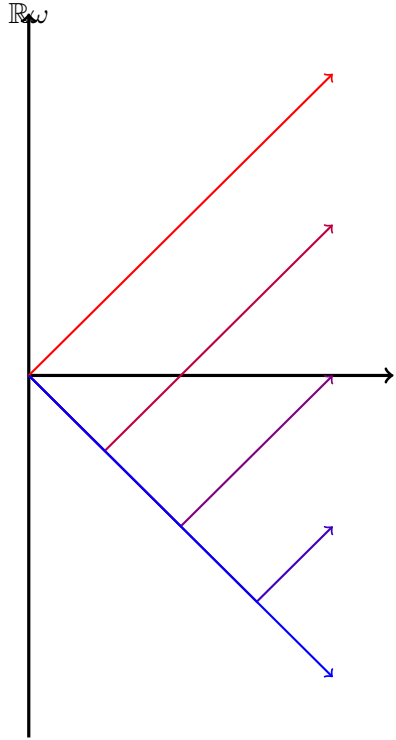
\begin{figure}[ht!]
\centering
\caption{$A_1$-type path crystal of tableaux for highest weight $\lambda = 2 \alpha$}
\label{label:A1_PathCrystal_LS}

\begin{tikzpicture}[auto, bend right, scale=1]
\path [draw, very thick, color = black, ->] (0,0) -- (
4.80000000000000 , 0 );
\node at ( 0, 4.80000000000000 ) {$\mathbb{R} \omega$};
\path [draw, very thick, color = black, ->] ( 0, -4.80000000000000 )
-- ( 0, 4.80000000000000 );
\definecolor{color_0}{rgb}{ 0.75,0.0,0.25 }
\path [draw, thick, color = color_0, ->] (0,0) -- (1,-1) -- (4,2);
\definecolor{color_1}{rgb}{ 0.25,0.0,0.75 }
\path [draw, thick, color = color_1, ->] (0,0) -- (3,-3) -- (4,-2);
\definecolor{color_2}{rgb}{ 0.5,0.0,0.5 }
\path [draw, thick, color = color_2, ->] (0,0) -- (2,-2) -- (4,0);
\definecolor{color_3}{rgb}{ 1.0,0.0,0.0 }
\path [draw, thick, color = color_3, ->] (0,0) -- (4,4);
\definecolor{color_4}{rgb}{ 0.0,0.0,1.0 }
\path [draw, thick, color = color_4, ->] (0,0) -- (4,-4);
\end{tikzpicture}

\end{figure}
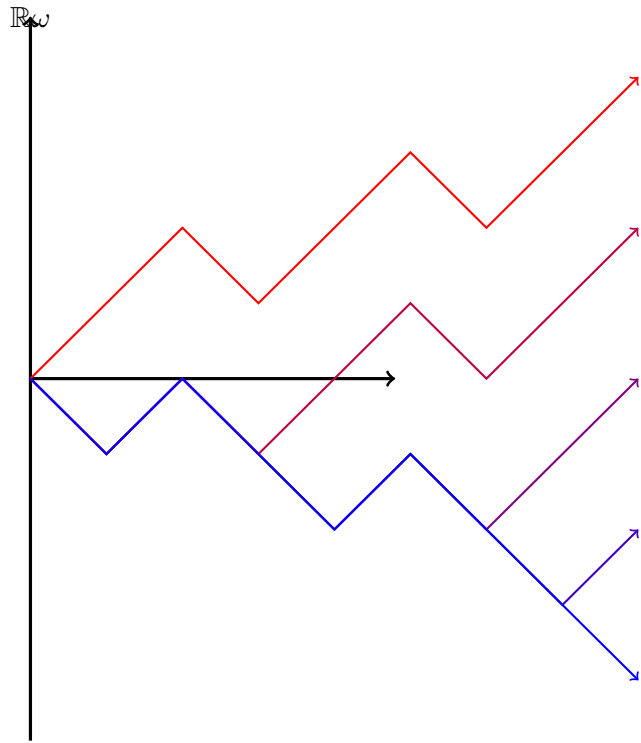
\begin{figure}[ht!]
\centering
\caption{$A_1$-type path crystal of tableaux for highest weight $\lambda = 2 \alpha$}
\label{label:A1_PathCrystal_Long}

\begin{tikzpicture}[auto, bend right, scale=1]
\path [draw, very thick, color = black, ->] (0,0) -- (
4.80000000000000 , 0 );
\node at ( 0, 4.80000000000000 ) {$\mathbb{R} \omega$};
\path [draw, very thick, color = black, ->] ( 0, -4.80000000000000 )
-- ( 0, 4.80000000000000 );
\definecolor{color_0}{rgb}{ 0.5,0.0,0.5 }
\path [draw, thick, color = color_0, ->] (0,0) -- (1,-1) -- (2,0) --
(4,-2) -- (5,-1) -- (6,-2) -- (8,0);
\definecolor{color_1}{rgb}{ 1.0,0.0,0.0 }
\path [draw, thick, color = color_1, ->] (0,0) -- (2,2) -- (3,1) --
(5,3) -- (6,2) -- (8,4);
\definecolor{color_2}{rgb}{ 0.25,0.0,0.75 }
\path [draw, thick, color = color_2, ->] (0,0) -- (1,-1) -- (2,0) --
(4,-2) -- (5,-1) -- (7,-3) -- (8,-2);
\definecolor{color_3}{rgb}{ 0.75,0.0,0.25 }
\path [draw, thick, color = color_3, ->] (0,0) -- (1,-1) -- (2,0) --
(3,-1) -- (5,1) -- (6,0) -- (8,2);
\definecolor{color_4}{rgb}{ 0.0,0.0,1.0 }
\path [draw, thick, color = color_4, ->] (0,0) -- (1,-1) -- (2,0) --
(4,-2) -- (5,-1) -- (8,-4);
\end{tikzpicture}

\end{figure}

For the $A_2$ type, we have produced the path crystal for $V(\lambda=\rho)$ in figure \ref{label:A2_PathCrystal_Vrho}, where one can observe six extremal paths and two paths ending at zero. This is exactly the same configuration as in the crystal of tableaux in figure \ref{label:A2_CrystalOfTableaux_Vrho}.
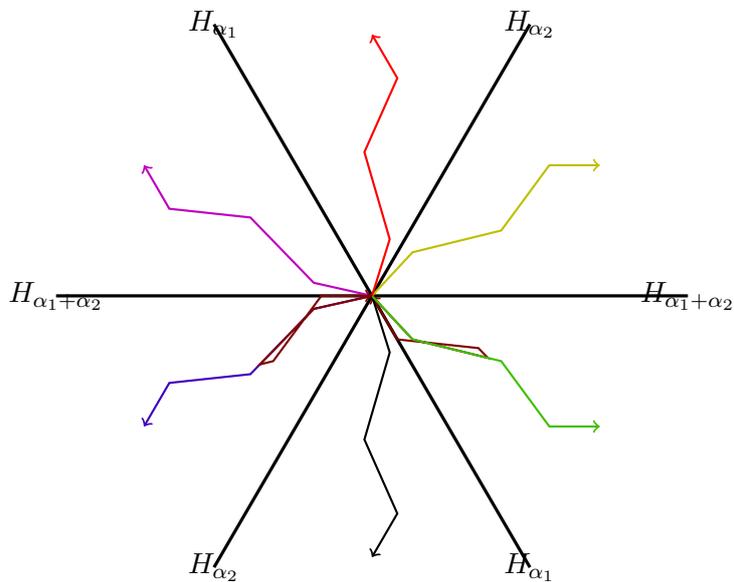
\begin{figure}[ht!]
\centering
\caption{$A_2$-type path crystal for highest weight $\lambda = \rho = \alpha_1 + \alpha_2$}
\label{label:A2_PathCrystal_Vrho}

\begin{tikzpicture}[auto, bend right, scale=2]
\node at(-1.0392,-1.8){$H_{\alpha_2}$};
\path [draw, very thick, color = black] (0,0) -- (-1.0392,-1.8);
\node at(-1.0392,1.8){$H_{\alpha_1}$};
\path [draw, very thick, color = black] (0,0) -- (-1.0392,1.8);
\node at(-2.0784,0.0){$H_{\alpha_1 + \alpha_2}$};
\path [draw, very thick, color = black] (0,0) -- (-2.0784,0.0);
\node at(1.0392,1.8){$H_{\alpha_2}$};
\path [draw, very thick, color = black] (0,0) -- (1.0392,1.8);
\node at(1.0392,-1.8){$H_{\alpha_1}$};
\path [draw, very thick, color = black] (0,0) -- (1.0392,-1.8);
\node at(2.0784,0.0){$H_{\alpha_1 + \alpha_2}$};
\path [draw, very thick, color = black] (0,0) -- (2.0784,0.0);
\definecolor{color_0}{rgb}{ 0.5,0.0,0.0 }
\path [draw, thick, color = color_0, ->] (0,0) -- (0.2667,-0.2887) --
(0.7625,-0.4114) -- (0.7,-0.3464) -- (0.1667,-0.2887) -- (0.0,0.0);
\definecolor{color_1}{rgb}{ 0.75,-0.75,0.75 }
\path [draw, thick, color = color_1, ->] (0,0) -- (-0.3833,0.0866) --
(-0.8,0.5196) -- (-1.3333,0.5773) -- (-1.5,0.866);
\definecolor{color_2}{rgb}{ 0.0,0.0,0.0 }
\path [draw, thick, color = color_2, ->] (0,0) -- (0.1167,-0.3753) --
(-0.05,-0.9526) -- (0.1667,-1.4433) -- (0.0,-1.732);
\definecolor{color_3}{rgb}{ 0.75,0.75,-0.75 }
\path [draw, thick, color = color_3, ->] (0,0) -- (0.2667,0.2887) --
(0.85,0.433) -- (1.1667,0.866) -- (1.5,0.866);
\definecolor{color_4}{rgb}{ 0.25,0.75,-0.75 }
\path [draw, thick, color = color_4, ->] (0,0) -- (0.2667,-0.2887) --
(0.85,-0.433) -- (1.1667,-0.866) -- (1.5,-0.866);
\definecolor{color_5}{rgb}{ 1.0,0.0,0.0 }
\path [draw, thick, color = color_5, ->] (0,0) -- (0.1167,0.3753) --
(-0.05,0.9526) -- (0.1667,1.4433) -- (0.0,1.732);
\definecolor{color_6}{rgb}{ 0.25,-0.75,0.75 }
\path [draw, thick, color = color_6, ->] (0,0) -- (-0.3833,-0.0866)
-- (-0.8,-0.5196) -- (-1.3333,-0.5773) -- (-1.5,-0.866);
\definecolor{color_7}{rgb}{ 0.5,0.0,0.0 }
\path [draw, thick, color = color_7, ->] (0,0) -- (-0.3833,-0.0866)
-- (-0.7375,-0.4547) -- (-0.65,-0.433) -- (-0.3333,0.0) -- (0.0,0.0);
\end{tikzpicture}

\end{figure}

In fact, the Littelmann path model goes beyond and can be generalized to all Lie groups in the Cartan-Killing classification (see fig.\ref{label:G2_PathCrystal_Vrho}).
\input{G2_PathCrystal_Vrho.tex}

\subsubsection{``Asymptotic'' representation theory}
Renormalisation is particularly easy to think about for paths. As such, building a continuous model that is the limit of Littelmann's is quite natural. This is the work of Biane and al. in \cite{bib:BBO} and \cite{bib:BBO2}. For example in \cite{bib:BBO} remark 5.8, one can see how asymptotic weight multiplicities appear by considering ``very long'' Littelmann paths.

\subsection{The philosophy of path models}
In general, we mean by path model, a device where one ``enumerates'' paths on a Euclidian space in order to extract representation-theoretic informations about a Lie group. Of course, ``counting'' paths in a continuous setting means computing the probability of certain events, under a certain canonical measure. Certain properties are expected:
\begin{itemize}
 \item The notion of weight is given by taking a path's endpoint.
 \item Irreducible representations should be in correspondence with connected path crystals.
 \item Tensor product should correspond to the concatenation of paths.
 \item Weight multiplicities will be given by the image measure of a natural measure through the weight map, hence a character formula.
 \item A Littlewood-Richardson rule.
\end{itemize}

\section{Pitman-type theorems in probability}

Pitman's theorem has known a remarkably long list of extensions, to the point it seems hard to draw a complete picture. For now, I can distinguish between several directions in generalizing the result.
\begin{itemize}
 \item The first direction is 'geometric lifting'. The usual Pitman transform $\Pc$ is the zero temperature limit of a smoother path transform $\Tc^q$, where $q$ can be seen as the temperature parameter:
$$ \Tc^q\left( \pi \right)(t) = \pi(t) + q \log( \int_0^t e^{ -\frac{2 \pi(s)}{q} }ds ) 
\stackrel{ q \rightarrow 0 }{\rightarrow} \Pc\left( \pi \right)(t)$$
It is worth mentioning that Pitman's path transform also has applications in physics, and more precisely for modelling random polymers in Brownian environnement ( O'Connell, Moriarty \cite{bib:MOC07}, O'Connell \cite{bib:OConnell} ). In such a case, the path $\pi$ is Brownian and the weight $e^{-\frac{2\pi(.)}{q}}$ is a Boltzmann weight.

This has strong links with tropical geometry a very recent field in mathematics. From a physical point of view, 'geometric lifting' means considering higher temperatures. Following those terms, one could say that the geometric lifting of Pitman's theorem was proven by Matsumoto and Yor in \cite{bib:MY00-1, bib:MY00-2}, in 2000. 

 \item A second direction is looking for non-trivial multidimensional extensions to Pitman's theorem. Here, the approach initiated by Biane, Bougerol and O'Connell (\cite{bib:BJ, bib:BBO, bib:BBO2}) is based on the intuition that such structure in probability can appear only from the world of rigid structures, Algebra and more specifically group theory.

In \cite{bib:BBO} were introduced Pitman path transforms associated to simple roots $\alpha$:
$$ \Pc_{\alpha^\vee}(\pi)(t) = \pi(t) - \inf_{0 \leq s \leq t}\alpha\left( \pi(s) \right) \alpha^\vee $$
There, it is also proven that these path transforms satisfy the braid relationships, giving thus a meaning to $\mathcal{P}_w$ for $w$ an element of a Weyl group. The transform $\mathcal{P}_{w_0}$ associated to the longest element $w_0$ plays a special role in the continuous Littelmann model, thus in representation theory, but also in probability: this transform folds Brownian motion in the Weyl chamber in such a way that it remains Markovian. The process thus obtained is a Brownian motion conditioned to stay inside that Weyl chamber. The construction goes beyond the cristallographic case and is valid for Coxeter groups.
\end{itemize}

For a global view we produced table \ref{tab:thesis_perspective} that puts this work in perspective with previous mathematical theory. In all cases, a miraculous Markov property appears, as well as a remarkable harmonic function and an intertwining measure. Finding such structures was our Ariadne's thread.

\begin{table}[ht!]
\begin{center}
  \caption{Perspectives}
  \label{tab:thesis_perspective}
  \begin{tabular}{| p{1.5cm} | p{6cm} | p{7cm} |}
  \hline
                   &  { \centering One dimensional setting / $SL_2$ }
                   &  { \centering Multidimensional setting / $G$ semi-simple group } \\
  \hline
  { \centering Crystal versions} &

Pitman and Rogers \cite{bib:RogersPitman}: If $B^{(\mu)}$ is a Brownian motion with drift $\mu$ then $B^{(\mu)}_t - 2 \inf_{ 0 \leq s \leq t} B^{(\mu)}_s$ is a Markov process with infinitesimal generator
$$ \frac{1}{2} \frac{d^2}{dx^2} + \frac{d}{dx} \log\left( h_\mu \right) \frac{d}{dx}$$

The function $h_\mu(x) = \frac{ \sinh( \mu x ) }{ \mu }$ satisfies the eigenfunction equation on $\mathbb{R}^+$, with Dirichlet boundary conditions:
$$ \frac{1}{2} \frac{d^2}{dx^2} h_\mu = \frac{ \mu^2 }{2} h_\mu$$

Intertwining measure: Uniform measure on $[0, r]$ when $\mu=0$.
                    &  

Biane, Bougerol, O'Connell \cite{bib:BBO}: In the context of the continuous Littelmann path model, consider the highest path transform $\Pc_{w_0}$. For $W$ a Brownian motion in the Cartan subalgebra, $\Pc_{w_0} W$ is a Brownian motion conditionned in the sense of Doob to stay in the Weyl chamber. It is a Markov process with infinitesimal generator:
$$ \frac{1}{2} \Delta + \nabla \log\left(h(x)\right) \cdot \nabla $$

The function $h(x)$ is (the unique up to a multiplicative constant) positive harmonic on the Weyl chamber with Dirichlet boundary conditions.

Intertwining measure: Continuous Duistermaat-Heckmann measure. \\
  \hline
  { \centering Geometric versions} &

Matsumoto and Yor \cite{bib:MY00-1, bib:MY00-2}: If $B^{(\mu)}$ is a Brownian motion with drift $\mu$ then $B^{(\mu)}_t + \log\left( \int_0^t e^{-2 B^{(\mu)}_s } ds \right)$ is a Markov process with infinitesimal generator 
$$ \frac{1}{2} \frac{d^2}{dx^2} + \frac{d}{dx} \log\left( K_\mu \right) \frac{d}{dx}$$

The function $K_\mu\left( x \right)$ satisfies the eigenfunction equation:
$$ \frac{1}{2} \frac{d^2}{dx^2} K_\mu - e^{-2x } K_\mu= \frac{ \mu^2 }{2} K_\mu$$

Intertwining measure: Generalized Inverse Gaussian law.
                    &
There is a highest path transform $\Tc_{w_0}$ in the context of a path model for geometric crystals. For $W^{(\mu)}$ a Brownian motion with drift $\mu$ in the Cartan subalgebra, $\Tc_{w_0} W^{(\mu)}$ is a Markov process with infinitesimal generator:
$$ \frac{1}{2} \Delta + \nabla \log\left(\psi_\mu\right) \cdot \nabla $$

The function $\psi_\mu(x)$ solves the quantum Toda eigenfunction equation:
$$ \half \Delta \psi_\mu - \sum_{\alpha \in \Delta } \frac{\langle \alpha, \alpha \rangle}{2} e^{ -\alpha(x) } \psi_\mu = \frac{\langle \mu, \mu \rangle}{2} \psi_\mu$$

Intertwining measure: A geometric Duistermaat-Heckmann measure. \\
  \hline
  \end{tabular}
\end{center} 
\end{table}
\chapter{Preliminaries}

\section{On Lie groups}
\index{$G$: Simply-connected complex semi-simple group}
Let $G$ be a simply-connected complex semi-simple group. The framework can be easily extented to the reductive case, since adding the center never causes much trouble. Such groups are just groups of complex matrices:

\begin{thm}[\cite{bib:Borel06}]
$G$ is an algebraic linear group and up to conjugation self-adjoint.
\end{thm}

Therefore, in our treatment, $G$ is a closed subgroup in $GL_p(\C)$, for $p \in \N$ large enough, and stable under the usual adjoint.

The Cartan subgroup $H$ is an abelian subgroup of maximal dimension. $H \approx \left(\C^*\right)^n$ and the integer $n$ is called the rank of $G$. $H$ is usually taken to be made of diagonal matrices.
The tangent space at the identity element is $\gfrak = T_e G \subset M_N(\C)$. It can be characterized using the usual matrix exponential as:
\index{$\gfrak$: Lie algebra of $G$}
$$ \gfrak = \{ x \in M_p(\C) \ | \ \forall t \in \R, \exp(tx) \in G\}$$
$\gfrak$ is endowed with the structure of a complex semi-simple Lie algebra $(\gfrak, [,])$. The Lie bracket $[,]$ is given by the commutator:
$$ \forall x, y \in \gfrak, [x, y] = xy - yx$$

\section{On Lie algebras}
As a standard reference for the structure of Lie algebras and their representation theory, we recommend \cite{bib:Humphreys72}.

\subsection*{Structure}
\index{$\hfrak$: Cartan subalgebra of $\gfrak$}
The Cartan subalgebra  $\hfrak = T_e H$ is a maximal abelian subalgebra in $\gfrak$.

The adjoint action valued in the space of endomorphisms of $\gfrak$, $ad: \gfrak \longrightarrow End(\gfrak)$ is defined as $ad(x)(y) = [x,y]$ for $x, y \in \gfrak$. It is the differential of the action $Ad: G \longrightarrow Aut(\gfrak)$:
$$\forall g \in G, Ad(g)(x) = g x g^{-1}$$
There is a symmetric bilinear form on $\gfrak$ called the Killing form:
$$K(x,y) := Tr\left( ad(x)\circ ad(y) \right)$$
By Cartan's criterion, since $\gfrak$ is semi-simple, the Killing form is non-degenerate. $K$ can then be transported to the dual $\gfrak^*$. Its restriction to $\hfrak$ (resp. $\hfrak^*)$ is in fact a scalar product written $\langle\cdot, \cdot \rangle$.

\index{$\Phi$: Root system}
\index{$\Phi^+$, $\Phi^-$: Positive and negative roots}
\index{$\Delta$: Simple roots}

A fundamental idea in the classification of these Lie groups is that the adjoint action is codiagonalizable, once restricted to $\hfrak$. For each $h \in \hfrak$, $ad(h) \in End(\gfrak)$ has eigenvalues $\alpha(h)$, and the dependency in $h$ is obviously linear. This gives rise to a family of linear forms $\{0\} \sqcup \Phi \subset \hfrak^*$ such that for all $h$, $ad(h)$ has eigenvalues $\{0\} \sqcup \left(\alpha(h)\right)_{\alpha \in \Phi}$. Hence a root-space decomposition:
$$\gfrak = \bigoplus_{\alpha \in \{0\} \sqcup \Phi} \gfrak_{\alpha} = \gfrak_0 \oplus \bigoplus_{\alpha \in \Phi^+} \left( \gfrak_{\alpha} \oplus \gfrak_{-\alpha} \right)$$
where for every $\alpha \in \{0\} \sqcup \Phi$:
$$\gfrak_\alpha = \left\{ x \in \gfrak \ | \ \forall h \in \hfrak, ad(h)(x) = \alpha(h)x \right\}$$
Here, $\gfrak_0 = \hfrak$ and $\Phi \subset \hfrak^*$ is called the set of roots. It is the disjoint union of $\Phi^+$, roots chosen to be called 'positive', and $\Phi^- := -\Phi^+$ the set of 'negative' roots. This choice uniquely determines $\Delta = (\alpha_i)_{i \in I} \subset \Phi^+$ a simple system spanning $\hfrak^*$ such that every positive root is a sum with positive integer coefficients of simple roots (\cite{bib:Humphreys90}) - and reciprocally, a simple system uniquely determines a positive system. Moreover, the simple system $\Delta$ forms a basis of $\hfrak^*$.

\index{$\afrak$: Real part of the Cartan subalgebra}
The Cartan subalgebra has a decomposition $\hfrak = \afrak + i \afrak$ with $\afrak$ chosen to be the real subspace of $\hfrak$ where roots are real valued.

For every root $\beta \in \Phi$, the coroot $\beta^\vee$ is the unique vector in $\hfrak$ such that 
$$\forall h \in \hfrak, K\left( \beta^\vee, h\right) = \frac{ 2 \beta(h)}{\langle \beta, \beta \rangle}$$
When identifying $\hfrak$ and $\hfrak^*$ thanks to the Killing form, it is customary to simply write $\beta^\vee = \frac{2 \beta}{\langle \beta, \beta \rangle}$.

The Cartan matrix $A = (a_{ij})_{1 \leq i,j \leq n} \in M_n(\Z)$ is the matrix with coefficients $a_{i,j} = \alpha_j\left( \alpha_i^\vee \right)$. It allows a complete classification of complex semi-simple algebras, and therefore of complex semi-simple groups. Dynkin diagrams are a convenient graphical way of representing Cartan matrices.

\subsection*{Classification}
A matrix $A \in M_n(\Z)$ is said to be symmetrizable if there are matrices $B$ and $D$ with $B$ symmetric and $D$ positive diagonal such that:
$$ A = D B$$
Any Cartan matrix of a semi-simple Lie algebra is symmetrizable, with the symmetric $B$ being positive definite. Moreover, diagonal elements are equal to $2$ while the others are $\leq 0$. And reciprocally, for a matrix $A$ satisfying those properties, a complex semi-simple Lie algebra with Cartan matrix $A$ can be constructed (theorem 2.111 in \cite{bib:Knapp02}). The simple Lie algebras are known, and are classified by types:
\begin{itemize}
 \item Type $A_n$: $\gfrak = \mathfrak{sl}_n$ $G = SL_n(\C)$
 \item Type $B_n$: $\gfrak = \mathfrak{so}_{2n+1}$ $G = SO_{2n+1}(\C)$
 \item Type $C_n$: $\gfrak = \mathfrak{sp}_{n}$ $G = Sp_{n}(\C)$
 \item Type $D_n$: $\gfrak = \mathfrak{so}_{2n}$ $G = SO_{2n}(\C)$
 \item Exceptional types: $E_6$, $E_7$, $E_8$, $F_4$, $G_2$.
\end{itemize}

\subsection*{Langlands dual}
\index{$G^\vee$: Langlands dual of $G$}
Since the transpose of Cartan matrix $A$ is still a Cartan matrix, there is a simply-connected complex semi-simple group $G^\vee$ whose Cartan matrix is the transpose of $A$. It is called the Langlands dual of $G$ or for short the L-group. Types $B$ and $C$ are dual to each other. The other types are self-dual ($ADE$, $F_4$, $G_2$).

Clearly, the associated simple root system $\Delta^\vee$ is formed by the simple coroots. $\Phi^\vee$ denotes the dual root system.

\subsection*{Simply-laced}
The $ADE$ types are called simply laced because their Dynkin diagrams have only single edges.

\subsection*{Splitting}
For each positive root $\alpha \in \Phi^+$, we can choose an $\mathfrak{sl}_2$-triplet $(e_\alpha, f_\alpha, h_\alpha) \in \gfrak_\alpha \times \gfrak_{-\alpha} \times \hfrak$ such that $[e_\alpha, f_\alpha] = h_\alpha$ and $h_\alpha = \alpha^\vee$. $(e_\alpha, f_\alpha, h_\alpha)_{\alpha \in \Delta}$ will be the set of simple $\mathfrak{sl}_2$-triplets.

The name '$\mathfrak{sl}_2$-triplet' becomes clear after exhibiting the Lie algebra homomorphisms $\phi_\alpha: \mathfrak{sl}_2 \longrightarrow \gfrak$ such that:
$$\left\{ \begin{array}{ll}
\phi_\alpha\left( x \right) = e_\alpha\\
\phi_\alpha\left( y \right) = f_\alpha\\
\phi_\alpha\left( h \right) = h_\alpha\\
\end{array} \right.$$
where $x = \left( \begin{array}{cc} 0 & 1 \\ 0 &  0 \end{array} \right) $, $y = \left( \begin{array}{cc} 0 & 0 \\ 1 &  0 \end{array} \right)$,
$ h = \left( \begin{array}{cc} 1 & 0 \\ 0 & -1 \end{array} \right) $.
The exponential map $e: \gfrak \rightarrow G$ lifts these homomorphisms from the Lie algebra $\gfrak$ to the group $G$:
 each $\phi_\alpha$ gives rise at the group level to a Lie group homomorphism that embed $SL_2$ in $G$ and that will be denoted in the same way. The following notations are common for $t \in \C$:
\begin{align*}
 t^{h_\alpha} & = e^{ \log(t) h_\alpha} = \phi_\alpha\left(  \left( \begin{array}{cc} t & 0 \\ 0 & t^{-1} \end{array}\right) \right), t \neq 0 \\
 x_\alpha(t) & = e^{ t e_\alpha} = \phi_\alpha\left(  \left( \begin{array}{cc} 1 & t \\ 0 & 1 \end{array}\right) \right)\\
 y_\alpha(t) & = e^{ t f_\alpha} = \phi_\alpha\left(  \left( \begin{array}{cc} 1 & 0 \\ t & 0 \end{array}\right) \right)\\
 x_{-\alpha}(t) & = y_\alpha(t) t^{-h_\alpha} = \phi_\alpha\left(  \left( \begin{array}{cc} t^{-1} & 0 \\ 1 & t \end{array}\right) \right)\\
 y_{-\alpha}(t) & = t^{-h_\alpha} x_\alpha(t) = \phi_\alpha\left(  \left( \begin{array}{cc} t^{-1} & 1 \\ 0 & t \end{array}\right) \right)
\end{align*}
\begin{rmk}
 $\log t$ is any determination of the natural logarithm on a simply connected domain. Since they differ by a multiple of $2i\pi$ and $h_\alpha$ is a coroot, the determination does not matter.
\end{rmk}

\subsection*{Subalgebras of $\gfrak$ and subgroups of G}

$$\gfrak = \mathfrak{n} \oplus \hfrak \oplus \mathfrak{u}$$
is a triangular decomposition where
\begin{itemize}
 \item $\hfrak$ is the Cartan subalgebra (unique up to conjugacy \cite{bib:Humphreys72})\\
 \item $\mathfrak{u}$ ( resp. $\mathfrak{n}$) is the algebra generated by the $(e_\alpha)_{\alpha \in \Phi^+}$ (resp. $(f_\alpha)_{\alpha \in \Phi^+}$). These generators are referred to as the Chevalley generators.\\
\end{itemize}
$$\mathfrak{b^+} = \hfrak \oplus \mathfrak{u} $$
$$\mathfrak{b} = \mathfrak{n} \oplus \hfrak $$
\linebreak
The corresponding subgroups are:
\begin{itemize}
 \item H a maximal torus with Lie algebra $\hfrak$\\
 \item $N$ lower (resp. $U$ upper) unipotent subgroup with Lie algebra $\nfrak$ (resp. $\mathfrak{u}$)\\
 \item $B$ lower (resp. $B^+$ upper) Borel subgroup with Lie algebra $\bfrak$ (resp. $\bfrak^+$)\\
\end{itemize}

The reader unfamiliar with such objects can think of type $A_n$, where $B$ (resp. $B^+$) is the lower (resp. upper) triangular matrices in the group $G=SL_{n+1}$. $N$ (resp. $U$) is the subset of $B$ (resp. $B^+$) with unit diagonal.

\section{Weyl group and root systems} 
\index{$W$: Weyl group}
\index{$C$: Weyl chamber}
To every linear form $\beta \in \hfrak^*$, define the associated reflection $s_\beta$ on $\hfrak$ with:
$$\forall \lambda \in \hfrak,  s_\beta \lambda = \lambda - \beta\left( \lambda \right) \beta^\vee $$
The reflections $(s_\alpha)_{\alpha \in \Delta}$ are called simple reflections and they generate a finite group $W$ called the Weyl group. Define $m_{s,s'}$ as the order of the element $s s'$. $W$ can be realized as $W = \textrm{Norm}( H )/ H $.\\
The choice of a positive system $\Delta$ fixes an open Weyl chamber:
$$C := \left\{ x \in \afrak | \forall \alpha \in \Delta, \alpha(x) > 0 \right\}$$
Then, a fundamental domain of action for the Weyl group on $\afrak$ is the closed Weyl chamber:
$$\bar{C} = \left\{ x \in \afrak \ | \ \alpha(x) \geq 0, \forall \alpha \in \Delta \right\}$$

For $w \in W$, the inversion set of $w$ is defined as:
$$ Inv(w) := \left\{ \beta \in \Phi^+, w      \beta \in \Phi^{-} \right\} $$
For $w \in W$, a reduced expression is given by writing $w$ as product of simple reflections with minimal length:
$$ w = s_{i_1} s_{i_2} \dots s_{i_{\ell}}$$
A reduced word is such a tuple ${\bf i} = \left( i_1, \dots, i_{\ell} \right)$ and the set of reduced words for $w \in W$ is denoted by $R(w)$.

\index{$\ell$: Length function on the Weyl group}
Since all reduced expressions have necessarily the same length, it defines unambiguously the length function $\ell: W \rightarrow \N$. It has another characterization as the cardinal of the inversion set:
\begin{thm}
\label{thm:weyl_group_properties}
For $w \in W$:
$$ \ell(w) = |Inv(w)|$$
\end{thm}
The unique longuest element is denoted by $w_0$ and we set $m=\ell(w_0)$.\\

\subsection*{Braid relationships and braid moves}
If $s, s' \in W$ are simple reflections, a braid relationship in $W$ is the equality between $d=m_{s,s'}$ terms:
$$ s s' s \dots = s' s s' \dots $$
A braid move or a $d$-move occurs when substituting $ s s' s \dots$ for $s' s s' \dots $ within a reduced word. An important theorem is the following:

\begin{thm}[Tits lemma]
\label{thm:tits_lemma}
Two reduced expressions of the same $w \in W$ can be derived from each other using braid moves. 
\end{thm}

\subsection*{Representatives of W in G}
A common set of representatives for the generating reflections $(s_i)_{i \in I}$ is taken as:
$$ \bar{s}_i = \phi_i\left(  \left( \begin{array}{cc} 0 & -1 \\ 1 & 0 \end{array}\right) \right) = e^{-e_i} e^{f_i} e^{-e_i} = e^{f_i} e^{-e_i} e^{f_i}$$
Another common choice is:
$$ \bar{\bar{s}}_i := \bar{s}_i^{-1} = \phi_i\left(  \left( \begin{array}{cc} 0 & 1 \\ -1 & 0 \end{array}\right) \right) = e^{e_i} e^{-f_i} e^{e_i} = e^{-f_i} e^{e_i} e^{-f_i}$$

\begin{thm}[ \cite{bib:KacPeterson}, lemma 2.3 ]
The Weyl group representatives $\bar{s}_i$ (resp. $\bar{\bar{s}}_i$) satisfy the braid relationships:
$$ \bar{s}_i \bar{s}_j \bar{s}_i \dots = \bar{s}_j \bar{s}_i \bar{s}_j \dots $$
It allows us to define unambiguously $\bar{w} = \bar{u} \bar{v}$ if $w = uv$ and $l(w) = l(u) + l(v)$.
\end{thm}

However they do not form a presentation of the Weyl group, since for example $ (\bar{s}_i)^2 = \phi_i(-id) \neq id $.\\

The representative $\bar{w}_0$ of the longest element $w_0$ has an important property:
\begin{proposition}[\cite{bib:BBBR92} lemma 4.9]
\label{proposition:w_0_action_ad}
Via the Ad action, $\bar{w}_0$ acts on the Chevalley generators as:
$$ \forall \alpha \in \Delta, Ad(\bar{w}_0)( e_\alpha ) = - f_\alpha$$
\end{proposition}

\subsection*{Positive roots enumerations }

It is standard that reduced expressions of Weyl group elements produce positive roots enumerations. See for instance \cite{bib:Humphreys90}.
\begin{lemma}
\label{lemma:positive_roots_enumeration}
Let $(i_1, \dots, i_k)$ be a reduced expression of $w \in W$. Then for $j=1 \dots k$:
$$\beta_{{\bf i}, j} := s_{i_1} \dots s_{i_{j-1}} \alpha_{i_j}$$
produces all the positive roots in $Inv(w)$. For $w = w_0$, it produces all positive roots.
\end{lemma}
When the chosen reduced expression is obvious from context, we will drop the subscript ${\bf i}$. In appendix \ref{appendix:positive_roots_enumeration}, we give useful identities related to those enumerations and examples.

\section{Commutation identities}
Let $t \in \mathbb{C}$, $a \in H$. Then the following commutation relationships hold (easy to check on $SL_2$ then use the embeddings):
\begin{eqnarray}
  \label{lbl:comm1}
  a x_\alpha(t) a^{-1} & = & x_\alpha( a^{\alpha} t )
\end{eqnarray}
\begin{eqnarray}
  \label{lbl:comm2}
  a y_\alpha(t) a^{-1} & = & y_\alpha( a^{-\alpha} t )
\end{eqnarray}
\begin{eqnarray}
  \label{lbl:comm3}
  x_\alpha(t) y_\alpha(t') & = & \left\{ \begin{array}{ll}
				y_\alpha( \frac{t'}{1+tt'} ) (1+tt')^{h_\alpha} x_\alpha( \frac{t'}{1+tt'} ) \textrm{ if $1+tt' \neq 0$ } \\
				y_\alpha( \frac{1}{t} ) t^{h_\alpha} \bar{\bar{s}}_\alpha = \bar{\bar{s}}_\alpha t^{-h_\alpha} x_\alpha(-t) \textrm{ otherwise }
                               \end{array} \right.
\end{eqnarray}

\section{Weights and coweights}
\index{$P$: Weight lattice}
\index{$P^+$: Dominant weights}
The fundamental weights $\left( \omega_\alpha \right)_{\alpha \in \Delta}$ form the dual basis of simple coroots. They are the elements in $\hfrak^*$ such that:
$$ \forall (\alpha, \beta) \in \Delta^2, \omega_\alpha\left( h_\beta \right) = \delta_{\alpha, \beta}$$
They form a $\Z$-basis of the weight lattice:
$$ P := \left\{ x \in \hfrak^* \ | \ \forall \alpha \in \Delta, x(h_\alpha) \in \Z \right\}
      = \bigoplus_{\alpha \in \Delta} \Z \omega_{\alpha}$$
The dominant weights are:
$$ P^+ := \bigoplus_{\alpha \in \Delta} \N \omega_{\alpha}$$

Similarly, define the fundamental coweights $\left( \omega_\alpha^\vee \right)_{\alpha \in \Delta} \subset \afrak$ as the dual basis of simple roots.

\section{Involutions}
\label{lbl:preliminaries_involutions}
Since, $w_0 \in W$ transforms all simple positive roots to simple negative roots, there is an involution on $\Delta$ (or equivalently the index set $I$) denoted by $*$ such that:
$$ \forall \alpha \in \Delta, \beta^* = -w_0 \alpha$$

An antimorphism on the group is a map $G \rightarrow G$ that becomes a group morphism once composed with the transpose. We define the following group antimorphisms by their actions on a torus element $a \in H$ and the one-parameters subgroups generated by the Chevalley generators. For convenience, we also give their action at the level of the Lie algebra.
\index{${}^T$: Transpose}
\index{$\iota$: Positive inverse or Kashiwara involution}
\index{$S$: Sch\"utzenberger involution}
\begin{itemize}
 \item The usual transpose:
       $$ \begin{array}{ccc}
           a^T = a          & x_i(t)^T     = y_i(t) & y_i(t)^T     = x_i(t)
          \end{array}
       $$
       $$ \forall \alpha \in \Delta, 
          \begin{array}{ccc}
           h_\alpha^T = h_\alpha &  e_\alpha^T = f_\alpha & f_\alpha^T = e_\alpha
          \end{array}
       $$
 \item The 'positive inverse':
       $$ \begin{array}{ccc}
           a^\iota = a^{-1} & x_i(t)^\iota = x_i(t) & y_i(t)^\iota = y_i(t)
          \end{array}
       $$
       $$ \forall \alpha \in \Delta, 
          \begin{array}{ccc}
           h_\alpha^\iota = -h_\alpha &  e_\alpha^\iota = e_\alpha & f_\alpha^\iota = f_\alpha
          \end{array}
       $$
       When acting on the enveloping algebra $\Uc(\gfrak)$ or its quantum deformation, it is often referred to as the Kashiwara involution ( \cite{bib:Kashiwara91} (1.3) ).
 \item Sch\"utzenberger involution: $S(x) = \bar{w}_0 \left( x^{-1} \right)^{\iota T} \bar{w}_0^{-1} = \bar{w}_0^{-1} \left( x^{-1} \right)^{\iota T} \bar{w}_0$ \\
       It acts as ( relation 6.4 in \cite{bib:BZ01} or using proposition \ref{proposition:w_0_action_ad}):
       $$ S\left( x_{i_1}(t_1) \dots x_{i_q}(t_q) \right) = x_{i_q^*}(t_q) \dots x_{i_1^*}(t_1)$$
       $$ \forall \alpha \in \Delta, 
          \begin{array}{ccc}
           S(h_\alpha) = h_{\alpha^*} &  S(e_\alpha) = e_{\alpha*} & S(f_\alpha) = f_{\alpha*}
          \end{array}
       $$
       Notice that $S = \iota \circ S \circ \iota$
\end{itemize}
More informations are given in section \ref{section:involutions}, detailing their effect on crystals.

\section{On the Bruhat and Gauss decompositions}
The Bruhat decomposition states that $G$ is the disjoint union of cells:
$$ G = \bigsqcup_{\omega \in W} B^+ \omega B^+  = \bigsqcup_{\tau \in W} B \tau B^+ $$
In the case of $GL_n$, the second decomposition is known in linear algebra as the LPU decomposition which states that every invertible matrix can be decomposed into the product of a lower triangular matrix $L$, a permutation matrix $P$ and an upper triangular matrix $U$. $P$ is unique, and the cell corresponding to $P = id$ is dense as it is the locus where all principal minors are non-zero. The $LU$ decomposition is of utmost importance in numerical analysis as it allows efficient inversion of matrices.\\

In the largest opposite Bruhat cell $B B^+ = N H U$, every element $g$ admits a unique Gauss decomposition in the form $g = n a u$ with $ n \in N$, $a \in H$, $u \in U$.\\
In the sequel, we will write $g = [g]_- [g]_0 [g]_+$, $[g]_- \in N$, $[g]_0 \in H$ and $[g]_+ \in U$ for the Gauss decomposition. Also $[g]_{-0} := [g]_- [g]_0$.\\
Useful identities that can be proven writing the full Gauss decomposition in two forms then identifying terms, when they exist:
\begin{eqnarray}
  \label{lbl:gauss1}
  \forall (g_1, g_2) \in N H U \times N H U, [g_1 g_2]_{0+} & = & [ [g_1]_{0+} g_2]_{0+}
\end{eqnarray}
\begin{eqnarray}
  \label{lbl:gauss2}
  \forall (g, a) \in N H U \times H, [g a]_{+} & = & a^{-1} [g]_{+} a
\end{eqnarray}

\section{The universal enveloping algebra}
\paragraph{Invariant differential operators:} 
Here we consider right invariant group actions. The same presentation can be done using the left group action.\\

Every $X \in \gfrak$ can be viewed as a left invariant differential operator of order $1$. Its action on smooth functions is given by:
$$ \forall f \in \Cc^\infty\left( G \right), X f(g) := \lim_{t \rightarrow 0} \frac{f(g e^{tX})-f(g)}{t}$$
From such a point of view, it is easy to envision invariant different operators of arbitrary order. They should be obtained by composing elements $X_1, X_2, \dots, X_k$ in $\gfrak$ acting as differential operators. Their identification is subject to possible relations due to the Lie bracket $[\ ,\ ]$.

\index{$\Uc(\gfrak)$: Universal enveloping algebra}
This notion is formalized in algebra as the universal enveloping algebra $\Uc(\gfrak)$.

\paragraph{Definition from universal property:} 
The universal enveloping algebra of $\gfrak$ is constructed as the quotient of the tensor algebra $\bigoplus_n \gfrak^{\otimes n}$ by the two sided ideal generated by $ab - ba - [a, b]$, $a, b \in \gfrak$.\\
It has the universal property that any Lie algebra homomorphism $f: \gfrak \rightarrow A$, where $A$ is a unital algebra, factors into $f = g \circ i$. $g: \Uc(\gfrak) \rightarrow A$ is unique and $i: \gfrak \rightarrow \Uc(\gfrak)$ is the inclusion.

\paragraph{Definition with generators and relations:} 
An alternative definition uses generators and relations, with only the Cartan matrix $A = (a_{ij})_{1 \leq i,j \leq n}$ as input data.
$\Uc(\gfrak)$ is the unital associative algebra generated by $F_i, H_i, E_i$ with $1 \leq i \leq n$ with relations:
\begin{align*}
 [H_i, H_i] & = 0\\
 [E_i, F_j] & = \delta_{i,j} H_i\\
 [H_i, E_j] & = a_{ij} H_i\\
 [H_i, F_j] & = -a_{ij} H_i\\
 \textrm{Serre relations} & \textrm{ for } i \neq j:\\
 0 & = \sum_{s=0}^{1-a_{ij}} (-1)^s \binom{1-a_{ij}}{s}E_i^{1-a_{ij}-s} E_j E_i^s\\
 0 & = \sum_{s=0}^{1-a_{ij}} (-1)^s \binom{1-a_{ij}}{s}F_i^{1-a_{ij}-s} F_j F_i^s
\end{align*}
Since such a definition strips the algebra structure to its bare minimum, it has proved to be a fruitful starting point for generalizing the construction to Kac-Moody Lie algebras or quantum groups.

\section{On the representation theory of semisimple Lie algebras}
We are only concerned by finite dimensional modules. A representation of $\gfrak$ or a $\gfrak$-module is a (complex) vector space $V$ endowed with an action $\gfrak \rightarrow End(V)$ that is a homorphism of Lie algebras. The Lie algebra structure on $End(V)$ is simply given by the commutator bracket $[a, b] = ab -ba$.

Since every homomorphism of Lie algebras lifts to a homomorphism of the corresponding simply connected Lie groups, every $\gfrak$-module $V$ lifts to a unique Lie group representation $G \rightarrow GL(V)$.\\
An irreductible or simple $\gfrak$-module is a $\gfrak$-module $V$ with no non-trivial submodules, the trivial submodules being the zero module and $V$ itself.

Note that any $\gfrak$-module can be equivalently seen as a $\Uc\left( \gfrak \right)$-module. There is a weight space decomposition:
$$ V = \bigoplus_{\mu \in P} V_\mu$$
where $V_\mu = \left\{ v \in V \ | \forall h \in \hfrak, h v = \mu(h) v \right\}$. The non-zero $V_\mu$ are called weight spaces, and their vectors weight vectors of weight $\mu$.

A highest weight vector in $V$ is a non-zero weight vector $v$ such that $\Uc\left( \ufrak \right) v = \{0\}$. In a simple module, there is one and only one highest weight vector, up to scalar multiplication.

\paragraph{Highest weight modules:}
A classical theorem identifies the isomorphism classes of simple modules.
\begin{thm}
There is a bijection between dominant weights $P^+$ and the isomorphism classes of simple $\gfrak$-modules. To every $\lambda \in P^+$ corresponds a unique highest weight module $V(\lambda)$.
\end{thm}
\index{$V(\lambda)$: Simple module with highest weight $\lambda \in P^+$}
For a simple module of highest weight $\lambda$, $V(\lambda)$, we choose a highest weight vector and denote it by $v_\lambda$.

\paragraph{Characters:}
For a $\gfrak$-module $V$, define the character $ch(V)$ as the function on $\hfrak$ defined as:
\index{$ch(V)$: Character associated to the $\gfrak$-module $V$}
$$ ch(V) = \sum_{\mu \in P} \dim V_\mu e^{\mu}$$
Characters encode an important amount of informations on the representation $V$. One of the goals of combinatorial representation theory is to develop combinatorial models to compute characters efficiently. Examples of combinatorial models are the combinatorics of Young tableaux for the $A_n$ type, Kashiwara crystals, the Littelmann path model and alcove walks.

\section{On Lusztig's canonical basis}
\label{section:preliminaries_canonical_basis}
In the sequel, we will never use Lusztig's canonical basis in itself. However, we will be extensively interested in its parametrizations. As such it is important to review this mathematical object. We will remain elusive concerning its precise definition, though. For more details, the reader could have a look at Morier-Genoud's excellent introduction, in French \cite{bib:Mor06}.

\index{$\Bfrak$: Canonical basis}
In the nineties, Lusztig introduced a basis $\Bfrak$ of the quantum group $\Uc_q\left( \nfrak \right)$ called the canonical basis. For $q=1$, one obtains a basis for the enveloping algebra.

\subsection*{Parametrizations}
There are two common parametrizations of the canonical basis. Both depend on a choice of reduced word for the longest Weyl group element $w_0$. Let ${\bf i} \in R(w_0)$ and $m = \ell(w_0)$.

The Lusztig parametrization is a bijection:
$$
 \begin{array}{llll}
[x_{\bf i}]: & \N^m              & \rightarrow & \Bfrak\\
             & (t_1, \dots, t_m) & \mapsto     & x_{\bf i}(t_1, \dots, t_m)
 \end{array}
$$
The string (or Kashiwara) parametrization uses the integer points of a convex polyhedral cone $\Cc_{\bf i} \subset \R_+^m$ (\cite{bib:Littelmann}), which we call the string cone. It is given by a bijection:
$$
 \begin{array}{llll}
[x_{\bf-i}]: & \Cc_{\bf i} \cap \N^m            & \rightarrow & \Bfrak\\
             & (c_1, \dots, c_m) & \mapsto     & x_{\bf-i}(c_1, \dots, c_m)
 \end{array}
$$

\subsection*{Kashiwara operators}
These are linear operators on $\Uc_q(\nfrak)$ defined by their action on the canonical basis $\Bfrak$. Let ${\bf i} \in R(w_0)$ and $\alpha = \alpha_{i_1}$. The Kashiwara operators $\tilde{e}_\alpha$ and $\tilde{f}_\alpha$ satisfy:

\begin{align}
\label{eqn:kashiwara_operator_f_1}
\tilde{f}_\alpha\left( [x_{\bf i}](t_1, \dots, t_m) \right) & = [x_{\bf i}](t_1 + 1, \dots, t_m)
\end{align}

\begin{align}
\label{eqn:kashiwara_operator_f_2}
\tilde{f}_\alpha\left( [x_{\bf-i}](c_1, \dots, c_m) \right) & = [x_{\bf-i}](c_1 + 1, \dots, c_m)
\end{align}

\begin{align}
\label{eqn:kashiwara_operator_e_1}
\tilde{e}_\alpha\left( [x_{\bf i}](t_1, \dots, t_m) \right) & = [x_{\bf i}](t_1 - 1, \dots, t_m) \textrm{ or } 0 \textrm{ if } t_1 = 0
\end{align}

\begin{align}
\label{eqn:kashiwara_operator_e_2}
\tilde{e}_\alpha\left( [x_{\bf-i}](c_1, \dots, c_m) \right) & = [x_{\bf-i}](c_1 - 1, \dots, c_m) \textrm{ or } 0 \textrm{ if } c_1 = 0
\end{align}

\begin{rmk}
 The Kashiwara operators are quasi-inverses of each other, in the sense that:
$$ \forall b \in \Bfrak, \tilde{e}_\alpha \circ \tilde{f}_\alpha(b) = b$$
$$ \forall b \in \Bfrak, \tilde{e}_\alpha(b) \neq 0 \Rightarrow \tilde{f}_\alpha \circ \tilde{e}_\alpha(b) = b$$
\end{rmk}

\subsection*{Compatibility properties}
Here consider $q=1$ and view $\Bfrak$ as a basis of the universal enveloping algebra $\Uc(\nfrak)$. The desirable properties of the canonical basis are compatibility properties regarding highest weight modules. Fix $\lambda \in P^+$ and consider the highest weight module $V(\lambda)$. $v_\lambda$ will denote a highest weight vector, unique up to a multiplicative scalar. It is well known that the canonical surjection:
$$
 \begin{array}{llll}
\pi_\lambda: & \Uc(\nfrak) & \rightarrow & V(\lambda)\\
             & n           & \mapsto     & n v_\lambda
 \end{array}
$$
has the kernel:
$$ \ker(\pi_\lambda) = \sum_{\alpha \in \Delta} \Uc(\nfrak) f_\alpha^{\lambda(\alpha^\vee)+1}$$
Hence the truncation:
$$ \Bfrak(\lambda) := \Bfrak - \Bfrak \cap \ker \pi_\lambda$$

\begin{thm}[ \cite{bib:Lusztig93} ]
 $\Bfrak(\lambda) v_\lambda$ is a basis for $V(\lambda)$ made of weight vectors.
\end{thm}
Therefore, the subsets $\left( \Bfrak(\lambda), \lambda \in P^+ \right)$ of the canonical basis $\Bfrak$, once identified with $\Bfrak(\lambda) v_\lambda$, form compatible bases of highest weight modules.

For $b \in \Bfrak(\lambda)$, denote by $\gamma(b)$ its weight in the representation $V(\lambda)$. If $b = x_{\bf i}(t_1, \dots, t_m)$ then, using the positive roots enumeration $(\beta_{{\bf i}, 1}, \dots, \beta_{{\bf i}, m})$ associated to ${\bf i}$:
\begin{align}
\label{eqn:weight_lusztig}
\gamma(b) & = \lambda - \sum_{j=1}^m t_j \beta_{{\bf i}, j}
\end{align}
If $b = x_{\bf-i}(c_1, \dots, c_m)$ then:
\begin{align}
\label{eqn:weight_string}
\gamma(b) & = \lambda - \sum_{j=1}^m c_j \alpha_{i_j}
\end{align}

\section{On total positivity}
The classical definition of a totally positive matrix is a matrix with all minors being positive. The subject itself dates back to the beginning of the XXth century and has many applications in combinatorics, graph theory and probability.

As he says himself, Lusztig got interested in the subject after Kostant pointed out that the combinatorics of the canonical basis are similar to the combinatorics of total positivity. Later, it was made explicit that totally positive varieties in $G$ have morally the same parametrizations as the canonical basis of the Langlands dual $G^\vee$, but in a different semi-field. The operation needed is tropicalization and will be detailed later.\\

\paragraph{Case of $GL_n$:}
There are two equivalent definitions for totally non-negative matrices.
\begin{thm}[Whitney \cite{bib:Whitney52}, Loewner \cite{bib:Lo55}, Cryer \cite{bib:Cr76} ]
\label{bib:total_positivity_gln}
An invertible $n \times n$ matrix is said to be totally non-negative if all its minors are $\geq 0$, or equivalently, if it has a decomposition
$$ y_{i_1}(t_1) \dots y_{i_m}(t_m) h x_{i_1}(t_1') \dots x_{i_m}(t_m')$$
where $h$ is diagonal with positive entries, $y_i(t) = e^{t E_{i+1,i}} = I_n + t E_{i+1,i}$, $x_i(t) = e^{t E_{i,i+1}} = I_n + t E_{i,i+1}$ (Jacobi matrices) for $t \geq 0$. Moreover, the space of totally non-negative matrices can be characterized as the semi-group generated by such elements.
\end{thm}
For more informations on the combinatorics of total positivity, see \cite{bib:Skandera03} and references therein. The link to the enumeration of non-intersecting paths hints directly to path models, and walks confined in cones.

\paragraph{Reductive case:}
In 1993, Lusztig generalized this definition to arbitrary complex reductive groups. The following sets are called totally non-negative parts of $G$.
\begin{itemize}
 \item The semi-group generated by $a \in H$ such that $a^\gamma>0$ for every weight $\gamma$: $ H_{>0} = A = exp\left( \afrak \right)$
 \item The semi-group generated by the $x_\alpha(t), t>0, \alpha \in \Delta$: $U_{\geq 0}$
 \item The semi-group generated by the $y_\alpha(t), t>0, \alpha \in \Delta$: $N_{\geq 0}$
 \item The totally non-negative of G is denoted $G_{\geq 0}$ and is formed by the semi-group generated by all of them.
\end{itemize}

Lusztig proved that totally non-negative elements admit a Gauss decomposition made of totally non-negative elements, and exhibited parametrizations as products of Jacobi matrices.
\begin{thm}[\cite{bib:Lusztig94} lemma 2.3]
\label{thm:totally_positive_gauss_decomposition}
Any element $g \in G_{\geq 0}$ has a unique Gauss decomposition $g = n a u$ with $n \in N_{\geq 0}$, $a \in A$ and $u \in U_{\geq 0}$.
\end{thm}

\begin{thm}[\cite{bib:Lusztig94} proposition 2.7, \cite{bib:BZ97}, Proposition 1.1]
\label{thm:totally_positive_parametrizations}
For any $w \in W$ with $k=\ell(w)$, every reduced word $\mathbf{i} = (i_1, \dots, i_k)$ in $R(w)$ gives rise to a parametrization of $U^{w}_{>0} := U_{\geq 0} \cap B w B$ by:
$$
 \begin{array}{cccc}
x_{\mathbf{i}}: & \mathbb{R}_{>0}^k & \rightarrow & U^{w}_{>0}\\
                & (t_1, \dots, t_k) & \mapsto     & x_{i_1}(t_1) \dots x_{i_k}(t_k)
 \end{array}
$$
\end{thm}
Hence the name of totally positive varieties for the sets $U^{w}_{>0}, w \in W$. The Bruhat decomposition tells us then that the previous maps have disjoint images and cover the entire non-negative part $U_{\geq 0} = \bigsqcup_{w \in W} U^{w}_{>0}$. Of course, after transpose, one has analogous parametrizations for $N^{w}_{>0} := N_{\geq 0} \cap B^+ w B^+$.\\

Afterwards, Berenstein, Fomin and Zelevinsky in a series of papers (\cite{bib:BZ97, bib:FZ99, bib:BZ01}) completed the picture by defining generalized minors on semi-simple groups, allowing to define the totally positive varieties as the locus where appropriate minors are positive.

\section{On generalized determinantal calculus}
We mean by determinantal calculus, the computations involving minors and relations among them.

\paragraph{Case of $GL_n$:} In the classical case, the minor $\Delta_{I, J}(x)$ of a matrix $x \in GL_n$ is obtained as the determinant of the submatrix with rows $I$ and columns $J$. $I$ and $J$ are subsets of $\{ 1, \dots, n\}$. It is well known that matrices having a Gauss decomposition are those having non-zero principal minors.\\

\paragraph{Complex reductive case:}
\index{$\Delta^{\omega_\alpha}$: Principal generalized minors}
This fact can be extended to all complex semi-simple groups provided that we construct generalized minors. For $x = n a u \in N H U$ dense subset of $G$, we define the generalized principal minors indexed by the fundamental weights as $\Delta^{\omega_i}(x) = a^{\omega_i}$. A useful result is:

\begin{proposition}[\cite{bib:FZ99} Corollary 2.5]
 An element $x \in G$ admits a Gauss decomposition if and only if
$$ \forall \alpha \in \Delta, \Delta^{\omega_\alpha}(x) \neq 0$$
\end{proposition}

\index{$\Delta_{u \omega_\alpha, v \omega_\alpha}$: Generalized minors}
Arbitrary minors are indexed by the fundamental weights and couples of Weyl group elements:
$$ \forall (u, v) \in W \times W, \Delta_{u \omega_i, v \omega_i}(x) := \Delta^{\omega_i}\left( \overline{u}^{-1} x \bar{v} \right)$$
For more details, see \cite{bib:BZ97, bib:FZ99, bib:BZ01}.

\paragraph{Representation theoretic definition:}
Generalized principal minors can also be written using a representation theoretic approach.

\begin{lemma}[\cite{bib:BBO} section 3]
Let $v_{\omega_i}$ be a highest weight vector for the representation $V(\omega_i)$ and $\langle .,. \rangle$ an invariant scalar product. With $v_{\omega_i}$ normalised, we have:
$$ \forall x \in G, \Delta^{\omega_i}(x) = \langle x v_{\omega_i}, v_{\omega_i} \rangle $$
\end{lemma}
\begin{proof}
On the dense subset $N H U \subset G$, write the Gauss decomposition $x = n a u$. Since the $U$ action fixes $v_{\omega_i}$ in the representation $V(\omega_i)$:
$$ \langle x v_{\omega_i}, v_{\omega_i} \rangle = \langle a u v_{\omega_i}, n^T v_{\omega_i} \rangle = \langle a v_{\omega_i},  v_{\omega_i} \rangle$$
The result holds using the fact that the torus $H$ acts multiplicatively on highest weight vectors:
$$a v_{\omega_i} = a^{\omega_i} v_{\omega_i}$$
\end{proof}

Then the generalized minors are given by:
$$ \forall (u, v) \in W \times W, \Delta_{u \omega_i, v \omega_i}(x) := \langle x \bar{v} v_{\omega_i}, \bar{u} v_{\omega_i} \rangle$$

\section{Criteria for total positivity}
\label{section:total_positivity_criteria}
Thanks to the previous generalized minors, Berenstein, Fomin and Zelevinsky gave criteria for total positivity. We will often make use of the following criterion for total positivity in the lower unipotent group $N$.

\begin{thm}
\label{thm:total_positivity_criterion}
The group element $x \in N$ is totally positive:
$$ x \in N^{w_0}_{>0}$$
if and only if:
$$ \forall w \in W, \forall \alpha \in \Delta, \Delta_{w \omega_\alpha, \omega_\alpha}(x) > 0$$
\end{thm}

In fact, there are more precise results and for all cells. An important feature is that one does not need to test all minors. Let $w \in W$ and for every reduced word ${\bf i} = (i_1, \dots, i_m) \in R(w)$, define the family of minors:
$$ F\left( {\bf i} \right) = \left\{ \Delta_{k, \bf i} :=  \Delta_{s_{i_k} \dots s_{i_m} \omega_{i_k}, \omega_{i_k}}, 1 \leq k \leq m\right\}$$
And:
$$ F\left( w \right) = \bigcup_{ {\bf i } \in R(w) } F( {\bf i} )$$
\begin{thm}[Total positivity criterion - theorem 1.5 in \cite{bib:BZ97} or theorem 1.11 in \cite{bib:FZ99} ]
\label{thm:extended_total_positivity_criterion}
Let $w \in W$. Then the following propositions are equivalent:
\begin{itemize}
 \item[(i)] $x \in N \cap B^+ w B^+$ is totally positive.
 \item[(ii)] For a certain ${\bf i} \in R(w)$, $\Delta(x) > 0$ for any $\Delta \in F({\bf i})$.
 \item[(iii)] $\Delta(x) > 0$ for any $\Delta \in F(w)$.
\end{itemize}
\end{thm}

Then, they also gave monomial formulas that gives the parameters in the factorization to a product of Jacobi matrices. We can cite \cite{bib:BZ97}, \cite{bib:FZ99}, \cite{bib:BZ01} as main references. Those formulas rely heavily on the generalized determinantal calculus and express the parameters of a group element $x \in  U \cap B w B$ as a function of the minors of $z = \eta_w(x)$, a twisted transformation of the element $x$.

This 'twist-map' is defined on the trace of each Bruhat cell. For each $w$ in the Weyl group, let $U^{w} := U \cap B w B $ and $\eta_w: U^w \rightarrow U^w$ is defined by saying that $\eta_w(x)$ is the unique element $z$ in $U \cap B w x^T$. Hence:
\index{$\eta_w$: Twist map on $U \cap B w B$}
\begin{align}
\label{eqn:def_twist}
\forall x \in U^w, \eta_w(x) & := [\bar{w} x^T]_+
\end{align}
A key fact is that $\eta_w$ restricts to an automorphism between the positive parts $U^w_{>0}$ (\cite{bib:BZ97} page 3, theorem 1.2, proposition 1.3).

\begin{rmk}
Here we mainly talk about total positivity in the unipotent group $U$, or $N$ applying the transpose involution to the previous theorem. The same machinery works for the entire group, but then it is necessary to cut the space into double Bruhat cells. It is the object of the article \cite{bib:FZ99}. 
\end{rmk}

\chapter{Littelmann path model for geometric crystals}
Let $G$ be a simply-connected complex semi-simple group with Lie algebra $\gfrak$. Let $\hfrak$ be its Cartan subalgebra and $\Delta$ the set of simple roots. $\afrak$ is the Cartan subalgebra of its split real form or equivalently the subspace of $\hfrak$, the Cartan subalgebra, where simple roots are real-valued.\\

In the nineties, Kashiwara introduced combinatorial objects called crystals that encode the representation theory of Lie algebra (see \cite{bib:Kashiwara95}). Littelmann's work (\cite{bib:Littelmann}, \cite{bib:Littelmann95}, \cite{bib:Littelmann97}) allows to realize them as paths in $\afrak^*$. And more recently, Berenstein and Kazhdan (\cite{bib:BK00}, \cite{bib:BK04}, \cite{bib:BK06}) defined geometric crystals as algebro-geometric objects that degenerate to Kashiwara crystals by tropicalization, using the fact that totally positive varieties 'know' everything about the combinatorics of crystals.\\

Here, we construct a path model for geometric crystals, in the same spirit as Littelmann. In fact, this will be a path model for the Langlands dual $G^\vee$. As such, crystal elements are paths in the real Cartan subalgebra $\afrak$ (instead of $\afrak^*$), the weight function is the endpoint of a path and tensor product of crystals is given by concatenation.\\

We start by defining a notion of geometric crystal in general. Then we consider the totally positive variety $\Bc = B_{>0}$, with relevant coordinate charts. $\Bc$ is the typical positive geometric crystal in the sense of Berenstein and Kazhdan. It will play the role of a 'group picture' for geometric crystals, before presenting the path model that morally sits on top.\\

Indeed, there is a projection $p$ that maps paths to Berenstein and Kazhdan's group picture. The projection of a path $\pi$ is given by the flow of a left invariant differential equation on $NA$ driven by $\pi$. The underlying invariant differential operator is closely related to the Casimir element in Kostant's Whittaker model. Furthermore, $p$ is a morphism of crystals that restricts to an isomorphism on connected components.\\

A connected crystal is naturally parametrized by $m = \ell(w_0)$ positive real parameters. These parameters are a geometric lifting of either the Lusztig parameters or the string parameters of the canonical basis. Moreover, isomorphism classes are indexed by a single vector $\lambda$ that is interpreted as a highest weight. In order to obtain the isomorphism class of a connected crystal $\langle \pi \rangle$ generated by the path $\pi$, a remarkable transform on paths $\Tc_{w_0}$ has to be applied. This transform is a geometric lifting of the Pitman operator $\Pc_{w_0}$. The highest weight $\lambda \in \afrak$ is the endpoint of $\Tc_{w_0} \pi$.

In the end, for every $T>0$, we prove that the following map is a bijection onto its image:
$$ \begin{array}{cccc}
  C_0\left( [0, T], \afrak \right) & \longrightarrow & \left( \Bc   , C\left( (0, T], \afrak \right) \right) \\
  \pi                              & \mapsto	     & \left( p(\pi), \Tc_{w_0} \pi \right)\\
  \end{array}
$$
This bijection can be interpreted as a geometric version of the Robinson-Schensted-Knuth correspondence.

\section{Geometric crystals}
\label{section:geom_lifting}

We made several although not essential modifications to the original setting of Berenstein and Kazhdan. Whereas they defined geometric crystals as being affine varieties over $\Q$, we will simply consider them as sets with structural maps. This will allow us to consider the geometric crystal of continuous paths valued in the Cartan subalgebra $\afrak$.

Moreover, instead of defining a positive structure separately (\cite{bib:BK06} section 3) we will be directly working at the level of the totally non-negative variety $G_{\geq 0}$. Finally, we will favor using the additive group $\R$ instead of the multiplicative $\R_{>0}$ hence the presence of numerous logarithms.

\subsection{Abstract geometric crystals}

\begin{definition}[Abstract geometric crystal]
\label{def:crystal}
An abstract crystal is a set $L$ equipped with
\begin{itemize}
 \item a weight map $\gamma: L \rightarrow \afrak$.
 \item $\varepsilon_\alpha, \varphi_\alpha: L \rightarrow \R$ defined for every $\alpha \in \Delta$
 \item $e^c_\alpha: L \rightarrow L$, $c \in \R$, $\alpha \in \Delta$
\end{itemize}
and satisfying the following properties for $\pi \in L$:
\begin{itemize}
 \item[(C1)] $\varphi_\alpha(\pi) = \varepsilon_\alpha(\pi) + \alpha\left( \gamma(\pi) \right)$
 \item[(C2)] $\gamma\left( e^c_\alpha \cdot \pi \right) = \gamma\left( \pi \right) + c \alpha^\vee$
 \item[(C3)] $\varepsilon_\alpha\left( e^c_\alpha \cdot \pi \right) = \varepsilon_\alpha\left( \pi \right) - c$
 \item[(C3')]$\varphi_\alpha\left( e^c_\alpha \cdot \pi \right) = \varphi_\alpha\left( \pi \right) + c$
 \item[(C4)] $e^._\alpha$ are actions: $e^0 = id$ and $e^{c+c'}_\alpha = e^c_\alpha \cdot e^{c'}_\alpha$
\end{itemize}
 Clearly, (C3) and (C3') are equivalent once (C1) and (C2) are assumed.
\end{definition}

Here, unlike the standard object defined by Kashiwara, there is no ghost element and the crystal has free actions. Moreover, we adopt a continuous setting in the spirit of \cite{bib:BBO2}. One could use the term 'free continuous crystal'.\\

Later on, we will also require a certain type of commutation relations between the actions $(e^._\alpha)_{\alpha \in \Delta}$, identified as Verma relations. Berenstein and Kazhdan refer to such structure as a 'pre-crystal' if Verma relations are not available. For more convenient notations, define $I$ to be a set of indices for the set of simple roots $\Delta$.

\paragraph{Generated crystals:} Given a subset $S$ of a crystal $L$, define $\langle S \rangle$ as the smallest subcrystal of $L$ containing $S$. Since intersections of crystals are crystals, we can define it as:
\begin{align*}
\langle S \rangle & := \cap_{ S \subset \mathcal{C} \textrm{ subcrystal of L } } \mathcal{C}\\
                  & = \left\{ e^{c_1}_{\alpha_{i_1}} \cdot e^{c_2}_{\alpha_{i_2}} \cdot \dots e^{c_l}_{\alpha_{i_l}} \cdot x | x \in S, l \in \mathbb{N}, \left( c_1, c_2, \dots, c_l\right) \in \mathbb{R}^l , \left( i_1, i_2, \dots, i_l\right) \in I^l  \right\}
\end{align*}

\paragraph{Connected components:} A crystal is connected if given two elements $x$ and $y$, there is $l \in \mathbb{N}$, $\left( c_1, c_2, \dots, c_l\right) \in \mathbb{R}^l$ and $\left( i_1, i_2, \dots, i_l\right) \in I^l$ such that:
$$ y = e^{c_1}_{\alpha_{i_1}} \cdot e^{c_2}_{\alpha_{i_2}} \cdot \dots e^{c_l}_{\alpha_{i_l}} \cdot x  $$
It is quite obvious that any crystal is the disjoint union of its connected components, since 'being connected' is an equivalence relation. Also a connected component is generated by any of its elements, and connected components are subcrystals.

\paragraph{Morphism of crystals:} A morphism of crystals is a map $\psi$ that preserves the structure. It is an automorphism if invertible, and the inverse map is a morphism.

\paragraph{q-Tensor product of crystals:} In the sequel, the crystal structure itself will depend on a parameter $q \geq 0$. For $q\geq0$, we define the $q$-tensor product of two crystals $B_1$ and $B_2$ as the set $ B_1 \otimes_q B_2 = B_1 \times B_2$ endowed with structural maps. For $q>0$, they are given by:
\begin{itemize}
 \item $\gamma\left( b_1 \otimes_q b_2 \right) = \gamma\left( b_1 \right) + \gamma\left( b_2 \right)$
 \item $\varepsilon_\alpha\left( b_1 \otimes_q b_2 \right) = \varepsilon_\alpha\left( b_1 \right) + 
         q \log\left( 1 + e^{\frac{ \varepsilon_\alpha(b_2) - \varphi_\alpha(b_1) }{q}}\right)$
 \item $\varphi_\alpha\left( b_1 \otimes_q b_2 \right) = \varphi_\alpha\left( b_2 \right) + 
         q \log\left( 1 + e^{\frac{ \varphi_\alpha(b_1) - \varepsilon_\alpha(b_2) }{q}}\right)$
 \item The actions are defined as 
       $e^c_\alpha\left( b_1 \otimes_q b_2 \right) = \left( e^{c_1}_\alpha \cdot b_1 \right) \otimes_q \left( e^{c_2}_\alpha \cdot b_2 \right)$\\
       where
       \begin{align*}
        c_1 & = q \log\left( \frac{ e^{\frac{c + \varphi_\alpha(b_1) }{q}} + e^{\frac{   \varepsilon_\alpha(b_2) }{q}} }
                                  { e^{\frac{    \varphi_\alpha(b_1) }{q}} + e^{\frac{   \varepsilon_\alpha(b_2) }{q}} } \right)\\
            & = q \log\left( e^{\frac{c}{q}} + e^{\frac{\varepsilon_\alpha(b_2) - \varphi_\alpha(b_1)}{q}} \right) - 
                q \log\left( 1 + e^{\frac{\varepsilon_\alpha(b_2) - \varphi_\alpha(b_1)}{q}} \right)\\
        c_2 & = q \log\left( \frac{ e^{\frac{    \varphi_\alpha(b_1) }{q}} + e^{\frac{   \varepsilon_\alpha(b_2) }{q}} }
                                  { e^{\frac{    \varphi_\alpha(b_1) }{q}} + e^{\frac{-c+\varepsilon_\alpha(b_2) }{q}} } \right)\\
            & = -q \log\left( e^{\frac{-c}{q}} + e^{\frac{\varphi_\alpha(b_1) - \varepsilon_\alpha(b_2)}{q}} \right) +
                q \log\left( 1 + e^{\frac{\varphi_\alpha(b_1) - \varepsilon_\alpha(b_2)}{q}} \right)\\
       \end{align*}

\end{itemize}
\begin{rmk}
 $$ c_1 + c_2 = c$$
\end{rmk}
By letting the parameter $q \rightarrow 0$, one recovers the same 'frozen' axioms for tensor product as in \cite{bib:BBO2}. As such, tensor product for $q=0$ is also well defined as:
\begin{itemize}
 \item $\gamma\left( b_1 \otimes b_2 \right) = \gamma\left( b_1 \right) + \gamma\left( b_2 \right)$
 \item $\varepsilon_\alpha\left( b_1 \otimes b_2 \right) = \varepsilon_\alpha\left( b_1 \right) + 
         \left( \varepsilon_\alpha(b_2) - \varphi_\alpha(b_1) \right)^+$
 \item $\varphi_\alpha\left( b_1 \otimes b_2 \right) = \varphi_\alpha\left( b_2 \right) + 
         \left( \varphi_\alpha(b_1) - \varepsilon_\alpha(b_2)\right)^+$
 \item $e^c_\alpha\left( b_1 \otimes b_2 \right) = \left( e^{c_1}_\alpha \cdot b_1 \right) \otimes \left( e^{c_2}_\alpha \cdot b_2 \right)$\\
       where
       \begin{align*}
        c_1 & = \max\left( c, \varepsilon_\alpha(b_2) - \varphi_\alpha(b_1) \right) - 
                \left( \varepsilon_\alpha(b_2) - \varphi_\alpha(b_1) \right)^+\\
        c_2 & = \min\left( c, \varepsilon_\alpha(b_2) - \varphi_\alpha(b_1) \right) +
                \left( \varphi_\alpha(b_1) - \varepsilon_\alpha(b_2) \right)^+\\
       \end{align*}
\end{itemize}

It is easy to check that one still obtains a crystal:

\begin{proposition}
\label{proposition:tensor_product_is_crystal}
For all $q\geq 0$, $B_1 \otimes_q B_2$ is a crystal.
\end{proposition}
\begin{proof}
Let us verify axioms for crystals from $(C_1)$ to $(C4)$, in the case $q>0$, then $q=0$ will follow by a limit argument.

$(C1)$ 
\begin{align*}
&  \varphi_\alpha\left( b_1 \otimes_q b_2 \right) - \varepsilon_\alpha\left( b_1 \otimes_q b_2 \right)\\
& = \varphi_\alpha\left( b_2 \right) + q \log\left( 1 + e^{\frac{ \varphi_\alpha(b_1) - \varepsilon_\alpha(b_2) }{q}}\right)\\
&  - \varepsilon_\alpha\left( b_1 \right) - q \log\left( 1 + e^{\frac{ \varepsilon_\alpha(b_2) - \varphi_\alpha(b_1) }{q}}\right)\\
& = \varphi_\alpha\left( b_2 \right) - \varepsilon_\alpha\left( b_1 \right)  + \varphi_\alpha\left(b_1\right) - \varepsilon_\alpha\left(b_2\right)\\
& = \alpha\left( \gamma\left(b_1\right) \right) + \alpha\left( \gamma\left(b_2\right) \right)\\
& = \alpha\left( \gamma\left(b_1 \otimes_q b_2 \right) \right)
\end{align*}

$(C2)$
\begin{align*}
& \gamma\left( e^c_\alpha \cdot (b_1 \otimes_q b_2) \right)\\
& = \gamma\left( e^{c_1}_\alpha \cdot b_1 \right) + \gamma\left( e^{c_2}_\alpha \cdot b_2 \right)\\
& = \gamma\left( b_1 \otimes_q b_2 \right) + \left( c_1 + c_2 \right)\alpha^\vee
\end{align*}
using the remark that $c=c_1+c_2$, the second axiom is checked.

$(C3)$ We will only check:
\begin{align*}
& \varepsilon_\alpha\left( e^c_\alpha \cdot (b_1 \otimes_q b_2) \right)\\
& = \varepsilon_\alpha\left( e^{c_1}_\alpha \cdot b_1 \right) + 
    q \log\left( 1 + e^{\frac{ \varepsilon_\alpha( e^{c_2}_\alpha \cdot b_2) - \varphi_\alpha( e^{c_1}_\alpha \cdot b_1) }{q}}\right)\\
& = -c_1 + \varepsilon_\alpha\left( b_1 \right) + 
    q \log\left( 1 + e^{\frac{ -c + \varepsilon_\alpha(b_2) - \varphi_\alpha(b_1) }{q}}\right)\\
& = -c_1 + \varepsilon_\alpha\left( b_1 \right) - c +
    q \log\left( e^c + e^{\frac{ \varepsilon_\alpha(b_2) - \varphi_\alpha(b_1) }{q}}\right)\\
& = \varepsilon_\alpha\left( b_1 \right) - c +
    q \log\left( 1 + e^{\frac{ \varepsilon_\alpha(b_2) - \varphi_\alpha(b_1) }{q}}\right)\\
& = -c + \varepsilon_\alpha\left( b_1 \otimes_q b_2 \right) 
\end{align*}

$(C4)$ We know that
$$ e^c_\alpha \cdot e^{c'}_\alpha \cdot \left( b_1 \otimes_q b_2 \right) = e^{c_1 + c_1'} \cdot b_1 \otimes_q e^{c_2 + c_2'} \cdot b_2 $$
{ \centering where }
\begin{align*}
c_1' & = q \log\left( \frac{ e^{\frac{c' + \varphi_\alpha(b_1) }{q}} + e^{\frac{   \varepsilon_\alpha(b_2) }{q}} }
                           { e^{\frac{    \varphi_\alpha(b_1) }{q}} + e^{\frac{   \varepsilon_\alpha(b_2) }{q}} } \right)\\
c_1  & = q \log\left( \frac{ e^{\frac{c  + \varphi_\alpha(e^{c_1'}_\alpha \cdot b_1) }{q}} + e^{\frac{   \varepsilon_\alpha(e^{c_2'}_\alpha \cdot b_2) }{q}} }
                           { e^{\frac{    \varphi_\alpha(e^{c_1'}_\alpha \cdot b_1) }{q}} + e^{\frac{   \varepsilon_\alpha(e^{c_2'}_\alpha \cdot b_2) }{q}} } \right)\\
c_2' & = c - c_1'\\
c_2  & = c - c_1\\
\end{align*}
We simplify:
\begin{align*}
c_1 & = q \log\left( \frac{ e^{\frac{c  + \varphi_\alpha(b_1) + c_1'}{q}} + e^{\frac{   \varepsilon_\alpha(b_2) - c_2'}{q}} }
                           { e^{\frac{    \varphi_\alpha(b_1) + c_1'}{q}} + e^{\frac{   \varepsilon_\alpha(b_2) - c_2'}{q}} } \right)\\
& = q \log\left( \frac{ e^{\frac{c + c' + \varphi_\alpha(b_1)}{q}} + e^{\frac{   \varepsilon_\alpha(b_2)}{q}} }
                           { e^{\frac{c'+ \varphi_\alpha(b_1)}{q}} + e^{\frac{   \varepsilon_\alpha(b_2)}{q}} } \right)
\end{align*}
Hence:
$$ c_1 + c_1' = q \log\left( \frac{ e^{\frac{c + c' + \varphi_\alpha(b_1)}{q}} + e^{\frac{   \varepsilon_\alpha(b_2)}{q}} }
                                  { e^{\frac{         \varphi_\alpha(b_1)}{q}} + e^{\frac{   \varepsilon_\alpha(b_2)}{q}} } \right)$$
and
$$ c_2 + c_2' = c + c' - \left( c_1 + c_1' \right)
              = q \log\left( \frac{ e^{\frac{ \varphi_\alpha(b_1)}{q}} + e^{\frac{        \varepsilon_\alpha(b_2)}{q}} }
                                  { e^{\frac{ \varphi_\alpha(b_1)}{q}} + e^{\frac{ -c-c'+ \varepsilon_\alpha(b_2)}{q}} } \right)$$
In the end:
$$ e^c_\alpha \cdot e^{c'}_\alpha \cdot \left( b_1 \otimes_q b_2 \right)
 = e^{c_1 + c_1'} \cdot b_1 \otimes_q e^{c_2 + c_2'} \cdot b_2
 = e^{c   + c   } \cdot \left( b_1 \otimes_q b_2 \right)$$

\end{proof}

\subsection{Geometric varieties and coordinates}
Geometric lifting is the general idea that computations in the tropical world using the semi-field $(\R, \min, +)$ have analogues in the  geometric world using $(\R_{>0}, +, .)$. More information is provided in subsection \ref{subsection:semifields_and_maslov}. Here, the reader will only need to have in mind that the rational maps in the sense of semi-fields, which are the rational and substraction free maps, preserve positivity and can tropicalized.

\index{$U_{>0}^{w_0}$: Geometric Lusztig variety}
\begin{definition}[Lusztig variety]
\label{def:geom_lusztig_variety}
Define the geometric Lusztig variety as:
$$U_{>0}^{w_0} := U \cap B w_0 B \cap G_{\geq 0}$$ 
\end{definition}
In the Cartan-Killing type $A$, $U_{>0}^{w_0}$ is nothing but the set of totally positive upper triangular matrices with unit diagonal. It is known since Lusztig that:
\begin{thm}[\cite{bib:BZ01}]
Every reduced word ${\bf i} \in R(w_0)$ gives rise to a bijection:
\label{lbl:geom_lusztig_parameters}
$$
 \begin{array}{cccc}
x_{\mathbf{i}}: & \mathbb{R}_{>0}^m & \rightarrow & U^{w_0}_{>0}\\
                & (t_1, \dots, t_m) & \mapsto     & x_{i_1}(t_1) \dots x_{i_m}(t_m)
 \end{array}
$$
Moreover, for ${\bf i}, {\bf i'} \in R(w_0)$, the maps $x_{\mathbf{i'}}^{-1} \circ x_{\mathbf{i}}: \R_{>0}^m \rightarrow \R_{>0}^m$ are rational and substraction free.
\end{thm}

Such a name is legitimate because changes of parametrization in the $G^\vee$ canonical basis where proven to be the tropicalization of the rational substraction free expression (\cite{bib:BZ01}, theorem 5.2):
\begin{align}
\label{eqn:geom_change_of_param_lusztig}
R_{\bf i,i'}\left( \bf t \right) & = x_{\bf i'}^{-1} \circ x_{\bf i}\left( \bf t \right)
\end{align}

\index{$C_{>0}^{w_0}$: Geometric Kashiwara (or string) variety}
\begin{definition}[Kashiwara (or string) variety]
\label{def:geom_kashiwara_variety}
Define the geometric string variety $C_{>0}^{w_0}$ as:
$$ C_{>0}^{w_0} := U \bar{w_0} U \cap B \cap G_{\geq 0}$$ 
\end{definition}
Of course, this definition depends on $\bar{w}_0$, our choice of representative for the longest element $w_0$ in the Weyl group.

\begin{thm}[\cite{bib:BZ01}]
\label{lbl:geom_string_parameters}
Every reduced word ${\bf i} \in R(w_0)$ gives rise to a bijection:
$$
 \begin{array}{cccc}
x_{-\mathbf{i}}: & \mathbb{R}_{>0}^m & \rightarrow & C_{>0}^{w_0}\\
                 & (c_1, \dots, c_m) & \mapsto     & x_{-i_1}(c_1) \dots x_{-i_m}(c_m)
 \end{array}
$$
Moreover, for ${\bf i}, {\bf i'} \in R(w_0)$, the maps $x_{-\mathbf{i'}}^{-1} \circ x_{-\mathbf{i}}: \R_{>0}^m \rightarrow \R_{>0}^m$ are rational and substraction free.
\end{thm}

In the same fashion, changes of parametrizations in string coordinates for the canonical basis (in $G^\vee$) are given by the tropicalization of (\cite{bib:BZ01}, theorem 5.2):
\begin{align}
\label{eqn:geom_change_of_param_string}
R_{\bf -i,-i'}\left( \bf c \right) & = x_{\bf -i'}^{-1} \circ x_{\bf -i} \left( \bf c \right)
\end{align}

A useful relationship between the maps $x_{\bf i}$ and $x_{\bf -i}$ is the following:
\begin{lemma}( \cite{bib:BZ01} Lemma 6.1 )
\label{lemma:change_of_coordinates_UC}
Let ${\bf i } = \left( i_1, \dots, i_j \right) \in R(w)$ a reduced expression and $\left( \beta^\vee_1, \dots, \beta^\vee_j \right)$ an associated positive coroots enumeration. Then the following statements are equivalent:
$$(i)   \left( x_{-i_1}(c_1) \dots x_{-i_j}(c_j) \right)^T = c_1^{-\alpha_{i_1}^\vee} \dots c_j^{-\alpha_{i_j}^\vee} x_{i_j}(t_j) \dots x_{i_1}(t_1) $$
$$(ii)  \forall 1 \leq k \leq j, t_k = c_k \prod_{l<k} c_l^{ \alpha_{i_k}(\alpha_{i_l}^\vee) }$$
$$(iii) \forall 1 \leq k \leq j, c_k = t_k \prod_{l<k} t_l^{ \beta_{k}(\beta_{l}^\vee) }$$
Moreover:
$$ \prod_{k=1}^j c_k^{\alpha_{i_k}^\vee} = \prod_{k=1}^j t_k^{-w^{-1} \beta_k^\vee}$$
\end{lemma}
\begin{proof}
The equivalence between the two first statements is immediate using commutation relations:
\begin{align*}
  & \left( x_{-i_1}(c_1) \dots x_{-i_j}(c_j) \right)^T\\
= &  c_j^{-\alpha_{i_j}^\vee} x_{i_j}(c_j) \dots c_1^{-\alpha_{i_1}^\vee} x_{i_1}(c_1)\\
= &  c_1^{-\alpha_{i_1}^\vee} \dots c_j^{-\alpha_{i_j}^\vee} \prod_{k=0}^{j-1} x_{i_{j-k}}(c_{j-k} \prod_{l<j-k} c_l^{ \alpha_{i_k}(\alpha_{i_l}^\vee) } )
\end{align*}
The equivalence between  the two last statements can be proved by induction over $j$. For $j=1$, it is immediate. Then, for $j \geq 1$, by induction hypothesis:
$$c_j = t_j \prod_{k=1}^{j-1} c_k^{-\alpha_{i_j}(\alpha_{i_k}^\vee)}
      = t_j \prod_{k=1}^{j-1} \left( t_k \prod_{l=1}^{k-1} t_l^{\beta_k(\beta_l^\vee)} \right)^{-\alpha_{i_j}(\alpha_{i_k}^\vee)} $$
Rearranging the double product gives:
\begin{align*}
c_j = & t_j \left( \prod_{k=1}^{j-1} t_k^{-\alpha_{i_j}(\alpha_{i_k}^\vee) } \right) \prod_{l=1}^{j-1} t_l^{-\sum_{k=l}^{j-1} \beta_k(\beta_l^\vee) \alpha_{i_j}(\alpha_{i_k}^\vee) }\\
= & t_j \prod_{k=1}^{j-1} t_k^{ -\alpha_{i_j}\left( \alpha_{i_k}^\vee + \sum_{l=k}^{j-1} \beta_l(\beta_k^\vee) \alpha_{i_l}^\vee \right) }
\end{align*}
Then using the second identity in lemma \ref{lbl:kumar} with $\lambda = \beta^\vee_k$:
\begin{align*}
  & \alpha_{i_k}^\vee + \sum_{l=k}^{j-1} \beta_l(\beta_k^\vee) \alpha_{i_l}^\vee\\
= & \alpha_{i_k}^\vee + \left( s_{i_1} \dots s_{i_k} \right)^{-1} \beta_k^\vee - \left( s_{i_1} \dots s_{i_{j-1}} \right)^{-1} \beta_k^\vee\\
= & - \left( s_{i_1} \dots s_{i_{j-1}} \right)^{-1} \beta_k^\vee
\end{align*}
In the end, as announced:
\begin{align*}
c_j = & t_j \prod_{k=1}^{j-1} t_k^{ \alpha_{i_j}\left( \left( s_{i_1} \dots s_{i_{j-1}} \right)^{-1} \beta_k^\vee \right) }\\
    = & t_j \prod_{k=1}^{j-1} t_k^{ \beta_j( \beta_k^\vee ) }\\
\end{align*}
The last equality is a straightforward calculation:
\begin{align*}
\prod_{k=1}^j c_k^{\alpha_{i_k}^\vee} 
= & \prod_{k=1}^j \left( t_k^{\alpha_{i_k}^\vee} \prod_{l < k} t_l^{\beta_k(\beta_l) \alpha_{i_k}^\vee}\right)\\
= & \prod_{k=1}^j t_k^{\alpha_{i_k}^\vee + \sum_{l=k+1}^j \beta_l(\beta_k) \alpha_{i_l}^\vee}
\end{align*}
Using again the lemma \ref{lbl:kumar}, we have:
$$\alpha_{i_k}^\vee + \sum_{l=k+1}^j \beta_l(\beta_k) \alpha_{i_l}^\vee = -w^{-1} \beta_k^\vee$$
\end{proof}

\paragraph{From Lusztig parametrization to string coordinates:}

Define 
\begin{align}
\label{eqn:geom_lusztig_2_kashiwara}
\forall u \in U \cap B w_0 B, \eta^{e, w_0}\left( u \right) & := [\bar{w}_0^{-1} u^T]_{-0}^{-1} = [\bar{w}_0^{-1} u^T]_{+} \bar{w}_0 S\left( u \right)^\iota
\end{align}
\begin{align}
\label{eqn:geom_kashiwara_2_lusztig}
\forall v \in B \cap U \bar{w}_0 U, \eta^{w_0, e}\left( v \right)& := [\left( \bar{w}_0 v^T\right)^{-1}]_+ 
\end{align}

\begin{thm}[ \cite{bib:BZ01}, corollary 5.6 ]
\label{thm:geom_from_lusztig_to_kashiwara}
The map $\eta^{e, w_0}$ is a bijection from $U \cap B w_0 B$ to $B \cap U \bar{w_0} U$ and restricts to a bijection from $U_{>0}^{w_0}$ to $C_{>0}^{w_0}$. The inverse map is $\eta^{e, w_0}$.
\end{thm}

$\eta^{e, w_0}$ gives then the correspondence between the Lusztig variety and the string variety. The tropicalization of this correspondence is a very interesting map. In the notations of section \ref{section:preliminaries_canonical_basis}, we have:
\begin{thm}[ \cite{bib:BZ01}, theorem 5.7 ]
\label{thm:tropical_from_lusztig_to_kashiwara}
Changes of parametrization for the $G^\vee$ canonical basis are obtained by tropicalizing the following rational substraction free expressions. Going from ${\bf i}$-Lusztig parameters to ${\bf i'}$-string parameters is achieved by tropicalizing:
$$ x_{\bf-i'} \circ \eta^{e, w_0} \circ x_{\bf i}$$
Conversely, in order to obtain ${\bf i}$-string parameters from ${\bf i'}$-Lusztig parameters, tropicalize:
$$ x_{\bf i'} \circ \eta^{w_0, e} \circ x_{\bf-i}$$
\end{thm}

\paragraph{Geometric crystal elements:}

\index{$\Bc(\lambda)$: Geometric crystal with highest weight $\lambda$}
\index{$\Bc$: Disjoint union of all highest weight geometric crystals}
\begin{definition}[Geometric crystals]
\label{def:geom_crystal}
Define the geometric crystal of highest weight $\lambda \in \afrak$ as the set:
$$\Bc\left(\lambda\right) := C_{>0}^{w_0} e^{\lambda}$$
The union of all highest weight crystals will be denoted by $\Bc$, which is nothing but the set of totally positive elements in $B$:
$$ \Bc = \bigsqcup_{ \lambda \in \afrak} \Bc(\lambda) $$
\end{definition}
Later, we will see that $\Bc$ is an abstract crystal in the sense of definition \ref{def:crystal}. Also, notice that the dependence in $\lambda$ is given by a flat torus fibration and the highest weight can easily be recovered from any element $x \in \Bc$ using the highest weight map:

\index{$hw$: Highest weight map on $\Bc$}
\begin{definition}[Highest and lowest weight, \cite{bib:BK06} relation 1.6]
\label{def:highest_lowest_weight}
Define the highest weight map $hw: \Bc \rightarrow \mathfrak{a}$ by:
$$ \forall x \in \Bc, hw(x) := \log[ \bar{w}_0^{-1} x ]_0$$
The lowest weight is given by:
$$ \forall x \in \Bc, lw(x) := \log[ \bar{w}_0^{-1} x^\iota ]_0^\iota = w_0 hw(x)$$
\end{definition}

Notice that highest weight crystals are disjoint in $\Bc$ and that $hw\left( \Bc(\lambda) \right) = \{ \lambda \}$

\begin{properties}
\phantomsection
\label{properties:hw}
\begin{itemize}
 \item[(i)] $hw$ can be extented to $B^+ w_0 B^+$ as:
$$\forall \left( z,u \right) \in U \times U, t \in H, hw\left( z \bar{w}_0 t u \right) = \log(t) $$
 \item[(ii)] $hw$ is an $U \times U$-invariant function.
 \item[(iii)] $$\forall \left(x, y \right) \in \mathfrak{a}^2, \forall g \in B^+ w_0 B^+, hw\left( e^x g e^y \right) = w_0 x + hw\left(g\right) + y$$ 
\end{itemize}
\end{properties}
\begin{proof}
\begin{itemize}
 \item[(i)] If $g = z \bar{w}_0 t u \in B^+ w_0 B^+$
\begin{align*}
hw\left( g \right) & = \log [\bar{w}^{-1}_0  g]_0\\
= & \log [\bar{w}^{-1}_0  z \bar{w}_0 t u]_0\\
= & \log(t)
\end{align*}
 \item[(ii)] Immediate from (i)
 \item[(iii)] 
\begin{align*}
hw\left( e^x g e^y \right) & = \log [\bar{w}^{-1}_0  e^x g e^y]_0\\
= & \log [\bar{w}^{-1}_0  e^x \bar{w}_0 \bar{w}^{-1}_0 g e^y]_0\\
= & \log [e^{w_0 x} \bar{w}^{-1}_0 g e^y]_0\\
= & w_0 x + hw\left(g\right) + y
\end{align*}
\end{itemize}
\end{proof}

Every element $x \in \Bc\left(\lambda\right)$ can be written using a certain associated parameter. The letters $u$ and $z$ will usually refer to an element in $U_{>0}^{w_0}$ and the letter $v$ will usually refer to an element in $C_{>0}^{w_0}$. An expression we will often use is:
\begin{thm}
For $x \in \Bc\left(\lambda\right)$, one can write uniquely:\\
$$ x = z \bar{w_0} e^{\lambda} u, z \in U^{w_0}_{>0}, u \in U^{w_0}_{>0}$$
Mapping $x$ to $z$ (resp. $u$) is a bijection from $\Bc(\lambda)$ to $U^{w_0}_{>0}$. The former will be refered to as the Lusztig parameter associated to $x$, and the latter the twisted Lusztig parameter.
\end{thm}
\begin{proof}
The existence is a consequence of the definition of $\Bc(\lambda)$. The uniqueness comes, as we will see, from exhibiting inverse maps that preserve total positivity.
\end{proof}

There is also the possibility of using a parameter $v \in C_{>0}^{w_0}$ that we will call the string or Kashiwara parameter associated to $x$. Such names are justified by the fact that these choices give a geometric lifting of the parametrizations for crystal bases.\\

From the previous theorem, it is obvious that $hw^{-1}( \{ \lambda \} ) = \Bc(\lambda)$. In all the following formulas, the group elements considered belong to the double Bruhat cell $B \cap B^+ w_0 B^+$ and thus, every Gauss decomposition that we use is allowed.\\

\index{$\varrho^L, \varrho^K, \varrho^T$: Lusztig, Kashiwara or twisted Lusztig parameter corresponding to an element in $\Bc$}
\begin{definition}[Parameters associated to a crystal element]
\label{def:crystal_parameter}
Define the following maps on $\Bc$:
$$
\begin{array}{cccc}
\varrho^L : &  \Bc                           & \longrightarrow & U^{w_0}_{>0}\\
	    &  x = z \bar{w}_0 e^{\lambda} u &   \mapsto       & z = [\bar{w}_0^{-1} x^\iota]_+^\iota
\end{array}$$

$$
\begin{array}{cccc}
\varrho^K : &  \Bc                           & \longrightarrow & C^{w_0}_{>0}\\
	    &  x                             &   \mapsto       & v = [\bar{w}_0^{-1} [x]_-]_{0+}^T
\end{array}$$

$$
\begin{array}{cccc}
\varrho^T : &  \Bc                           & \longrightarrow & U^{w_0}_{>0}\\
	    &  x = z \bar{w}_0 e^{\lambda} u &   \mapsto       & u = [\bar{w}_0^{-1} x]_+
\end{array}$$

For $x \in \Bc$, the group elements $z = \varrho^L(x)$, $v = \varrho^K(x)$ and $z = \varrho^T(x)$ will be referred to as the Lusztig, Kashiwara and twisted Lusztig parameters associated to $x$.
\end{definition}

The following property shows that all highest weight crystals share the same parametrizations, hinting to the compatibility properties of the canonical basis. Recall that $\eta^{w_0, e}$ is given in equation (\ref{eqn:geom_kashiwara_2_lusztig}).

\index{$b_\lambda^L, b_\lambda^K, b_\lambda^T$: Geometric crystal element with corresponding Lusztig, Kashiwara or twisted Lusztig parameter}
\begin{proposition}
\label{proposition:crystal_param_maps}
Once restricted to $\Bc(\lambda)$ the maps $\varrho^L$, $\varrho^K$ and $\varrho^T$ are invertible with inverses:
$$
\begin{array}{cccc}
b_\lambda^L: &  U^{w_0}_{>0} & \longrightarrow & \Bc(\lambda)\\
	     &  z            &   \mapsto       & x = [z \bar{w}_0]_{-0} e^\lambda = z \bar{w}_0 e^\lambda \left( e^{-\lambda} [z \bar{w}_0]_+^{-1} e^{\lambda}\right)
\end{array}$$

$$
\begin{array}{cccc}
b_\lambda^K: &  C^{w_0}_{>0} & \longrightarrow & \Bc(\lambda)\\
	     &  v            &   \mapsto       & x = [\eta^{w_0, e}\left( v \right) \bar{w}_0]_{-0} e^\lambda = [(\bar{w}_0 v^T)^{-1} ]_+ \bar{w}_0 v^T [v^T]_0^{-1}e^\lambda
\end{array}$$

$$
\begin{array}{cccc}
b_\lambda^T: &  U^{w_0}_{>0} & \longrightarrow & \Bc(\lambda)\\
	     &  u            &   \mapsto       & x = S \circ \iota \left( e^{-\lambda} [ \bar{w}_0^{-1} u^T ]_{+} e^{\lambda} \right) \bar{w}_0 e^\lambda u 
\end{array}$$
\end{proposition}
\begin{rmk}
It is easy to see that with such definitions, $\eta^{e, w_0}\left( z \right) = v$.
\end{rmk}

In the sequel, we will try to stick to the letters $z$, $v$ and $u$ when dealing with each choice of parameter. The figure \ref{fig:geom_parametrizations} shows the different charts for $\Bc(\lambda)$, together with the inverse maps $b^L_\lambda$, $b^K_\lambda$ and $b^T_\lambda$. Since those parametrizations will be important to us, we reproduce that commutative diagram among other ones in appendix \ref{appendix:parametrizations_reminder}.

\begin{figure}[htp!]
\centering
\begin{tikzpicture}[baseline=(current bounding box.center)]
\matrix(m)[matrix of math nodes, row sep=5em, column sep=6em, text height=3ex, text depth=1ex, scale=2]
{
                    & x \in \Bc(\lambda) &                    \\
 z \in U^{w_0}_{>0} & v \in C^{w_0}_{>0} & u \in U^{w_0}_{>0} \\
};
\path[->, font=\scriptsize] (m-1-2) edge node[above]{$\varrho^L$} (m-2-1);
\path[->, font=\scriptsize] (m-1-2) edge node[right]{$\varrho^K$} (m-2-2);
\path[->, font=\scriptsize] (m-1-2) edge node[above]{$\varrho^T$} (m-2-3);

\draw [->] (m-2-1) to [bend left=30]  node[above, swap]{$b^L_\lambda$} (m-1-2);
\draw [->] (m-2-2) to [bend left=45]  node[auto, swap]{$b^K_\lambda$} (m-1-2);
\draw [->] (m-2-3) to [bend right=30] node[above, swap]{$b^T_\lambda$} (m-1-2);

\draw [->] (m-2-1) to [bend left=0 ]  node[above, swap]{$\eta^{e, w_0}$} (m-2-2);
\draw [->] (m-2-2) to [bend left=30]  node[auto , swap]{$\eta^{w_0, e}$} (m-2-1);
\end{tikzpicture}
\caption{Charts for the highest weight geometric crystal $\Bc(\lambda)$}
\label{fig:geom_parametrizations}
\end{figure}
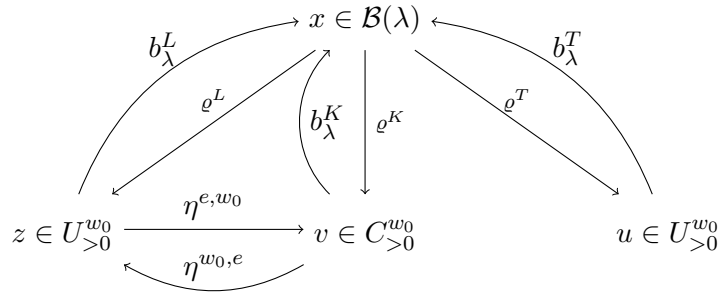

\begin{proof}
The fact that these charts preserve total positivity is dealt with later in theorem \ref{thm:charts_positivity}. Let us start by writing:
$$ x = z \bar{w}_0 e^{\lambda} u $$

Lusztig parameters: It is easy to see that since $x \in B$, we have $x = [x]_{-0} = [z \bar{w}_0]_{-0} e^{\lambda}$. The second expression is obtained directly by the identity $[z \bar{w}_0]_{-0} = z \bar{w}_0 [z \bar{w}_0]_+^{-1}$.\\
In order to obtain $z$ from $x$, making use of the anti-automorphism $\iota$ gives:
$$ x^\iota = u^\iota e^{-\lambda} \bar{w}_0 z^\iota$$
Hence:
\begin{align*}
[ \bar{w}_0^{-1} x^\iota ]_+^\iota & = [ \bar{w}_0^{-1} u^\iota e^{-\lambda} \bar{w}_0 z^\iota ]_+^\iota\\
& = \left( z^\iota \right)^\iota\\
& = z
\end{align*}

Twisted Lusztig parameters are treated in a similar way. Write:
\begin{align*}
x = & [x^\iota]_{-0}^\iota\\
= & [ u^\iota e^{-\lambda} \bar{w}_0 ]_{-0}^\iota\\
= & \left([ u^\iota e^{-\lambda} \bar{w}_0 ]_{+}^\iota\right)^{-1} \bar{w}_0 e^\lambda u \\
= & S \circ \iota (y) \bar{w}_0 e^\lambda u
\end{align*}
where we have used the involutive automorphism $S \circ \iota$ defined in section \ref{section:involutions}:
\begin{align*}
y & = S \circ \iota \left( \left([ u^\iota e^{-\lambda} \bar{w}_0 ]_{+}^\iota\right)^{-1} \right)\\
& = \bar{w}_0^{-1} [ \bar{w}_0^{-1} e^{-\lambda} \left(u^\iota\right)^T ]_{-}^\iota \bar{w}_0\\
& = [ \bar{w}_0^{-1} e^{-\lambda} u^T e^{\lambda} ]_{+}\\
& = e^{-\lambda} [ \bar{w}_0^{-1} u^T ]_{+} e^{\lambda}
\end{align*}
And in order to obtain $u$ from $x$, write $[\bar{w}_0^{-1} x]_+ = [\bar{w}_0^{-1} z \bar{w}_0 e^\lambda u ]_+ = u$.

Finally, for the Kashiwara parameter, if $x \in \Bc(\lambda)$ and $\varrho^K(x) = v$, then:
\begin{align*}
\eta^{w_0, e}(v) = & [ \left( \bar{w}_0 v^T \right)^{-1} ]_+\\
= & [ [\bar{w}_0^{-1} [x]_- ]_{0+}^{-1} \bar{w}_0^{-1} ]_+\\
= & [ [x]_-^{-1} \bar{w}_0 [\bar{w}_0^{-1} [x]_- ]_{-} \bar{w}_0^{-1} ]_+\\
= & \bar{w}_0 [\bar{w}_0^{-1} [x]_- ]_{-} \bar{w}_0^{-1}
\end{align*}
Therefore:
\begin{align*}
x = & [ \bar{w}_0 [ \bar{w}_0^{-1} x ]_{-0} ]_{-0}\\
= & [ \bar{w}_0 [ \bar{w}_0^{-1} [x]_- ]_{-} ]_{-0} e^{\lambda}\\
= & [ \eta^{w_0, e}(v) \bar{w}_0 ]_{-0} e^\lambda
\end{align*}
Another possible expression is indeed:
\begin{align*}
x = & \eta^{w_0, e}(v) \bar{w}_0 [ \eta^{w_0, e}(v) \bar{w}_0 ]_{+}^{-1} e^\lambda\\
= & \eta^{w_0, e}(v) \bar{w}_0 [ \left( \bar{w}_0 v^T \right)^{-1} \bar{w}_0 ]_{+}^{-1} e^\lambda\\
= & \eta^{w_0, e}(v) \bar{w}_0 v^T [v^T]_0^{-1} e^\lambda
\end{align*}
\end{proof}

\begin{thm}
\label{thm:charts_positivity}
All maps $\varrho^L$, $\varrho^{K}$ and $\varrho^T$ (and their inverses) are rational and substraction free once written in coordinates. Thus they preserve total positivity.
\end{thm}
\begin{proof}
Notice that:
$$ \eta^{e, w_0} \circ \varrho^L = \varrho^K $$
$$ \iota \circ \varrho^L \circ \iota = \varrho^T$$
We already know that $\eta^{e, w_0}$ and its inverse are rational substraction free once written in the appropriate charts (theorems \ref{thm:geom_from_lusztig_to_kashiwara} and \ref{thm:tropical_from_lusztig_to_kashiwara}). The same goes for $\iota$ as $x_{\bf i^{op}}^{-1} \circ \iota \circ x_{\bf i} = id$ and $x_{\bf i'}^{-1} \circ \iota \circ x_{\bf i}$. Therefore, the theorem will be proved by dealing only with the mappings $\varrho^L$ and $b_\lambda^L$.

In order to further reduce the problem, introduce the twist map studied in \cite{bib:BZ97}:
$$
\begin{array}{cccc}
\eta_{w_0} : & \left( U \cap B w_0 B \right) & \longrightarrow & \left( U \cap B w_0 B \right)\\
	    &  z                             &   \mapsto       & [ \bar{w}_0^{-1} z^T]_+
\end{array}$$
It is easy to see that $\eta_{w_0}^{-1} = \iota \circ \eta_{w_0} \circ \iota$. Moreover, one can show that $x_{\bf i'}^{-1} \circ \eta_{w_0} \circ x_{\bf i}$ is rational and subtraction free, for every reduced words ${\bf i'}$ and ${\bf i}$.

Technically, in proposition \ref{proposition:crystal_param_maps}, we only proved that the following correspondence for the Lusztig parametrization is bijective:
$$
\begin{array}{ccc}
  \left( B \cap U \bar{w_0} U \right) e^\lambda & \longrightarrow & U \cap B w_0 B\\
  x = [z \bar{w}_0]_{-0} e^{\lambda}            &   \mapsto       & z = [\bar{w}_0^{-1} x^\iota]_+^\iota
\end{array}$$
Therefore, after getting rid of the dependence in $\lambda$, we will consider:
$$
\begin{array}{cccc}
\varphi: & B \cap U \bar{w_0} U   & \longrightarrow & U \cap B w_0 B\\
         & x = [z \bar{w}_0]_{-0} &   \mapsto       & z = [\bar{w}_0^{-1} x^\iota]_+^\iota
\end{array}$$
and prove that $x_{\bf i}^{-1} \circ \varphi \circ x_{\bf -i}(c_1, \dots, c_m)$ is rational and subtraction free in the variables $(c_1, \dots, c_m)$, hence preserving total positivity. Applying the monomial change of variable in lemma \ref{lemma:change_of_coordinates_UC}:
$$   \left( x_{\bf-i}(c_1, \dots, c_m) \right)^T = c_1^{-\alpha_{i_1}^\vee} \dots c_m^{-\alpha_{i_m}^\vee} x_{\bf i^{op}}(t_m, \dots, t_1) $$
we obtain $x_{\bf i}^{-1} \circ \eta_{w_0}^{-1} \circ x_{\bf i^{op}}(t_m, \dots, t_1)$, which we know is rational and subtraction free, as well as its inverse. All intermediate rearrangements were also rational and substraction free, hence the result.
\end{proof}

\subsection{The weight map}

\index{$\gamma$: Weight map}
\begin{definition}
\label{def:geom_weight_map}
Define the weight map $\gamma: \Bc\left( \lambda \right) \rightarrow \afrak$ by:
$$ e^{\gamma(x)} = [x]_0$$ 
\end{definition}

This weight map is the geometric analogue of the classical weight map for crystal bases. The similarity is particularly obvious when comparing the following result with equations \ref{eqn:weight_lusztig} and \ref{eqn:weight_string}. It uses the dual root system as the geometric crystal on $G$ has the properties of the  canonical basis for $G^\vee$.
\begin{thm}
\label{thm:geom_weight_map}
Let $x \in \Bc(\lambda)$, $z = \varrho^L(x)$, $v = \varrho^{K}(x)$ and $u = \varrho^T(x)$. Write for ${\bf i} \in R(w_0)$
$$ z = x_{ \bf i}\left( t_1, \dots, t_m \right)$$
$$ v = x_{-\bf i}\left( c_1, \dots, c_m \right)$$
$$ u = x_{ \bf i}\left( t_1', \dots, t_m' \right)$$
And $e^{-\tilde{t}_k} = t_k$, $e^{-\tilde{t}_k'} = t_k'$ and $e^{-\tilde{c}_k} = c_k$. Then, in terms of coordinates, the weight map is given by:
\begin{align*}
\gamma(x) & = \lambda - \sum_{k=1}^m \tilde{t}_k \beta^\vee_k\\
& = w_0 \left( \lambda - \sum_{k=1}^m \tilde{t}'_k \beta_{k}^\vee \right)\\
& = \lambda - \sum_{k=1}^m \tilde{c}_k \alpha^\vee_{i_k}
\end{align*}
\end{thm}
\begin{proof}
For $x = [z \bar{w}_0]_{-0} e^\lambda$ with $z = x_{\bf i}\left( t_1, \dots, t_m\right)$. We apply lemma \ref{lemma:change_of_coordinates_UC} to the opposite reduced word ${\bf i^{op}} = \left( i_m, \dots, i_1 \right)$. As such, the positive roots enumeration is reversed as well as the order of the parameters $t_1, \dots, t_m$. There is a $b = x_{-\bf i^{op}}\left(c_1, \dots, c_m \right)^T \in B^+ \cap N \bar{w}_0^{-1} N$ such that:
$$z = \left( \prod_{j=1}^m t_j^{-w_0 \beta_{{\bf i^{op}},m-j+1}^\vee} \right) b$$
The exponents can be simplified as:
\begin{align*}
  & -w_0 \beta_{ {\bf i^{op}} ,m-j+1 }^\vee\\
= & -w_0 s_{i_m} \dots s_{i_{j+1}} \alpha_{i_j}^\vee\\
= & s_{i_1} \dots s_{i_{j-1}} \alpha_{i_j}^\vee\\
= & \beta_j^\vee
\end{align*}
Hence:
$$u = \left( \prod_{j=1}^m t_j^{\beta_j^\vee} \right) b$$
Because $b \bar{w}_0 \in N U$, we have:
\begin{align*}
  & [x]_0\\
= & [u \bar{w}_0]_0 e^{\lambda }\\
= & \left[\left( \prod_{j=1}^m t_j^{\beta_j^\vee} \right) b \bar{w}_0\right]_0 e^{\lambda}\\
= & e^{\lambda} \left( \prod_{j=1}^m t_j^{\beta_j^\vee} \right) 
\end{align*}

We can deduce the weight map expression in terms of $\left( t'_1, t_2', \dots, t_m'\right)$ quite simply from the above proof. Notice that applying $\iota$ to $x$, changes $u$ to $u'^\iota$, $\lambda$ to $-w_0 \lambda$ and $\gamma(x)$ to $-\gamma(x)$. As such, using the expression found for the weight map in Lusztig coordinates, while considering the opposite word ${\bf i^{op}}$:
$$ -\gamma(x) = -w_0 \lambda - \sum_{k=1}^m \tilde{t}'_{m-k+1} \beta_{{\bf i^{op}},k}^\vee $$
Hence:
\begin{align*}
 \gamma(x) & = w_0 \lambda + \sum_{k=1}^m \tilde{t}'_{m-k+1} \beta_{{\bf i^op},k}^\vee\\
& = w_0 \lambda + \sum_{k=1}^m \tilde{t}'_{m-k+1} (-w_0 s_{i_1} \dots s_{i_{m-k-1}}) \alpha_{i_{m-k}}^\vee \\
& = w_0 \left( \lambda - \sum_{k=1}^m \tilde{t}'_k \beta_{k}^\vee \right)
\end{align*}

In the string parametrization $v = \varrho^S(x) = x_{-\bf i}\left( c_1, \dots, c_m\right)$. By definition \ref{def:crystal_parameter}:
$$ v = [ \bar{w}_0^{-1} [x]_- ]_{0+}^T$$
Hence:
\begin{align*}
\prod_{j=1}^m c_j^{-\alpha_{i_j}^\vee}
= & [v]_0 \\
= & [\bar{w}_0^{-1} [x]_-]_0\\
= & [\bar{w}_0^{-1} x]_0 [x]_0^{-1}\\
= & e^{\lambda} [x]_0^{-1}
\end{align*}
Rearranging the equation yields the result.

\end{proof}

\subsection{Examples}

We illustrate the previous coordinate systems and maps by a few examples for different semi-simple groups. We will take $x \in \Bc$ and write in coordinates:
$$ z = \varrho^L(x) \in U^{w_0}_{>0}$$
$$ v = \varrho^K(x) \in C^{w_0}_{>0}$$

\subsubsection{$A_1$-type:}
$$G = SL_2 = \left\{ x  = \begin{pmatrix} a & c  \\ b & d \end{pmatrix} \ | ad - bc = 1 \right\}$$
$$\gfrak = \mathfrak{sl}_2 = \left\{ x \in M_2(\C) \ | tr(x) = 0 \right\}$$
$$H = \left\{ x  = \begin{pmatrix} a & 0  \\ 0 & a^{-1} \end{pmatrix}, a \in \C^* \right\}$$
$$\hfrak = \C \alpha^\vee $$
where $\alpha^\vee = \begin{pmatrix} 1 & 0  \\ 0 & -1 \end{pmatrix}$.\\

The disjoint union of all highest weight crystals is $\Bc$:
$$\Bc = \left\{ \begin{pmatrix} a     & 0  \\ b & a^{-1} \end{pmatrix} \ | \ a>0, b>0 \right\}$$
For $x = \begin{pmatrix} a     & 0  \\ b & a^{-1} \end{pmatrix} \in \Bc$, if:
$$ \lambda = hw( x ) $$
$$ z = \begin{pmatrix} 1     & t  \\ 0 & 1      \end{pmatrix}$$
$$ v = \begin{pmatrix} c^{-1} & 0 \\ 1 & c     \end{pmatrix}$$
then, in terms of the matrix $x$, we have:
$$ \lambda = \log(b) \alpha^\vee$$
$$ t = c = \frac{a}{b}$$

\begin{rmk}[The reductive case: $GL_2$]
\label{lbl:reductive_gl2_example}
If $G = GL_2$, then one can factor the group thanks to the determinant $G \approx SL_2 \times \C^*$. The extra dimension can be treated separately.
\end{rmk}

\subsubsection{$A_2$-type:}
$$G = SL_3(\C) $$
$$\gfrak = \mathfrak{sl}_2 = \left\{ x \in M_3(\C) \ | tr(x) = 0 \right\}$$
$$H = \left\{ x  = \begin{pmatrix} a & 0 & 0 \\ 0 & b & 0 \\ 0 & 0 & c \end{pmatrix}, abc = 1, (a,b,c) \in (\C^*)^3 \right\}$$
$$\hfrak = \C \alpha_1^\vee \oplus \C \alpha_2^\vee$$
where $\alpha_1^\vee = \begin{pmatrix} 1 & 0 & 0 \\ 0 & -1 & 0 \\ 0 & 0 &  0 \end{pmatrix}$ and $\alpha_2^\vee = \begin{pmatrix} 0 & 0 & 0 \\ 0 &  1 & 0 \\ 0 & 0 & -1  \end{pmatrix}$.\\

The disjoint union of all highest weight crystals is given by lower triangular totally positive matrices:
$$\Bc = \left\{ \begin{pmatrix} a  & 0 & 0 \\ b & c & 0 \\ d & e & f \end{pmatrix} \ | \  acf=1; a,b,c,d,e,f>0; be-dc>0  \right\}$$
For a crystal element $x = \begin{pmatrix} a  & 0 & 0 \\ b & c & 0 \\ d & e & f \end{pmatrix} \in \Bc$, if:
$$ \lambda = hw(x)$$
$$ z = x_{\bf 121}(t_1, t_2, t_3)
     = \begin{pmatrix} 1  & t_1 + t_3 & t_1 t_2 \\ 0 & 1 & t_2 \\ 0 & 0 & 1 \end{pmatrix}$$
$$ v = x_{\bf-121}(c_1, c_2, c_3)
     = \begin{pmatrix} \frac{1}{c_1 c_3}  & 0 & 0 \\ c_3^{-1}+\frac{c_1}{c_2} & \frac{c_1 c_3}{c_2} & 0 \\ 1 & c_3 & c_2 \end{pmatrix}$$
then the correspondence $\eta^{e, w_0}\left( z \right) = v$ gives:
$$ \left( t_1, t_2, t_3 \right) = \left( c_1, c_3, c_2 c_3^{-1}\right)$$
Moreover, we have:
$$ e^\lambda = \begin{pmatrix} d & 0 & 0 \\ 0 & \frac{be-dc}{d} & 0 \\ 0 & 0 & \frac{1}{be-dc} \end{pmatrix} $$
\begin{align*}
x & = \begin{pmatrix} t_1 t_2  & 0 & 0 \\ t_2 & t_3 t_1^{-1} & 0 \\ 1 & \frac{t_1+t_3}{t_1 t_2} & \frac{1}{t_2 t_3} \end{pmatrix} e^\lambda\\ 
& = \begin{pmatrix} c_1 c_3 & 0 & 0 \\ c_3 & c_1^{-1} c_2 c_3^{-1} & 0 \\ 1 & \frac{c_1 c_3+c_2}{c_1 c_3^2} & \frac{1}{c_2} \end{pmatrix} e^\lambda
\end{align*}

\subsection{Geometric crystals in the sense of Berenstein and Kazhdan}

Now, we will explain why $\Bc$ is a positive geometric crystal in the sense of Berenstein and Kazhdan (\cite{bib:BK00}, \cite{bib:BK06}) using their framework. Their construction starts with the notion of unipotent bicrystal. In our case, the unipotent bicrystal is simply the cell $B \cap B^+ w_0 B^+$. Then it can be decorated with structural maps and endowed with a positive structure. This tantamounts to restricting the structural maps to $\Bc$, the totally positive part. The structural maps we inherit satisfy the axioms in definition \ref{def:crystal} and more.\\

Define the fundamental additive $N$-character $e^{\chi_\alpha^-}: N \rightarrow \C$ by:
$$ \forall (\alpha,\beta) \in \Delta^2, \forall t \in \R, \chi_\alpha^-(e^{t f_\beta})) = t \delta_{\alpha, \beta}$$
where $\delta_{.,.}$ is the Kronecker delta. It is naturally extended to $B$ by setting $\forall x \in B, \chi_\alpha^-(x) = \chi_\alpha^-([x]_-)$

\begin{thm}
\label{thm:geom_crystal_is_crystal}
The set $\Bc$ is an abstract geometric crystal once endowed with the structural maps:
\begin{itemize}
 \item  $\begin{array}{cccc}
         \gamma : &  \Bc  & \rightarrow & \mathfrak{a}     \\
		  &  g    & \mapsto     & \log\left([g]_0\right)
        \end{array}$
 \item For $x \in \Bc$:
       \begin{align*}
          \varepsilon_\alpha(x) & := \chi_\alpha^-(x)\\
          \varphi_\alpha(x) & := \chi_\alpha^- \left( x^\iota \right) = \alpha\left( \gamma(x) \right) + \varepsilon_\alpha(x)
       \end{align*}
 \item $e_\alpha^c \cdot x = [x_\alpha\left( \frac{e^c-1}{e^{\varepsilon_\alpha(x)}} \right) x]_{-0}
                      = x_\alpha\left( \frac{e^c-1}{e^{\varepsilon_\alpha(x)}} \right) x x_\alpha\left( \frac{e^{-c}-1}{e^{\varphi_\alpha(x)}} \right)$
\end{itemize}
\end{thm}
An important fact to keep in mind is that the previous group product uses a $U \times U$ action, and the right action is there to exactly balance the left action. As such, the resulting group element is still in $\Bc \subset B$.

So far, we made the choice of working directly with the totally positive elements. Only in this subsection, we will work outside of the totally positive varieties, in order to present Berenstein and Kazhdan's construction, from which theorem \ref{thm:geom_crystal_is_crystal} follows immediately.

\paragraph{ The unipotent bicrystal $({ \bf X }, p)$: }
A unipotent bicrystal is a couple $({ \bf X }, p)$ such that ${ \bf X }$ is a $U \times U$ variety, meaning a set with a right and left action of $U$, and $p: { \bf X } \rightarrow G$ a $U \times U$-equivariant application, meaning it is an application such that the action of $U \times U$ on ${ \bf X }$ and $G$ commute.

Here pick ${\bf X} := B^+ w_0 B^+$ with the natural left and right group action of $U$. And $p: B^+ w_0 B^+ \hookrightarrow G$ is the inclusion map.\\

\paragraph{ The unipotent crystal $X^-$:} Following \cite{bib:BK06} section 2, we define $X^-$ to be the unipotent crystal associated to $({ \bf X }, p)$ by:
$$ X^- := p^{-1}(B) = B \cap B^+ w_0 B^+$$
It is nothing but the largest double Bruhat cell inside of the Borel subgroup $B$. Here we are dealing with a unipotent bicrystal of type $w_0$ (\cite{bib:BK06}, claim 2.6).

\paragraph{ The positive structure $\Theta_{\bf X}$: } Now fix $\lambda \in \afrak$. For every ${\bf i}$ consider the charts:
$$
\begin{array}{cccc}
b_\lambda^L \circ x_{\bf i}: &  \R_{>0}^m & \rightarrow & \Bc(\lambda)\\
b_\lambda^K \circ x_{\bf i}: &  \R_{>0}^m & \rightarrow & \Bc(\lambda)\\
b_\lambda^T \circ x_{\bf i}: &  \R_{>0}^m & \rightarrow & \Bc(\lambda)\\
\end{array}$$
In the language of \cite{bib:BK06}, they are the restrictions to the positive octant $\R_{>0}^m$ of toric charts from $\C^m$ to $\left( B \cap U \bar{w}_0 U \right) e^{\lambda}$. Moreover, because of theorem \ref{thm:charts_positivity}, these toric charts are positively equivalent, defining the same positive structure $\Theta_{\bf X}$ on $\left( B \cap U \bar{w}_0 U \right) e^{\lambda}$. When looking only at the image of $\R_{>0}^m$ through those charts, one is dealing only with $\Bc(\lambda)$.

\paragraph{ The positive geometric crystals $\Fc( { \bf X }, p, \Theta_{\bf X} )$: }

By proposition 2.25 in \cite{bib:BK06}, the unipotent bicrystal $\left( { \bf X }, p\right)$ gives rise to a geometric crystal $\Fc( { \bf X }, p) = \left( { \bf X }^-, \gamma, \varphi_\alpha, \varepsilon_\alpha, e_\alpha^. | \alpha \in \Delta \right)$.  By lemma 3.30 in \cite{bib:BK06}, one gets a positive geometric crystal $\Fc( { \bf X }, p, \Theta_{\bf X} )$ meaning that these structural maps respect the positive structure. Therefore, we can  restrict them to $\Bc$, which proves theorem \ref{thm:geom_crystal_is_crystal}. Notice that the notation for $\varepsilon_\alpha$ and $\varphi_\alpha$ are reversed compared to \cite{bib:BK06}.

\paragraph{ Tensor product of geometric crystals } Given two geometric crystals $X$ and $Y$, each one endowed with maps $\left( \gamma, \varphi_\alpha, \varepsilon_\alpha, e_\alpha^. | \alpha \in \Delta \right)$, Berenstein and Kazhdan define the tensor product $X \otimes Y$ as the set $X \times Y$ endowed with the following maps:
\begin{align*}
\gamma( x \otimes y ) & = \gamma(x) + \gamma(y)\\
\varepsilon_\alpha( x \otimes y ) & =  \varepsilon_\alpha(x) + \log\left( 1 + e^{ \varepsilon_\alpha(y) - \varphi_\alpha(x) } \right)\\
\varphi_\alpha( x \otimes y ) & = \varphi_\alpha(y) + \log\left( 1 + e^{\varphi_\alpha(x) - \varepsilon_\alpha(y) } \right)\\
e^c_\alpha \cdot \left(x \otimes y \right) & = e^{c_1}_\alpha \cdot x \otimes e^{c_2}_\alpha \cdot y\\
 & \textrm{ where } \\
 & c_1 = \log\left(e^c + e^{ \varepsilon_\alpha(y) - \varphi_\alpha(x) } \right) - \log\left( 1 + e^{ \varepsilon_\alpha(y) - \varphi_\alpha(x) }\right)\\
 & c_2 = -\log\left(e^{-c} + e^{ \varepsilon_\alpha(y) - \varphi_\alpha(x) } \right) + \log\left( 1 + e^{ \varepsilon_\alpha(y) - \varphi_\alpha(x) }\right)
\end{align*}
Claim 2.16 in \cite{bib:BK06} asserts that $X \otimes Y$ is a geometric crystal. Notice that this definition is the same as our $q$-tensor product of crystals when $q=1$, and in proposition \ref{proposition:tensor_product_is_crystal}, we in fact checked that claim.

\subsection{Additional structure}
\label{subsection:additional_structure}

\paragraph{Invariant under crystal action:}
At this level, it is easy to see that the highest weight is invariant under the crystal actions $e^._\alpha$:

\begin{lemma}
\label{lemma:hw_is_invariant}
$$ \forall x \in \Bc, \forall \alpha \in \Delta, \forall c \in \R, hw\left( e^c_\alpha \cdot x \right) = hw(x)$$
\end{lemma}
\begin{proof}
Notice that $\Bc = \left( B \cap B^+ w_0 B^+ \right)_{\geq 0} = \left( B \cap U \bar{w}_0 U \right)_{\geq 0} \cdot A$. Also, the crystal actions $e^._\alpha, \alpha \in \Delta$ are given by an action of $U \times U$, leaving the $A$ factor invariant. This factor is nothing but $e^{hw(.)}$, hence the result.
\end{proof}

\paragraph{Verma relations:}
Following (\cite{bib:BK00}), for an abstract crystal $L$ and any word ${\bf i} = \left(i_1, \dots, i_k \right) \in I^k$ (not necessarily reduced), define the map:
$$ \begin{array}{cccc}
e_{\bf i}^.:  & \afrak \times L    & \rightarrow & L\\
              & \left(t, x \right) & \mapsto     & e_{\bf i}^t = e_{\alpha_{i_1}}^{ \beta^{(1)}(t) } \cdot e_{\alpha_{i_2}}^{ \beta^{(2)}(t) } \dots e_{\alpha_{i_k}}^{ \beta^{(k)}(t) } \cdot x
   \end{array}
$$
where $\beta^{(j)} = s_{i_k} \dots s_{i_{j+1}}(\alpha_{i_j})$.

The relations appearing in the next lemma are called Verma relations. If they hold, one can define unambiguously $e_w = e_{\bf i}$ for ${\bf i} \in I^k$ if $w = s_{i_1} \dots s_{i_k}$.
\begin{lemma}[lemma 2.1 \cite{bib:BK00}]
The following proposition are equivalent:
\begin{itemize}
 \item[(i)] For any ${\bf i} \in I^k$ and ${\bf i'} \in I^{k'}$, if:
$$ w = s_{i_1} \dots s_{i_k} = s_{i_1'} \dots s_{i_{k'}'}$$
Then $e_{\bf i} = e_{\bf i'}$.
 \item[(ii)] The following relations hold for every $c_1, c_2 \in \R$:
$$ e^{c_1}_\alpha \cdot e^{c_2}_\beta = e^{c_2}_\beta \cdot e^{c_1}_\alpha$$
if $\alpha(\beta) = \beta(\alpha) = 0$;
$$ e^{c_1 }_\alpha \cdot e^{2c_1 + c_2 }_\beta \cdot e^{c_1 + c_2 }_\alpha \cdot e^{c_2 }_\beta 
 = e^{c_2 }_\beta \cdot e^{c_1 + c_2 }_\alpha \cdot e^{2c_1 + c_2 }_\beta \cdot e^{c_1 }_\alpha $$
if $\alpha(\beta^\vee) = -1$, $\beta(\alpha^\vee) = -2$;
$$ e^{c_1 }_\alpha \cdot e^{3c_1 + c_2 }_\beta \cdot e^{2c_1 + c_2 }_\alpha \cdot e^{3c_1 + 2c_2 }_\beta \cdot e^{c_1 + c_2 }_\alpha \cdot e^{c_2 }_\beta 
 = e^{c_2 }_\beta \cdot e^{c_1 + c_2 }_\alpha \cdot e^{3c_1 + 2c_2 }_\beta \cdot e^{2c_1 + c_2 }_\alpha \cdot e^{3c_1 + c_2 }_\beta \cdot e^{c_1 }_\alpha $$
if $\alpha(\beta^\vee) = -1$, $\beta(\alpha^\vee) = -3$.
\end{itemize}
\end{lemma}
\begin{proof}
$(i) \Rightarrow (ii)$: 
If ${\bf i}$ and ${\bf i'}$ are reduced expressions, then by Tits lemma (theorem \ref{thm:tits_lemma}), one can obtain ${\bf i'}$ from ${\bf i}$ using braid moves. If $\alpha$ and $\beta$ are simple roots in $\Delta$ and satisfy a $d$-term braid relationship $s_\alpha s_\beta s_\alpha \dots  = s_\beta s_\alpha s_\beta \dots $, then we obtain:
$$ \forall t \in \afrak, 
e^{ \alpha(t) }_\alpha \cdot e^{ (s_\alpha \beta)(t) }_\beta \cdot e^{ (s_\alpha s_\beta \alpha)(t) }_\alpha \dots = e^{ \beta(t) }_\beta \cdot e^{ (s_\beta \alpha)(t) }_\alpha \cdot e^{ (s_\alpha s_\beta \alpha)(t) }_\alpha \dots $$
In particular, writing this equation for $t$ in the span of the coweights $\omega_\alpha^\vee$ and $\omega_\beta^\vee$, we find the Verma relations. This is the classical rank $2$ reduction. The list of relations corresponds to the root systems $A_1 \times A_1$, $A_2$, $BC_2$, $G_2$.

$(ii) \Rightarrow (i)$
Conversely, the Verma relations imply $e_{\bf i} = e_{\bf i'}$ for reduced words. If ${\bf i}$ and ${\bf i'}$ are not reduced, it is well known that one can reduce them by using braid moves and by deleting equal successive indices, as they correspond to a product of the form $s_\alpha^2 = id$. Therefore, we only need to notice that if ${\bf i}$ contains two equal successive indices and ${\bf i'}$ is the word obtained by deleting them, then $e_{\bf i} = e_{\bf i'}$. Indeed, if ${\bf i} \in I^k$, $k \in N$ and $i_j = i_{j+1}$ for a certain $j$ then for all $t \in \afrak$:
$$ \forall t \in \afrak, 
   e^{\left( s_{i_k} \dots s_{i_{j+1}} (\alpha_{i_{j  }}) \right)(t)}_{\alpha_{i_{j  }}}
   e^{\left( s_{i_k} \dots s_{i_{j+2}} (\alpha_{i_{j+1}}) \right)(t)}_{\alpha_{i_{j+1}}} = id $$

\end{proof}

\begin{proposition}
\label{proposition:group_verma}
For the geometric crystal $\Bc$, Verma relations hold.
\end{proposition}
\begin{proof}
See \cite{bib:BK00}. The proof is carried by direct computations in the group.
\end{proof}

\paragraph{W-action on the crystal:} In general for any abstract crystal $L$, as soon as the Verma relations hold, one can define a $W$ action on $L$. If $w = s_{i_1} \dots s_{i_k}$, define:
$$ \forall x \in L, w \cdot x = e^{-\gamma(x)}_w \cdot x$$
\begin{proposition}[ \cite{bib:BK00} ]
The $W$ action on a crystal is well defined, and the weight map is equivariant with respect to this action.
\end{proposition}
\begin{proof}
Equivariance is easily checked on simple reflections, as for $\alpha \in \Delta$ and $x \in L$:
\begin{align*}
  & \gamma\left( s_\alpha \cdot x \right)\\
= & \gamma\left( e^{-\alpha(\gamma(x))}_\alpha \cdot x \right)\\
= & \gamma\left( x \right) - \alpha\left(\gamma(x)\right) \alpha^\vee\\
= & s_\alpha( \gamma(x) )
\end{align*}
Then it carries on to all elements in the Weyl group by writing them as products of simple reflections, once we know we have defined an action.

Now, in order to check we have an action, consider $w = u v \in W$. By induction on the length, one can suppose that equivariance for $v$ holds ($\gamma( v \cdot x ) = v \gamma(x)$). If ${\bf i} \in I^k$ (resp. ${\bf i'} \in I^l$) is a word giving $u$ (resp. $v$), then their concatenation gives $w$. Moreover, with $\beta^{(j)}, 1 \leq j \leq k+l$ the corresponding roots and $t \in \afrak$:
\begin{align*}
e^{t}_w & = e_{\alpha_{i_1 }}^{\beta^{(1  )}(t) } \dots e_{\alpha_{i_k }}^{\beta^{(k  )}(t) } \cdot
            e_{\alpha_{i_1'}}^{\beta^{(k+1)}(t) } \dots e_{\alpha_{i_l'}}^{\beta^{(k+l)}(t) }\\
& = e_{\alpha_{i_1 }}^{v\beta^{(1  )}(vt) } \dots e_{\alpha_{i_k }}^{v\beta^{(k  )}(vt) } \cdot
    e_{\alpha_{i_1'}}^{ \beta^{(k+1)}( t) } \dots e_{\alpha_{i_l'}}^{ \beta^{(k+l)}( t) }\\
& = e^{vt}_u \cdot e^{t}_v
\end{align*}
Finally, it is easy to check that:
\begin{align*}
  & u \cdot (v \cdot x)\\
= & e^{-\gamma(v\cdot x)}_u \cdot \left( e^{-\gamma(x)}_v \cdot x \right)\\
= & e^{-v \gamma(x)}_u \cdot e^{-\gamma(x)}_v \cdot x\\
= & e^{-\gamma(x)}_w \cdot x\\
= & w \cdot x
\end{align*}
\end{proof}

\section{Group-theoretic path transforms}
\subsection{Paths on the solvable group B}

Let $C( \mathbb{R}^+, \afrak )$ be the set of continuous paths valued in the real Cartan subalgebra $\afrak$. In the following, for every path $X \in C( \mathbb{R}^+, \afrak )$ we want to introduce $B$-valued processes that are solution of a certain differential equation driven by $X$. The differential equation can be understood as being formal if $X$ fails to be regular enough so that the differential equation has a meaning.

One can also note that all the algebraic operations on group elements can be interpreted as matrix operations in any finite dimensional representation of the group $G$.\\

Let $\left( B_t(X) \right)_{ t \in \mathbb{R}^+ }$ be the $B$-valued path, driven by $X$ and solution of the following equation:
\begin{align}
\label{lbl:process_B_ode}
\left\{ \begin{array}{ll}
dB_t(X) = B_t(X) \left( \sum_{\alpha \in \Delta} f_\alpha dt + dX_t\right) \\
B_0(X) = \exp( X_0 )
\end{array} \right.
\end{align}

The following expression is easy to check (\cite{bib:BBO} after transpose) and can be taken as a definition when discarding the smoothness assumption on $X$:
\index{$B_t(\pi)$: $B$-valued path driven by $\pi$}
\begin{thm}
\begin{align}
\label{lbl:process_B_explicit}
B_t(X) & = \left( \sum_{k \geq 0} \sum_{ i_1, \dots, i_k } \int_{ t \geq t_k \geq \dots \geq t_1 \geq 0} e^{ -\alpha_{i_1}(X_{t_1}) \dots -\alpha_{i_k}(X_{t_k}) } dt_1 \dots dt_k f_{i_1} \cdot f_{i_2} \dots f_{i_k} \right) e^{X_t}
\end{align} 
\end{thm}
By convention, the term for $k=0$ is the identity element. Also, later, we will take $X$ to be a semi-martingale and view equation \ref{lbl:process_B_ode} as a stochastic differential equation (SDE) written in Stratonovich convention.\\

When $X$ is differentiable, equation \ref{lbl:process_B_ode} has to be understood the following way. In any finite dimensional group representation $V$, $B_.(X)$ is viewed as $GL(V)$-valued function of the time parameter:
$$ B_.(X): \R_+ \longrightarrow GL(V)$$
It is the solution of the system of ordinary differential equations written in matrix form as: 
\begin{align}
\left\{ \begin{array}{ll}
\frac{dB(X)}{dt}(t) = B_t(X) \left( \sum_{\alpha \in \Delta} f_\alpha + \frac{dX}{dt}(t)\right) \\
B_0(X) = \exp( X_0 )
\end{array} \right.
\end{align}

\begin{example}[$A_1$-type]
In the case of $SL_2$:
$$dB_t(X) = B_t(X) \begin{pmatrix} dX_t & 0\\ dt & -dX_t \end{pmatrix}  $$
Solving the differential equation leads to:
$$ B_t(X) = \begin{pmatrix} e^{X_t} & 0 \\ e^{X_t} \int_0^t e^{-2X_s}ds & e^{-X_t} \end{pmatrix} $$
\end{example}
\begin{example}[$A_2$-type]
For the canonical representation of $SL_3$, $\afrak = \{ x \in \mathbb{R}^3 | x_1 + x_2 + x_3 = 0 \}$:
$$dB_t(X) = B_t(X) \begin{pmatrix} dX^1_t & 0 & 0\\ dt & dX^2_t & 0 \\ 0 & dt & dX^3_t \end{pmatrix} $$
Solving the differential equation leads to:
$$ B_t(X) = \begin{pmatrix} e^{X^1_t}                                                               & 0                                        & 0
                         \\ e^{X^1_t} \int_0^t e^{-\alpha_1(X_s)}                                   & e^{X^2_t}                                & 0
                         \\ e^{X^1_t} \int_0^t e^{-\alpha_1(X_s) }ds \int_0^s e^{-\alpha_2(X_u) }du & e^{X^2_t} \int_0^t e^{-\alpha_2(X_s) }ds & e^{X^3_t} \end{pmatrix} $$
where $\left( \alpha_1 = (1,-1, 0), \alpha_2 = (0, 1, -1) \right) $ are the simple roots.
\end{example}

Now define $\left( A_t(X) \right)_{ t \in \mathbb{R}^+ }$ and $\left( N_t(X) \right)_{ t \in \mathbb{R}^+ }$ via the $NA$ decomposition of $B_.(X) = N_.(X) A_.(X)$:
\begin{align}
\label{lbl:process_A_explicit}
A_t(X) & = e^{X_t}
\end{align}
\begin{align}
\label{lbl:process_N_explicit}
N_t(X) & = \sum_{k \geq 0} \sum_{ i_1, \dots, i_k } \int_{ t \geq t_k \geq \dots \geq t_1 \geq 0 } e^{ -\alpha_{i_1}(X_{t_1}) \dots -\alpha_{i_k}(X_{t_k}) } f_{i_1} \cdot f_{i_2} \dots f_{i_k} dt_1 \dots dt_k
\end{align}

\begin{lemma}
$A_.(X)$ and $N_.(X)$ are solution of the following equations:
\begin{align}
\label{lbl:process_A_ode}
dA_t(X) = A_t(X) dX_t, & \quad A_0(X) = \exp( X_0 )
\end{align}
\begin{align}
\label{lbl:process_N_ode}
dN_t(X) = N_t(X) \left( \sum_{\alpha \in \Delta} e^{ -\alpha\left(X_t\right) }f_\alpha dt \right), & \quad N_0(X) = id
\end{align}
\end{lemma}
\begin{proof}
It is quite trivial for the $A$-part. Then, since $B_t(X) = N_t(X) A_t(X)$, we have by differentiation (Stratonovich differentiation rule in the stochastic case):
\begin{align*}
   & \ dN_t(X) A_t(X) + N_t(X) dA_t(X) = B_t(X) \left(\sum_{\alpha \in \Delta} f_\alpha dt + dX_t\right) \\
\Leftrightarrow & \ dN_t(X) A_t(X) + N_t(X) A_t(X) dX_t = N_t(X) A_t(X) \left(\sum_{\alpha \in \Delta} f_\alpha dt + dX_t\right) \\
\Leftrightarrow & \ dN_t(X) A_t(X) = N_t(X) A_t(X) \left(\sum_{\alpha \in \Delta} f_\alpha dt\right) \\
\Leftrightarrow & \ dN_t(X) = N_t(X) A_t(X) \left(\sum_{\alpha \in \Delta} f_\alpha dt\right) A_t(X)^{-1}\\
\Leftrightarrow & \ dN_t(X) = N_t(X) \left(\sum_{\alpha \in \Delta} e^{-\alpha(X_t)} f_\alpha dt\right)
\end{align*}
The last step uses the $Ad$ action of the torus on the Chevalley generators.
\end{proof}

\subsection{Group considerations}

Morally speaking, $B_.(X)$ is obtained by infinitesimal increments that are totally non-negative. Therefore, as totally non-negative  matrices form a semigroup, the following theorem is no surprise. For the convenience of the reader, we recall the proof.
\begin{thm}[\cite{bib:BBO}, lemma 3.4 - Total positivity of the flow $B_.$]
\label{thm:flow_B_total_positivity}
Let $X \in C( \mathbb{R}^+, \afrak )$. Then for all $t \geq 0$, $B_t(X)$ is totally non-negative. More precisely:
$$ \forall t>0, B_t(X) \in N^{w_0}_{>0} A$$
\end{thm}
\begin{proof}
For $t=0$, $B_0(X) = e^{X_0} \in A$ which is totally non-negative.\\
For $t>0$, clearly we need to prove that $N_t(X) \in N^{w_0}_{>0}$ or equivalently, thanks to theorem \ref{thm:total_positivity_criterion} that all minors $\Delta_{ w \omega_i, \omega_i }\left( N_t(X) \right), 1 \leq i \leq n, w \in W$ are positive:
\begin{align*}
  & \Delta_{ w \omega_i, \omega_i }\left( N_t(X) \right)\\
= & \langle N_t(X) v_{\omega_i}, \bar{w} v_{\omega_i} \rangle\\
= & \sum_{k \geq 0} \sum_{ i_1, \dots, i_k } \int_{ t \geq t_k \geq \dots \geq t_1 \geq 0} e^{ -\alpha_{i_1}(X_{t_1}) \dots -\alpha_{i_k}(X_{t_k}) } dt_1 \dots dt_k \\
  & \ \ \langle f_{i_1} \cdot f_{i_2} \dots f_{i_k} v_{\omega_i}, \bar{w} v_{\omega_i} \rangle
\end{align*}
Because of lemma 7.4 in \cite{bib:BZ01}, we have that:
$$\forall k \in \N, \forall w \in W, \langle f_{i_1} \cdot f_{i_2} \dots f_{i_k} v_{\omega_i}, \bar{w} v_{\omega_i} \rangle \geq 0$$
and therefore, we have a sum of non-negative terms. In order to see that $\Delta_{ w \omega_i, \omega_i }\left( N_t(X) \right)$ is strictly positive, only one of them needs to be non-zero.\\
As $\bar{w} v_{\omega_i}$ generates the one dimensional weight space $V(\omega_i)_{w \omega_i}$, there is some sequence $i_1, \dots, i_k$ such that $\alpha_{i_1} + \dots + \alpha_{i_k} = \omega_i - w \omega_i$ and
$\bar{w} v_{\omega_i}$ is proportional to $f_{i_1} \cdot f_{i_2} \dots f_{i_k} v_{\omega_i}$. Hence a non-zero scalar product.
\end{proof}

The following path transform will play a fundamental role in the sequel.
\index{$T_g$: Path transform}
\begin{definition}
 When it exists, for $g \in G$ and $X$ a continuous path in $\afrak$, define:
$$ T_g X(t) := \log \left[ g B_t(X)\right]_0$$
The previous expression makes sense when $g B_t(X)$ has a Gauss decomposition and $\left[ g B_t(X)\right]_0 \in A$, in order to be able to consider its logarithm.
\end{definition}

This path transform has the property:
\begin{thm}[\cite{bib:BBO2}, proposition 6.4]
Let $X$ be a continuous path in $\afrak$ and $g \in G$. Assume that $g B_t(X)$ has a Gauss decomposition on an open time interval $J$. Then $[ g B_t(X) ]_{-0}$ solves for $t \in J$:
$$d [ g B_t(X) ]_{-0} = [g B_t(X)]_{-0} \left( \sum_{\alpha} f_\alpha dt + d\left( T_g X \right)_t \right)$$
\end{thm}

There are certain sets $D \subset G$ such that for $g \in D$, $g B_t(X), t \geq 0$ has always a Gauss decomposition. The following will play an important role:
$$ D := N \cdot A \cdot U_{\geq 0} $$
\begin{proposition}
For $g \in D$, we have a well-defined path transform:
$$ T_g: C( \mathbb{R}^+, \afrak ) \rightarrow C( \mathbb{R}^+, \afrak )$$
such that for $X \in C( \mathbb{R}^+, \afrak )$, $T_g X$ is the unique path in $\afrak$ such that:
$$ [g B_t(X)]_{-0} = [g]_- B_t(T_g X) $$
\end{proposition}
\begin{proof}
We only need to prove that the path transform is well defined. As for all $t \geq 0$, $B_t(X) \in G_{\geq 0}$, we have that $( A U_{\geq 0}) B_t(X) \subset G_{\geq 0}$, since the totally non-negative matrices form a semigroup. Hence $D B_t(X) \subset N G_{\geq 0}$, which is a set whose elements admit a Gauss decomposition (see theorem \ref{thm:totally_positive_gauss_decomposition}).

In order to prove that the equation driving $[g B_t(X)]_{-0}$ is of the required form, use the previous theorem.
\end{proof}

\begin{properties}
\label{lbl:path_transform_properties}
  Let $X$ be a continuous path. Then:\\
\begin{itemize}
 \item[(i)]$\forall g_1, g_2 \in G$ such that $g_1 \in D, g_1 g_2 \in D$, we have:
$$ T_{g_1 g_2} = T_{g_1 [g_2]_{-}} \circ T_{g_2} $$ 
In particular, $\forall u_1, u_2 \in U_{\geq 0}, T_{u_1 u_2} = T_{u_1} \circ T_{u_2}$.
 \item[(ii)]  $\forall g \in D, \forall n \in N, T_{ng} = T_g$
 \item[(iii)] $\forall g \in D, \forall a \in A, T_{a g} X = T_g X + \log a$
 \item[(iv)]  $\forall g \in D, x \in \afrak, T_g( X + x) = T_{ g e^x } (X)$
 \item[(v)]   If $\alpha \in \Delta$ and $g=x_\alpha(\xi), \xi>0$, then:
$$ \left(T_g X \right)_t = X_t + \log\left( 1 + \xi \int_0^t e^{-\alpha( X_s)} ds \right) \alpha^{\vee}$$
\end{itemize}
\end{properties}
\begin{proof}
The proof uses the properties of the Gauss decomposition in equations \ref{lbl:gauss1} and \ref{lbl:gauss2}.

(i) \begin{eqnarray*}
& & [ g_1 g_2 B_t(X) ]_{-0} = [ g_1 [g_2 B_t(X)]_{-0} ]_{-0}\\
\Rightarrow & & [ g_1 g_2]_- B_t(T_{g_1 g_2} X) = [  g_1 [g_2]_- B_t(T_{g_2} X)]_{-0} = [  g_1 [g_2]_- ]_- B_t( T_{ g_1 [g_2]_{-} } \circ T_{g_2} X)\\
\Rightarrow & & B_t(T_{g_1 g_2} X) = B_t( T_{ g_1 [g_2]_{-} } \circ T_{g_2} X)
\end{eqnarray*}

(ii) $[n g B_t(X)]_{-0} = [n g]_{-} B(T_{gn}X)$ by definition. And on the other hand, is it also equal to $n [g B_t(X)]_{-0} = n [g]_- B_t(T_gX)$

(iii) $[agB_t(X)]_{-0} = [ag]_{-} B(T_{a g}X) = a [g]_{-} a^{-1} B(T_{a g}X) $ by definition. And on the other hand, is it also equal to $a[gB_t(X)g]_{-0} = a [g]_- B_t(T_gX)$. 
Then, we have $a^{-1} B_t(T_{ag}(X) = B_t(T_g X)$.

(iv) One can check that $B_t(X+x) = \exp(x) B_t(X)$

(v) Direct computation, using the embedding from $SL_2$ into the closed subgroup of $G$ whose Lie algebra is generated by the $\mathfrak{sl}_2$-triplets $(e_\alpha, f_\alpha, h_\alpha)$. One can also use the lemma in the next subsection.
\end{proof}

\subsection{Extension of the path transform}
Now, looking at property $(v)$, it is natural to expect the path transform $T_{x_\alpha(\xi)}$ to be extended to negative values of $\xi$, although this will depend on the path taken as input. Let us examine first when a Gauss decomposition exists for $x_\alpha(\xi) B_t(X)$, or equivalently $x_\alpha(\xi) N_t(X)$.

\begin{lemma}
For different $\alpha$ and $\beta$ in $\Delta$:
$$ \forall t \geq 0, \Delta^{\omega_\beta}\left( x_\alpha(\xi) N_t(X) \right) = 1$$
$$ \forall t \geq 0, \Delta^{\omega_\alpha}\left( x_\alpha(\xi) N_t(X) \right) = 1 + \xi \int_0^t e^{-\alpha(X_s)} ds$$
\end{lemma}
\begin{proof}
The first identity is a consequence of proposition 2.2 in \cite{bib:FZ99}. For the second, we start by using equation \ref{lbl:process_N_explicit} and work with the highest weight representation $V(\omega_\alpha)$. We have:
\begin{align*}
  & \Delta^{\omega_\alpha}\left( x_\alpha(\xi) N_t(X) \right) \\
= & \langle \exp(\xi e_\alpha) N_t(X) v_{\omega_\alpha}, v_{\omega_\alpha} \rangle\\
= & \sum_{k \geq 0} \sum_{i_1, i_2, \dots, i_k} \int_{t\geq t_1 \geq \dots \geq t_k \geq 0} e^{ -\alpha_{i_1}(X_{t_1}) - \dots \alpha_{i_k}(X_{t_k})} dt_1 \dots dt_k \\
  & \ \ \langle \exp( \xi e_\alpha) f_{i_1} \dots f_{i_k}  v_{\omega_\alpha}, v_{\omega_\alpha} \rangle
\end{align*}
Hence, as we will write the expansion $e^{\xi e_\alpha} = \sum_{n \in \N} \frac{\xi^n}{n!} e_\alpha^n$, we need to consider vectors of the form:
$$e^n_\alpha f_{i_1} \dots f_{i_k} v_{\omega_\alpha}$$
Now notice that:
$$ \forall n \in \N, \forall k \in \N, e^n_\alpha f_{i_1} \dots f_{i_k} v_{\omega_\alpha} \in V(\omega_\alpha)_\mu$$
where $V(\omega_\alpha)_\mu$ is the weight space in $V(\omega_\alpha)$ corresponding to the weight
$$ \mu = \omega_\alpha + n \alpha - \sum_{j=1}^k \alpha_{i_j}$$
Weight spaces in $V(\omega_\alpha)$ corresponding to different weights are orthogonal under the invariant scalar product $\langle., . \rangle$. Therefore, if $k \neq n$ or there is a $j$ such that $\alpha_{i_j} \neq \alpha$, we have:
$$ \langle e^n_\alpha f_{i_1} \dots f_{i_k} v_{\omega_\alpha}, v_{\omega_\alpha} \rangle = 0$$
Moreover, in the representation  $V(\omega_\alpha)$, we have:
$$ \forall n\geq 2, f_\alpha^n v_{\omega_\alpha} = 0$$
Therefore, most terms are zero:
\begin{align*}
  & \Delta^{\omega_\alpha}\left( x_\alpha(\xi) N_t(X) \right) \\
= & \langle v_{\omega_\alpha}, v_{\omega_\alpha} \rangle + \xi \langle e_\alpha f_\alpha v_{\omega_\alpha}, v_{\omega_\alpha} \rangle \int_0^t e^{-\alpha(X_s)}ds\\
= & 1 + \xi \int_0^t e^{-\alpha(X_s)}ds
\end{align*}
The last equality is due to the fact that:
\begin{align*}
  & e_\alpha f_\alpha v_{\omega_\alpha}\\
= & [e_\alpha, f_\alpha] v_{\omega_\alpha} + f_\alpha e_\alpha v_{\omega_\alpha}\\
= & h_\alpha v_{\omega_\alpha}\\
= & v_{\omega_\alpha}
\end{align*}
\end{proof}

This lemma encourages us to consider paths only up to a certain horizon $T$ and when applying $T_{x_\alpha(\xi)}$ to $X \in C\left([0; T], \afrak\right)$, one can only take $\xi \in (-\frac{1}{\int_0^T e^{-\alpha(X)}}; +\infty)$. This can be summarized in the following proposition.
\begin{proposition}
 For every $\alpha \in \Delta$ and $c \in \R$, there is a path transform:
$$ e^c_\alpha: C\left([0; T], \afrak\right) \longrightarrow C\left([0; T], \afrak\right)$$
such that for $X \in C\left([0; T], \afrak\right)$:
$$ e^c_\alpha \cdot X = T_{x_\alpha\left(\frac{e^c-1}{\int_0^T e^{-\alpha(X)}}\right) }\left( X \right) $$
\end{proposition}
This path transform is in fact the corner stone of the geometric path model we will now present. For instance, as we will see, we have:
$$ e^{c+c'}_\alpha = e^{c}_\alpha e^{c'}_\alpha$$
It is in fact the geometric lifting of the Littelmann operators.

\section{The geometric path model}
In this section, we define a continuous family of Littelmann path models depending on a parameter $q$, as well as $q$-tensor product. For $q \rightarrow 0$, we recover the continuous 'frozen' setting presented at the beginning of \cite{bib:BBO2}. Since all $q$-Littelmann models for $q>0$ are equivalent in certain sense, our study will focus on the $q=1$ case and prove that tensor product of crystals is given by the concatenation of their elements.\\

The path transforms described in the previous section naturally appear as the building blocks for the Littelmann operators $e^c_\alpha$. We will also benefit from the construction by Berenstein and Kazhdan (\cite{bib:BK00, bib:BK04, bib:BK06}) while exhibiting Verma relations and finally we show that a simple projection exists between the path model and the group picture $\Bc$.

\subsection{Path models}

\subsubsection{Definitions }
Let $C\left( [0; T], \afrak \right)$ be set of $\afrak$-valued continuous functions on $[0, T]$. Its elements are loosely referred to as paths in $\afrak$. We call a 'model' a candidate for becoming a crystal. Hence a path model will be a set of paths endowed with structure maps. The subscript $0$ will indicate that they are starting at zero.

\begin{definition}
\label{def:path_model}
A path crystal $L$ is a subset $C_0\left( [0; T], \afrak \right)$, where $\afrak$ is the real Cartan subalgebra, endowed with maps $\gamma$, $\varepsilon_\alpha$, $\varphi_\alpha$ and actions $\left(e^._\alpha\right)_{ \alpha \in \Delta }$ such that
\begin{itemize}
 \item The 'weight' map $\gamma: L \rightarrow \afrak$ gives the endpoint:
       $$ \gamma(\pi) = \pi(T) $$
 \item $L$ is an abstract geometric crystal as in definition \ref{def:crystal}.
\end{itemize}
\end{definition}

\subsubsection{Duality}
Define the duality map $\iota: C_0\left( [0; T], \afrak \right) \rightarrow C_0\left( [0; T], \afrak \right)$ that associates to each path $\pi$ its dual $\pi^\iota$. It is defined as $$\pi^\iota(t) = \pi(T-t) - \pi(T)$$
A crystal structure is said to behave well with respect to duality if:
$$e^{-c}_\alpha = \iota \circ e^c_\alpha \circ \iota$$
$$\varphi_\alpha = \varepsilon_\alpha \circ \iota $$

\subsubsection{Continuous q-Littelmann model}
Here we define a family of path models indexed by $q$, a parameter that can be understood as temperature. When $q=0$, we recover the continuous path model introduced and studied in \cite{bib:BBO2} as the continuous counterpart of Littelmann's path model (\cite{bib:Littelmann95} \cite{bib:Littelmann97}).

\paragraph{Models:}
When $q>0$, a continuous $q$-Littelmann model is a subset $L$ of $C_0\left( [0; T], \afrak \right)$ endowed with the structure $L_q = \left( \gamma, \left( \varepsilon_\alpha, \varphi_\alpha, e^._\alpha \right)_{\alpha \in \Delta} \right)$: 
\begin{itemize}
 \item The 'weight' map $\gamma: L \rightarrow \afrak$ is the endpoint:
       $$ \gamma(\pi) = \pi(T) $$
 \item $\varepsilon_\alpha, \varphi_\alpha: L \rightarrow \R$ defined for every $\alpha \in \Delta$ as
$$\varepsilon_\alpha( \pi ) := q \log\left( \int_0^T e^{ -q^{-1} \alpha(\pi(s)) }ds \right)$$
$$\varphi_\alpha( \pi ) := \alpha\left( \pi(T) \right) + q \log\left( \int_0^T e^{-q^{-1} \alpha(\pi(s))}ds \right)$$
 \item $e^c_\alpha: L \rightarrow L$, $c \in \R$, $\alpha \in \Delta$ defined as
$$ e^c_\alpha \cdot \pi(t) := \pi(t) + q \log\left( 1 + \frac{ e^{q^{-1} c} - 1 }{e^{q^{-1}\varepsilon_\alpha(\pi)}}\int_0^t e^{ -q^{-1} \alpha(\pi(s)) }ds \right) \alpha^\vee $$
\end{itemize}
When $q=0$, we take as defining axioms the limit $q \rightarrow 0$:
$$\varepsilon_\alpha\left( \pi \right) = -\inf_{0 \leq s \leq T} \alpha\left( \pi(s) \right) $$
$$\varphi_\alpha\left( \pi \right) = \alpha\left( \pi(T) \right)-\inf_{0 \leq s \leq T} \alpha\left( \pi(s) \right) $$
$$\forall 0<t<T, e^c_\alpha\left( \pi \right)(t) = \pi(t) + \inf_{0 \leq s \leq T} \alpha\left( \pi(s) \right) \alpha^\vee 
    - \min\left( \inf_{0 \leq s \leq t} \alpha\left( \pi(s) \right) - c,
                 \inf_{t \leq s \leq T} \alpha\left( \pi(s) \right) \right) \alpha^\vee $$
$$ e^c_\alpha\left( \pi \right)(0) = \pi(0) = 0$$
$$ e^c_\alpha\left( \pi \right)(T) = \pi(T) + c \alpha^\vee$$
Indeed, the limits for $\varepsilon_\alpha$ and $\varphi_\alpha$ are an immediate application of the Laplace method. 
The latter limit comes from re-arranging the expression before using the Laplace method as well:
\begin{align*}
e^c_\alpha \cdot \pi(t) & = \pi(t) + q \log\left( 1 + \frac{ e^{q^{-1} c} - 1 }{e^{q^{-1}\varepsilon_\alpha(\pi)}}\int_0^t e^{ -q^{-1} \alpha(\pi(s)) }ds \right) \alpha^\vee\\
& = \pi(t) + q \log\left( e^{q^{-1} c} \int_0^t e^{ -q^{-1} \alpha(\pi(s)) }ds + \int_t^T e^{ -q^{-1} \alpha(\pi(s)) }ds \right) \alpha^\vee \\
& \ \ \ - q \log\left( \int_0^T e^{ -q^{-1} \alpha(\pi(s)) }ds \right) \alpha^\vee\\
& \stackrel{ q \rightarrow 0 }{ \rightarrow }
    \pi(t) + \inf_{0 \leq s \leq T} \alpha\left( \pi(s) \right) \alpha^\vee 
    - \alpha^\vee \min\left( \inf_{0 \leq s \leq t} \alpha\left( \pi(s) \right) - c,
                             \inf_{t \leq s \leq T} \alpha\left( \pi(s) \right) \right)
\end{align*}
Notice that the condition $0<t<T$ is essential in Laplace method, as certain integral terms disappear at $t=0$ and $t=T$. We claim that this expression is exactly the same as the one defining the generalized Littelmann operators defined in \cite{bib:BBO2} section 3. We discuss that point in the next subsection '$q=0$ limit'.

\paragraph{Crystals:}
A $q$-Littelmann model that satisfies the crystal axioms is called a $q$-Littelmann crystal. An important fact is that for $q>0$, all the continuous $q$-Littelmann structures $\left( L_q \right)_{q>0}$ on $C_0\left( [0; T], \afrak \right)$ are equivalent. That is why we can restrict our attention to the case $q=1$. We will use the term 'Geometric Littelmann crystal' to refer to the $q=1$ model as it is the path model for the geometric crystals introduced by Berenstein and Kazhdan. The link will be made clear further in the presentation.\\

The fact that $q$-Littelmann crystal structures on $C_0\left( [0; T], \afrak \right)$ are equivalent for $q>0$ can easily be checked by using the rescaling on reals and on paths.
$$ \forall x \in \R, \psi_{q,q'}\left( x \right) = \frac{q'}{q} x$$
$$ \forall \pi \in C_0\left( [0; T], \afrak \right), \psi_{q,q'}\left( \pi \right) = \frac{q'}{q} \pi$$
$\psi_{q,q'}$ intertwines structural maps. Let $L_q = \left( \gamma, \left( \varepsilon_\alpha, \varphi_\alpha, e^._\alpha \right)_{\alpha \in \Delta} \right)$ and 
    $L_{q'} = \left( \gamma', \left( \varepsilon^{'}_{\alpha}, \varphi^{'}_{\alpha}, e^{'.}_{\alpha} \right)_{\alpha \in \Delta} \right)$ 
the two continuous Littelmann structures on $C_0\left( [0; T], \afrak \right)$ associated to $q$ and $q'$. We have:
\begin{align*}
\gamma^{'}\left( \psi_{q,q'}\left( \pi \right) \right) & = \psi_{q,q'}\left( \gamma\left( \pi \right) \right)\\
\varepsilon^{'}_{\alpha}\left( \psi_{q,q'}\left( \pi \right) \right) & = \psi_{q,q'}\left( \varepsilon_\alpha\left( \pi \right) \right)\\
\varphi^{'}_{\alpha}\left( \psi_{q,q'}\left( \pi \right) \right) & = \psi_{q,q'}\left( \varphi_\alpha\left( \pi \right) \right)\\
e^{'\psi_{q,q'}(c)}_{\alpha} \cdot \psi_{q,q'}\left( \pi \right) & = \psi_{q,q'}\left(  e^{c}_\alpha \cdot \pi \right)\\
\end{align*}

\begin{rmk}
 Our choice of describing this relationship as 'equivalence' and not 'isomorphism' in the strict sense is because the real parameter $c$ in the actions $e^c_\alpha, \alpha \in \Delta$ is rescaled. The structural maps $\varepsilon_\alpha, \varphi_\alpha$ are also rescaled.\\
In fact, this transformation is more easily seen as a change of underlying semifields as we will see much later in chapter \ref{chapter:degenerations}.
\end{rmk}

Another fact worth mentioning is the commutation with respect to tensor product: If $B_q$ and $B_q^{'}$ are 
$q$-Littelmann crystals then $$ \psi_{q, q'}\left( B_q \otimes_q B_q^{'} \right) = \psi_{q, q'}\left( B_q \right) \otimes_{q'} \psi_{q, q'}\left( B_q^{'} \right) $$

\subsection{Classical Littelmann model as a limit}
In \cite{bib:BBO2} definition 3.3, the continuous path model described, which in fact coincides with Littelmann's original definition, has the following structural maps:
$$\varepsilon_\alpha\left( \pi \right) = -\inf_{0 \leq s \leq T} \alpha\left( \pi(s) \right) $$
$$\varphi_\alpha\left( \pi \right) = \alpha\left( \pi(T) \right)-\inf_{0 \leq s \leq T} \alpha\left( \pi(s) \right) $$
$$\textrm{If } 0 \leq c \leq \varepsilon_\alpha\left( \pi \right), \Ec^c_\alpha\left( \pi \right)(t) = \pi(t)
  - \min\left(0, -c - \inf_{0 \leq s \leq T} \alpha\left( \pi(s) \right) - \inf_{0 \leq s \leq t} \alpha\left( \pi(s) \right) \right)
 $$
$$\textrm{If } - \varphi_\alpha\left( \pi \right) \leq c \leq 0, \Ec^c_\alpha\left( \pi \right)(t) = \pi(t)
  - \min\left(-c, \inf_{t \leq s \leq T} \alpha\left( \pi(s) \right) - \inf_{0 \leq s \leq T} \alpha\left( \pi(s) \right) \right)
 $$
In this subsection, we give explicit indications on why these are exactly the same maps as our $q=0$ limit. The identification is quite immediate except when it comes to recognizing our actions $e_\alpha^.$. Recall that:
$$\forall 0<t<T, e^c_\alpha\left( \pi \right)(t) = \pi(t) + \alpha^\vee \inf_{0 \leq s \leq T} \alpha\left( \pi(s) \right) - \alpha^\vee \min\left( \inf_{0 \leq s \leq t} \alpha\left( \pi(s) \right) - c, \inf_{t \leq s \leq T} \alpha\left( \pi(s) \right) \right)$$
$$ e^c_\alpha\left( \pi \right)(0) = \pi(0) = 0$$
$$ e^c_\alpha\left( \pi \right)(T) = \pi(T) + c \alpha^\vee$$
It shows that our description has at least an advantage at $q=0$: Only one formula for $e_\alpha^c$ independently of the sign of $c$.\\

Notice that for paths starting from zero, if $c \geq 0$:
\begin{align*}
  & \min\left( \inf_{0 \leq s \leq t} \alpha(\pi(s)) - c, \inf_{t \leq s \leq T} \alpha(\pi(s)) \right)\\
= & \min\left( \inf_{0 \leq s \leq t} \alpha(\pi(s)) - c, \inf_{0 \leq s \leq t} \alpha(\pi(s)), \inf_{t \leq s \leq T} \alpha(\pi(s)) \right)\\
= & \min\left( \inf_{0 \leq s \leq t} \alpha(\pi(s)) - c, \inf_{0 \leq s \leq T} \alpha(\pi(s)) \right)\\
\end{align*}
While if $c \leq 0$:
\begin{align*}
  & \min\left( \inf_{0 \leq s \leq t} \alpha(\pi(s)) - c, \inf_{t \leq s \leq T} \alpha(\pi(s)) \right)\\
= & -c + \min\left( \inf_{0 \leq s \leq t} \alpha(\pi(s)), c + \inf_{t \leq s \leq T} \alpha(\pi(s)) \right)\\
= & -c + \min\left( \inf_{0 \leq s \leq t} \alpha(\pi(s)), \inf_{t \leq s \leq T} \alpha(\pi(s)), c + \inf_{t \leq s \leq T} \alpha(\pi(s)) \right)\\
= & -c + \min\left( \inf_{0 \leq s \leq T} \alpha(\pi(s)), c + \inf_{t \leq s \leq T} \alpha(\pi(s)) \right)\\
= & \min\left( -c + \inf_{0 \leq s \leq T} \alpha(\pi(s)), \inf_{t \leq s \leq T} \alpha(\pi(s)) \right)
\end{align*}
Replacing in each case, $\min\left( \inf_{0 \leq s \leq t} \alpha(\pi(s)) - c, \inf_{t \leq s \leq T} \alpha(\pi(s)) \right)$ in our expression for $e^c_\alpha$ recovers $\Ec^c, c \geq 0$ and $\Ec^c, c \leq 0$.\\

\begin{rmk}
The 'cutting' conditions $-\varphi_\alpha(\pi) \leq c \leq \varepsilon_\alpha(\pi)$ will appear naturally later. For now, we can notice that in order for $e^c_\alpha$ to preserve continuity at $t=0$, one needs $c \leq \varepsilon_\alpha(\pi)$. In order to preserve continuity at $t=T$, we need $-\varphi_\alpha(\pi) \leq c$.
\end{rmk}

\subsection{A rank 1 example}
In rank $1$, crystal actions on paths in $\afrak$ are in fact one dimensional, and via projection $\afrak$ can be considered as $\R$.

\paragraph{Connected crystal at $\bf q=1$:}
Let $\pi \in C_0\left( [0; T], \afrak \right)$ be a path and $\langle \pi \rangle$ be the connected crystal generated by $\pi$:
$$\langle \pi \rangle = \left\{ \pi_c = e^c_\alpha \cdot \pi, c \in \R \right\} = \left\{ t \mapsto \pi(t) + \log\left( 1 + (e^c - 1)\frac{\int_0^t e^{-\alpha\left( \pi \right)}}{\int_0^T e^{-\alpha\left( \pi \right)}} \right) \alpha^\vee \right\}$$ 
Notice that there is an extremal element $\eta = e^{-\infty}_\alpha \cdot \pi$ that does not belong to the crystal, as it diverges at its endpoint ($t=T$):
$$\eta\left(t\right) = \pi\left(t\right) + \log\left(1 - \frac{\int_0^t e^{-\alpha\left( \pi \right)}}{\int_0^T e^{-\alpha\left( \pi \right)}} \right)\alpha^\vee $$

The transform $e^{-\infty}_\alpha$ is a projection as it gives $\eta$ when applied to any element of the crystal, as a consequence of $e^{-\infty}_\alpha \cdot e^{c}_\alpha = e^{-\infty}_\alpha$. As such, it is clearly not an injective map. However there is only one real number that is lost in this process, and it is in fact $\int_0^T e^{-\alpha\left( \pi \right)}$. This basic remark will be a key element in parametrizing path crystals.

\subsection{Geometric Littelmann model} 

As announced, we will now restrict our attention to the $q=1$ case, which we call the geometric case. In the next subsection we prove there is a projection morphism to $\Bc$ the typical crystal in the sense of Berenstein and Kazhdan.

\subsubsection{Geometric Littelmann Crystal}
 A geometric Littelmann crystal $L$ is a subset of $C_0\left( [0; T], \mathfrak{a} \right)$ endowed with
\begin{itemize}
 \item A 'weight' map $\gamma: L \rightarrow \mathfrak{a}$ defined as
$$ \gamma(\pi) = \pi(T) $$
 \item For every $\alpha \in \Delta$, maps $\varepsilon_\alpha, \varphi_\alpha$ defined as:
$$\varepsilon_\alpha( \pi ) := \log\left( \int_0^T e^{ -\alpha(\pi(s)) }ds \right)$$
$$\varphi_\alpha( \pi ) := \varepsilon_\alpha \circ \iota ( \pi )
                        = \alpha\left( \pi(T) \right) + \log\left( \int_0^T e^{-\alpha(\pi(s))}ds \right)$$
 \item The actions $\left(e^._\alpha\right)_{ \alpha \in \Delta }$ defined as:
$$ e^c_\alpha \cdot \pi(t) := \pi(t) + \log\left( 1 + \frac{ e^c - 1 }{e^{\varepsilon_\alpha(\pi)}}\int_0^t e^{ -\alpha(\pi(s)) }ds \right) \alpha^\vee $$
\end{itemize}

\begin{example}
 The whole set $C_0\left( [0; T], \mathfrak{a} \right)$ is a crystal.
\end{example}

The following properties show that it is indeed a crystal in the usual sense:
\begin{properties}
A geometric Littelmann crystal is a geometric path crystal in the sense that for $\pi \in L$, the following properties are satisfied:
\begin{enumerate}
 \item[(i)]   $\varphi_\alpha(\pi) = \varepsilon_\alpha(\pi) + \alpha\left( \gamma(\pi) \right)$
 \item[(ii)]  $\gamma\left( e^c_\alpha \cdot \pi \right) = \gamma\left( \pi \right) + c \alpha^\vee$
 \item[(iii)] $\varepsilon_\alpha\left( e^c_\alpha \cdot \pi \right) = \varepsilon_\alpha\left( \pi \right) - c$
 \item[(iv)]  $\varphi_\alpha\left( e^c_\alpha \cdot \pi \right) = \varphi_\alpha\left( \pi \right) + c$
 \item[(v)]   $e^._\alpha$ are indeed actions as $e^{c}_\alpha \cdot e^{c'}_\alpha = e^{c+c'}_\alpha $
 \item[(vi)]  The Littelmann action behaves well with respect to time-reversal:
              $$e^{-c}_\alpha = \iota \circ e^c_\alpha \circ \iota$$
              $$\varphi_\alpha = \varepsilon_\alpha \circ \iota $$
 \end{enumerate}
\end{properties}
\begin{proof}
\begin{enumerate}
 \item[(i)] Obvious.
 \item[(ii)]  \begin{align*}
		 & \gamma\left( e^c_\alpha \cdot \pi \right)\\
               = & \pi(T) + \log\left( 1 + \frac{ e^c - 1 }{e^{\varepsilon_\alpha(\pi)}}\int_0^T e^{ -\alpha(\pi(s)) }ds \right) \alpha^\vee\\
               = & \pi(T) + \log\left( 1 + ( e^c - 1 ) \right) \alpha^\vee\\
               = & \gamma\left( \pi \right) + c \alpha^\vee
              \end{align*}
 \item[(iii)] \begin{align*}
                 & \varepsilon_\alpha\left( e^c_\alpha \cdot \pi \right)\\
               = & \log\left( \int_0^T \frac{ e^{ -\alpha(\pi(s)) } }{\left( 1 + \frac{ e^c - 1 }{e^{\varepsilon_\alpha(\pi)}}\int_0^s e^{ -\alpha(\pi(u)) }du\right)^2 }ds \right)\\
               = & \log\left( \frac{e^{\varepsilon_\alpha(\pi)}}{ e^c - 1 } \left( \int_0^T -\frac{d}{ds}\left( \frac{ 1 }{ 1 + \frac{ e^c - 1 }{e^{\varepsilon_\alpha(\pi)}}\int_0^s e^{ -\alpha(\pi(u)) }du } \right)ds \right) \right)\\
               = & \log\left( \frac{e^{\varepsilon_\alpha(\pi)}}{ e^c - 1 } \left( 1 - e^{-c} \right) \right)\\
               = & \varepsilon_\alpha(\pi) - c\\
              \end{align*}
 \item[(iv)]  Obvious using (i), (ii) and (iii).
 \item[(v)]    \begin{align*}
                & \left( e^{c}_\alpha \cdot e^{c'}_\alpha \cdot \pi \right) (t)\\
              = & \left( e^{c'}_\alpha \cdot \pi \right)(t) + \log\left( 1 + \frac{e^c - 1}{e^{ \varepsilon\left( e^{c'}_\alpha \cdot \pi \right)}}\int_0^t e^{ -\alpha\left( e^{c'}_\alpha \cdot \pi (s) \right) } ds \right) \alpha^{\vee}\\
              = & \left( e^{c'}_\alpha \cdot \pi \right)(t) + \log\left( 1 + \frac{e^c - 1}{e^{ \varepsilon\left( \pi \right) - c'}}\int_0^t \frac{ e^{ -\alpha\left( \pi (s) \right) } }{\left( 1 + \frac{ e^{c'} - 1 }{e^{\varepsilon_\alpha(\pi)}}\int_0^s e^{ -\alpha(\pi(u)) }du \right)^2 } ds \right) \alpha^{\vee}\\
              = & \left( e^{c'}_\alpha \cdot \pi \right)(t) + \log\left( 1 + \frac{e^c - 1}{e^{ \varepsilon\left( \pi \right) - c'}} \frac{ e^{\varepsilon\left(\pi\right)} }{e^{c'} - 1}\int_0^t -\frac{d}{ds}\left( \frac{ 1 }{ 1 + \frac{ e^{c'} - 1 }{e^{\varepsilon_\alpha(\pi)}}\int_0^s e^{ -\alpha(\pi(u)) }du } \right) ds \right) \alpha^{\vee}\\
              = & \left( e^{c'}_\alpha \cdot \pi \right)(t) + \log\left( 1 + e^{c'}\frac{e^c - 1}{e^{c'} - 1}\left( 1 - \frac{ 1 }{ 1 + \frac{ e^{c'} - 1 }{e^{\varepsilon_\alpha(\pi)}}\int_0^t e^{ -\alpha(\pi(s)) }ds } \right) \right) \alpha^{\vee}\\
              = & \left( e^{c'}_\alpha \cdot \pi \right)(t) + \log\left( 1 + e^{c'} \frac{ \frac{ e^{c} - 1 }{e^{\varepsilon_\alpha(\pi)}}\int_0^t e^{ -\alpha(\pi(s)) }ds }{ 1 + \frac{ e^{c'} - 1 }{e^{\varepsilon_\alpha(\pi)}}\int_0^t e^{ -\alpha(\pi(s)) }ds } \right) \alpha^{\vee}\\
              = & \left( e^{c'}_\alpha \cdot \pi \right)(t) + \log\left( \frac{ 1 + \frac{ e^{c+c'} - 1 }{e^{\varepsilon_\alpha(\pi)}}\int_0^t e^{ -\alpha(\pi(s)) }ds }{ 1 + \frac{ e^{c'} - 1 }{e^{\varepsilon_\alpha(\pi)}}\int_0^t e^{ -\alpha(\pi(s)) }ds } \right) \alpha^{\vee}\\
              = & \left( e^{c+c'}_\alpha \cdot \pi \right)(t)\\
               \end{align*}
 \item[(vi)]   \begin{align*}
                & \left( \iota \circ e^c_\alpha \circ \iota \right)(\pi) (t)\\
              = & \iota\left( t \mapsto \pi^\iota(t) + \log\left( 1 + \frac{ e^c - 1 }{e^{\varepsilon_\alpha(\pi^\iota)}}\int_0^t e^{ -\alpha(\pi^\iota(s)) }ds \right) \alpha^\vee \right)(t)\\
              = & \pi(t) + \iota\left( t \mapsto \log\left( 1 + \frac{ e^c - 1 }{e^{\varphi_\alpha(\pi)}}\int_0^t e^{ -\alpha(\pi^\iota(s)) }ds \right) \alpha^\vee \right)(t)\\
              = & \pi(t) + \iota\left( t \mapsto \log\left( 1 + \frac{ e^c - 1 }{e^{\varphi_\alpha(\pi)}} e^{\alpha(\pi(T))} \int_{T-t}^T e^{ -\alpha(\pi(s)) }ds \right) \alpha^\vee \right)(t)\\
              = & \pi(t) + \iota\left( t \mapsto \log\left( 1 + \frac{ e^c - 1 }{e^{\varepsilon_\alpha(\pi)}} \int_{T-t}^T e^{ -\alpha(\pi(s)) }ds \right) \alpha^\vee \right)(t)\\
              = & \pi(t) + \log\left( 1 + \frac{ e^c - 1 }{e^{\varepsilon_\alpha(\pi)}} \int_t^T e^{ -\alpha(\pi(s)) }ds \right) \alpha^\vee
                         - \log\left( 1 + \frac{ e^c - 1 }{e^{\varepsilon_\alpha(\pi)}} \int_0^T e^{ -\alpha(\pi(s)) }ds \right) \alpha^\vee \\
              = & \pi(t) + \log\left( 1 + \frac{ e^c - 1 }{e^{\varepsilon_\alpha(\pi)}} \left( \int_0^T e^{ -\alpha(\pi(s)) }ds - \int_0^t e^{ -\alpha(\pi(s)) }ds \right) \right) \alpha^\vee - c \alpha^\vee \\
              = & \pi(t) + \log\left( e^c - \frac{ e^c - 1 }{e^{\varepsilon_\alpha(\pi)}} \int_0^t e^{ -\alpha(\pi(s)) }ds \right) \alpha^\vee - c \alpha^\vee \\
              = & \pi(t) + \log\left( 1 + \frac{ e^{-c} - 1 }{e^{\varepsilon_\alpha(\pi)}} \int_0^t e^{ -\alpha(\pi(s)) }ds \right) \alpha^\vee
              \end{align*}
              As for $\varphi_\alpha = \varepsilon_\alpha \circ \iota $, it is obvious.
\end{enumerate}
\end{proof}

\subsubsection{Tensor products of crystals and concatenation of paths}
In this subsection, we will see that the seemingly complicated definition for the tensor product of crystals is in fact easily coded within a path model using the concatenation of paths. Define the concatenation of two paths $\pi_1: [0,T] \rightarrow \afrak$ and $\pi_2: [0,S] \rightarrow \afrak$ as 
the path $\pi_1 \ast \pi_2: [0,T+S] \rightarrow \afrak$ given by:
$$ \pi_1 \ast \pi_2 \left(t\right)
   = \left\{\begin{array}{cc}
             \pi_1(t)                    & \textrm{ if } 0 \leq t \leq T \\
             \pi_1(T) + \pi_2\left( t - T\right) & \textrm{ otherwise }
            \end{array}
     \right.
$$

\begin{thm}
\label{thm:concatenation_is_isomorphism}
$$ \begin{array}{cccc}
   \theta: & C_0\left( [0; T], \afrak \right) \otimes C_0\left( [0; S], \afrak \right) & \rightarrow & C_0\left( [0; T+S], \afrak \right)\\
       & \pi_1 \otimes \pi_2                                                 & \mapsto     & \pi_1 \ast \pi_2
   \end{array}
$$
is a crystal isomorphism. In fact, the following properties are true:
\begin{itemize}
	\item[(i)] $\gamma\left( \pi_1 \ast \pi_2 \right) = \gamma\left( \pi_1 \otimes \pi_2 \right)$
	\item[(ii)] $\varepsilon_\alpha\left( \pi_1 \ast \pi_2 \right) = \varepsilon_\alpha\left( \pi_1 \otimes \pi_2 \right)$ or equivalently $\varphi_\alpha\left( \pi_1 \ast \pi_2 \right) = \varphi_\alpha\left( \pi_1 \otimes \pi_2 \right)$
	\item[(iii)] $e^c_\alpha\left( \pi_1 \ast \pi_2 \right) = \theta\left( e^c_\alpha\left( \pi_1 \otimes \pi_2 \right) \right)$
\end{itemize}
\end{thm}
\begin{proof}
Given those properties, $\theta$ clearly transports the crystal structure. The fact that it is invertible with a morphism as inverse map is obvious. Let us show these relations:
\begin{itemize}
	\item[(i)] \begin{align*}
	             & \gamma\left( \pi_1 \ast \pi_2 \right)\\
	           = & \pi_1 \ast \pi_2\left( T + S\right)\\
	           = & \pi_1(T) + \pi_2(S)\\
	           = & \gamma\left( \pi_1 \otimes \pi_2 \right)
	           \end{align*}
	\item[(ii)] \begin{align*}
	             & \varepsilon_\alpha\left( \pi_1 \ast \pi_2 \right)\\
	           = & \log\left( \int_0^{T+S} e^{-\alpha\left( \pi_1 \ast \pi_2(s) \right) } ds \right)\\
	           = & \log\left( \int_0^{T} e^{-\alpha\left( \pi_1 (s) \right) } ds + e^{ -\alpha\left( \pi_1(T)\right) } \int_0^{S} e^{-\alpha\left( \pi_2(s) \right) } ds \right)\\
	           = & \varepsilon_\alpha\left( \pi_1 \right) + \log\left( 1 + e^{\varepsilon_\alpha\left(\pi_2\right) - \alpha\left( \gamma(\pi_1) \right) - \varepsilon_\alpha\left(\pi_2\right)} \right)\\
	           = & \varepsilon_\alpha\left( \pi_1 \right) + \log\left( 1 + e^{\varepsilon_\alpha\left(\pi_2\right) - \varphi_\alpha\left(\pi_2\right)} \right)\\
	           = & \varepsilon_\alpha\left( \pi_1 \otimes \pi_2 \right)	           
	            \end{align*}
	\item[(iii)] One first needs to remember that
	             $$  e^c_\alpha\left( \pi_1 \otimes \pi_2 \right) = e^{c_1}_\alpha \cdot \pi_1 \otimes e^{c_2}_\alpha \cdot \pi_2 $$
	             with
	             $$ e^{c_1} = \frac{ e^{c+\varphi_\alpha\left( \pi_1 \right) } + e^{ \varepsilon_\alpha\left(\pi_2\right) } }
	                               { e^{  \varphi_\alpha\left( \pi_1 \right) } + e^{ \varepsilon_\alpha\left(\pi_2\right) } }  $$
	             $$ e^{c_2} = \frac{ e^{  \varphi_\alpha\left( \pi_1 \right) } + e^{ \varepsilon_\alpha\left(\pi_2\right) } }
	                               { e^{  \varphi_\alpha\left( \pi_1 \right) } + e^{-c+ \varepsilon_\alpha\left(\pi_2\right) } }  $$
	            So that we will compute both parts of the concatenated path $\theta\left( e^c_\alpha\left( \pi_1 \otimes \pi_2 \right) \right)$ separately:
	            - The first half is, using $0 \leq t \leq T$:
	            $$ \theta\left( e^c_\alpha\left( \pi_1 \otimes \pi_2 \right) \right)(t) = e^{c_1}_\alpha \cdot \pi_1(t) = \pi_1(t) + \alpha^\vee \log\left( 1 + \frac{e^{c_1}-1}{e^{\varepsilon_\alpha\left(\pi_1\right)}} \int_0^t e^{-\alpha\left( \pi_1(s)\right)}ds \right)$$
	            But as:
	            \begin{align*}
	              & \frac{e^{c_1}-1}{e^{\varepsilon_\alpha\left(\pi_1\right)}}\\
	            = & \frac{\frac{ e^{c+\varphi_\alpha\left( \pi_1 \right) } + e^{ \varepsilon_\alpha\left( \pi_2 \right) } }{ e^{  \varphi_\alpha\left( \pi_1 \right) } + e^{ \varepsilon_\alpha\left( \pi_2 \right) }  }-1}{ e^{\varepsilon_\alpha\left(\pi_1\right)} }\\
	            = & \frac{ e^{\varphi_\alpha\left( \pi_1 \right) } \left( e^c - 1 \right) }
	                     { e^{\varepsilon_\alpha\left(\pi_1\right)}\left( e^{  \varphi_\alpha\left( \pi_1 \right) } + e^{ \varepsilon_\alpha\left( \pi_2 \right) } \right) } \\
	            = & \frac{ e^c - 1  }
	                     { e^{\varepsilon_\alpha\left(\pi_1\right)} + e^{ \varepsilon_\alpha\left( \pi_2 \right) - \alpha\left( \gamma\left(\pi_1\right) \right) } } \\
	            = & \frac{ e^c - 1  }{ e^{\varepsilon\left( \pi_1 \ast \pi_2 \right)} }
	            \end{align*}
	            We get:
	            $$ \forall 0 \leq t \leq T, \theta\left( e^c_\alpha\left( \pi_1 \otimes \pi_2 \right) \right)(t) =  e^c_\alpha \cdot \left( \pi_1 \ast \pi_2 \right)(t)$$
	            - Moving on to the second half:
              \begin{align*}
                & \theta\left( e^c_\alpha\left( \pi_1 \otimes \pi_2 \right) \right)(T + t)\\
              = & e^{c_1}_\alpha \cdot \pi_1(T) + e^{c_2}_\alpha \cdot \pi_2(t)\\
              = & \pi_1(T) + c_1 \alpha^\vee + \pi_2(t) + \log\left( 1 + \frac{e^{c_2}-1}{e^{\varepsilon_\alpha\left(\pi_2\right)}} \int_0^t e^{-\alpha\left( \pi_2(s)\right)}ds \right) \alpha^\vee \\
              = & \pi_1 \ast \pi_2 (T+t) + c_1 \alpha^\vee + \log\left( 1 + \frac{e^{c_2}-1}{e^{\varepsilon_\alpha\left(\pi_2\right) }} \int_0^t e^{-\alpha\left( \pi_2(s)\right)}ds \right)\alpha^\vee
              \end{align*}
              Then, as:
              \begin{align*}
	              & \frac{e^{c_2}-1}{e^{\varepsilon_\alpha\left(\pi_2\right)}}\\
	            = & \frac{\frac{ e^{\varphi_\alpha\left( \pi_1 \right) } + e^{ \varepsilon_\alpha\left( \pi_2 \right) } }{ e^{  \varphi_\alpha\left( \pi_1 \right) } + e^{ -c + \varepsilon_\alpha\left( \pi_2 \right) }  }-1}{ e^{\varepsilon_\alpha\left(\pi_2\right)} }\\
	            = & \frac{ 1 - e^{-c}  }
	                     { e^{\varphi_\alpha\left(\pi_1\right)} + e^{ -c +\varepsilon_\alpha\left( \pi_2 \right)}} \\
	            \end{align*}
	            We get:
	            \begin{align*}
                & \theta\left( e^c_\alpha\left( \pi_1 \otimes \pi_2 \right) \right)(T + t)\\
              = & \pi_1 \ast \pi_2 (T+t) + c_1 \alpha^\vee + \log\left( 1 + \frac{ e^{c}-1 }{ e^{ c + \varphi_\alpha\left(\pi_1\right) } + e^{\varepsilon_\alpha\left(\pi_2\right) }} \int_0^t e^{-\alpha\left( \pi_2(s)\right)}ds \right) \alpha^\vee \\
              = & \pi_1 \ast \pi_2 (T+t) + c_1 \alpha^\vee + \\
                & \quad \log\left( 1 + \frac{ e^{c}-1 }{ e^{ c + \varphi_\alpha\left(\pi_1\right) } + e^{\varepsilon_\alpha\left(\pi_2\right) }} e^{\alpha\left( \gamma(\pi_1) \right) } \left( \int_0^{T+t} e^{-\alpha\left( \pi_1 \ast \pi_2(s)\right)}ds - e^{\varepsilon_\alpha\left( \pi_1 \right) }\right) \right) \alpha^\vee \\
              \end{align*}
              But 
              \begin{align*}
                & 1- \frac{ e^{c}-1 }{ e^{ c + \varphi_\alpha\left(\pi_1\right) } +  e^{ \varepsilon_\alpha\left( \pi_2 \right) }} e^{\alpha\left( \gamma(\pi_1) \right) + \varepsilon_\alpha\left( \pi_1 \right) }\\
              = & 1- \frac{ e^{c}-1 }{ e^{ c + \varphi_\alpha\left(\pi_1\right) } +  e^{ \varepsilon_\alpha\left( \pi_2 \right) }} e^{ \varphi_\alpha\left( \pi_1 \right) }\\
              = & \frac{ e^{ \varphi_\alpha\left( \pi_1 \right) } + e^{ \varepsilon_\alpha\left( \pi_2 \right) } }{ e^{ c + \varphi_\alpha\left(\pi_1\right) } +  e^{ \varepsilon_\alpha\left( \pi_2 \right) }}\\
              = & e^{-c_1}
              \end{align*}
              So that:
              \begin{align*}
                & \theta\left( e^c_\alpha\left( \pi_1 \otimes \pi_2 \right) \right)(T + t)\\
              = & \pi_1 \ast \pi_2 (T+t) + c_1 \alpha^\vee + \log\left( e^{-c_1} + \frac{ e^{c}-1 }{ e^{ c + \varphi_\alpha\left(\pi_1\right) } + e^{\varepsilon_\alpha\left(\pi_2\right) }} e^{\alpha\left( \gamma(\pi_1) \right) } \int_0^{T+t} e^{-\alpha\left( \pi_1 \ast \pi_2(s)\right)}ds \right) \alpha^\vee \\
              = & \pi_1 \ast \pi_2 (T+t) + \log\left( 1 + e^{c_1} \frac{ e^{c}-1 }{ e^{ c + \varphi_\alpha\left(\pi_1\right) } + e^{\varepsilon_\alpha\left(\pi_2\right) }} e^{\alpha\left( \gamma(\pi_1) \right) } \int_0^{T+t} e^{-\alpha\left( \pi_1 \ast \pi_2(s)\right)}ds \right) \alpha^\vee \\
              = & \pi_1 \ast \pi_2 (T+t) + \log\left( 1 + \frac{ \frac{ e^{c}-1 }{ e^{ c + \varphi_\alpha\left(\pi_1\right) } + e^{\varepsilon_\alpha\left(\pi_2\right) }} e^{\alpha\left( \gamma(\pi_1) \right) } }{ 1 - \frac{ e^{c}-1 }{ e^{ c + \varphi_\alpha\left(\pi_1\right) } + e^{\varepsilon_\alpha\left(\pi_2\right) }} e^{\alpha\left( \gamma(\pi_1) \right) + \varepsilon_\alpha\left( \pi_1 \right) }} \int_0^{T+t} e^{-\alpha\left( \pi_1 \ast \pi_2(s)\right)}ds \right) \alpha^\vee \\
              = & \pi_1 \ast \pi_2 (T+t) + \log\left( 1 + \frac{ \left( e^{c}-1 \right) e^{\alpha\left( \gamma(\pi_1) \right) } }{  e^{ c + \varphi_\alpha\left(\pi_1\right) } + e^{\varepsilon_\alpha\left(\pi_2\right) } - \left( e^{c}-1 \right) e^{\varphi_\alpha\left( \pi_1 \right) }} \int_0^{T+t} e^{-\alpha\left( \pi_1 \ast \pi_2(s)\right)}ds \right) \alpha^\vee \\
              = & \pi_1 \ast \pi_2 (T+t) + \log\left( 1 + \frac{ e^{c}-1 }{ e^{ \varphi_\alpha\left(\pi_1\right) } + e^{-\alpha\left( \gamma(\pi_1) \right)+\varepsilon_\alpha\left(\pi_2\right) } } \int_0^{T+t} e^{-\alpha\left( \pi_1 \ast \pi_2(s)\right)}ds \right) \alpha^\vee \\
              = & \pi_1 \ast \pi_2 (T+t) + \log\left( 1 + \frac{ e^{c}-1 }{e^{\varepsilon_\alpha\left(\pi_1 \ast \pi_2(s) \right)}} \int_0^{T+t} e^{-\alpha\left( \pi_1 \ast \pi_2(s)\right)}ds \right) \alpha^\vee \\
              = & e^c_\alpha\left( \pi_1 \ast \pi_2 \right)( T + t )
              \end{align*}
\end{itemize}
\end{proof}

\subsection{Projection on the group picture}
A little lemma shows how the Littelmann operators are linked to the transform $T_g$:
\begin{lemma}
 For $g = x_\alpha\left( \frac{e^c-1}{e^{\varepsilon_\alpha(\pi)}} \right)$, one has:
$$\forall \pi \in C_0\left( [0, T], \afrak \right), e^c_\alpha \pi = T_g \pi$$
\end{lemma}
Such an expression for the group element $x_\alpha\left( \frac{e^c-1}{e^{\varepsilon_\alpha(\pi)}} \right)$ is no coincidence, as it looks very similar to the left action on the crystal $\Bc$.\\
Let $L$ be a geometric Littelmann crystal. We define a projection map:
$$ \begin{array}{cccc}
    p: & L   & \rightarrow & \Bc = \left( B \cap B^+ w_0 B^+ \right)_{\geq 0}\\
       & \pi & \mapsto     & B_T\left(\pi\right)
   \end{array}
$$
We claim that $p$ is surjective provided that $L$ is large enough. A nice Brownian proof could be developped using the following idea: when $\pi$ is a Brownian motion, the random variable $B_T\left(\pi\right)$ has a density with support equal to $\Bc$. We will not pursue such a lead, as we will prove a more precise statement, dealing with the parametrization of connected components in the next section. For now, let us show that $p$ transports structures:

\begin{thm}
$p$ is a morphism of abstract crystals, as the following properties hold:
\begin{itemize}
 \item[(i)]   $\gamma\left( \pi \right) = \gamma\left( p\left( \pi \right) \right)$
 \item[(ii)]  $p \circ \iota = \iota \circ p$ where $\iota$ stands for duality on the left-hand side and for positive inverse on the right-hand side.
 \item[(iii)] $\varepsilon_\alpha = \varepsilon_\alpha \circ p$ or equivalently $\varphi_\alpha = \varphi_\alpha \circ p$
 \item[(iv)] $e^c_\alpha \cdot p \left( \pi \right) = p\left( e^c_\alpha \cdot \pi \right)$
\end{itemize}
\end{thm}
\begin{proof}
\begin{itemize}
 \item[(i)]   \begin{align*}
                & \gamma\left( p\left( \pi \right) \right)\\
              = & \gamma\left( B_T\left(\pi\right) \right)\\
              = & \pi(T)\\
              = & \gamma\left( \pi \right)
              \end{align*}
 \item[(ii)]  \begin{align*}
                & p \circ \iota \left( \pi \right)\\
              = & p\left( \pi^\iota \right)\\
              = & \left(\sum_{k \geq 1} \sum_{ i_1, \dots, i_k } \int_{ T \geq t_k \geq \dots \geq t_1 \geq 0} e^{ -\alpha_{i_1}(\pi^\iota(t_1)) \dots -\alpha_{i_k}(\pi^\iota(t_k)) } f_{i_1} f_{i_2} \dots f_{i_k} dt_1 \dots dt_k\right)e^{\pi^\iota(T)} + e^{ \pi^\iota(T) }\\
              \end{align*}
              Since:
              \begin{align*}
                & e^{ -\alpha_{i_1}(\pi^\iota(t_1)) \dots -\alpha_{i_k}(\pi^\iota(t_k)) } f_{i_1} f_{i_2} \dots f_{i_k}\\
              = & e^{ -\alpha_{i_1}(\pi(T-t_1)) \dots -\alpha_{i_k}(\pi(T-t_k)) + \alpha_{i_1}(\pi(T)) \dots + \alpha_{i_k}(\pi(T)) } f_{i_1} f_{i_2} \dots f_{i_k}\\
              = & e^{ -\alpha_{i_1}(\pi(T-t_1)) \dots -\alpha_{i_k}(\pi(T-t_k)) } e^{ -\pi(T) } f_{i_1} f_{i_2} \dots f_{i_k} e^{ \pi(T) }\\
              \end{align*}
              We get:
              \begin{align*}
                & p \circ \iota \left( \pi \right)\\
              = & \left( \sum_{k \geq 1} \sum_{ i_1, \dots, i_k } \int_{ T \geq t_k \geq \dots \geq t_1 \geq 0} e^{ -\alpha_{i_1}(\pi(T-t_1)) \dots -\alpha_{i_k}(\pi(T-t_k)) } e^{-\pi(T)} f_{i_1} f_{i_2} \dots f_{i_k} e^{ \pi(T) } dt_1 \dots dt_k\right) \\
                & \ \ \ e^{-\pi(T)} + e^{-\pi(T)} \\
              = &  e^{-\pi(T)}\left( id + \sum_{k \geq 1} \sum_{ i_1, \dots, i_k } \int_{ T \geq t_k \geq \dots \geq t_1 \geq 0} e^{ -\alpha_{i_1}(\pi(T-t_1)) \dots -\alpha_{i_k}(\pi(T-t_k)) }f_{i_1} f_{i_2} \dots f_{i_k} dt_1 \dots dt_k\right) \\
              = & \left( \left(\sum_{k \geq 1} \sum_{ i_1, \dots, i_k } \int_{ T \geq t_k \geq \dots \geq t_1 \geq 0} e^{ -\alpha_{i_1}(\pi(T-t_1)) \dots -\alpha_{i_k}(\pi(T-t_k)) } f_{i_k} \dots f_{i_2} f_{i_1} dt_1 \dots dt_k + id\right) e^{ \pi(T) } \right)^{ \iota}\\
              = & \left( \left(\sum_{k \geq 1} \sum_{ i_1, \dots, i_k } \int_{ T \geq t_1 \geq \dots \geq t_k \geq 0} e^{ -\alpha_{i_1}(\pi(t_1)) \dots -\alpha_{i_k}(\pi(t_k)) } f_{i_k} \dots f_{i_2} f_{i_1} dt_1 \dots dt_k + id\right) e^{ \pi(T) } \right)^{ \iota}\\
              = & \left( \left(\sum_{k \geq 1} \sum_{ i_1, \dots, i_k } \int_{ T \geq t_k \geq \dots \geq t_1 \geq 0} e^{ -\alpha_{i_1}(\pi(t_1)) \dots -\alpha_{i_k}(\pi(t_k)) } f_{i_1} f_{i_2} \dots f_{i_k} dt_1 \dots dt_k + id\right) e^{ \pi(T) } \right)^{ \iota}\\
              = & \left( B_T\left(\pi\right) \right)^\iota\\
              = & \iota \circ p \left( \pi \right)
              \end{align*}
 \item[(iii)] The two propositions are equivalent as in both the path model and Berenstein and Kazhdan's model, $\varepsilon_\alpha = \varphi_\alpha \circ \iota$.
              \begin{align*}
                & \varepsilon_\alpha \circ p (\pi)\\
              = & \chi_\alpha\left( B_T\left(\pi\right) \right)\\
              = & \chi_\alpha\left( N_T\left(\pi\right) \right)
             \end{align*}
             All that remains to be proven is $e^{ \chi_\alpha\left( N_T\left(\pi\right) \right) }= \int_0^T e^{-\alpha\left(\pi(s)\right) } ds $. Both expressions coincide for $T=0$, and have the same derivatives with respect to $T$.
 \item[(iv)] \begin{align*}
               & e^c_\alpha \cdot p \left( \pi \right)\\
             = & e^c_\alpha \cdot B_T\left(\pi\right)\\
             = & x_\alpha\left( \frac{e^c-1}{e^{\varepsilon_\alpha\left( B_T\left(\pi\right) \right)}} \right) \cdot B_T\left(\pi\right) \cdot x_\alpha\left( \frac{e^{-c}-1}{e^{\varphi_\alpha\left( B_T\left(\pi\right) \right)}} \right)\\
             = & B_T\left( T_g \pi \right)\\
               & \quad \textrm{ where } g = x_\alpha\left( \frac{e^c-1}{e^{\varepsilon_\alpha\left( \pi \right)}} \right)\\
             = & B_T\left( e^c_\alpha \cdot \pi \right)\\
             = & p\left( e^c_\alpha \cdot \pi \right)
             \end{align*}
\end{itemize} 
\end{proof}

\subsection{Verma relations}
Thanks to the previous subsection, we know that the geometric crystals given by the path model, in a certain sense, sit above the group picture $\Bc$. The Verma relations are also valid at the path level: Given a geometric Littelmann crystal $L$, for any ${\bf i} \in I^k$, $k \in \N$ consider the map $e_{\bf i}^.$ as in subsection \ref{subsection:additional_structure}.

The analogue of proposition \ref{proposition:group_verma} holds:
\begin{proposition}
 In the Littelmann geometric path model, $e_{\bf i}$ depends only on:
 $$ w = s_{i_1} \dots s_{i_k} \in W$$
\end{proposition}
\begin{proof}
For $\pi \in C_0\left( [0, T], \afrak \right)$, $t \in \afrak$ and ${\bf i}, {\bf i'}$ words defining the same Weyl group element consider:
$$ \eta  = e^t_{\bf i } \cdot \pi$$
$$ \eta' = e^t_{\bf i'} \cdot \pi$$
Now let us prove that $\eta = \eta'$. Because the Littelmann path operators can be expressed thanks to the operator $T_.$, there are two elements $u, u' \in U$ such that:
$$ \eta = T_u \pi, \eta' = T_{u'} \pi$$
Furthermore, after applying the crystal morphism $p = B_T(.)$:
$$ B_T(\eta) = [u B_T(\pi)]_{-0}, B_T(\eta') = [u' B_T(\pi)]_{-0}$$
But since the Verma relations hold for the group picture (proposition \ref{proposition:group_verma}):
$$ B_T(\eta) = B_T(\eta')$$
Now, using the fact that $[u B_T(\pi)]_{-0} = u B_T(\pi) [u B_T(\pi)]_{+}^{-1}$ write:
\begin{align*}
g & = [\bar{w}_0^{-1} B_T(\eta)^\iota]_+\\
  & = [\bar{w}_0^{-1} \left( [u B_T(\pi)]_+^{-1} \right)^\iota B_T(\pi)^\iota u^\iota]_+\\
  & = [\bar{w}_0^{-1} B_T(\pi)^\iota ]_+ u^\iota
\end{align*}
Symetrically, $g=[\bar{w}_0^{-1} B_T(\pi)^\iota ]_+ u^\iota=[\bar{w}_0^{-1} B_T(\pi)^\iota ]_+ (u')^\iota$. Hence $u=u'$ and $\eta=\eta'$.
\end{proof}

Notice that one can also prove the Verma relations for the path model by direct computation on paths, though it is more complicated. Instead of group operations, one has to do multiple integrations by parts. Here we only give a partial sketch in the simply-laced case, using a classical procedure.

\paragraph{ $\mathbf{A_1}$ case:} In the case that $\alpha\left( \beta^\vee \right) = \beta\left( \alpha^\vee \right) = 0$, the actions $e_\alpha^.$ and $e_\beta^.$ commute, which proves the required Verma relation for type $A_1$:
$$ e^{c_1}_\alpha \cdot e^{c_2}_\beta = e^{c_2}_\beta \cdot e^{c_1}_\alpha$$

\paragraph{ $\mathbf{A_2}$ case:} By writing $t = c_1 \omega_1 + c_2 \omega_2$, the Verma relationship becomes
$$ e^{c_1 }_\alpha \cdot e^{c_1 + c_2 }_\beta \cdot e^{c_2 }_\alpha 
 = e^{c_2 }_\beta \cdot e^{c_1 + c_2 }_\alpha \cdot e^{c_1 }_\beta $$

A tedious computation gives the following lemma, that we give without proof:
\begin{lemma}
\label{label:lemma_verma_A2}
If $\alpha\left( \beta^\vee \right) = \beta\left( \alpha^\vee \right) = -1$, then
\begin{align*}
        & e^{c_1 }_\alpha \cdot e^{c_1 + c_2 }_\beta \cdot e^{c_2 }_\alpha \cdot \pi (t) = \pi(t)\\
        & + \alpha^\vee \log\left( 1 + \left( e^{c_1 + c_2} - 1 \right)
\frac
{\int_0^t e^{-\alpha(\pi(s))} \left( 1 + \left( e^{c_1} - 1 \right) \frac{\int_0^s e^{ -\beta(\pi(u)) }du }{\int_0^T e^{ -\beta(\pi(u)) }du } \right)}
{\int_0^T e^{-\alpha(\pi(s))} \left( 1 + \left( e^{c_1} - 1 \right) \frac{\int_0^s e^{ -\beta(\pi(u)) }du }{\int_0^T e^{ -\beta(\pi(u)) }du } \right)} \right) \\
        & + \beta^\vee  \log\left( 1 + \left( e^{c_1 + c_2} - 1 \right)
\frac
{\int_0^t e^{-\beta(\pi(s))} \left( 1 + \left( e^{c_2} - 1 \right) \frac{\int_0^s e^{ -\alpha(\pi(u)) }du }{\int_0^T e^{ -\alpha(\pi(u)) }du } \right)}
{\int_0^T e^{-\beta(\pi(s))} \left( 1 + \left( e^{c_2} - 1 \right) \frac{\int_0^s e^{ -\alpha(\pi(u)) }du }{\int_0^T e^{ -\alpha(\pi(u)) }du } \right)} \right)
\end{align*}
\end{lemma}
By inspecting the formula, one realizes that is symmetric in $\alpha$ and $c_1$ on the one hand, and $\beta$ and $c_2$ on the other hand. As such, by swapping those variables in the left hand-side term, one gets the Verma relation for type $A_2$:
$$ e^{c_1 }_\alpha \cdot e^{c_1 + c_2 }_\beta \cdot e^{c_2 }_\alpha 
 = e^{c_2 }_\beta \cdot e^{c_1 + c_2 }_\alpha \cdot e^{c_1 }_\beta $$

\paragraph{ $\mathbf{ADE}$ case:} Root systems from the ADE classifications have Dynkin diagrams with single edges as $\alpha\left( \beta^\vee \right) = -1$ in all cases. Hence, the only Verma relations needed are of type $A_1$ and $A_2$ and $e_w$ is unambiguously defined for any element $w$ in a Weyl group of $ADE$ type.

\section{Parametrizing geometric path crystals}
In this section, we will show how connected components of a geometric Littelmann path crystal are parametrized by the totally positive group elements. In the same fashion as in the group picture (section \ref{section:geom_lifting}), we will associate to every path its Lusztig parameter, an element in $U^{w_0}_{>0}$ or equivalently a Kashiwara parameter in $C^{w_0}_{>0}$. For such an endeavour, we will need to consider paths whose endpoint or starting point is not defined anymore, as we will move to the 'edges' of the crystal. Such paths will be referred to as extended paths. The highest weight path will be given by a path transform $\Tc_{w_0}$.

More precisely, we fix a time horizon $T>0$, then consider a path $\pi \in C_0\left( [0, T], \afrak \right)$ and the crystal $\langle \pi \rangle$ it generates. We will show how Lusztig parameters and Kashiwara (or string) parameters can be retrieved from a path. The expressions found for string coordinates are geometric liftings of the formulas used in the classical Littelmann path model. Basically, we construct maps for every reduced word ${\bf i} \in R(w_0)$, $m = \ell(w_0)$:
$$\varrho_{\bf i}^L : \langle \pi \rangle \rightarrow \left(\R_{>0} \right)^m$$
$$\varrho_{\bf i}^K : \langle \pi \rangle \rightarrow \left(\R_{>0} \right)^m$$
The maps $\varrho_{\bf i}^L$ (resp. $\varrho_{\bf i}^K$) give the Lusztig (resp. Kashiwara) parameters for a path, in the sense that the diagram \ref{fig:parametrizations_diagram} is commutative (corollaries \ref{corollary:commutative_diagram_kashiwara} and \ref{corollary:commutative_diagram_lusztig}). This completes the group picture from figure \ref{fig:geom_parametrizations}. Such an important commutative diagram is also reproduced in appendix \ref{appendix:parametrizations_reminder}.

\begin{figure}[htp!]
\centering
\begin{tikzpicture}[baseline=(current bounding box.center)]
\matrix(m)[matrix of math nodes, row sep=3em, column sep=5em, text height=3ex, text depth=1ex, scale=1.2]
{ 
U_{>0}^{w_0} & \left( \R_{>0} \right)^m & \langle \pi \rangle    & \left( \R_{>0} \right)^m & C_{>0}^{w_0}\\
             &                          & \Bc(\lambda)\\
};
\path[->, font=\scriptsize] (m-1-3) edge node[above, scale=1.5]{$\varrho_{\bf i}^K$} (m-1-4);
\path[->, font=\scriptsize] (m-1-4) edge node[above, scale=1.5]{$x_{-\bf i}$} (m-1-5);
\path[->, font=\scriptsize] (m-1-3) edge node[above, scale=1.5]{$\varrho_{\bf i}^L$} (m-1-2);
\path[->, font=\scriptsize] (m-1-2) edge node[above, scale=1.5]{$x_{ \bf i}$} (m-1-1);

\path[->, font=\scriptsize] (m-1-3) edge node[left,  scale=1.5]{$p$} (m-2-3);
\draw[->] (m-2-3) to [bend right=30] node[right]{$p^{-1}$} (m-1-3);

\path[->, font=\scriptsize] (m-1-1) edge node[below, scale=1.5]{$b^L_\lambda$} (m-2-3);
\draw[->] (m-2-3) to [bend left=30] node[below]{$\varrho^L$} (m-1-1);
\path[->, font=\scriptsize] (m-1-5) edge node[below, scale=1.5]{$b^K_\lambda$} (m-2-3);
\draw[->] (m-2-3) to [bend right=30] node[below]{$\varrho^K$} (m-1-5);
\end{tikzpicture} 
\caption{Parametrizations for a connected crystal $\langle \pi \rangle$, with $\pi \in C_0([0, T], \afrak)$ and $\lambda = \Tc_{w_0} \pi(T)$}
\label{fig:parametrizations_diagram}
\end{figure}
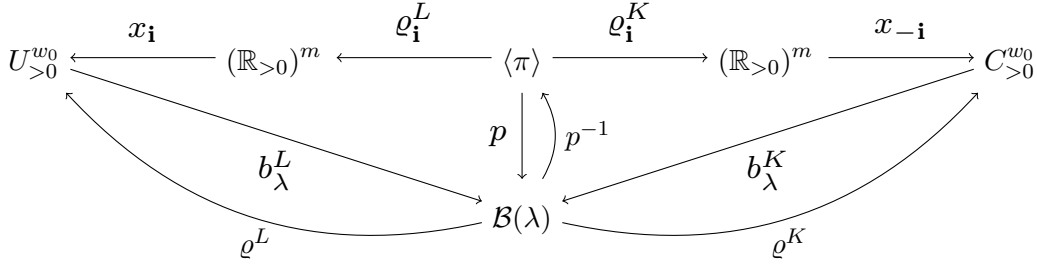

A variant of Littelmann's independence theorem is then proved, with the projection map $p$ being an automorphism of crystals once restricted to connected components (theorem \ref{thm:geometric_rsk_correspondence}). We interpret it as a geometric version of the Robinson-Schensted-Knuth correspondence, for which we also give a dynamical version (theorem \ref{thm:dynamic_rsk_correspondence}).

Finally, for the purpose of greater generality, we will also consider a possibly infinite time horizon, an essential ingredient for chapter \ref{chapter:random_crystals}. In such a case, paths will need to have a drift inside the Weyl chamber (subsection \ref{subsection:infinite_T}).

\subsection{Extended paths and related path transforms}
A novelty in the path model approach to geometric crystals is the appearance of extended paths at the 'edges' of geometric crystals. This allows a simple compactification that does not involve the geometry of Bruhat cells. A visual sketch is given in figure \ref{figure:ExtremalPaths}.

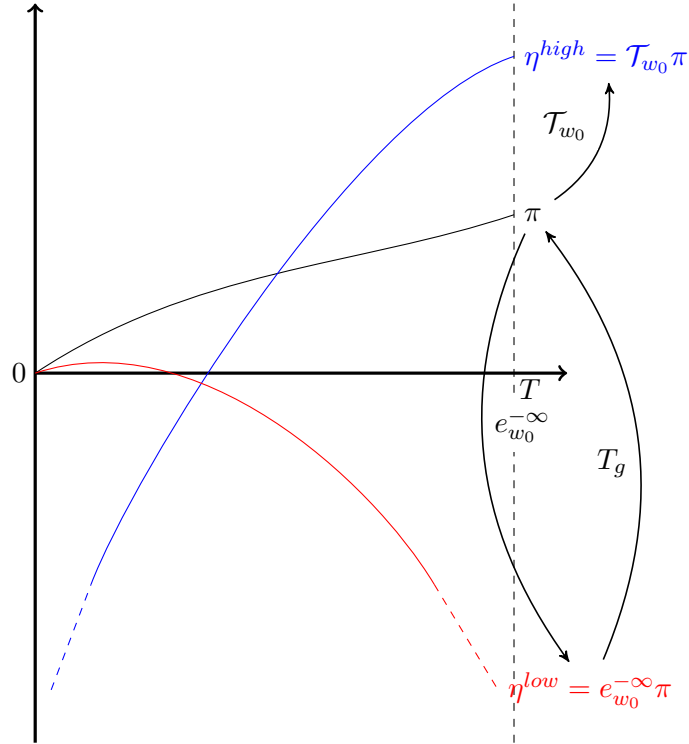
\begin{figure}[ht!]
\centering
\caption{Sketch of extremal paths corresponding to a geometric Littelmann path $\pi$ with Lusztig parameter $g \in U^{w_0}_{>0}$ }
\label{figure:ExtremalPaths}

\begin{tikzpicture}[auto, scale=0.7]
\node at ( -0.3, 0 ) {$0$};
\node at ( 9.3, -0.3 ) {$T$};
\path [draw, very thick, color = black, ->] (0,0) -- (10 , 0 );
\path [draw, very thick, color = black, ->] ( 0, -7 ) -- ( 0, 7 );

\path [draw, dashed, color = black] (9,-7) -- (9 , 7 );

\definecolor{red  }{rgb}{ 0.5,0.0,0.0 }
\definecolor{blue }{rgb}{ 0.0,0.0,0.5 }
\definecolor{black}{rgb}{ 0.0,0.0,0.0 }

\draw (0,0) .. controls (3,2) and (6,2)  .. (9,3) node (pi) [right]{$\pi$};

\draw[dashed, color=blue] (0.3,-6) -- (1.05,-4);
\draw[color=blue] (1.05, -4) .. controls (1.8,-2) and (6, 5) .. (9, 6) node (high) [right]{$\eta^{high} = \Tc_{w_0} \pi$};

\draw[color=red] (0,0) .. controls (3,1) and (6.3,-2) .. (7.5, -4);
\draw[dashed, color=red] (7.5, -4) -- (8.7, -6) node (low) [right]{$\eta^{low} = e^{-\infty}_{w_0} \pi$};

\Edge[lw=0.1cm,style={post, bend right,color=black,}, label=$T_g$](low)(pi)
\Edge[lw=0.1cm,style={post, bend right,color=black,}, label=$\Tc_{w_0}$](pi)(high)
\Edge[lw=0.3cm,style={post, bend right,color=black,}, label=$e^{-\infty}_{w_0}$](pi)(low)

\end{tikzpicture}

\end{figure}

\subsubsection{High path transforms}
A first example giving extended paths was introduced in \cite{bib:BBO}: for every simple root $\alpha$, define the following transform of a continuous path $\pi \in C(\R_+^*, \afrak )$, such that $e^{-\alpha(\pi)}$ is integrable at the neighborhood of $0$:
$$ \forall t > 0, \Tc_\alpha(\pi)(t) := \pi(t) + \log\left(\int_0^t e^{-\alpha(\pi(s))} ds\right) \alpha^{\vee} $$
Notice that $\Tc_\alpha(\pi)$ is not defined at zero. A 'tropicalization' gives the Pitman operators $\Pc_{\alpha^\vee} = \lim_{\varepsilon \rightarrow 0} \varepsilon \Tc_{\alpha} \varepsilon^{-1}$:
$$ \forall t>0, \Pc_{\alpha^\vee}(\pi)(t) = \pi(t) - \inf_{0\leq s \leq t} \alpha\left(\pi(s)\right) \alpha^{\vee}$$
It was also proven that the $\left(\Tc_\alpha\right)_{\alpha \in \Delta}$ satisfy the braid relationships. We give a simpler proof shortly.
\begin{thm}[\cite{bib:BBO}]
\label{thm:braid_relations_Tc}
If $\bar{w} = \bar{s}_{i_1} \dots \bar{s}_{i_k}$ is a representative in $G$ of $w\in W$ written in a reduced fashion and $\pi$ is a continuous path, then:
$$ e^{\Tc_w(\pi)_t} := [\bar{w}^{-1} B_t(\pi)]_0$$ 
is well defined for $t>0$ and
$$ \Tc_w =  \Tc_{\alpha_{i_k}} \circ \dots \circ \Tc_{\alpha_{i_2}} \circ \Tc_{\alpha_{i_1}}$$
Moreover, the operators $(\Tc_\alpha)_{\alpha \in \Delta}$ satisfy the braid relationships.
\end{thm}

The path operator $\Tc_{w_0}$ arises then in a very natural way as the highest weight path transform for the geometric Littelmann model. Indeed, considering a connected geometric crystal, because crystal actions are free, there is no such thing as a dominant path that could be preferred, unlike the $q=0$ case considered in \cite{bib:BBO2}, or the original setting considered by Littelmann. Hence the idea of finding an invariant under crystal actions that will play that role. In the group picture, we have already introduced a notion of highest weight in definition \ref{def:highest_lowest_weight} which fullfills that purpose as invariant (lemma \ref{lemma:hw_is_invariant}). Now that we have at our disposal the projection map $p$, it is natural to transport the definition of highest weight from the group picture. 

\begin{definition}
If $\pi \in C( [0, T], \afrak )$ the associated highest weight is given by:
$$ hw(\pi) := hw( B_T(\pi) )$$
\end{definition}

And evidently:
\begin{proposition}
If $\pi \in C( [0, T], \afrak )$ then:
$$  hw(\pi) = \Tc_{w_0}(\pi)(T)$$
\end{proposition}

A key difference with the 'crystallized' case obtained in \cite{bib:BBO, bib:BBO2} is that the path $\Tc_{w_0} \pi$ does not belong to the crystal generated by $\pi$. For instance, it is not defined at zero. It is however an invariant. This difference would explain why the naturality of $\Tc_{w_0}$ has been so elusive, so far.\\

Now, going back to discussing the braid relations, we make a simple remark:
\begin{align*}
\forall t>0, T_{\xi^{-\alpha^\vee} x_\alpha\left( \xi \right) }( \pi )(t)
& = \pi(t) + \log\left( \frac{1}{\xi} + \int_0^t e^{-\alpha(\pi(s))} ds\right) \alpha^{\vee}\\
& \stackrel{ \xi \rightarrow +\infty }{\longrightarrow}  \Tc_\alpha(\pi)(t)
\end{align*} 
This will allow us to give a simpler proof of the braid relations for $(\Tc_\alpha)_{\alpha \in \Delta}$ using the path transform properties \ref{lbl:path_transform_properties} as suggested at the end of \cite{bib:BBO2} (section 6.6).

\begin{proof}[Proof of theorem \ref{thm:braid_relations_Tc}]
We will show that the operators $(\Tc_\alpha)_{\alpha \in \Delta}$ satisfy the braid relationships as a consequence of the fact that the representatives $(\bar{s}_\alpha)_{\alpha \in \Delta}$ also satisfy them.

Let $\pi$ a continuous path in $\afrak$ and fix $g \in H_{>0} U^w_{>0}$ such that:
$$ g = \xi_k^{h_{i_k}} x_{i_k} \left( \frac{1}{\xi_k} \right) \dots \xi_1^{h_{i_1}} x_{i_1}\left( \frac{1}{\xi_1} \right) $$
for parameters $\xi_i > 0$.\\
We will make use of the approximants:
\begin{align} 
\label{lbl:weyl_group_approximants}
\bar{s}_i(t) & := \phi_i\left(  \left( \begin{array}{cc} t & -1 \\ 1 & 0 \end{array}\right) \right) = y_i\left( -\frac{1}{t} \right) t^{h_i} x_i\left( \frac{1}{t} \right)\\
\bar{\bar{s}}_i(t) & := \phi_i\left(  \left( \begin{array}{cc} t & 1 \\ -1 & 0 \end{array}\right) \right) = y_i\left( \frac{1}{t} \right) t^{h_i} x_i\left( -\frac{1}{t} \right)
\end{align}
which converges respectively to $\bar{s}_i$ and $\bar{\bar{s}}_i = \bar{s}_{i_1}^{-1}$ as the parameter goes to zero. Let us start by writing:
\begin{align*}
\forall t>0, e^{T_g\left( \pi \right)(t)}
& = \left[g B_t(\pi)\right]_0\\
& = \left[\xi_k^{h_{i_k}} x_{i_k}\left( \frac{1}{\xi_k} \right)
    \dots \xi_1^{h_{i_1}} x_{i_1}\left( \frac{1}{\xi_1} \right)
          B_t(\pi)\right]_0\\
& = \left[ y_{i_k}\left( \frac{1}{\xi_k} \right) \bar{\bar{s}}_{i_k}( \xi_k )
     \dots y_{i_1}\left( \frac{1}{\xi_1} \right) \bar{\bar{s}}_{i_1}( \xi_1 )
           B_t(\pi) \right]_0\\ 
\end{align*}
Here we need to make $\xi_j$ successively go to zero in the decreasing order $j=k, \dots, 1$. Let us prove by induction that at the step $k \geq j > 1$, we get the quantity:
\begin{align}
 \label{eqn:quantity_j}
 \left[ \overline{ \overline{s_{i_k} \dots s_{i_{j+1}} }}
                y_{i_j}\left( \frac{1}{\xi_j} \right) \bar{\bar{s}}_{i_j}( \xi_j ) 
          \dots y_{i_1}\left( \frac{1}{\xi_1} \right) \bar{\bar{s}}_{i_1}( \xi_1 )
           B_t(\pi) \right]_0
\end{align}
Since $B_t(\pi)$ is totally positive inside $B$ (theorem \ref{thm:flow_B_total_positivity}), then for all $j=k, \dots, 1$, it is also the case for 
$$ g_j = \xi_j^{h_{i_j}} x_{i_j}\left( \frac{1}{\xi_j} \right)
    \dots \xi_1^{h_{i_1}} x_{i_1}\left( \frac{1}{\xi_1} \right) B_t(\pi)$$
Therefore, the minors $\Delta^{\omega_\alpha}\left( \overline{s_{i_{j+1}} \dots s_{i_k}}^{-1} g_j\right)$ are non zero, a topologically open property that stays valid for the $\xi_j$ in a neighborhood of zero. Hence, taking those limits and considering those Gauss decompositions is allowed.

First, at step $j=k$, we can get rid of $y_{i_k} \left( \frac{1}{\xi_k} \right) \in N$:
\begin{align*}
e^{T_g\left( \pi \right)(t)}
& = \left[ y_{i_k}\left( \frac{1}{\xi_k} \right) \bar{\bar{s}}_{i_k}( \xi_k )
     \dots y_{i_1}\left( \frac{1}{\xi_1} \right) \bar{\bar{s}}_{i_1}( \xi_1 )
           B_t(\pi) \right]_0\\ 
& \stackrel{ \xi_k \rightarrow 0 }{\longrightarrow}
    \left[ \bar{\bar{s}}_{i_k}
     \dots y_{i_1}\left( \frac{1}{\xi_1} \right) \bar{\bar{s}}_{i_1}( \xi_1 )
           B_t(\pi) \right]_0
\end{align*}
Now, assume that equation (\ref{eqn:quantity_j}) is proven for step $j$:
$$ \left[ \overline{ \overline{s_{i_k} \dots s_{i_{j+1}} }}
                y_{i_j}\left( \frac{1}{\xi_j} \right) \bar{\bar{s}}_{i_j}( \xi_j ) 
          \dots y_{i_1}\left( \frac{1}{\xi_1} \right) \bar{\bar{s}}_{i_1}( \xi_1 )
           B_t(\pi) \right]_0$$
Here, write:
$$ \overline{ \overline{s_{i_k} \dots s_{i_{j+1}} }} y_{i_j}\left( \frac{1}{\xi_j} \right)
 = \exp\left( \frac{1}{\xi_j} Ad(\overline{ \overline{s_{i_k} \dots s_{i_{j+1}} }}) f_{i_j} \right) 
   \overline{ \overline{s_{i_k} \dots s_{i_{j+1}} }}$$
And since $Ad(\overline{ \overline{s_{i_k} \dots s_{i_{j+1}} }}) f_{i_j} \in \mathfrak{g}_{ -s_{i_k} \dots s_{i_{j+1}} \alpha_{i_j} } \subset \nfrak$, one has:
\begin{align*}
  & \left[ \overline{ \overline{s_{i_k} \dots s_{i_{j+1}} }}
                  y_{i_j}\left( \frac{1}{\xi_j} \right) \bar{\bar{s}}_{i_j}( \xi_j ) 
            \dots y_{i_1}\left( \frac{1}{\xi_1} \right) \bar{\bar{s}}_{i_1}( \xi_1 )
            B_t(\pi) \right]_0\\
= & \left[ \exp\left( \frac{1}{\xi_j} Ad(\overline{ \overline{s_{i_k} \dots s_{i_{j+1}} }}) f_{i_j} \right) 
           \overline{ \overline{s_{i_k} \dots s_{i_{j+1}} }} \bar{\bar{s}}_{i_j}( \xi_j ) 
            \dots y_{i_1}\left( \frac{1}{\xi_1} \right) \bar{\bar{s}}_{i_1}( \xi_1 )
            B_t(\pi) \right]_0\\
= & \left[ \overline{ \overline{s_{i_k} \dots s_{i_{j+1}} }} \bar{\bar{s}}_{i_j}( \xi_j ) 
            \dots y_{i_1}\left( \frac{1}{\xi_1} \right) \bar{\bar{s}}_{i_1}( \xi_1 )
            B_t(\pi) \right]_0\\
& \stackrel{ \xi_j \rightarrow 0}{\longrightarrow}
    \left[ \overline{ \overline{s_{i_k} \dots s_{i_{j+1}} s_{i_j} }}
            \dots y_{i_1}\left( \frac{1}{\xi_1} \right) \bar{\bar{s}}_{i_1}( \xi_1 )
            B_t(\pi) \right]_0
\end{align*}
The previous limit gives step $j-1$.\\

At the end, we get:
$$\left[\overline{ \overline{s_{i_k} \dots s_{i_1} }} B_t(\pi) \right]_0 = \left[\bar{w}^{-1} B_t(\pi) \right]_0 = \exp\left( \Tc_w \pi(t) \right)$$ 
On the other hand, because the group elements belong to the appropriate sets, we can use the composition property among properties \ref{lbl:path_transform_properties}:
\begin{align*}
\forall t>0, T_g\left( \pi \right)(t)
& = T_{ \xi_1^{h_{i_k}} x_{i_k} \left( \frac{1}{\xi_k} \right) } \circ \dots \circ T_{ \xi_1^{h_{i_1}} x_{i_1} \left( \frac{1}{\xi_1} \right) }\left( \pi \right)(t)\\
& \longrightarrow  \Tc_{k} \circ \dots \circ \Tc_{1}(\pi)(t)
\end{align*}
The previous limit makes sense if and only if for every $j=k, \dots, 1$, $\exp\left(-\alpha_{i_j}( \Tc_{s_{i_1} \dots s_{i_{j-1}}}\pi) \right)$ is integrable. Later, we have a much precise description of this integrability property, but for now, we already know thanks to the previous computation that the Gauss decompositions exist at every level. Therefore, the highest path transforms $(\Tc_\alpha)_{\alpha \in \Delta}$ must have been applied to paths with the appropriate integrability property. Identifying both limits, the braid relationships are proven:
$$\forall t>0,  \Tc_{k} \circ \dots \circ \Tc_{1}(\pi)(t) = \Tc_w \pi(t)$$
\end{proof}

\subsubsection{Low path transforms}
Define $e^{-\infty}_\alpha: C_0\left( [0; T), \afrak \right) \rightarrow C_0\left( [0; T), \afrak \right)$ as
$$ \forall 0 \leq t < T, e^{-\infty}_\alpha \cdot \pi(t) := \pi(t) + \log\left( 1 - \frac{ \int_0^t e^{-\alpha(\pi)} }{ \int_0^T e^{-\alpha(\pi)} } \right) \alpha^\vee$$
Notice that $T$ is excluded and that this path transform makes sense even if $\int_0^T e^{\alpha(\pi)} = \infty$. The notation obviously comes from the fact that $e^{-\infty}_\alpha = \lim_{c \rightarrow -\infty} e^c_\alpha$, hence the name of 'low' path transforms.\\

Clearly, $e^{-\infty}_\alpha$ is a projection in the sense that $e^c_\alpha \cdot e^{-\infty}_\alpha = e^{-\infty}_\alpha \cdot e^c_\alpha = e^{-\infty}_\alpha$ and it stabilizes paths $\pi$ such that $\int_0^T e^{-\alpha(\pi)} = +\infty$. In fact, we can associate such transforms to each element of the Weyl group:
\begin{definition}
Given a reduced expression $w=s_{i_1} \dots s_{i_l}$ for $w \in W$ and $\ell(w) = l$, $e^{-\infty}_w$ is defined unambiguously as 
$$ e^{-\infty}_w = e^{-\infty}_{\alpha_{i_l}} \dots e^{-\infty}_{\alpha_{i_1}} $$
\end{definition}
\begin{proof}
A first proof uses remark \ref{lemma:high_low_duality} and the braid relations for $\left( \Tc_w \right)_{w \in W}$. A second proof consists of using the Verma relations in order to check the claim in the case of a braid move: $w = s_i s_j s_i \dots = s_j s_i s_j \dots$. Indeed let $\left( \beta_1, \beta_2, \beta_3, \dots \right)$ and  $\left( \beta_1', \beta_2', \beta_3', \dots \right)$ the two positive roots enumerations associated to each reduced word.
$$\forall t \in \afrak, e_w^t = e^{ \beta_1(t) }_{\alpha_i} \cdot e^{ \beta_2(t) }_{\alpha_j} \dots = e^{ \beta_1'(t) }_{\alpha_j} \cdot e^{ \beta_2'(t) }_{\alpha_i} \dots $$
Then take $t = - M \mu$ with $\mu$ in the open Weyl chamber, $M$ a real number, and have $M$ go to $+\infty$.
\end{proof}

The name of 'high' path transforms for $\left( \Tc_w \right)_{w \in W}$ is justified by the fact that they are dual to 'low' path transforms:
\begin{lemma}
\label{lemma:high_low_duality}
$$ e^{-\infty}_{\alpha} \circ \iota = \iota \circ \Tc_{\alpha}  $$
And for $w \in W$:
$$ e^{-\infty}_{w} \circ \iota = \iota \circ \Tc_{w} $$ 
\end{lemma}
\begin{proof}
The first identify is a quick computation. The second one is a consequence.
\end{proof}

\begin{rmk}
Notice that it does not make sense to apply the duality map $\iota$ after $e^{-\infty}_w$ for $w \in W$, since it produces a path lacking an endpoint. Though, an extended duality holds, as we will see.

Moreover, the transforms $\left( \Tc_w \right)_{w \in W}$ are not projections.
\end{rmk}

Of course $e^{-\infty}_{w_0}$ is special projection as:
$$\forall \alpha \in \Delta, \forall c \in \mathbb{R}, e^{-\infty}_{w_0} \cdot e^{c}_{\alpha} = e^{c}_{\alpha} \cdot e^{-\infty}_{w_0} = e^{-\infty}_{w_0}$$
The following proposition shows that it is constant on the crystal's components:
\begin{proposition}
\label{proposition:partial_connectedness_criterion}
 If $\pi_1$ and $\pi_2$ are connected then $e^{-\infty}_{w_0} \cdot \pi_1 = e^{-\infty}_{w_0} \cdot \pi_2$
\end{proposition}
\begin{proof}
It is quite obvious. Being connected means that there are real numbers $\left(c_1, \dots, c_l \right)$ and indices 
$\left( i_1, \dots, i_l \right)$ such that:
$$\pi_2 = e^{c_1}_{\alpha_{i_1}} \dots e^{c_l}_{\alpha_{i_l}} \pi_1$$
Then:
$$e^{-\infty}_{w_0} \cdot \pi_2 = e^{-\infty}_{w_0} \cdot e^{c_1}_{\alpha_{i_1}} \dots e^{c_l}_{\alpha_{i_l}} \pi_1 = e^{-\infty}_{w_0} \cdot \pi_1$$
\end{proof}

In theorem \ref{thm:connectedness_criterion}, we will see that the converse is true giving a connectedness criterion.

\subsubsection{A certain property of the Weyl co-vector}
While moving to the edges of geometric crystals, we will obtain paths that go to infinity in possibly finite time. The direction taken to go to infinity will be of utmost importance, involving the Weyl co-vector:
$$ \rho^\vee := \sum_{\alpha \in \Delta} \omega_\alpha^\vee = \frac{1}{2} \sum_{ \beta \in \Phi^+ } \beta^\vee$$

We will need a little property linking $\rho^\vee$ and the weak Bruhat order.
\begin{lemma}
  \label{lemma:rho_property}
  Let $w = s_{i_1} s_{i_2} \dots s_{i_k} \in W$ with $\ell(w) = k$. It defines a positive roots enumeration 
  $\left( \beta_1, \beta_2, \dots, \beta_k \right)$. Then:
  \begin{itemize}
   \item $$\rho^\vee - w \rho^\vee = \beta_1^\vee + \beta_2^\vee + \dots + \beta_k^\vee$$ 
   \item $\ell(s_\alpha w) = \ell(w) + 1$ if and only if $-\alpha\left( \rho^\vee - w \rho^\vee \right) \geq 0$
  \end{itemize}
\end{lemma}
\begin{proof}
 The first statement comes as an application of formula ($\ref{lbl:kumar}$). Concerning the second, following Bourbaki ( \cite{bib:Bourbaki}, Ch. V, \S 3, Th. 1, (ii)), $\ell\left(s_\alpha w\right) = 1 + \ell\left(w\right)$ is equivalent to saying that $C$ and $w\left( C \right)$ are on the same side of the wall associated to $\alpha$. As the Weyl co-vector is inside $C$, it tantamounts to $\alpha\left( w \rho^\vee \right) > 0$. In the end:
$$ -\alpha\left( \rho^\vee - w \rho^\vee \right) > -1 $$
The proof is finished once we notice that the left-hand side is an integer.
\end{proof}
\begin{rmk}
If $w$ is taken as $w_0$ the longest element, we recover the identity we used to define the Weyl co-vector.
\end{rmk}

\subsubsection{Extended path types}
The paths transforms $\left( \Tc_w \right)_{w \in W}$ ( resp. $\left( e^{-\infty}_w \right)_{w \in W}$) give paths that lack a starting (resp. an ending) point. In order to examine the possible asymptotics, let us first start by a simple lemma:
\begin{lemma}
 \label{lemma:start_asymptotics}
 If $\pi \in C\left( [0, T], \afrak \right)$ then for all $w \in W$:
$$ \Tc_w \pi(t) = w^{-1} \pi(0) + \log(t)\left( \rho^\vee - w^{-1} \rho^\vee \right) + c_w + o(1)$$
where $c_w$ is a constant depending only on $w$ and $o(1)$ goes to zero as $t \rightarrow 0$.
\end{lemma}
\begin{proof}
By induction on $\ell(w)$. If $\ell(w)=0$, then $w=e$ and the result is obvious ($c_e=0$).\\
If $w = u s_\alpha$ with $u \in W$ and $\ell(w) = \ell(u) + 1$, then:
$$\Tc_w = \Tc_\alpha \circ \Tc_{u}$$
Using the induction hypothesis, for $s>0$:
$$e^{-\alpha(\Tc_u\pi(s))} = e^{-\alpha\left( u^{-1} \pi(0) + c_u\right) + o(1)} s^{ -\alpha( \rho^\vee - u \rho^\vee) }$$
Lemma \ref{lemma:rho_property} applied to $w^{-1}$ tells us that $-\alpha( \rho^\vee - u^{-1} \rho^\vee) \geq 0$, and $e^{-\alpha(\Tc_u\pi(s))}$ is integrable at the neighborhood of zero. Since integrating equivalents is allowed, and using that $\alpha\left( \rho^\vee \right) = 1$:
\begin{align*}
  & \int_0^t e^{-\alpha(\Tc_u \pi)}\\
= & e^{-\alpha\left( u^{-1} \pi(0) + c_u\right) + o(1)} \int_0^t s^{ -\alpha( \rho^\vee - u^{-1} \rho^\vee) } ds\\
= & e^{-\alpha\left( u^{-1} \pi(0) + c_u\right) + o(1)} \frac{t^{1-\alpha( \rho^\vee - u^{-1} \rho^\vee) }}{1-\alpha( \rho^\vee - u^{-1} \rho^\vee)}\\
= & e^{-\alpha\left( u^{-1} \pi(0) + c_u\right) + o(1)} \frac{t^{ (u\alpha)( \rho^\vee) }}{(u\alpha)( \rho^\vee)}
\end{align*}
Then:
\begin{align*}
\Tc_w \pi(t) & = \Tc_{u}\pi(t) + \alpha^\vee \log \int_0^t e^{-\alpha(\Tc_u \pi)}\\
& = u^{-1} \pi(0) + \log(t)\left( \rho^\vee - u^{-1}\rho^\vee \right) + c_u + \\
& \ \ -\alpha\left( u^{-1} \pi(0) + c_u\right)\alpha^\vee  + \log \frac{t^{ (u\alpha)( \rho^\vee) }}{(u\alpha)( \rho^\vee)}\alpha^\vee  + o(1)\\
& = w^{-1} \pi(0) + \log(t) \left( \rho^\vee - u^{-1} \rho^\vee + (u\alpha)(\rho^\vee) \alpha^\vee \right) + s_\alpha c_u - \alpha^\vee \log\left( (u\alpha)( \rho^\vee)\right) + o(1)\\
& = w^{-1} \pi(0) + \log(t) \left( \rho^\vee - w^{-1} \rho^\vee \right) + s_\alpha c_u - \alpha^\vee \log\left( (u\alpha)( \rho^\vee)\right) + o(1)
\end{align*}
Set $c_w = s_\alpha c_u - \alpha^\vee \log\left( (u\alpha)( \rho^\vee)\right)$.\\
Finally, in order to prove that $c_w$ depends only on $w$ and not the reduced expression used, we invoke the fact that the asymptotic development is unique and $\Tc_w$ depends only on $w$.
\end{proof}

For completeness, we give an explicit expression for the constants $c_w, w \in W$, without proof, as we will not use them:
\begin{lemma}
$$ \forall w \in W, c_w = w^{-1} \sum_{\beta \in Inv(w)} \log\left( \beta(\rho^\vee) \right) \beta^\vee$$
where $Inv(w)$ is the set of inversions of $W$.
\end{lemma}

Lemma \ref{lemma:start_asymptotics} suggests to allow paths with undefined starting point, and to distinguish between them depending on their asymptotic behaviour at $t=0$, hence a definition:

\begin{definition}
\label{def:high_path_types}
For $T > 0$, we say that a path $\pi \in C\left( (0, T], \afrak \right)$ is a high path of type $w \in W$ when the following asymptotic development holds at $0$:
$$ \pi(t) = \log(t)\left( \rho^\vee - w^{-1} \rho^\vee \right) + c_w + o(1)$$
The set of all high paths of type $w \in W$ in $C\left( (0, T], \afrak \right)$ is denoted $C^{high}_w\left( (0, T], \afrak \right)$.
\end{definition}
\begin{rmk}
As $c_e = 0$, $C^{high}_e\left( (0, T], \afrak \right) = C_0\left( [0, T], \afrak \right)$.
\end{rmk}

By duality, having in mind remark \ref{lemma:high_low_duality}:
\begin{definition}
\label{def:low_path_types}
For $T > 0$, we say that a path $\pi \in C_0\left( [0; T), \afrak \right)$ is a low path of type $w \in W$ when the following asymptotic development holds at $T$:
$$ \exists C_\pi, \pi(t) = C_\pi + \log(T-t)\left( \rho^\vee - w^{-1} \rho^\vee \right) + c_w + o(1)$$
The set of all low paths of type $w \in W$ in $C\left( [0, T), \afrak \right)$ is denoted $C^{low}_w\left( [0, T), \afrak \right)$.
\end{definition}
\begin{rmk}
A low path $\pi$ of type $e$ has a continuous extension at $t=T$ by letting $\pi(T)=C_\pi$, hence:
$$C^{low}_e = C_0\left( [0, T], \afrak \right)$$
\end{rmk}

Both high and low paths will be referred to as extended paths. Clearly, the extended paths of type $w_0$ deserve a special name. As such, high (resp. low) paths of type $w_0$ will be referred to as highest (resp. lowest) paths.

\subsection{String parameters for paths}
\label{subsection:string_params}

\subsubsection{Definition}
Let ${\bf i} \in R(w_0)$. For any path $\pi \in C_0\left( [0, T], \afrak \right)$, define $\varrho_{\bf i}^K(\pi)$ as the sequence of numbers ${ \bf c} = \left( c_1, c_2, \dots, c_m \right)$ recursively as:
$$ c_k = \frac{1}{\int_0^T e^{-\alpha_{i_k}\left( \Tc_{s_{i_1} \dots s_{i_{k-1}}} \pi \right)}}$$
\begin{thm}
\label{thm:string_params_paths}
The map $\varrho_{\bf i}^K$ is well-defined on $C_0\left( [0, T], \afrak \right)$ and takes values in $\R_{>0}^m$. Moreover, for $\pi \in C_0\left( [0, T], \afrak \right)$, the $m$-tuple $\varrho_{\bf i}^K(\pi)$ allows to recover $\pi$ from the highest weight path $\Tc_{w_0} \pi$.
\end{thm}
The proof is given soon after a few discussions.

\subsubsection{Extracting string parameters}
Building up on lemma \ref{lemma:start_asymptotics}, we will see when it is possible to apply $\left( \Tc_\alpha \right)_{\alpha \in \Delta}$ depending on a path's type. The importance of the weak Bruhat order is quite remarkable. Also, the use of the geometric Pitman operator $\Tc_\alpha$ for $\alpha \in \Delta$ corresponds to the loss of exactly one real number. 
\begin{proposition}
\label{proposition:extracting_string}
Let $w \in W$ and $\alpha \in \Delta$ such that $\ell(w s_\alpha) = \ell(w) + 1$.\\
(1) If $\pi$ is a high path of type $w$ then $c = \frac{1}{\int_0^T e^{-\alpha(\pi)}} > 0$ and $\eta = \Tc_\alpha \pi$ has type $w s_\alpha$.\\
(2) Reciprocally, given $\eta$ , a high path with type $w s_\alpha$, and a positive $c>0$, there is a unique high path $\pi$ of type $w$ such that $c = \frac{1}{\int_0^T e^{-\alpha(\pi)}} > 0$ and $\eta = \Tc_\alpha \pi$. It is given by:
$$ \forall 0 < t \leq T, \pi(t) = \eta(t) + \log\left( c + \int_t^T e^{-\alpha(\eta)} \right) \alpha^\vee$$
\end{proposition}
\begin{proof}
(1) Using lemma \ref{lemma:rho_property}, if $\pi$ is a high path of type $w$ and $\ell(w s_\alpha) = \ell(w) + 1$, then $e^{-\alpha(\pi)}$ is integrable at the neighborhood of zero, hence $c>0$.

Thanks to the proof of \ref{lemma:start_asymptotics}, we have seen that $\eta = \Tc_\alpha \pi$ will have the right asymptotics at $t=0$, so that it will be of type $w s_\alpha$.\\
(2) By composing the equality $\eta = \Tc_\alpha \pi$ with $e^{-\alpha(.)}$:
\begin{align*}
  & e^{-\alpha(\eta(t))}\\
= & \frac{ e^{-\alpha(\pi(t))} }{\left(\int_0^t e^{-\alpha(\pi)}\right)^2}\\
= & - \frac{d}{dt} \frac{1}{\int_0^t e^{-\alpha(\pi)}}
\end{align*}
Then after integration between $t>0$ and $T$:
$$ \forall t>0, \int_t^T e^{-\alpha(\eta)} = \frac{1}{\int_0^t e^{-\alpha(\pi)}} - c$$
Reinjecting this relation in the definition of $\eta$:
$$ \forall t>0, \pi(t) = \eta(t) + \alpha^\vee \log\left( c + \int_t^T e^{-\alpha(\eta)} \right)$$
Finally, all that is left is to check that $\pi$ has the type $w$. The asymptotic development at $0$ follows from a computation similar to the proof of lemma \ref{lemma:start_asymptotics}. Indeed, since $\eta$ is of type $w s_\alpha$, we have the following asymptotics for $e^{-\alpha(\eta)}$ at zero:
$$ e^{-\alpha(\eta(t))} = e^{-\alpha(c_{w s_\alpha}) + o(1)}t^{-\alpha(\rho^\vee - s_\alpha w^{-1}\rho^\vee) }$$
As $-\alpha(\rho^\vee - s_\alpha w^{-1}\rho^\vee) \leq -1$ (lemma \ref{lemma:rho_property}), integrating $e^{-\alpha(\eta)}$ gives a divergent integral at zero. Therefore, discarding the other end of the integral, we have the following equivalent for $t \rightarrow 0$:
\begin{align*}
  & \int_t^T e^{-\alpha(\eta)}\\
= & e^{-\alpha(c_{w s_\alpha}) + o(1)} \int_t^{+\infty} s^{-\alpha(\rho^\vee - s_\alpha w^{-1}\rho^\vee) } ds\\
= & e^{-\alpha(c_{w s_\alpha}) + o(1)}
    \frac{\left[ s^{1-\alpha(\rho^\vee - s_\alpha w^{-1}\rho^\vee) } \right]_t^{+\infty}}
         { 1-\alpha(\rho^\vee - s_\alpha w^{-1}\rho^\vee) }\\
= & \frac{ e^{-\alpha(c_{w s_\alpha}) + o(1)} }
         { (w \alpha)\left( \rho^\vee \right) t^{(w \alpha)\left( \rho^\vee \right)} }\\
\end{align*}
Because $\ell(w s_\alpha) = \ell(w) +1$, $w \alpha$ is a positive root and $w \alpha(\rho^\vee)>0$. Thus, we can write:
\begin{align*}
  & \pi(t)\\
= & \log(t)\left( \rho^\vee - s_\alpha w^{-1} \rho^\vee \right) + c_{w s_\alpha} + \log\left( c +  \frac{ e^{-\alpha(c_{w s_\alpha}) + o(1)} }{ (w \alpha)\left( \rho^\vee \right) t^{(w \alpha)\left( \rho^\vee \right)} } \right) \alpha^\vee + o(1)\\
= & \log(t)\left( \rho^\vee - s_\alpha w^{-1} \rho^\vee - (w \alpha)(\rho^\vee) \alpha^\vee \right) + s_\alpha c_{w s_\alpha} - \log\left( (w \alpha)( \rho^\vee ) \right)\alpha^\vee + o(1)\\
= & \log(t)\left( \rho^\vee - w^{-1} \rho^\vee \right) + s_\alpha c_{w s_\alpha} - \log\left( (w \alpha)( \rho^\vee ) \right)\alpha^\vee + o(1)\\
\end{align*}
Noticing that $s_\alpha c_{w s_\alpha} - \log\left( (w \alpha)( \rho^\vee ) \right)\alpha^\vee = c_w$ concludes the proof.
\end{proof}

\begin{proof}[Proof of theorem \ref{thm:string_params_paths}]
Start with $\pi \in C_0\left( [0, T], \afrak \right)$. It is a high path of type $e$. When composing the geometric Pitman operators $\left( \Tc_\alpha \right)_{\alpha}$ while respecting the weak Bruhat order, we obtain paths whose types are climbing the Hasse diagram, until we reach $\Tc_{w_0}\pi$. At each step, exactly one positive real number is lost using proposition \ref{proposition:extracting_string}. 
\end{proof}

This drawing sums up the situation in the case of $A_2$:

\begin{figure}[htp!]
\begin{center}
\begin{tikzpicture}
\useasboundingbox (0,0) rectangle (7.0cm,10.0cm);
\definecolor{black}{rgb}{0.0,0.0,0.0}
\definecolor{white}{rgb}{1.0,1.0,1.0}
\Vertex[style={minimum size=1, shape=circle}, x=2.5, y=10, L=$w_0$]{v0}
\Vertex[style={minimum size=1, shape=circle}, x=4,   y=7, L=$s_1 s_2$]{v1}
\Vertex[style={minimum size=1, shape=circle}, x=1,   y=7, L=$s_2 s_1$]{v2}
\Vertex[style={minimum size=1, shape=circle}, x=1,   y=4, L=$s_1$]{v3}
\Vertex[style={minimum size=1, shape=circle}, x=4,   y=4, L=$s_2$]{v4}
\Vertex[style={minimum size=1, shape=circle}, x=2.5, y=1, L=$e$]{v5}
\Edge[lw=0.1cm,style={post, color=black,},](v5)(v3)
\Edge[lw=0.1cm,style={post, color=black,},](v5)(v4)
\Edge[lw=0.1cm,style={post, color=black,},](v3)(v2)
\Edge[lw=0.1cm,style={post, color=black,},](v4)(v1)
\Edge[lw=0.1cm,style={post, color=black,},](v2)(v0)
\Edge[lw=0.1cm,style={post, color=black,},](v1)(v0)
\Edge[lw=0.1cm,style={post, bend left,color=black,}, label=$\Tc_{\alpha_1}$](v5)(v3)
\Edge[lw=0.1cm,style={post, bend left,color=black,}, label=$\Tc_{\alpha_2}$](v3)(v2)
\Edge[lw=0.1cm,style={post, bend left,color=black,}, label=$\Tc_{\alpha_1}$](v2)(v0)
\end{tikzpicture}
\end{center}
\label{lbl:string_hasse_diagram}
\caption{Extracting string parameters and climbing Hasse diagram of type $A_2$} 
\end{figure}
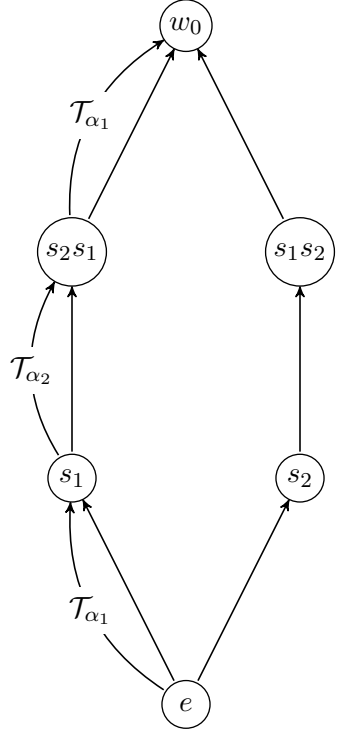

\subsubsection{Inversion lemma}
The inversion lemma is a bijective correspondence between $N_T(\pi)$ and the string parameters $\varrho_{\bf i}^K(\pi)$. Its proof is inspired from theorem 6.5 in \cite{bib:BBO2}.

\begin{thm}
\label{thm:string_inversion_lemma}
For ${\bf i} \in R(w_0)$ and $\pi \in C_0\left( [0, T], \afrak \right)$:
$$ x_{\bf -i} \circ \varrho_{\bf i}^K\left( \pi \right) = [\bar{w_0}^{-1} N_T(\pi)]_{0+}^T$$
or equivalently:
$$ N_T(\pi) = [ \left( \bar{w}_0 \left( x_{\bf -i} \circ \varrho_{\bf i}^K\left( \pi \right) \right)^T \right)^{-1} ]_-^{-1}$$
\end{thm}
\begin{proof}
In fact, this works with $w \in W$ and ${\bf i} = \left(i_1, \dots, i_j\right) \in R(w)$. Write:
\begin{align*}
  & [\bar{w}^{-1} N_T(\pi)]_{0+}\\
= & [\bar{w}^{-1} N_T(\pi)]_-^{-1} \bar{w}^{-1} N_T(\pi)\\
= & [\bar{w}^{-1} N_T(\pi)]_-^{-1} \bar{s}^{-1}_{i_j} \dots \bar{s}^{-1}_{i_1} N_T(\pi)\\
= &     [\bar{w}^{-1} N_T(\pi)]_-^{-1} \bar{s}^{-1}_{i_j} [\overline{ s_{i_1} \dots s_{i_{j-1}} }^{-1} N_T(\pi)]_-\\
  & \ \ [\overline{ s_{i_1} \dots s_{i_{j-1}} }^{-1} N_T(\pi)]_-^{-1} \bar{s}^{-1}_{i_{j-1}} [\overline{ s_{i_1} \dots s_{i_{j-2}} }^{-1} N_T(\pi)]_-\\
  & \ \ [\overline{ s_{i_1} \dots s_{i_{j-2}} }^{-1} N_T(\pi)]_-^{-1} \bar{s}^{-1}_{i_{j-2}} \dots \\
  & \ \ \dots \bar{s}^{-1}_{i_2} [\bar{s}^{-1}_{i_1} N_T(\pi)]_-\\
  & \ \ [\bar{s}^{-1}_{i_1} N_T(\pi)]_-^{-1} \bar{s}^{-1}_{i_1} N_T(\pi)\\
= &     [\bar{s}^{-1}_{i_j}     [\overline{ s_{i_1} \dots s_{i_{j-1}} }^{-1} N_T(\pi)]_- ]_{0+}\\
  & \ \ [\bar{s}^{-1}_{i_{j-1}} [\overline{ s_{i_1} \dots s_{i_{j-2}} }^{-1} N_T(\pi)]_- ]_{0+}\\
  & \ \ \dots \\
  & \ \ [\bar{s}^{-1}_{i_1}     N_T(\pi)]_{0+}
\end{align*}
Notice that each element $x_k = [\bar{s}^{-1}_{i_{k}} [\bar{s}^{-1}_{i_{k-1}} \dots \bar{s}^{-1}_{i_1} N_T(\pi)]_- ]_{0+}$, $ 1 \leq k \leq j$, in the previous product, belongs to the reduced Bruhat cell $N \bar{s}^{-1}_{i_k} N \cap B^+$. Using theorem 4.5 in \cite{bib:BZ01} $x_k = y_{-i_k}(c_k)$ where:
\begin{align*}
c_k^{-\alpha_{i_k}^\vee} = & [x_k]_0 \\
= & [\bar{s}^{-1}_{i_{k}} [\overline{ s_{i_1} \dots s_{i_{k-1}} }^{-1}N_T(\pi)]_- ]_{0}\\
= & [\bar{s}^{-1}_{i_{k}} [\overline{ s_{i_1} \dots s_{i_{k-1}} }^{-1} B_T(\pi)]_- ]_{0}\\
= & [\overline{ s_{i_1} \dots s_{i_{k}} }^{-1} B_T(\pi) [\overline{ s_{i_1} \dots s_{i_{k-1}} }^{-1} B_T(\pi)]_{0+}^{-1} ]_{0}\\
= & e^{ \Tc_{s_{i_1} \dots s_{i_k}} \pi - \Tc_{s_{i_1} \dots s_{i_{k-1}}} \pi }\\
\end{align*}
Hence we obtain exactly the string parameters
$$ c_k = \frac{1}{\int_0^T e^{-\alpha_{i_k}( \Tc_{s_{i_1} \dots s_{i_{k-1}}} \pi )}}$$
And
$$[\bar{w}^{-1} N_T(\pi)]_{0+} = y_{-i_j}(c_j) \dots y_{-i_1}(c_1)$$
Taking the transpose concludes the proof.\\
Finally, in order to see that the second expression can be deduced from the first, write 
\begin{align*}
  & N_T(\pi)\\
= & [ \left( \bar{w}_0 [\bar{w_0}^{-1} N_T(\pi)]_-^{-1} \bar{w}_0^{-1} N_T(\pi) \right)^{-1} ]_-^{-1}\\
= & [ \left( \bar{w}_0 [\bar{w_0}^{-1} N_T(\pi)]_{0+} \right)^{-1} ]_-^{-1}\\
= & [ \left( \bar{w}_0 \left( x_{\bf -i} \circ \varrho_{\bf i}^K\left( \pi \right) \right)^T \right)^{-1} ]_-^{-1}
\end{align*}
\end{proof}

As a corollary, we get the commutativity of the right side in the diagram \ref{fig:parametrizations_diagram}:
\begin{corollary}
\label{corollary:commutative_diagram_kashiwara}
$$ \forall {\bf i} \in R(w_0), \forall \pi \in C_0\left( [0,T], \afrak \right), x_{\bf -i} \circ \varrho_{\bf i}^K\left( \pi \right) = \varrho^K \circ p (\pi)$$
\end{corollary}
\begin{proof}
Theorem \ref{thm:string_inversion_lemma} and definition \ref{def:crystal_parameter} because:
$$ [p(\pi)]_- = N_T(\pi)$$
\end{proof}

\subsection{Lusztig parameters for paths}
\label{subsection:lusztig_param}
In the same fashion, we define Lusztig parameters for a path and show how to extract them.

\subsubsection{Definition}
Let ${\bf i} \in R(w_0)$. For any path $\pi \in C_0\left( [0, T], \afrak \right)$, define $\varrho_{\bf i}^L(\pi)$ as the sequence of numbers ${ \bf t } = \left( t_1, t_2, \dots, t_m \right)$ recursively as:
$$ t_k = \frac{1}{\int_0^T e^{-\alpha_{i_k}\left( e_{s_{i_1} \dots s_{i_{k-1}}}^{-\infty} \pi \right)}}$$
\begin{thm}
\label{thm:lusztig_params_paths}
The map $\varrho_{\bf i}^L$ is well-defined on $C_0\left( [0, T], \afrak \right)$ and takes values in $\R_{>0}^m$. Moreover, for $\pi \in C_0\left( [0, T], \afrak \right)$, the $m$-tuple $\varrho_{\bf i}^L(\pi)$ allows to recover $\pi$ from the lowest weight path $e_{w_0}^{-\infty} \pi$.
\end{thm}
Let us now explain how to prove this theorem carefully using duality.

\subsubsection{Extracting Lusztig parameters}
Let $w \in W$. It is quite obvious that if $\pi' \in C^{high}_w\left( (0, T], \afrak\right)$ then $\iota(\pi')$ makes sense and belongs to $C^{low}_w\left( [0, T), \afrak\right)$. The symmetric situation is not quite true, since one cannot apply the duality map $\iota$ to low paths. But we still have:
\begin{lemma}[Extended duality lemma]
\label{lemma:extended_duality}
Let $w \in W$. If $\pi \in C^{low}_w\left( [0, T), \afrak\right)$ is a low path of type $w$, there is a unique $\pi' \in C^{high}_w\left( (0, T], \afrak\right)$ such that:
$$ \pi = \left( \pi' \right)^\iota$$
\end{lemma}
\begin{proof}
Simply take for $0 \leq t < T$, $\pi'(t) = \pi(T-t) - C_\pi$, where $C_\pi$ is the constant appearing in the definition of low paths. It is immediate that it satisfies all the requirements.
\end{proof}
\begin{rmk}
Mapping $\pi$ to $\pi'$ can be thought of as an extension of the duality map to low paths.
\end{rmk}
\begin{rmk}
Putting together the previous lemma and lemma \ref{lemma:high_low_duality}, one sees that the path transforms $\left( e^{-\infty}_w \right)_{w \in W}$ and $\left( \Tc_w \right)_{w \in W}$ are in extended duality.
\end{rmk}

Thus, we can easily prove analogous statements to the previous case:
\begin{proposition}
\label{proposition:extracting_lusztig}
Let $w \in W$ and $\alpha \in \Delta$ such that $\ell(w s_\alpha) = \ell(w) + 1$.\\
(1) If $\pi$ is a low path of type $w$ then $c = \frac{1}{\int_0^T e^{-\alpha(\pi)}} > 0$ and $\eta = e_\alpha^{-\infty} \pi$ has type $w s_\alpha$.\\
(2) Reciprocally, given $\eta$, a low path with type $w s_\alpha$, and a positive $c>0$, there is a unique low path $\pi$ of type $w$ such that $c = \frac{1}{\int_0^T e^{-\alpha(\pi)}} > 0$ and $\eta = e_\alpha^{-\infty} \pi$. It is given by:
$$ \forall 0 \leq t < T, \pi(t) = T_{x_\alpha(c)} \eta(t) = \eta(t) + \log\left( 1 + c \int_0^t e^{-\alpha(\eta)} \right)\alpha^\vee$$
\end{proposition}
\begin{proof}
(1) Using lemma \ref{lemma:extended_duality}, we get $\pi'$ a high path of type $w$. Using proposition \ref{proposition:extracting_string}, we get that:
$$c = \frac{1}{\int_0^T e^{-\alpha(\pi)}} = \frac{e^{-\alpha(\pi'(T))}}{\int_0^T e^{-\alpha(\pi')}} > 0$$
And:
\begin{align*}
\eta & = e_\alpha^{-\infty} \pi\\
     & = e_\alpha^{-\infty} \circ \iota(\pi')\\
     & = \iota \circ \Tc_\alpha(\pi')
\end{align*}
which is a low path of type $w s_\alpha$ since $\Tc_\alpha(\pi') \in C_{w s_\alpha}^{high}\left( (0, T], \afrak \right)$.\\
(2) Again, using the extended duality lemma, there exists $\eta' \in C_{w s_\alpha}^{high}\left( (0, T], \afrak \right)$ such that $\eta = \iota(\eta')$. The result is proven by using (2) in proposition \ref{proposition:extracting_string}. In order to recover $\pi$ from $\eta$, rather than rearranging the formula from proposition \ref{proposition:extracting_string}, let us direcly solve:
$$ \eta(t) = \pi(t) + \alpha^\vee \log\left( 1 - c \int_0^t e^{-\alpha( \pi )} \right) $$
By evaluating the $e^{-\alpha(.)}$ on each side of the previous equality:
$$ e^{-\alpha\left(\eta(t)\right)} = \frac{ e^{-\alpha\left(\pi(t)\right)} }{ \left( 1 - c \int_0^t e^{-\alpha( \pi)}\right)^2 }$$
This expression can be integrated and rearranged as:
$$ \left( 1 + c \int_0^t e^{ -\alpha(\eta) } \right) \left( 1 - c \int_0^t e^{-\alpha(\pi)} \right) = 1$$
As such, since $\eta$ has type $w s_\alpha$, we have that $\int_0^T e^{-\alpha(\eta)} = +\infty$. Hence:
$$ c = \frac{1}{\int_0^T e^{-\alpha(\pi)} }$$
Also by replacing $\log\left( 1 - c \int_0^t e^{-\alpha(\pi)} \right)$ by $-\log\left( 1 + c \int_0^t e^{ -\alpha(\eta) } \right)$, we get the result:
$$ \pi(t) = \eta(t) + \log\left(1 + c \int_0^t e^{-\alpha(\eta)} \right) \alpha^\vee = T_{x_\alpha\left( c \right)} \eta (t)$$
\end{proof}

\begin{proof}[Proof of theorem \ref{thm:lusztig_params_paths}]
Apply iteratively proposition \ref{proposition:extracting_lusztig}. Successive projections give low paths whose types goes down the Hasse diagram. At every composition, exactly one positive real parameter is lost.  
\end{proof}

This drawing illustrates the situation in the case of $A_2$:

\begin{figure}[htp!]
\label{lbl:lusztig_hasse_diagram}
\caption{Extracting Lusztig parameters and going down the Hasse diagram of type $A_2$} 
\begin{center}
\begin{tikzpicture}
\useasboundingbox (0,0) rectangle (7.0cm,10.0cm);
\definecolor{black}{rgb}{0.0,0.0,0.0}
\definecolor{white}{rgb}{1.0,1.0,1.0}
\Vertex[style={minimum size=1, shape=circle}, x=2.5, y=10, L=$e$]{v0}
\Vertex[style={minimum size=1, shape=circle}, x=4,   y=7, L=$s_2$]{v1}
\Vertex[style={minimum size=1, shape=circle}, x=1,   y=7, L=$s_1$]{v2}
\Vertex[style={minimum size=1, shape=circle}, x=1,   y=4, L=$s_1 s_2$]{v3}
\Vertex[style={minimum size=1, shape=circle}, x=4,   y=4, L=$s_2 s_1$]{v4}
\Vertex[style={minimum size=1, shape=circle}, x=2.5, y=1, L=$w_0$]{v5}
\Edge[lw=0.1cm,style={post, color=black,},](v3)(v5)
\Edge[lw=0.1cm,style={post, color=black,},](v4)(v5)
\Edge[lw=0.1cm,style={post, color=black,},](v2)(v3)
\Edge[lw=0.1cm,style={post, color=black,},](v1)(v4)
\Edge[lw=0.1cm,style={post, color=black,},](v0)(v2)
\Edge[lw=0.1cm,style={post, color=black,},](v0)(v1)
\Edge[lw=0.1cm,style={post, bend left,color=black,}, label=$e^{-\infty}_{\alpha_1}$](v3)(v5)
\Edge[lw=0.1cm,style={post, bend left,color=black,}, label=$e^{-\infty}_{\alpha_2}$](v2)(v3)
\Edge[lw=0.1cm,style={post, bend left,color=black,}, label=$e^{-\infty}_{\alpha_1}$](v0)(v2)
\end{tikzpicture}
\end{center}
\end{figure}
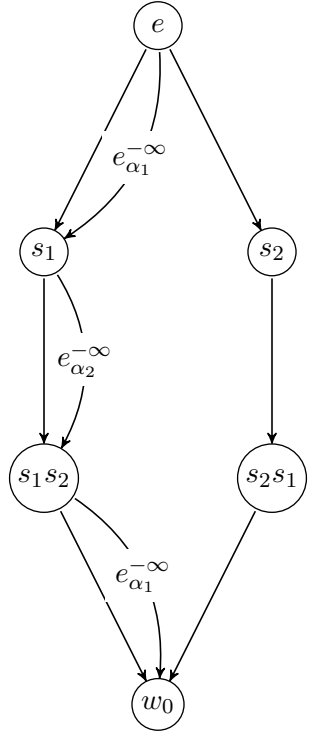
 
This time, a path transform we already encountered appears:
\begin{thm}
\label{thm:from_lowest_path_2_path}
Given a lowest path $\eta$, a reduced word ${\bf i} = \left( i_1, \dots, i_m\right) \in R\left(w_0\right)$ and strictly positive parameters $\left(t_1, \dots, t_m\right) \in \R_{>0}^m$, there is a unique path $\pi \in C_0\left( [0, T], \afrak \right)$ such that:
\begin{itemize}
 \item $e^{-\infty}_{w_0} \pi = \eta$
 \item $\eta_j = e^{-\infty}_{s_{i_1} \dots s_{i_j}} \pi = e^{-\infty}_{\alpha_{i_j}} \cdot \eta_{j-1}$
 \item $t_j = \frac{1}{\int_0^T e^{-\alpha_{i_j}( \eta_{j-1} ) } }$
\end{itemize}
  It is given by:
$$ \pi = T_z \eta$$
where
$$ z = x_{i_1}\left( t_1 \right) \dots x_{i_m}\left( t_m \right) \in U^{w_0}_{>0}$$
\end{thm}
\begin{proof}
At each level, $ \eta_{j-1} = T_{x_{\alpha_{i_j}}\left( \xi_j \right)} \eta_j$. Then, using the composition property among properties \ref{lbl:path_transform_properties}:
$$ \pi = T_{x_{\alpha_{i_1}}\left( \xi_1 \right)} \circ T_{x_{\alpha_{i_2}}\left( \xi_2 \right)} \circ \dots \circ T_{x_{\alpha_{i_j}}\left( \xi_j \right)}\left( \eta \right)
       = T_z\left( \eta \right)$$
\end{proof}
\begin{rmk}
A similar statement holds for paths of type $w$, using group elements in $U^{w}_{>0}$. And this $z \in U^{w_0}_{>0}$ is the Lusztig parameter, as we will see shortly.
\end{rmk}

\begin{corollary}
\label{corollary:connected_component_bijection}
A connected component generated by a path $\pi_0$ can be parametrized by the totally positive part $U^{w_0}_{>0}$ 
thanks to the bijection:
 $$\begin{array}{ccc}
    U^{w_0}_{>0} & \rightarrow & \langle \pi_0 \rangle\\
          z      & \mapsto     & T_z\left( e^{-\infty}_{w_0} \cdot \pi_0 \right) 
   \end{array}$$
\end{corollary}
\begin{proof}
Recall that $\eta = e^{-\infty}_{w_0} \pi$ does not depend on $\pi \in \langle \pi_0 \rangle$, but only on the connected component. And every path in $\pi \in \langle \pi_0 \rangle$ is uniquely determined by an $u \in U^{w_0}_{>0}$ such that $\pi = T_u \eta$ thanks to the previous theorem.
\end{proof}

\subsubsection{Inversion lemma}
Again, we have a bijective correspondence between $N_T(\pi)$ and the Lusztig parameters $\varrho_{\bf i}^L(\pi) \in \R_{>0}^{\ell(w_0)}$:

\begin{thm}
\label{thm:inversion_lemma_lusztig}
For $\eta \in C^{low}_{w_0}\left( [0, T), \afrak \right)$, let $\pi = T_g \eta$ be is a crystal element with Lusztig parameters encoded by $g = x_{i_1}\left( t_1 \right) \dots x_{i_m}\left( t_m \right) = x_{\bf i} \circ \varrho_{\bf i}^L(\pi) \in U^{w_0}_{>0}$. Then:
$$ N_T\left( \pi \right) = [ g \bar{w}_0 ]_-$$
or equivalently
$$ g = [\bar{w}_0^{-1} N_T\left( \pi \right)^\iota ]_+^\iota = [\bar{w}_0^{-1} B_T\left( \pi \right)^\iota ]_+^\iota$$
\end{thm}
\begin{proof}
This can be deduced from theorem \ref{thm:string_inversion_lemma}. However, we choose to give a separate proof that is easily adapted to the case of $T=\infty$. First, in order to see that both identities are equivalent, since $\bar{w}_0^{-1} \left([g \bar{w}_0]_{0+}^{-1}\right)^\iota \bar{w}_0 \in B$, write:
\begin{align*}
g = & [g^\iota]_+^\iota\\
= & [\bar{w}_0^{-1} \left([g \bar{w}_0]_{0+}^{-1}\right)^\iota \bar{w}_0 g^\iota]_+^\iota\\
= & [\bar{w}_0^{-1} \left( g \bar{w}_0 [g \bar{w}_0]_{0+}^{-1}\right)^\iota]_+^\iota
\end{align*}
Therefore:
$$N_T\left( \pi \right) = [ g \bar{w}_0 ]_- = g \bar{w}_0 [g \bar{w}_0]_{0+}^{-1}$$
if and only if:
$$ g = [\bar{w}_0^{-1} N_T\left( \pi \right)^\iota]_+^\iota $$
One can also add the torus part and write:
$$ g = [\bar{w}_0^{-1} N_T\left( \pi \right)^\iota ]_+^\iota = [\bar{w}_0^{-1} B_T\left( \pi \right)^\iota ]_+^\iota$$

Now, let us prove the above statement using a similar decomposition to the one used in the proof of theorem \ref{thm:string_inversion_lemma}:
\begin{align*}
      & \left[\bar{w}_0^{-1} B_T\left( \pi^\iota \right) \right]_+\\
    = & \left[\bar{w}_0^{-1} B_T\left( \pi^\iota \right) \right]_{-0}^{-1} \bar{w}_0^{-1} B_T\left( \pi^\iota \right) \\
    = & \left[\bar{w}_0^{-1} B_T\left( \pi^\iota \right) \right]_{-0}^{-1} \bar{s}_{i_m}^{-1}
         \left[\overline{s_{i_1} \dots s_{i_{m-1}}}^{-1} B_T\left( \pi^\iota \right) \right]_{-0}\\
& \cdot \left[\overline{s_{i_1} \dots s_{i_{m-1}}}^{-1} B_T\left( \pi^\iota \right) \right]_{-0}^{-1} \bar{s}_{i_{m-1}}^{-1}
         \left[\overline{s_{i_1} \dots s_{i_{m-2}}}^{-1} B_T\left( \pi^\iota \right) \right]_{-0}\\
& \cdot \left[\overline{s_{i_1} \dots s_{i_{m-2}}}^{-1} B_T\left( \pi^\iota \right) \right]_{-0}^{-1} \bar{s}_{i_{m-2}}^{-1}
         \left[\overline{s_{i_1} \dots s_{i_{m-3}}}^{-1} B_T\left( \pi^\iota \right) \right]_{-0}\\
& \dots \dots \dots \dots\\
& \cdot \left[\overline{s_{i_1}}^{-1} B_T\left( \pi^\iota \right) \right]_{-0}^{-1} \bar{s}_{i_1}^{-1}
         B_T\left( \pi^\iota \right)\\
    = & \left[ \bar{s}_{i_m}^{-1}
         \left[\overline{s_{i_1} \dots s_{i_{m-1}}}^{-1} B_T\left( \pi^\iota \right) \right]_{-0} \right]_+\\
& \cdot \left[ \bar{s}_{i_{m-1}}^{-1}
         \left[\overline{s_{i_1} \dots s_{i_{m-2}}}^{-1} B_T\left( \pi^\iota \right) \right]_{-0} \right]_+\\
& \cdot \left[ \bar{s}_{i_{m-2}}^{-1}
         \left[\overline{s_{i_1} \dots s_{i_{m-3}}}^{-1} B_T\left( \pi^\iota \right) \right]_{-0} \right]_+\\
& \dots \dots \dots \dots\\
& \cdot \left[ \bar{s}_{i_1}^{-1}
         B_T\left( \pi^\iota \right) \right]_+
\end{align*}
In the previous equation, we have a product of $m$ terms, each of the form $x_j = [ \bar{s}_{i_j}^{-1} b ]_{-0}^{-1} \bar{s}_{i_j}^{-1} b = [ \bar{s}_{i_j}^{-1} b ]_+$, $1 \leq k \leq m$ with $b = \left[ \overline{s_{i_1} \dots s_{i_{j-1}}}^{-1}B_T\left( \pi^\iota\right) \right]_{-0}$. As it belongs to the reduced double Bruhat cell $U \cap B s_{i_j} B$, we can use theorem theorem 4.5 in \cite{bib:BZ01} and write $x_j = x_{\alpha_{i_j}}\left( t_j \right)$. The quantity $t_j$ will be computed at the end.
Hence:
$$ [\bar{w}_0^{-1} B_T\left( \pi^\iota \right) ]_+ = x_{i_m}\left( t_m \right) \dots x_{i_2}\left( t_2 \right) x_{i_1}\left( t_1 \right)$$
or equivalently:
$$ [\bar{w}_0^{-1} B_T\left( \pi \right)^\iota ]_+^\iota = x_{i_1}\left( t_1 \right) x_{i_2}\left( t_2 \right) \dots x_{i_m}\left( t_m \right)$$

Now, all that is left is to prove that the Lusztig parameters for $\pi$ are nothing but the quantities $t_j, j=1, \dots, m$. We have:
\begin{align*}
x_{\alpha_{i_j}}\left( t_j \right)
= & \left[ \bar{s}_{i_j}^{-1} \left[\overline{s_{i_1} \dots s_{i_{j-1}}}^{-1} B_T\left( \pi^\iota \right) \right]_{-0} \right]_+\\
= & \left[\overline{s_{i_1} \dots s_{i_{j}}}^{-1} B_T\left( \pi^\iota \right) \right]_{-0}^{-1}
    \bar{s}_{i_j}^{-1} \left[\overline{s_{i_1} \dots s_{i_{j-1}}}^{-1} B_T\left( \pi^\iota \right) \right]_{-0}\\
= & e^{-\Tc_{s_{i_1} \dots s_{i_j}} \circ \iota(\pi)(T)}
    \left[\overline{s_{i_1} \dots s_{i_{j}}}^{-1} B_T\left( \pi^\iota \right) \right]_{-}^{-1}
    \bar{s}_{i_j}^{-1} \left[\overline{s_{i_1} \dots s_{i_{j-1}}}^{-1} B_T\left( \pi^\iota \right) \right]_{-}
    e^{\Tc_{s_{i_1} \dots s_{i_{j-1}}} \circ \iota(\pi)(T)}\\
= & e^{-\Tc_{s_{i_1} \dots s_{i_j}} \circ \iota(\pi)(T)} y_j e^{\Tc_{s_{i_1} \dots s_{i_{j-1}}} \circ \iota(\pi)(T)}\\
\end{align*}
where $y_j = y_{-\alpha_{i_j}}(c_j) \in B^+ \cap N \bar{s}_{i_j}^{-1} N$. Necessarily:
\begin{align*}
c_j^{-\alpha_{i_j}^\vee} & = [y_j]_0\\
& = e^{\Tc_{s_{i_1} \dots s_{i_j}} \circ \iota(\pi)(T) - \Tc_{s_{i_1} \dots s_{i_{j-1}}} \circ \iota(\pi)(T)}\\
& = \exp\left( \log\int_0^T e^{ -\alpha_{i_j}(\Tc_{s_{i_1} \dots s_{i_{j-1}}} \circ \iota(\pi))} \alpha_{i_j}^\vee\right)
\end{align*}
Therefore:
$$c_j = \frac{1}{\int_0^T e^{ -\alpha_{i_j}(\Tc_{s_{i_1} \dots s_{i_{j-1}}} \circ \iota(\pi))}}$$
And:
\begin{align*}
t_j & = c_j \exp\left( -\alpha_{i_j}\left( \Tc_{s_{i_1} \dots s_{i_{j-1}}} \circ \iota(\pi)(T) \right) \right)\\
& = \frac{1}{\int_0^T e^{ -\alpha_{i_j}(\iota \circ \Tc_{s_{i_1} \dots s_{i_{j-1}}} \circ \iota(\pi))}}\\
& = \frac{1}{\int_0^T e^{ -\alpha_{i_j}( e^{-\infty}_{s_{i_1} \dots s_{i_{j-1}}}(\pi) )}}
\end{align*}
\end{proof}

Again, as a corollary, we have the commutativity of the left side in the diagram \ref{fig:parametrizations_diagram}:
\begin{corollary}
\label{corollary:commutative_diagram_lusztig}
$$ \forall {\bf i} \in R(w_0), \forall \pi \in C_0\left( [0,T], \afrak \right), x_{\bf i} \circ \varrho_{\bf i}^L\left( \pi \right) = \varrho^L \circ p (\pi)$$
\end{corollary}
\begin{proof}
 Theorem \ref{thm:inversion_lemma_lusztig} and definition \ref{def:crystal_parameter}.
\end{proof}

\subsection{Crystal actions in coordinates}

The actions $\left( e^._\alpha \right)_{\alpha \in \Delta}$ have a very simple expression in the appropriate charts for a connected crystal. Clearly, this property is a geometric lifting of the equations from  (\ref{eqn:kashiwara_operator_f_1}) to (\ref{eqn:kashiwara_operator_e_2}) for Kashiwara operators.
 
\begin{proposition}
\label{proposition:actions_in_coordinates}
Consider a path $\pi \in C_0\left( [0, T], \afrak \right)$, ${\bf i} \in R(w_0)$ and $\alpha = \alpha_{i_1}$. If:
$$ \varrho^L_{\bf i}(\pi) = \left( t_1, \dots, t_m \right)$$
$$ \varrho^K_{\bf i}(\pi) = \left( c_1, \dots, c_m \right)$$
Then for every $\xi \in \R$:
$$ \varrho^L_{\bf i}\left( e^\xi_\alpha \cdot \pi \right) = \left( e^\xi t_1, \dots, t_m \right)$$
$$ \varrho^K_{\bf i}\left( e^\xi_\alpha \cdot \pi \right) = \left( e^\xi c_1, \dots, c_m \right)$$
\end{proposition}

\begin{rmk}
Here, we do not require twisting the Kashiwara operators like in \cite{bib:BZ01} section 5.2.
\end{rmk}

We will need the following important property of the highest and lowest path transforms.
\begin{lemma}
For $\pi \in C_0\left( [0, T], \afrak \right)$, $\alpha \in \Delta$ and $\xi \in \R$::
$$ e^{-\infty}_\alpha \cdot \left( e^\xi_\alpha \cdot \pi \right) = e^{-\infty}_\alpha \cdot \pi$$
$$ \Tc_\alpha \cdot \left( e^\xi_\alpha \cdot \pi \right) = \Tc_\alpha \cdot \pi$$
\end{lemma}
\begin{proof}
The first identity is obvious and has already been referred to. The second one can be proved either by direct computation or by using the extended duality between $e^{-\infty}_\alpha$ and $\Tc_\alpha$.
\end{proof}

\begin{proof}[Proof of proposition \ref{proposition:actions_in_coordinates}]
Write:
$$ \varrho^L_{\bf i}\left( e^\xi_\alpha \cdot \pi \right) = \left( t_1', \dots, t_m' \right)$$
$$ \varrho^K_{\bf i}\left( e^\xi_\alpha \cdot \pi \right) = \left( c_1', \dots, c_m' \right)$$
The previous lemma along with the definitions for Lusztig and Kashiwara parameters (see subsections \ref{subsection:lusztig_param} and \ref{subsection:string_params}) tell us that:
$$ \forall j \geq 2, t_j = t_j', c_j = c_j'$$
Moreover:
$$ t_1 = c_1 = \frac{1}{\int_0^T e^{-\alpha(\pi)} }$$
and:
\begin{align*}
  & t_1' = c_1'\\
= & \frac{1}{\int_0^T e^{-\alpha(e^\xi \cdot \pi)} }\\
= & \frac{1}{e^{\varepsilon_\alpha(e^\xi \cdot \pi) }}\\
= & \frac{1}{e^{\varepsilon_\alpha(\pi) - \xi}}\\
= & e^\xi t_1 = e^\xi c_1
\end{align*}
\end{proof}

\subsection{Connectedness criterion}
Because of the previous investigations, it is easy to give a simple criterion that forces two paths to belong to the same connected component.

\begin{thm}
\label{thm:connectedness_criterion}
Consider two paths $\pi$ and $\pi'$ in $C_0\left( [0,T], \afrak \right)$. The following propositions are equivalent:
\begin{itemize}
 \item[(i)] $\pi$ and $\pi'$ are connected.
 \item[(ii )] $$ \left( e^{-\infty}_{w_0}(\pi )_t, 0 \leq t < T \right)
               = \left( e^{-\infty}_{w_0}(\pi')_t, 0 \leq t < T \right) $$
 \item[(iii)] $$ \left( \Tc_{w_0}(\pi )_t, 0 < t \leq T \right)
               = \left( \Tc_{w_0}(\pi')_t, 0 < t \leq T \right) $$
\end{itemize}
\end{thm}
\begin{proof}
Proposition \ref{proposition:partial_connectedness_criterion} says that (i) implies (ii), while corollary \ref{corollary:connected_component_bijection} gives the converse.

In order to prove the equivalence between (i) and (iii), notice that $\pi$ and $\pi'$ are connected if and only if the same holds for their duals. Therefore, we have an equivalence between (i) and:
$$ \left( e^{-\infty}_{w_0} \circ \iota (\pi )_t, 0 \leq t < T \right)
               = \left( e^{-\infty}_{w_0} \circ \iota (\pi')_t, 0 \leq t < T \right) $$
Applying lemma \ref{lemma:high_low_duality}, we have the result.
\end{proof}

\subsection{Geometric RSK correspondence and Littelmann's independence theorem}

\begin{thm}[Static version]
\label{thm:geometric_rsk_correspondence}
Let $\pi_0 \in C_0\left([0,T], \afrak\right)$ and $\langle \pi_0 \rangle$ be the connected crystal it generates. Set $\lambda = \Tc_{w_0} \pi_{0}\left( T \right) $. Then the projection is an automorphism of crystals:
$$ \begin{array}{cccc}
    p: & \langle \pi_0 \rangle & \rightarrow & \Bc(\lambda)\\
       &         \pi           &   \mapsto   & B_T\left( \pi \right)
   \end{array}
 $$
\end{thm}
\begin{proof}
The arriving set is indeed the appropriate one. The inverse map is constructed using the corollary \ref{corollary:connected_component_bijection}. 
\end{proof}

\begin{thm}[Dynamical version]
\label{thm:dynamic_rsk_correspondence}
For each $T>0$, define the set:
$$ \Pc_T := \left\{ (x, \eta) \in \Bc \times C^{high}_{w_0}\left( [0,T], \afrak \right) \ | \ hw(x) = \eta(T) \right\}$$
Then, we have a bijection:
$$ \begin{array}{cccc}
  RSK: & C_0\left( [0, T], \afrak \right) & \longrightarrow & \Pc_T\\
       &         \pi                      &   \mapsto       & \left( B_T\left( \pi \right), (\Tc_{w_0} \pi_t; 0 < t \leq T ) \right)
   \end{array}
 $$
\end{thm}
\begin{proof}
By theorem \ref{thm:string_params_paths}, the knowledge of a path $\pi \in C_0\left( [0, T], \afrak \right)$ is exactly equivalent to that of knowing the string parameters and the highest weight path. And thanks to \ref{corollary:commutative_diagram_lusztig}, the string parameters along with the highest weight $\lambda = \Tc_{w_0} \pi(T)$ are encoded by $B_T(\pi) \in \Bc(\lambda)$.
\end{proof}

With this presentation, the analogy with the classical RSK correspondence is quite clear. The path $\pi$ plays the role of a word. Elements in $\Bc$, the crystal elements, play the role of semi-standard tableaux. Highest paths of type $w_0$ play the role of shape dynamic. Finally the condition:
$$ hw(x) = \Tc_{w_0}\pi(T)$$
is the equivalent of saying that the 'P' tableau and the 'Q' tableau have the same shape.

\begin{rmk}
The previous theorem means that in the path model, there is a relatively large amount of automorphic crystals. In the group picture, however the only crystal automorphism of $\Bc\left( \lambda \right)$ is the identity. In that sense, the group picture is ``minimal''.
\end{rmk}

As a corollary, we have the geometric analogue of Littelmann's independence theorem
\begin{thm}[Geometric Littelmann independence theorem]
\label{thm:geometric_littelmann_independence}
For any connected geometric path crystal $L \subset C_0\left( [0, T], \afrak \right)$, the crystal structure only depends on $\lambda = \Tc_{w_0} \pi \left( T \right)$ for any $\pi \in L$.
\end{thm}

\subsection{The case of infinite time horizon}
\label{subsection:infinite_T}

Most of the previous results carry on the case where $T=\infty$. We will explain how to proceed in order to construct Lusztig parameters for a path $\pi \in C_0(\R_+, \afrak)$.

Clearly, low path transforms can be applied to paths in $C_0\left( \R_+, \afrak \right)$. And:
\begin{lemma}
Let $\pi \in C_0\left( \R_+, \afrak \right)$ such that $\pi(t) \underset{t \rightarrow \infty}{\sim} \mu t$ with $\alpha\left(\mu\right)>0$. Then:
$$\int_0^\infty e^{-\alpha(\pi)} < \infty$$
And
$$ e^{-\infty}_\alpha \pi(t) \underset{t \rightarrow \infty}{\sim} s_\alpha(\mu) t$$
\end{lemma}
\begin{proof}
The first assertion is clear. For the second, note that:
$$ \log \int_t^\infty e^{-\alpha\left( \pi(s) \right)}ds \underset{t \rightarrow \infty}{\sim} -\alpha\left(\mu\right) t $$
Therefore:
\begin{align*}
e^{-\infty}_\alpha \pi(t) & = \pi(t) + \log\left( 1 - \frac{ \int_0^t e^{\alpha(\pi)} }{ \int_0^\infty e^{\alpha(\pi)} } \right) \alpha^\vee\\
& = \pi(t) + \log\left(\int_t^\infty e^{\alpha(\pi)}\right) \alpha^\vee 
           - \log\left(\int_0^\infty e^{\alpha(\pi)} \right) \alpha^\vee\\
& \underset{t \rightarrow \infty}{\sim} s_\alpha(\mu) t
\end{align*}
\end{proof}

Hence the idea, that in this case, path types should depend on asymtotical behavior. 
\begin{definition}[Low path types in infinite horizon]
\label{def:low_path_types_infinite}
We say that a path $\pi \in C_0\left( \R_+, \afrak \right)$ is a low path of type $w \in W$ when it has a drift in $wC$:
$$ \exists \mu \in wC, \pi(t) \underset{t \rightarrow \infty}{\sim} \mu t$$
The set of all low paths of type $w \in W$ in $C\left( \R_+, \afrak \right)$ is denoted $C^{low}_w\left( [0, T), \afrak \right)$.
\end{definition}

Lusztig parameters also have a straightforward definition. Let ${\bf i} \in R(w_0)$. For any path $\pi \in C_0\left( \R_+, \afrak \right)$ with drift in the Weyl chamber, define $\varrho_{\bf i}^L(\pi)$ as the sequence of numbers ${ \bf t } = \left( t_1, t_2, \dots, t_m \right)$ recursively as:
$$ t_k = \frac{1}{\int_0^\infty \exp\left(-\alpha_{i_k}\left( e_{s_{i_1} \dots s_{i_{k-1}}}^{-\infty} \pi \right) \right)}$$
Thanks to the previous lemma, all $t_j$ are $>0$. Finally the inversion lemma \ref{thm:inversion_lemma_lusztig} is valid with infinite horizon. It is proven by simply taking $T$ to infinity.

\subsection{Minimality of group picture}
\label{subsection:minimality}
Thanks to Littelmann's independence theorem, we have seen that there are a lot of different but isomorphic path crystals. And all of them project (thanks to the map $p$) to a certain $\Bc(\lambda)$, what we called the 'group picture'. Now one can ask the question of how minimal this group picture is. A reasonable answer can be the fact that there are very few crystal morphisms on the group picture $\Bc(\lambda)$. 

\begin{thm}
Let $f: \Bc(\lambda) \rightarrow \Bc(\mu)$ be a map such that:
 $$ \left\{
    \begin{array}{ccc}
    \forall \alpha \in \Delta, \forall c \in \R, & f \circ e^c_\alpha = e^c_\alpha \circ f\\
    \forall \alpha \in \Delta,                   & \varepsilon_\alpha \circ f = \varepsilon_\alpha\\
   \end{array} \right.$$
If $\lambda = \mu$, then $f=id$.
\end{thm}
\begin{proof}
Let $x \in \Bc(\lambda)$ and $y = f(x) \in \Bc(\mu)$. We start by lifting the problem to the path model. This means that we consider $\pi$ and $\pi'$ in $C_0\left( [0,T], \afrak \right)$ such that:
$$ x = p(\pi) = B_T(\pi), y = p(\pi') = B_T(\pi')$$
and
$$ \Tc_{w_0} \pi(T) = \lambda, \Tc_{w_0} \pi(T) = \mu$$
Therefore, we see $\Bc(\lambda) \approx \langle \pi \rangle$ and $\Bc(\mu) \approx \langle \pi' \rangle$ as path crystals. Let us prove that for all $k \in \N$, ${\bf i} \in I^k$ and $\left( c_1, \dots, c_k \right) \in \R^k$:
$$ \frac{1}{\int_0^T \exp\left( -\alpha_{i_k}\left( e^{c_{i_1}}_{\alpha_{i_1}} \cdots e^{c_{i_{k-1}}}_{\alpha_{i_{k-1}}} \pi' \right) \right) }
 = \frac{1}{\int_0^T \exp\left( -\alpha_{i_k}\left( e^{c_{i_1}}_{\alpha_{i_1}} \cdots e^{c_{i_{k-1}}}_{\alpha_{i_{k-1}}} \pi \right) \right) }
$$
Indeed, using the same notations in the path model and in the group picture, this is equivalent to:
\begin{align*}
  & \varepsilon_{\alpha_{i_k}}\left( e^{c_{i_1}}_{\alpha_{i_1}} \cdots e^{c_{i_{k-1}}}_{\alpha_{i_{k-1}}} \pi' \right)\\
= & \varepsilon_{\alpha_{i_k}}\left( e^{c_{i_1}}_{\alpha_{i_1}} \cdots e^{c_{i_{k-1}}}_{\alpha_{i_{k-1}}} y \right)\\
= & \varepsilon_{\alpha_{i_k}}\left( e^{c_{i_1}}_{\alpha_{i_1}} \cdots e^{c_{i_{k-1}}}_{\alpha_{i_{k-1}}} f(x) \right)\\
= & \varepsilon_{\alpha_{i_k}}\circ f \left( e^{c_{i_1}}_{\alpha_{i_1}} \cdots e^{c_{i_{k-1}}}_{\alpha_{i_{k-1}}} x \right)\\
= & \varepsilon_{\alpha_{i_k}}\left( e^{c_{i_1}}_{\alpha_{i_1}} \cdots e^{c_{i_{k-1}}}_{\alpha_{i_{k-1}}} \pi \right)
\end{align*}
Taking all the $c_j \rightarrow -\infty$ and ${\bf i}$ a positive root enumeration, one finds that the Lusztig parameters of $\pi$ and $\pi'$ coincide (see subsection \ref{subsection:lusztig_param}). Then the same goes for $x$ and $y$. And if $\lambda = \mu$, we have $x=y$.

Note that the condition $\lambda = \mu$ can be deduced from elsewhere if for instance $\gamma(x) = \gamma(y)$. Hence the following remark.
\end{proof}

\begin{rmk}
\label{rmk:minimality}
$\lambda = \mu$ is for instance implied by:
$$ \gamma \circ f = \gamma $$
\end{rmk}

\section{Involutions and crystals }
\label{section:involutions}

\subsection{Kashiwara involution}
\label{subsection:involution_iota}
The Kashiwara involution was defined at the level of the enveloping algebra as the unique anti-automorphism satisfying
$$e_\alpha^\iota = e_\alpha, \quad f_\alpha^\iota = f_\alpha \quad h_\alpha^\iota = -h_\alpha $$
It can naturally be lifted to a group anti-automorphism, and we have seen on the path model that it is the group picture counterpart of duality: $\pi \mapsto \pi\left( T-t \right) - \pi\left( T \right)$.\\

\subsection{Sch\"utzenberger involution}
\label{subsection:involution_S}
This map was originally introduced by Sch\"utzenberger as an involution on semi-standard tableaux of a given shape. As tableaux with $n$ letters of a given shape $\lambda$ can be identified with the highest weight crystal $\Bfrak\left( \lambda \right)$ of type $A_n$, it can be seen as an involution on highest weight crystals.\\

\subsubsection{Definition in the group picture}

\begin{definition}[Sch\"utzenberger involution on $G$]
$$ \forall x \in G, S\left( x \right) = \overline{w}_0^{-1} ( x^{-1} )^{iT} \overline{w}_0 = \overline{w}_0 ( x^{-1} )^{iT} \overline{w}_0^{-1} $$ 
\end{definition}

\begin{rmk}
 This is defined ``coordinate free'' on the entire group.
\end{rmk}
\begin{rmk}
 Both definitions agree because $\overline{w}_0^2$ belongs to $\mathcal{Z}(G)$, the center of $G$ as a consequence of the following lemma.
\end{rmk}
\begin{lemma}
 For each $w$, define $t(w) = \overline{w} \overline{w^{-1}}$. We have:
$$ t(w) = e^{i\pi \left( \rho^\vee - w \rho^\vee \right)}$$
In particular:
$$ t(w_0) = \overline{w}_0^2 = e^{i 2 \pi \rho^\vee} $$
\end{lemma}
\begin{proof}
Using the notations from preliminaries:
\begin{align*}
t(w) & = \bar{s}_{i_1} \dots \bar{s}_{i_l} \bar{s}_{i_l} \dots \bar{s}_{i_1}\\
& = \phi_{i_l}\left( e^{i \pi h} \right)^{s_{i_1} \dots s_{i_{l-1}}}
    \phi_{i_{l-1}}\left( e^{i \pi h} \right)^{s_{i_1} \dots s_{i_{l-2}}}
    \phi_{i_2}\left( e^{i \pi h} \right)^{s_{i_1}}
    \phi_{i_1}\left( e^{i \pi h} \right)\\
& = \exp\left( i\pi \sum_{k=1}^l s_{i_1} \dots s_{i_{k-1}} h_{i_k} \right)\\
& = \exp\left( i\pi \sum_{k=1}^l \beta_k^\vee \right)\\
& = \exp\left( i\pi \left( \rho^\vee - w \rho^\vee\right) \right)\\
\end{align*} 
\end{proof}

\begin{properties}
\phantomsection
\label{properties:involution_S}
\begin{itemize}
 \item $S$ is an involutive anti-automorphism on the group.
 \item $$ \forall k \in \N, \left(i_1, \dots, i_k\right) \in I^k, S\left( x_{i_1}\left( t_1 \right) \dots x_{i_k}\left( t_k \right) \right) = x_{i_k^*}\left( t_k \right) \dots x_{i_1^*}\left( t_1 \right)$$
 \item $$ S\left( \bar{w_0} \right) = \bar{w}_0$$
 \item $$ \forall x \in B^+ w_0 B^+, hw\left( S\left( x \right) \right) = hw\left( x \right)$$
\end{itemize}
\end{properties}
\begin{proof}
It is easy to see that $S$ is an anti-automorphism as the composition of three anti-automorphisms (inverse, transpose and Kashiwara involution $\iota$) and an automorphism (conjugation by $\bar{w_0}$), the first property. The second property is known to Berenstein and Zelevinsky (relation 6.4 in \cite{bib:BZ01}) and is a consequence of $Ad(\bar{w_0})e_i = -f_{i^*}$ (proposition \ref{proposition:w_0_action_ad}). The rest is easy to check by direct computation. 
\end{proof}

Since $S$ stabilizes the geometric crystal $\Bc$ and preserves the highest weight, it is indeed an involution on highest weight crystals $\Bc\left( \lambda \right)$.

\subsubsection{Definition in the path model}
\begin{definition}[Sch\"utzenberger involution on paths]
$$ S\left( \pi \right) = - w_0 \pi^\iota $$ 
\end{definition}
As usual, this definition in fact agrees with the group picture after projection.
\begin{thm}
\label{thm:involution_S_automorphism}
 $$ p \circ S = S \circ p$$
 where on the right-hand side, $S$ stands for the Sch\"uzenberger involution on the group, and the left-hand side it is considered  in the path model.
\end{thm}
\begin{proof}
 Consider a smooth path $\pi$. Since $S\circ \iota$ is an automorphim, the left invariant equation solved by $S \circ \iota\left( B_t\left(\pi \right) \right)$ is:
\begin{align*}
d S \circ \iota\left( B_t\left(\pi \right) \right) & = S \circ \iota\left( B_t\left(\pi \right) \right)
                                                       d(S \circ \iota)\left( \sum f_\alpha + d\pi_t \right) \\
& = S \circ \iota\left( B_t\left(\pi \right) \right) \left( \sum_{\alpha \in \Delta} f_{\alpha^*} -w_0 d\pi_t \right)\\
& = S \circ \iota\left( B_t\left(\pi \right) \right) \left( \sum_{\alpha \in \Delta} f_{\alpha  } -w_0 d\pi_t \right)\\
\end{align*}
Then:
$$ S\circ \iota\left( B_t\left(\pi\right) \right) = B_t\left( -w_0 \pi \right)$$
Replacing $\pi$ by $\pi^\iota$ gives the result for all smooth paths:
$$ S\left( B_t\left(\pi\right) \right) = B_t\left( -w_0 \pi^\iota \right)$$
The smoothness assumption can then be discarded.
\end{proof}

\begin{rmk}
In the path model, $S$ does not preserve connected components, but thanks to the previous theorem, it stabilizes highest weight crystals in the group picture. As such it preserves isomorphism classes of path crystals using Littelmann's independence theorem.
\end{rmk}

As pointed out in \cite{bib:BBO2} p. 1552 lemma 4.19 the following can be taken as a definition for the Sc\"utzenberger involution for $A_n$ crystals. We prove the analogous statement in the geometric setting:
\begin{thm}
The Sch\"utzenberger involution is the unique map $S$ on geometric crystals (resp. path crystals up to crystal isomorphism) such that:
\begin{itemize}
 \item $\gamma \circ S\left( x \right) = w_0 \gamma\left( x \right)$
 \item $\varepsilon_\alpha \circ S\left( x \right) = \varphi_{\alpha^*}\left( x \right) $
       or equivalently
       $\varphi_\alpha \circ S\left( x \right) = \varepsilon_{\alpha^*}\left( x \right) $
 \item $\forall c \in \mathbb{R}, e^c_\alpha \cdot S\left(x\right) = S\left( e^{-c}_{\alpha^*} \cdot x \right)$
\end{itemize}
\end{thm}
\begin{proof}
Computations can be carried out very easily both in the group or on the path model.

For uniqueness, if $S$ and $S'$ satisfy those properties, then $S S'$ is an automorphism of (path) crystals. In the group picture, there is no crystal automorphism aside from the identity (see subsection \ref{subsection:minimality}), hence the uniqueness up to isomorphism.
\end{proof}

\chapter{Canonical measure on crystals and superpotential}

For classical discrete crystals, the measure of interest is the counting measure. For instance, the number of elements gives the dimension of the associated module. More precisely, for $\lambda \in P^+$ is a dominant weight, consider $\Bfrak(\lambda)$ the Kashiwara crystal with highest weight $\lambda$. The canonical measure is simply:
\begin{align}
\label{eqn:canonical_measure_discrete}
& \sum_{b \in \Bfrak(\lambda)} \delta_{b} 
\end{align}
where $\delta_b$ stands for the Dirac measure at the element $b$. The character $ch(V_\lambda)$ is the Laplace transform of the induced measure on weights:
\begin{align}
\label{eqn:character_discrete}
\forall \mu \in \afrak, ch\left(V_\lambda\right)(\mu) & = \sum_{b \in \Bfrak(\lambda)} e^{\langle \mu, wt(b)\rangle}
\end{align}
This allows us to define a canonical probability measure on $\Bfrak(\lambda)$ with spectral parameter $\mu$ by:
\begin{align}
\label{eqn:canonical_probability_measure_discrete}
& \frac{1}{ch\left(V_\lambda\right)(\mu)}\sum_{b \in \Bfrak(\lambda)} e^{\langle \mu, wt(b)\rangle} \delta_{b} 
\end{align}

Then, as discussed in \cite{bib:BBO2}, the object of interest in the case of continuous crystals is the Lebesgue measure on the polytope that parametrizes continuous crystals. In the geometric setting, it is natural to wonder what can play the role of canonical measure.

We start in sections \ref{section:toric_reference_measure} and \ref{section:superpotential} by presenting some surprising ingredients involved: the seemingly innocent toric reference measure $\omega$ and the superpotential map $f_B$.

Then, the canonical measure on $\Bc(\lambda)$ is defined. The reason why it is natural uses extensively the tools of probability and will be the subject of the next chapter. For now, we will rather focus on drawing consequences. Its image through the weight map is the geometric analogue of the Duistermaat-Heckman measure. Whittaker functions are defined and play the role of characters in this geometric representation-theoretic setting.

Finally a geometric Littlewood-Richardson rule is investigated. We prove in theorem \ref{thm:geom_littlewoodrichardson} that the tensor product of crystals each endowed with the canonical measure, is nothing but a convolution measure thanks to the group structure. Moreover, it disintegrates with respect to the family of canonical measures indexed by their highest weight. As such, the decomposition of a tensor product into its connected components carries on to canonical measures. Finally, we write a product formula for Whittaker functions that is interpreted as a linearization formula for characters with positive coefficients.
\section{The toric reference measure}
\label{section:toric_reference_measure}

Let us start with a simple object:
\begin{definition}[The measure $\omega_{toric}$]
 Define the measure $\omega_{toric}$ on $\R_{>0}^m$ by:
 $$ \omega_{toric} = \prod_{j=1}^m \frac{dt_j}{t_j}$$
\end{definition}
\begin{rmk}
 Notice that $\omega_{toric}$ is nothing but the flat measure in logarithmic coordinates, or the Haar measure on the multiplicative torus $\R_{>0}^m$.
\end{rmk}

$\omega_{toric}$ has the remarkable property that it is invariant under changes of parametrization for both the Lusztig variety and Kashiwara variety.

\begin{thm}[\cite{bib:GLO1} lemma 3.1 and \cite{bib:Rietsch07} theorem 7.2]
 \label{thm:omega_invariance}
 For all reduced words ${\bf i}$ and ${\bf i'}$ in $R(w_0)$, the image measure of $\omega_{toric}$ through the map $x_{ \bf i'}^{-1} \circ x_{ \bf i}$ (resp. $x_{ -\bf i'}^{-1} \circ x_{ -\bf i}$) is itself. Meaning that if the following change of variables hold:
$$ \left( t_1', t_2', \dots, t_m' \right) = x_{ \bf i'}^{-1} \circ x_{ \bf i}\left(t_1, \dots, t_m \right)$$
$$ \left( c_1', c_2', \dots, c_m' \right) = x_{-\bf i'}^{-1} \circ x_{-\bf i}\left(c_1, \dots, c_m \right)$$
Then:
$$ \prod_{j=1}^m \frac{dt_j'}{t_j'} = \prod_{j=1}^m \frac{dt_j}{t_j}$$
$$ \prod_{j=1}^m \frac{dc_j'}{c_j'} = \prod_{j=1}^m \frac{dc_j}{c_j}$$
\end{thm}
\begin{proof}
It is essentially the content of lemma 3.1 in \cite{bib:GLO1} for classical types and theorem 7.2  in \cite{bib:Rietsch07}, both proved by direct computation. A proof without any computation comes as a side product of theorem  \ref{thm:canonical_measure} and will be given in section \ref{section:some_proofs}. 
\end{proof}

This allows us to define a measure on $\Bc(\lambda)$ that has virtually the same expression regardless of the chosen parametrization. It is nothing more than the image measure of $\omega_{toric}$ under any of the usual parametrizations of $\Bc(\lambda)$.
\index{$\omega_{\Bc(\lambda)}$: Toric reference measure on highest weight crystal $\Bc(\lambda)$ }
\begin{thm}
 \label{thm:omega_invariance_on_crystal}
 There is a unique measure on $\Bc(\lambda)$ denoted by $\omega_{\Bc(\lambda)}$ such that, for all measurable functions $\varphi: \Bc(\lambda) \rightarrow \R_+$ and reduced word ${\bf i} \in R(w_0)$:
\begin{align*}
\int_{\Bc(\lambda)} \varphi(x) \omega_{\Bc(\lambda)}( dx )
& = \int_{\R_{>0}^m} \varphi \circ b_\lambda^{L} \circ x_{ \bf i}\left( t_1, \dots, t_m \right) \prod_{j=1}^m \frac{dt_j}{t_j}\\
& = \int_{\R_{>0}^m} \varphi \circ b_\lambda^{K} \circ x_{-\bf i}\left( c_1, \dots, c_m \right) \prod_{j=1}^m \frac{dc_j}{c_j}\\
& = \int_{\R_{>0}^m} \varphi \circ b_\lambda^{T} \circ x_{ \bf i}\left( t_1, \dots, t_m \right) \prod_{j=1}^m \frac{dt_j}{t_j}
\end{align*}
Moreover, $\omega$ is invariant with respect to crystal actions $e^._\alpha, \alpha \in \Delta$, meaning that:
$$ \forall c \in \R, \forall \alpha \in \Delta, \int_{\Bc(\lambda)} \varphi(x) \omega_{\Bc(\lambda)}( dx )
= \int_{\Bc(\lambda)} \varphi(e^c_\alpha \cdot x) \omega_{\Bc(\lambda)}( dx ) $$
\end{thm}
\begin{proof}
 See section \ref{section:some_proofs}.
\end{proof}
\begin{notation}
When the choice of highest weight crystal is clear from context, we will simply write $\omega$ instead of $\omega_{\Bc(\lambda)}$
\end{notation}

Although we will not use this result in the sequel, $\omega_{toric}$ can be easily seen as linked to the Haar measure on $U_\R$, the real group generated by $\{ x_\alpha(t), \alpha \in \Delta, t \in \R \}$:
\begin{proposition}[Essentially \cite{bib:GLO1} Proposition 3.1 for classical types]
Let ${\bf i} \in R\left( w_0 \right)$ and parametrize group elements $u \in U_\R$ thanks to their Lusztig parametrization:
$$u = x_{ \bf i}\left( t_1, \dots, t_m\right) $$
Using the Weyl vector $\rho = \omega_1 + \dots + \omega_n$, define the measure $du$ as:
$$ du = \prod_j\left( t_j^{ \rho\left( \beta_j^\vee \right) } \frac{dt_j}{t_j} \right)$$
The measure $du$ does not depend on a choice of reduced word ${\bf i}$ and is the Haar measure with normalizing condition:
$$ \int_{U_{>0}} e^{-\chi(u) } du = \prod_{j=1}^m \Gamma\left( \rho(\beta_j^\vee ) \right)$$
where $\chi$ is the principal character on $U$ such that:
$$\forall t \in \C, \forall \alpha \in \Delta, \chi(e^{t e_\alpha}) = t$$
\end{proposition}
\begin{proof}
The normalizing condition is clear. For proving that the definition of $du$ does not depend on the choice of ${\bf i}$, use the fact that (\ref{def:geom_weight_map}):
$$[u w_0]_0 = \prod_j t_j^{\beta_j^\vee}$$
Hence, we can write the measure in the following form:
$$du = [u w_0]_0^{\rho} \prod_{j=1}^m \frac{dt_j}{t_j}$$

Now that we know that the measure $du$ does not depend on the choice of reduced word ${\bf i}$, one can choose $\alpha_{i_1}$ as any root $\alpha$. Therefore $u = x_{ \bf i}\left( t_1, \dots, t_m\right)$ is a product of the form $u = x_\alpha\left( t_1 \right) \dots$, and $du = dt_1 \dots$. It is indeed invariant under the one parameter subgroup $x_\alpha\left( \R \right)$ for all $\alpha$, hence the left invariance. It is also right invariant because $U_\R$ is unimodular.
\end{proof}

\section{The superpotential}
\label{section:superpotential}

We will use the elementary additive unipotent characters $\chi_\alpha: U \rightarrow \C$ given by:
$$ \forall t \in \C, \forall \alpha, \beta \in \Delta, \chi_\alpha( e^{t e_\beta} ) = t \delta_{\alpha, \beta}$$
The principal character $\chi: U \rightarrow \C$ is defined as:
$$ \chi = \sum_{\alpha \in \Delta} \chi_\alpha$$

Recall that $\Bc$ is the set of totally positive, lower triangular elements in $G$ (see definition \ref{def:geom_crystal}).
\index{$f_B$: Superpotential map}
\begin{definition}
\label{definition:f_B}
Define the superpotential $f_B$ on $\Bc$ as the map:
$$
\begin{array}{cccc}
f_B: &          \Bc           & \rightarrow & \R_{>0} \\
     &   x = z \bar{w_0} t u  & \mapsto     & \chi(z) + \chi(u)
\end{array}
$$
Recall that, in the decomposition $x = z \bar{w_0} t u$, we have $z \in U^{w_0}_{>0}$, $u \in U^{w_0}_{>0}$ and 
$t \in A$.
\end{definition}

\begin{example}
\label{lbl:example_superpotential_rank1}
In rank one case, by writing for $x \in \Bc(\lambda)$:
$$ x = \begin{pmatrix} t & 0 \\ 1 & t^{-1} \end{pmatrix} e^{\lambda \alpha^\vee}
     = \begin{pmatrix} 1 & t  \\ 0 & 1 \end{pmatrix}
       \begin{pmatrix} 0 & -1 \\ 1 & 0 \end{pmatrix}
       \begin{pmatrix} e^{\lambda} & 0 \\ 0 & e^{-\lambda} \end{pmatrix}
       \begin{pmatrix} 1 & \frac{e^{-2\lambda}}{t} \\ 0 & 1 \end{pmatrix}
       $$
We find the 'superpotential' (cf \cite{bib:Rietsch11}):
$$f_B\left(x\right) = t + \frac{e^{-2 \lambda}}{t}$$
\end{example}

The following properties are easy to prove:
\begin{properties}
\label{properties:superpotential_properties}
For $x \in \Bc$:
$$ f_B\left( x \right) = f_B\left( x^\iota \right)
                       = f_B \circ S \left( x \right)$$
$$ \forall w \in W, f_B\left( w \cdot x \right) = f_B\left( x \right)$$
$$ f_B\left( e^c_\alpha \cdot x \right) = f_B(x) + \frac{e^{c}-1}{e^{\varepsilon_\alpha(x)}} + \frac{e^{-c}-1}{e^{\varphi_\alpha(x)}}$$
\end{properties}
\begin{proof}
The first relations are immediate. The invariance under the Weyl group action on the crystal can be checked easily on the reflections $s_\alpha, \alpha \in \Delta$. For the last one, use theorem \ref{thm:geom_crystal_is_crystal} and the superpotential's definition.
\end{proof}

\begin{notation}
Let $x$ be an element in $\Bc(\lambda)$ with $u = \varrho^T(x)$. Define the twist map as in \cite{bib:BZ97} by:
$$ \eta_{w_0}( u ) = [\bar{w}_0^{-1} u^T]_+$$
Thanks to the definition \ref{def:crystal_parameter}:
$$f_B(x) = \chi\left( e^{-\lambda} \eta_{w_0}(u) e^{\lambda} \right) + \chi(u)$$
The main obstruction for $f_B$ to have a simple expression is this twist.
\end{notation}

Using the notations and formulas in proposition \ref{proposition:crystal_param_maps}, we have the following semi-explicit expressions:
\begin{proposition}[Semi-explicit expressions in coordinates]
\label{proposition:semi_explicit_expressions}
Let $x \in \Bc(\lambda)$ and:
$$ v = \varrho^K(x) = x_{-\bf i}\left( c_1, \dots, c_m \right)$$
$$ u = \varrho^T(x) = x_{ \bf i}\left( t_1, \dots, t_m \right)$$
Then:
\begin{align*}
f_B(x) = & \chi \circ S \circ \iota \left( e^{-\lambda} [\bar{w}_0^{-1} u^T ]_+ e^{\lambda}\right) + \chi(u)
       =   \chi\left( e^{-\lambda} \eta_{w_0}(u) e^{\lambda}\right) + \sum_{j=1}^m t_j\\
       = & \chi\left( \eta^{w_0, e}(v) \right) + \chi\left(e^{-\lambda} v^T [v^T]_0^{-1} e^\lambda \right)
       =   \chi\left( \eta^{w_0, e}(v) \right) + \sum_{k=1}^m e^{-\alpha_{i_k}(\lambda)} c_k^{-1} \prod_{j=k+1}^m c_j^{-\alpha_{i_k}(\alpha_{i_j}^\vee)}
\end{align*}
\end{proposition}

This seemingly simple and innocent map has recently appeared in two quite different circumstances.
\paragraph{Geometric crystals:} 
Berenstein and Kazhdan use the map $f_B$ to 'cut' the discrete free crystals $\Bfrak^{free}(\lambda)$ obtained by tropicalizing $\Bc(\lambda)$ and its structural maps. Then they obtain normal Kashiwara crystals by setting (\cite{bib:BK06}):
$$ \Bfrak\left( \lambda \right) = \left\{b \in \Bfrak^{free}(\lambda) \ | \ [f_B]_{trop}(b) \geq 0 \right\}$$
The surprise is in the fact that a simple function like $f_B$ encodes exactly the string cones. As stated in the introduction of (\cite{bib:BK04}), exhibiting such a function, along with its properties, proves a corollary of the Local Langlands conjectures.

\paragraph{Mirror symmetry:} 
Rietsch used the same function in her mirror symmetric construction of the quantum Toda lattice for general type in \cite{bib:Rietsch11} (definitions in section 6). It is also related to Givental's construction, using however different coordinates (\cite{bib:Givental}).


Mirror symmetry is a phenomenon first noticed in string theory: Morally speaking, mirror symmetry says that a certain string theory called the A-model with underlying Calabi-Yau manifold $X$ is equivalent to another string theory, the B-model, on an entirely different Calabi-Yau manifold $Y$. $X$ and $Y$ are called mirrors to each other. The A-model has Gromov-Witten invariants as coupling constants which capture a symplectic structure on $X$, while the B-model captures a complex structure on $Y$.

Such a phenomenon has been generalized to other cases than Calabi-Yau. In the context of mirror symmetry for the flag manifolds $G/B$, the mirror variety is encoded by a pair $(X, e^{-W} )$ where $X$ is a complex variety and $W$ is the Landau-Ginzburg 'superpotential'. For $\P^1(\C) \approx G/B$, with $G=SL_2$, one finds $X = \C^*$ and:
$$ W(t) = t + \frac{e^{-2\lambda}}{t}$$
This is exactly $f_B(x)$ if $x \in \Bc(\lambda)$ has Lusztig parameter $t$ (see example \ref{lbl:example_superpotential_rank1} ).

We have no explanation as why mirror symmetry appears while investigating geometric crystals and we will not try to push into that direction. However, in the context of combinatorial representation theory, the strength of our approach is that the Landau-Ginzburg potential $e^{-f_B(x)} \omega(dx)$ will appear naturally as a canonical measure on geometric crystals. The measure is canonical in the sense that, once normalized into a probability measure, it is precisely the distribution of a random crystal element conditionnally to its highest weight. A random crystal element is given by a Brownian motion in the context of our geometric Littelmann path model.

\subsection{Algebraic structure of \texorpdfstring{$f_B$}{the superpotential}}

The map $f_B$ has a nice and deep algebraic structure. A first flavour that is sufficient for our needs is:

\begin{thm}
\label{thm:superpotential_structure}
If $x \in \Bc(\lambda)$ with twisted Lusztig parameter
$$u = \rho^T(x) = x_{ \bf i}\left( t_1, \dots, t_m \right)$$
then $f_B(x)$ is a Laurent polynomial in the variables $\left( t_1, \dots, t_m \right)$ with positive coefficients.
\end{thm}

Before giving a proof, let us observe that $f_B$ has an expression in term of generalized minors:
\begin{lemma}[ Variant of corollary 1.25 in \cite{bib:BK06} ]
 For $x \in \Bc(\lambda)$, with twisted Lusztig parameter $u = \varrho^T(x)$, we have:
$$ f_B\left( x \right) = \sum_{\alpha \in \Delta} e^{-\alpha(\lambda)} \frac{ \Delta_{s_\alpha \omega_\alpha, w_0 \omega_\alpha}( u )}{\Delta_{\omega_\alpha, w_0 \omega_\alpha}( u )} + \Delta_{\omega_\alpha, s_\alpha \omega_\alpha}(u)$$ 
\end{lemma}
\begin{proof}
 The character $\chi_\alpha$ can also be expressed as a minor (\cite{bib:BK06} relation 1.8):
$$ \forall u \in U^{w_0}_{>0}, \chi_\alpha(u) = \Delta_{\omega_\alpha, s_\alpha \omega_\alpha}(u)$$
Therefore, if (proposition \ref{proposition:crystal_param_maps}):
$$x = b_\lambda^T(u) = S \circ \iota\left( e^{-\lambda} [ \bar{w}_0^{-1} u^T ]_+ e^{\lambda}\right) \bar{w}_0 e^\lambda u$$
Using, $\chi \circ S \circ \iota = \chi$, we have:
\begin{align*}
f_B(x) & =  \chi \circ S \circ \iota\left( e^{-\lambda} [ \bar{w}_0^{-1} u^T ]_+ e^{\lambda}\right) + \chi(u)\\
& = \sum_{\alpha \in \Delta} e^{-\alpha(\lambda)} \chi_\alpha( [ \bar{w}_0^{-1} u^T ]_+ ) + \chi_\alpha(u)\\
& = \sum_{\alpha \in \Delta} e^{-\alpha(\lambda)} \Delta_{\omega_\alpha, s_\alpha \omega_\alpha}( [ \bar{w}_0^{-1} u^T ]_+ ) + \Delta_{\omega_\alpha, s_\alpha \omega_\alpha}(u)
\end{align*}
Moreover:
\begin{align*}
\Delta_{\omega_\alpha, s_\alpha \omega_\alpha}( [ \bar{w}_0^{-1} u^T ]_+ ) & = \Delta^{\omega_\alpha}( [ \bar{w}_0^{-1} u^T ]_+ \bar{s}_\alpha )\\
& = \Delta^{\omega_\alpha}( [ \bar{w}_0^{-1} u^T ]_0^{-1} [ \bar{w}_0^{-1} u^T ]_{0+} \bar{s}_\alpha )\\
& = \frac{ \Delta^{\omega_\alpha}( \bar{w}_0^{-1} u^T \bar{s}_\alpha )}
         { \Delta^{\omega_\alpha}( \bar{w}_0^{-1} u^T )}\\
& = \frac{ \Delta^{\omega_\alpha}( \bar{s}_\alpha^{-1} u \bar{w}_0 )}
         { \Delta^{\omega_\alpha}( u \bar{w}_0 )}
\end{align*}
Hence the result.
\end{proof}

\begin{proof}[Proof of theorem \ref{thm:superpotential_structure}]
Thanks to the previous lemma, all we need to know is that
$$ \Delta_{\omega_\alpha, s_\alpha \omega_\alpha}(u) $$
$$\frac{ \Delta_{s_\alpha \omega_\alpha, w_0 \omega_\alpha}( u )}{\Delta_{\omega_\alpha, w_0 \omega_\alpha}( u )}$$
is a Laurent polynomial with positive coefficients in the variables $t_j$. The first one is easy to deal with as:
$$ \Delta_{\omega_\alpha, s_\alpha \omega_\alpha}(u) = \chi(u) = \sum_{\alpha_{i_j} = \alpha} t_j$$
For the second one, using \cite{bib:BZ01} theorem 5.8, each of the minors $\Delta_{s_\alpha \omega_\alpha, w_0 \omega_\alpha}( u )$ and  $\Delta_{\omega_\alpha, w_0 \omega_\alpha}( u )$ are linear combinations of monomials with positive coefficients. The latter has only one monomial term by applying corollary 9.5 in \cite{bib:BZ01} (with $u=e$, $w=w_0$ and $\gamma = w_\alpha$).
\end{proof}

\subsubsection{Examples in rank 2}
Let $x$ be an element in $\Bc(\lambda)$ with $u = \varrho^T(x)$. For each classical Cartan-Killing type, we specify a reduced expression ${\bf i}$ for $w_0$ that gives rise to a parametrization of $u \in U^{w_0}_{>0}$:
$$ u = x_{i_1}\left( t_1 \right) \dots x_{i_m}\left( t_m \right)$$
We give explicit expressions in term of the $t_j$ variables for $f_B(x)$, while computing as an intermediary step the twist:
$$ \eta_{w_0}( u ) = [\bar{w}_0^{-1} u^T]_+$$

\begin{itemize}
 \item $A_2$: $w_0 = s_1 s_2 s_1$ $u = x_1(t_1) x_2(t_2) x_1(t_3)$
       \begin{align}
       \label{eqn:example_twist_A2}
       \eta_{w_0}\left( u \right) & = x_1\left( \frac{1}{t_1\left(1 + \frac{t_1}{t_3}\right)} \right)
                                      x_2\left( \frac{1 + \frac{t_1}{t_3} }{t_2} \right)
                                      x_1\left( \frac{1}{t_1 + t_3} \right) 
       \end{align}
       $$ f_B(x) = t_1 + t_2 + t_3 + e^{-\alpha_1(\lambda)} \frac{1}{t_1} + e^{-\alpha_2(\lambda)}\left( \frac{1}{t_2} + \frac{t_1}{t_2 t_3} \right)$$
 \item $B_2$: $w_0 = s_1 s_2 s_1 s_2$ $u = x_1(t_1) x_2(t_2) x_1(t_3) x_2(t_4)$
       \begin{align}
       \label{eqn:example_twist_B2}
       \eta_{w_0}\left( u \right) & = x_1\left( \frac{1}{t_1\left( 1 + \frac{t_1}{t_3}\left(1 + \frac{t_2}{t_4} \right)^2\right)} \right)
                                      x_2\left( \frac{\left(1 + \frac{t_1}{t_3}\left(1 + \frac{t_2}{t_4} \right)^2\right)}{t_2\left(1 + \frac{t_2}{t_4}\right)} \right) \\
                                  & \quad \cdot x_1\left( \frac{ \left(1 + \frac{t_2}{t_4} \right)^2}{t_1 \left(1 + \frac{t_2}{t_4} \right)^2 + t_3} \right)
                                      x_2\left( \frac{1}{t_2 + t_4} \right)
       \end{align}
       $$ f_B(x) = t_1 + t_2 + t_3 + t_4 + e^{-\alpha_1(\lambda)} \frac{1}{t_1} + e^{-\alpha_2(\lambda)}\left( \frac{1}{t_2} + \frac{t_1 t_2}{t_3 t_4^2} + \frac{t_1}{t_3 t_4}\right)$$
 \item $C_2$: $w_0 = s_1 s_2 s_1 s_2$ $u = x_1(t_1) x_2(t_2) x_1(t_3) x_2(t_4)$
       \begin{align}
       \label{eqn:example_twist_C2}
       \eta_{w_0}\left( u \right) & = x_1\left( \frac{1}{t_1\left( 1 + \frac{t_1}{t_3}\left(1 + \frac{t_2}{t_4} \right)\right)} \right)
                                      x_2\left( \frac{\left(1 + \frac{t_1}{t_3}\left(1 + \frac{t_2}{t_4} \right)\right)^2}{t_2\left(1 + \frac{t_2}{t_4}\right)} \right) \\
                                  & \quad \cdot x_1\left( \frac{ \left(1 + \frac{t_2}{t_4} \right)}{t_1 \left(1 + \frac{t_2}{t_4} \right) + t_3} \right)
                                      x_2\left( \frac{1}{t_2 + t_4} \right)
       \end{align}
       $$ f_B(x) = t_1 + t_2 + t_3 + t_4 + e^{-\alpha_1(\lambda)} \frac{1}{t_1} + e^{-\alpha_2(\lambda)}\left( \frac{1}{t_2} + \frac{2 t_1 }{t_2 t_3} + \frac{t_1^2}{t_2 t_3^2} + \frac{t_1^2}{t_3^2 t_4}\right)$$
\end{itemize}

\subsubsection{Link to cluster algebras}
In the examples of the previous subsection, we witness the so-called Laurent phenomenon: When computing the characters $\chi_\alpha \circ \eta_{w_0} (u)$ that are a priori just rational expression in the variables $(t_1, \dots, t_m)$, many simplifications occur and we end up with a Laurent polynomial with positive coefficients.

The Laurent phenomenon is a characteristic of cluster algebras. A cluster algebra is a commutative algebra with a specific set of chosen generators called clusters. Here we are concerned with the coordinate algebra of the double Bruhat cell $G^{w_0, e} := B^+ w_0 B^+ \cap B$, $\C[G^{w_0,e}]$, which has the structure of a cluster algebra (theorem 2.10 \cite{bib:BFZ05}). The clusters $\Delta\left({ \bf i }\right)$ are made of generalized minors and indexed by the possible reduced words ${ \bf i } \in R(w_0)$. In the notations of section \ref{section:total_positivity_criteria}, the minors in $\Delta({\bf i})$ are the principal minors along with those from the family $F({\bf i})$:
$$ \Delta({\bf i}) := F({\bf i}) \bigsqcup \left\{ \Delta^{\omega_\alpha}, \alpha \in \Delta \right\}$$

It is well known that the $t_j$ are linked to minors in $x$ from a specific 'cluster' $\Delta({\bf i_0})$ via an invertible monomial transformation (\cite{bib:BZ97}, \cite{bib:BZ01}). The theory of cluster algebras indicates that every minor is a Laurent polynomial in the variables of a previously fixed cluster. The fact that those Laurent polynomials have positive coefficients is still a quite open conjecture.  The link with cluster algebras is quite clear at this point: the minors appearing in the superpotential are variables in a cluster than can be obtained via seed mutation of the cluster $\Delta\left({ \bf i_0 }\right)$. Laurent phenomenon and coefficient's positivity is expected.

Therefore, theorem \ref{thm:superpotential_structure} is not a suprise given that $f_B$ can be expressed in term of generalized minors. We were able to prove it without any reference to the general theory of cluster algebras because the minors involved in our situation were not very complicated. A complete understanding of the underlying cluster algebra would provide more explicit versions of theorem \ref{thm:superpotential_structure}.

\subsubsection{Computations using the geometric path model}
The superpotential $f_B$ has the following expression in term of the geometric path model:
\begin{lemma}
\label{lemma:superpotential_path_model}
Let $x \in \Bc(\lambda)$ and $\pi \in C_0\left( [0, T], \afrak \right)$. If $x$ has twisted Lusztig parameter $u = \varrho^T(x) = x_{\bf i}\left( t_1, \dots, t_m\right)$, while $\pi$ has usual Lusztig parameters $\varrho^L_{\bf i}(\pi) = (t_1, \dots, t_m)$, we have:
$$ f_B(x) = \sum_{j=1}^m t_j + \sum_{\alpha \in \Delta} e^{-\alpha(\lambda)} \int_0^T e^{-\alpha\left(\pi(s)\right)}ds$$
\end{lemma}
\begin{proof}
By theorem \ref{thm:inversion_lemma_lusztig}, $N_T(\pi) = [u \bar{w}_0]_-$. Hence:
\begin{align*}
\int_0^T e^{-\alpha(\pi)} & = \chi_\alpha^-\left( N_T(\pi) \right)\\
& = \chi_\alpha^-\left( [u \bar{w}_0]_- \right)\\
& = \chi_\alpha\left( [\bar{w}_0^{-1} u^T]_+ \right)
\end{align*}
Recalling that:
$$ f_B(x) = \chi(u) + \sum_{\alpha} e^{-\alpha(\lambda)} \chi_\alpha\left( [\bar{w}_0^{-1} u^T]_+ \right)$$
finishes the proof.
\end{proof}

The geometric path model allows the computation of minors using integration by parts, while keeping the positivity property obvious. We illustrate this claim by an explicit computation in the $A_2$ type. Choose ${\bf i} = (1,2,1)$ and consider as in the previous lemma a path $\pi \in C_0\left( [0,T], \afrak \right)$ such that:
$$ \varrho^L_{\bf i}(\pi) = (t_1, t_2, t_3)$$
Then:
$$ \int_0^T e^{-\alpha_1(\pi)} = \frac{1}{t_1}$$
$$ \int_0^T e^{-\alpha_2(\pi)} = \frac{1}{t_2} + \frac{t_1}{t_2 t_3}$$
\begin{proof}
Recall from subsection \ref{subsection:string_params} that if we write for $j=1,2,3$:
$$ \eta_j = e^{-\infty}_{s_{i_1} \dots s_{i_j}} \cdot \pi$$ 
We have, with the convention $\eta_0 = \pi$:
$$ t_j = \frac{1}{\int_0^T e^{-\alpha_{i_j}(\eta_{j-1}) } }$$
$$ \eta_{j-1} = T_{x_{\alpha_{i_j}(t_j)}} \eta_j$$

The first identity comes from the definition of Lusztig parameters for a path. The second one needs a little more work.
\begin{align*}
  & \int_0^T e^{-\alpha_2(\pi)}\\
= & \int_0^T ds e^{-\alpha_2\left( \eta_1(s) \right)}\left( 1 + t_1 \int_0^s e^{-\alpha_1(\eta_1) } \right)^{-\alpha_2(\alpha_1^\vee)}\\
= & \int_0^T e^{-\alpha_2(\eta_1)} + t_1 \int_0^T ds e^{-\alpha_2\left( \eta_1(s) \right)} \int_0^s e^{-\alpha_1(\eta_1) }\\
= & \frac{1}{t_2} + t_1 \int_0^T ds e^{-\alpha_2\left( \eta_1(s) \right)} \int_0^s e^{-\alpha_1(\eta_1) }
\end{align*}
Moreover, using an integration by parts and the fact that $\int_0^T e^{-\alpha_2(\eta_2)} = \infty$:
\begin{align*}
  & \int_0^T ds e^{-\alpha_2\left( \eta_1(s) \right)} \int_0^s e^{-\alpha_1(\eta_1) }\\
= & \frac{1}{t_2} \int_0^T -\frac{d}{ds}\left( \frac{1}{1+t_2 \int_0^s e^{-\alpha_2(\eta_2)}} \right) ds \int_0^s e^{-\alpha_1(\eta_1) }\\
= & \frac{1}{t_2} \left[ \frac{-\int_0^s e^{-\alpha_1(\eta_1) }}{1+t_2 \int_0^s e^{-\alpha_2(\eta_2)}} \right]_0^T + \frac{1}{t_2} \int_0^T ds \frac{e^{-\alpha_1(\eta_1(s)) }}{1+t_2 \int_0^s e^{-\alpha_2(\eta_2)}}\\
= & \frac{1}{t_2} \int_0^T ds \frac{e^{-\alpha_1(\eta_1(s)) }}{1+t_2 \int_0^s e^{-\alpha_2(\eta_2)}}\\
= & \frac{1}{t_2} \int_0^T e^{-\alpha_1(\eta_2) }\\
= & \frac{1}{t_2 t_3}
\end{align*}
\end{proof}

\subsection{Existence and uniqueness of minimum on \texorpdfstring{$\Bc(\lambda)$}{a highest weight crystal}}

Because of the following theorem, $f_B$ deserves the name of potential as it behaves like a potential well on $\Bc(\lambda) \approx \R_{>0}^m$: level sets are compact.

\begin{thm}[\cite{bib:Rietsch11} proposition 11.3]
\label{thm:superpotential_well}
For $M>0$, consider the set:
$$ K_M(\lambda) = \left\{ x \in \Bc(\lambda) | f_B(x) \leq M \right\}$$
If $M$ is large enough, $K_M(\lambda)$ is a non-empty compact set.
\end{thm}
\begin{proof}
First parametrize $x \in K_M(\lambda)$ by $t = (t_1, \dots, t_m) \in \R_{>0}^m$ such that:
$$ x_{\bf i}(t_1, \dots, t_m) = u = \varrho^T(x)$$
Then:
\begin{align*}
f_B(x) & = \chi\left( e^{-\lambda} \eta_{w_0}(u) e^{\lambda} \right) + \chi(u)\\
       & = \chi\left( e^{-\lambda} \eta_{w_0}(u) e^{\lambda} \right) + \sum_{j=1}^m t_j
\end{align*}
Clearly, the condition $x \in K_M(\lambda)$ implies $t_j \leq M$ for $j=1, \dots, m$. All we need is to prove that the components of $t$ are bounded away from zero.

Here we can produce two arguments, the first one is the geometric argument produced by Rietsch. We give a quick sketch. Extend $\eta_{w_0}$ to the totally positive part of the flag manifold in the following way:
$$
\begin{array}{cccc}
\eta_{w_0}: & \left(G/B\right)_{\geq 0} & \rightarrow & \left( B \backslash G\right)_{\geq 0} \\
            &  u B                      & \mapsto     & B w_0 u^T
\end{array}
$$
It is straightforward that the extended $\eta_{w_0}$ maps $B^+ B \cap B w B$ to $B w_0 B \cap B w_0 w^{-1} B^+$. The idea is that if some of the $t_j$ go to zero, our group element $u B$ leaves the cell
$$U^{w_0}_{>0} B \subset B^+ B \cap B w_0 B$$
and exits to a cell of type $U^{w}_{>0} B$, $\ell\left( w \right) < \ell\left( w_0 \right)$. If:
$$ B u' = B x_{ \bf i}\left(t_1', \dots, t_m'\right) = \eta_{w_0}(u)$$
then, as some of the $t_j$ go to zero, $B u'$ heads to $B w_0 B \cap B w_0 w^{-1} B^+$. However, in order to reach it, some of the parameters $t_j'$ need to go infinity.

The second argument is based on the geometric path model, and is simpler given previous results. Using \ref{lemma:superpotential_path_model}, if $\pi \in C_0([0, T], \afrak)$ is a path with Lusztig parameters $t_1, \dots, t_m$, then:
$$ f_B(x) = \sum_{j=1}^m t_j + \sum_{\alpha \in \Delta} e^{-\alpha(\lambda)} \int_0^T e^{-\alpha\left(\pi(s)\right)}ds$$
As some of the $t_j$ go to zero, the path $\pi$ converges to $\eta \in C_0([0, T[, \afrak)$, an extended path of type $\ell(w) > \ell(e)$ (see definition \ref{def:low_path_types}). Hence, because of the divergent behavior at the endpoint, there is an $\alpha$ such that:
$$ \int_0^T e^{-\alpha(\pi)} = \infty$$
This cannot happen on $K_M(\lambda)$, and the $t_j$ are indeed bounded away from zero.
\end{proof}

\begin{corollary}
$f_B$ reaches a minimum inside of $\Bc(\lambda)$
\end{corollary}

The following answers the uniqueness question raised by Rietsch in \cite{bib:Rietsch11}:
\begin{thm}
 \label{thm:superpotential_minimum}
 The superpotential $f_B$ reaches its minimum on $\Bc(\lambda)$ at a unique non degenerate point $m_\lambda$. Moreover, it is fixed by the Sch\"utzenberger involution and has zero weight:
$$ \gamma( m_\lambda ) = 0$$
\end{thm}
\begin{proof}
We use logarithmic coordinates $\xi_j = \log\left( t_j \right)$. In such coordinates, denoting by ${ \bf \xi } \in \R^m$ the coordinate vector and $\langle ,\rangle$ the usual Euclidian scalar product, $f_B(x)$ has the form (theorem \ref{thm:superpotential_structure}):
$$ f_B(x) = \sum_{i \in I} c_i e^{ \langle a_i, { \bf \xi} \rangle } $$
where $c_i$ are positive coefficients and $a_i \in \Z^m$ encode exponents. Among the $a_i, i \in I$, there is the Euclidian canonical basis $(e_j)_{1 \leq j \leq m} $ because of the term:
$$ \sum_{j=1}^m t_j = \sum_{j=1}^m t^{ \langle e_j, {\bf \xi} \rangle }$$
Now, it is easy to see that $f_B({\bf \xi})$ is strictly convex as for all $v \in \R^m$:
\begin{align*}
  & \sum_{k,l=1}^m v_k v_l \frac{\partial^2 f}{\partial \xi_k \partial \xi_l}\\
= & \sum_{k,l=1}^m \sum_{i \in I} c_i v_k a_{i,k} v_l a_{i,l} e^{ \langle a_i, \xi\rangle }\\
= & \sum_{i \in I} c_i \langle  v, a_i \rangle^2 e^{ \langle a_i, \xi\rangle }\geq 0
\end{align*}
The hessian matrix is also everywhere non-degenerate: the previous inequality is strict as soon as $v$ is non zero because among the $a_i$, there is the canonical Euclidian basis. Uniqueness for $m_\lambda$ follows.

The Sch\"utzenberger involution $S$ stabilizes $\Bc(\lambda)$ and $f_B \circ S = f_B$. Because of the minimum's uniqueness, one must have $S(m_\lambda) = m_\lambda$, hence $\gamma( m_\lambda ) = w_0 \gamma( m_\lambda )$ which implies that the weight is zero.

Another way of seeing that $\gamma(m_\lambda)=0$ consists in computing the first order condition for a point $x \in \Bc(\lambda)$ being an extremal point:
\begin{align*}
  & f_B\left( e^c_\alpha \cdot x\right) -  f_B\left( x\right)\\
= & \frac{ e^{c} - 1}{ e^{\varepsilon_\alpha(x)} } + \frac{ e^{-c} - 1}{ e^{\varphi_\alpha(x)} } \\
= & c e^{\varepsilon_\alpha(x)} \left( 1 - e^{\varphi_\alpha(x) -\varepsilon_\alpha(x)} \right) + o(c)\\
= & c e^{\varepsilon_\alpha(x)} \left( 1 - e^{\alpha\left( \gamma(x) \right)} \right) + o(c)
\end{align*}
If $x$ critical:
$$ \forall \alpha, \alpha\left( \gamma(x) \right) = 0$$
Hence $\gamma(x) = 0$.\\
\end{proof}

The exact computation of this minimum would be interesting for example in computing the precise behaviour of Whittaker function $\psi_\mu(\lambda)$, that we will introduce later, as $\lambda$ goes to '$-\infty$' and the semiclassical limit for the quantum Toda equation.

\subsection{An estimate}
The following estimate is crucial in order to prove the integrability of $e^{-f_B(x)} \omega(dx)$ on $\Bc(\lambda)$.

\begin{thm}
\label{thm:superpotential_estimate}
There are rational exponents $n_j>0$ depending only on the group such that for all $x \in \Bc(\lambda)$:
$$ f_B\left( x \right)  \geq \sum_{j=1}^m t_j + \frac{ e^{-\max_\alpha \alpha(\lambda)} }{ \prod_{j=1}^m t_j^{n_j} }  $$
where $x$ is parametrized as $\varrho^T(x) = x_{ \bf i }\left( t_1, \dots, t_m \right)$. 
\end{thm}
\begin{proof}
If:
$$ u = \varrho^T(x) = x_{ \bf i }\left( t_1, \dots, t_m \right)$$
Then using definition \ref{def:crystal_parameter} and proposition \ref{proposition:semi_explicit_expressions}:
\begin{align*}
  & f_B(x)\\
= & \chi\left( u\right) + \sum_{\alpha} e^{-\alpha(\lambda)} \chi_\alpha\left( [ \bar{w}_0^{-1} u^T ]_+ \right)\\
\geq & \sum_{j=1}^m t_j + e^{-\max_\alpha \alpha(\lambda)} \chi\left( [ \bar{w}_0^{-1} u^T ]_+ \right)
\end{align*}
Moreover, in terms of the variable $t = \left( t_1, \dots, t_m\right)$,
$$L(t) = \chi\left( [ \bar{w}_0^{-1} u^T ]_+ \right)$$
is a Laurent polynomial with positive integer coefficients (theorem \ref{thm:superpotential_structure}). We write it as:
$$ L(t) = \sum_{i \in I} c_i \frac{1}{t^{a_i}}$$
Here $I$ is an index set, $a_i = \left( a_i^1, \dots, a_i^m \right) \in \Z^m$ for $i \in I$ are exponent vectors and $c_i \in \N^*, i \in I$ are the Laurent polynomial's coefficients. We use the notation $t^a := \prod_{j=1}^m t_j^{a^j}$ for $a \in \Z^m$.

In order to prove the theorem, we will focus on the lattice cone
$$\Cc(\N) := \left\{ \sum_{i \in I} \lambda_i a_i, \lambda_i \in \N \right\}$$
and prove that:
\begin{align}
\label{eqn:lattice_cone_nonempty}
\Cc(\N) \cap \left(\N^*\right)^m & \neq \emptyset  
\end{align}
Once this result obtained, pick an $m$-tuple $v \in \Cc(\N) \cap \left(\N^*\right)^m$ and call $M = \sum v_j \in \N^*$ the sum of its components. Then, for $t \in \left(\R_+^*\right)^m$:
\begin{align*}
     & L(t)^M\\
=    & \sum_{ i_1, i_2, \dots, i_M } c_{i_1} \dots c_{i_M}  \frac{1}{t^{a_{i_1} \dots a_{i_M}}}\\
\geq & \sum_{ i_1, i_2, \dots, i_M } \frac{1}{t^{a_{i_1} + \dots + a_{i_M}}}\\
=    & \sum_{ a \in \Cc(\N), \sum_{j=1}^m a^j = M} \frac{1}{t^a}\\
\geq & \frac{1}{t^v}
\end{align*}
Letting $n_j = \frac{v_j}{M} \in \Q_+^*$ will finish the proof.

Now, let us go back to proving identity (\ref{eqn:lattice_cone_nonempty}). For the purpose of using a density argument, define the convex cones:
$$\Cc(\Q_+) := \left\{ \sum_{i \in I} \lambda_i a_i, \lambda_i \in \Q_+ \right\}$$
$$\Cc(\R_+) := \left\{ \sum_{i \in I} \lambda_i a_i, \lambda_i \in \R_+ \right\}$$
The convex cone $\Cc(\R_+)-\R_+^m = \left\{ a-b, a \in \Cc(\R_+), b \in \R_+^m \right\}$ cannot entirely lie in a linear half-space. If it was the case, denote by $H$ such a half-space defined by a normal direction $x \in \R^m$:
$$ H := \left\{ y \in \R^m | \langle y , x \rangle \geq 0 \right\}$$
We have:
$$ \forall i \in I, \langle a_i, x \rangle \geq 0$$
And since $-\R_+^m \subset H$, we necessarily have $x_j \leq 0$ for $j=1, \dots, m$. Now, for $\lambda \in \R$:
\begin{align*}
  & L\left( e^{\lambda x_1}, \dots, e^{\lambda x_m} \right)\\
= & \sum_{i \in I } c_i e^{ -\lambda \langle a_i, x \rangle}
\end{align*}
Because $x$ is non zero with $x_j \leq 0$, taking $\lambda \rightarrow \infty$ forces at least one of the components of $t = \left( e^{\lambda x_1}, \dots, e^{\lambda x_m} \right)$ to zero, while $L(t)$ stays bounded. This contradicts theorem \ref{thm:superpotential_well}. We have then proved indeed that the convex cone $\Cc(\R_+)-\R_+^m$ cannot entirely lie in a linear half-space. Moreover, it is well known that the only convex cone in $\R^m$ that is not included in a half-space is $\R^m$, forcing $\Cc(\R_+)-\R_+^m = \R^m$. Therefore, $\Cc(\Q_+)-\Q_+^m$ is dense in $\R^m$ and $\Cc(\Q_+) \cap \left(\Q_+^*\right)^m$ is not empty. This implies the identity (\ref{eqn:lattice_cone_nonempty}).
\end{proof}

\section{Canonical measure}
\label{section:canonical_measure}

\index{$e^{ - f_B(x) } \omega(dx)$: Canonical measure on a geometric crystal}
\begin{definition}[Canonical measure on geometric crystals]
Define the canonical measure on $\Bc(\lambda)$ as the measure:
$$ e^{ - f_B(x) } \omega(dx)$$
\end{definition}

Recall we have defined the reference measure $\omega(dx)$ for $x \in \Bc(\lambda)$ as being given in either of the coordinates:
$$\varrho^L(x) = x_{\bf  i }\left( t_1,  \dots, t_m \right)$$
$$\varrho^K(x) = x_{\bf -i }\left( c_1 , \dots, c_m \right)$$
$$\varrho^T(x) = x_{\bf  i }\left( t_1', \dots, t_m'\right)$$
$$\omega(dx) = \prod_{j=1}^m \frac{dt_j}{t_j} = \prod_{j=1}^m \frac{dc_j}{c_j} = \prod_{j=1}^m \frac{dt_j'}{t_j'} $$

Since both $f_B$ and $\omega$ are invariant with respect to the $W$ action on the geometric crystal $\Bc(\lambda)$, the same holds for the canonical measure.

Originally, the Duistermaat-Heckman measure was used to refer to the asymptotic weight multipliticities for a very large finite dimensional representation of a semisimple group (\cite{bib:Heckman82}, \cite{bib:GS90} section 33). It is also the image measure of the uniform measure on a continuous crystal under the weight map (\cite{bib:BBO2} section 5.3 ). One can use the Littelmann path model for very long paths to recover easily the Duistermaat-Heckman measure as asymptotic weight multiplicities (\cite{bib:BBO} remark 5.8). Thus, now that we have identified a natural measure on geometric crystals, we will take virtually the same definition.
\index{$DH^\lambda$: Duistermaat-Heckman measure on $\afrak$}
\begin{definition}
For $\lambda \in \afrak$, define the geometric Duistermaat-Heckman measure $DH^\lambda$ on $\mathfrak{a}$ as the image of the canonical measure under the weight map $\gamma$.
\end{definition}
We will see that this measure intertwines the Laplacian on $\afrak$ and the quantum Toda Hamiltonian, or equivalently Brownian motion and the Whittaker process.

The Fourier-Laplace transform of the Duistermaat-Heckman measure plays the role of character, analogously to equation \ref{eqn:character_discrete}. In the geometric setting, it is a representation-theoretic definition of Whittaker functions. In a way, this is a geometric lifting of the famous Harish-Chandra Itzykzon Zuber formula.

\index{$\psi_\mu$: Whittaker functions}
\begin{definition}[Whittaker functions]
Whittaker functions are defined as the Laplace transform of the geometric Duistermaat-Heckman measure. For $\lambda \in \afrak$ and $\mu \in \hfrak$, it is given by:
\label{def:whittaker_functions}
\begin{align*}
\psi_\mu(\lambda) & = \int_\afrak e^{\langle \mu, k\rangle} DH^\lambda(dk)\\
& = \int_{\Bc(\lambda)} e^{ \langle \mu, \gamma(x) \rangle - f_B(x) } \omega(dx)
\end{align*}
\end{definition}
We will see in the next section that the integral is finite and that the Whittaker functions are well behaved. This semi-explicit integral formula given for Whittaker functions hints directly to the work of $\cite{bib:GLO1, bib:GLO2, bib:Givental}$. It is not so easy to link their formulae to ours, because of the multiple choices of coordinates. Notice however that our approach makes the choice of totally positive matrices a natural integration cycle.

In the next subsection, we will see that such functions are well-behaved and link them to Jacquet's original definition. Unless otherwise stated, the $\mu$ parameter will be taken in $\afrak$, making $\psi_\mu$ into a positive function.

\index{$C_\mu(\lambda)$: Canonical random variable on $\Bc(\lambda)$ with spectral parameter $\mu$}
\begin{definition}[Canonical probability measure with spectral parameter $\mu \in \afrak$]
\label{def:canonical_probability_measure}
For a spectral parameter $\mu \in \afrak$, define $C_\mu(\lambda)$ as a $\Bc(\lambda)$-valued random variable whose distribution satisfies for every bounded measurable function on $\Bc(\lambda)$:
\begin{align}
\label{eqn:canonical_measure_density}
\E\left( \varphi(C_\mu(\lambda)) \right) & = \frac{1}{\psi_\mu(\lambda)}  \int_{\Bc(\lambda)} \varphi(x) e^{ \langle \mu, \gamma(x) \rangle - f_B(x) } \omega(dx)
\end{align}
We will refer to its law as the canonical probability measure on $\Bc(\lambda)$ with spectral parameter $\mu$. And $C_\mu(\lambda)$ will be referred to as a canonical random variable on $\Bc(\lambda)$ with spectral parameter $\mu$.
\end{definition}
It is the geometric analogue of the probability measure defined in \ref{eqn:canonical_probability_measure_discrete}.

\begin{properties}
\label{properties:canonical_probability_measure}
 For $\lambda \in \afrak$ and $\mu \in \afrak$:
 \begin{itemize}
  \item[(i)]   $$ S\left( C_\mu(\lambda) \right)     \stackrel{\Lc}{=} C_{w_0 \mu}(\lambda) $$
  \item[(ii)]  $$ \iota\left( C_\mu(\lambda) \right) \stackrel{\Lc}{=} C_{-\mu}(-w_0 \lambda) $$
  \item[(iii)] $W$-invariance:
               $$\forall w \in W, C_{w\mu}(\lambda)  \stackrel{\Lc}{=} C_\mu(\lambda)$$
 \end{itemize}
\end{properties}
\begin{proof}
 In section \ref{section:some_proofs}.
\end{proof}

\section{Whittaker functions as geometric characters}
\label{section:whittaker_functions}
We defined Whittaker functions $\psi_\mu$ as the Laplace transform of measure induced on weights by the canonical measure. This definition is different from Jacquet's original definition as an integral on the unipotent group (\cite{bib:Jacquet67}). Whittaker functions are of special interest in number theory for instance. They appear in the Fourier expansion of Maass forms (see Goldfeld \cite{bib:Goldfeld}, chapter 5). A good knownledge of their properties is therefore essential. Our approach has the advantage to define well-behaved functions using integrals that converge rapidly for all $\mu$. Moreover, the integrands are positive. Finally, a lot of structure is exhibited thanks to the underlying geometric crystals: Whittaker functions play the role of characters in the theory.\\

Define $b: \hfrak \rightarrow \C$ as the meromorphic function
$$ b(\mu) := \prod_{\beta \in \Phi^+} \Gamma\left( \langle \beta^\vee, \mu \rangle \right)$$
It allows to define a natural normalization in our setting.

\begin{thm}
 \label{thm:whittaker_functions_properties}
 The Whittaker function 
$$ \psi_\mu(\lambda) = \int_{\Bc(\lambda)} e^{ \langle \mu, \gamma(x) \rangle - f_B(x) } \omega(dx) $$
satisfies the following:
 \begin{itemize}
  \item[(i)]   $\psi_\mu(\lambda)$ is an entire function in $\mu \in \hfrak = \afrak \otimes \C \approx \C^n$.
  \item[(ii)]  $\psi_\mu$ is invariant in $\mu$ under the Weyl group's action.
  \item[(iii)] For $\mu \in C$, the Weyl chamber, we have a probabilistic representation of the Whittaker function using $W^{(\mu)}$ a Brownian motion on $\afrak$ with drift $\mu$:
               $$ \psi_\mu\left( \lambda \right) = b(\mu) e^{ \langle \mu, \lambda \rangle }\E_\lambda\left( \exp\left( - \sum_{\alpha \in \Delta} \half \langle \alpha, \alpha \rangle \int_0^\infty ds \ e^{-\alpha(W_s^{(\mu)})} \right) \right)$$
               and is the unique solution to the quantum Toda eigenequation:
               $$ \half \Delta \psi_\mu(x) - \sum_{\alpha \in \Delta } \half \langle \alpha, \alpha \rangle e^{-\alpha(x) } \psi_\mu(x) = \half \langle \mu, \mu \rangle \psi_\mu(x)$$
               such that $\psi_\mu(x) e^{-\langle \mu, x \rangle }$ is bounded with growth condition $\psi_\mu(x) e^{-\langle \mu, x \rangle } \stackrel{ x \rightarrow \infty, x \in C }{\longrightarrow} b(\mu)$
 \end{itemize}
\end{thm}
\begin{proof}
\begin{itemize}
 \item[(i)] In coordinates, thanks to the estimate in theorem \ref{thm:superpotential_estimate} and the weight map expression in \ref{thm:geom_weight_map}, we see that
$$ \phi(\mu, x) := \exp\left( \langle \mu, \gamma(x) \rangle - f_B(x) \right) \omega(dx)$$
is holomorphic in $\mu \in \hfrak$ and integrable in the $x$ parameter uniformly for $\mu$ in a compact set. The same holds for partial derivatives w.r.t to $\mu$. Thus, integration in the $x$ parameter will give a holomorphic function whose domain is all of $\hfrak$. Hence, $\psi_\mu(\lambda)$ is entire in the $\mu$ parameter.
 \item[(ii)] Invariance under the Weyl group action is a consequence of the invariance for $e^{-f_B(x)} \omega(dx)$, and equivariance for the weight map.
 \item[(iii)] In proposition \ref{proposition:link_with_canonical}, we prove that the probabilistic representation coincides indeed with the previous definition. For the characterization as the unique solution of the above PDE, see proposition \ref{proposition:whittaker_characterization_boc}.
\end{itemize}

\end{proof}

\begin{example}
Once we choose a coordinate chart, plenty of explicit formulas are available. For instance, we can parametrize the elements $x \in \Bc(\lambda)$ thanks to $\left( c_1, \dots, c_m \right) \in \R^m$ such that:
$$ \varrho^K(x) = x_{\bf -i}\left( e^{-c_1}, \dots, e^{-c_m} \right)$$
for a certain ${\bf i} \in R(w_0)$. Then the formula in definition \ref{def:whittaker_functions} becomes:
\begin{itemize}
 \item $A_1$ case:
       $$ \psi_\mu(\lambda) = \int_{\R} \exp\left( \mu(\lambda-c)-e^{-c} - e^{c-2 \lambda} \right) dc$$
       This is a well-known formula for the Bessel function of the second kind also known as the MacDonald function.
 \item $A_2$ case, ${\bf i} = (1,2,1)$:
       \begin{align*}
       \psi_\mu(\lambda) & = \int_{\R^3} dc_1 dc_2 dc_3 e^{ \langle \mu, \lambda-c_1 \alpha_1^\vee - c_2 \alpha_2^\vee - c_3 \alpha_1^\vee \rangle} \exp( - e^{-c_1} - e^{-c_3} - e^{-(c_2-c_3)} \\
       & \ \ \         - e^{c_1-\alpha_1(\lambda-c_2 \alpha_2^\vee-c_3\alpha_1^\vee)}
                       - e^{c_2-\alpha_2(\lambda-c_3\alpha_1^\vee)}
                       - e^{c_3-\alpha_1(\lambda)} ) 
       \end{align*}
\end{itemize}
\end{example}

\subsection{Jacquet's Whittaker function}
Now, we link our definition of Whittaker functions to Jacquet's original definition in \cite{bib:Jacquet67}. In his thesis, Jacquet introduced Whittaker functions in the more general case of algebraic semi-simple groups over a locally compact field. We will mainly follow the presentation of Hashizume (\cite{bib:Hashizume82}) who deals with Whittaker functions on real Lie groups. In fact, the Whittaker function we considered is the Whittaker function on $G_0$ the split real subgroup of the complex Lie group $G$. The following definitions are valid only in the scope of this subsection.

A real form $\gfrak_0$ of $\gfrak$ is a real Lie algebra whose complexification is $\gfrak$:
$$ \gfrak = \gfrak_0 \otimes \C$$
The real form $\gfrak_0$ is said to be split or normal if for any Cartan decomposition:
$$ \gfrak_0 = \kfrak_0 + \pfrak_0$$
a Cartan subalgebra of $\gfrak_0$ can be taken in $\pfrak_0$. Any complex semi-simple Lie algebra $\gfrak$ has a split real form that is unique up to isomorphism (Ch IX theorem 5.10 in \cite{bib:Helgason78}).

Let $G_0$ be the connected real subgroup of $G$ whose Lie algebra $\gfrak_0$ is the split real form of $\gfrak$. The Cartan subalgebra in $\gfrak_0$ can be taken as the subset of $\hfrak$ where the roots take real values. Hence, it is nothing but $\afrak$.

The list of possible groups we are concerned with is:
\begin{itemize}
 \item Type $A_n$: $G_0 = SL_n(\R) \subset G = SL_n(\C)$
 \item Type $B_n$: $G_0 = SO(n+1, n) \subset G = SO_{2n+1}(\C)$
 \item Type $C_n$: $G_0 = Sp_n(\R) \subset G = Sp_n(\C)$
 \item Type $D_n$: $G_0 = SO(n, n) \subset G = SO_{2n}(\C)$
 \item Split real forms of the complex exceptionnal types.
\end{itemize}

Let $G_0 = U_0 A K_0$ be an Iwasawa decomposition of $G_0$. Here $U_0$ is the upper unipotent subgroup, $A$ the Cartan subgroup of $G_0$ and $K_0$ a maximal compact subgroup. The Iwasawa decomposition for a group element $g \in G_0$ is written:
$$ g = u(g) h(g) k(g), u(g) \in U_0, h(g) \in A, k(g) \in K_0$$
Let $\psi$ be a non-degenerate (multiplicative) unitary character on $U_0$, meaning that:
$$ \forall t \in \R, \forall \alpha \in \Delta, \psi\left( e^{t e_\alpha} \right) = e^{ i t \eta_\alpha }$$
where $\eta_\alpha \in \R^*$.

The Haar measure on $U_0$ will be simply denoted by $du$ and is normalized in the following way: consider the Lebesgue measure on $\ufrak_0$ and define $du$ as the image measure through the exponential map $\exp: \ufrak_0 \longrightarrow U_0$ which is a diffeomorphism. It is a (left and right) Haar measure (theorem 1.2.10 in \cite{bib:CG04}). The Lebesgue measure on $\ufrak_0$ is defined using as underlying Euclidian scalar product for $(x, y) \in \ufrak_0 \times \ufrak_0$, $K(x, y^T)$ where $K$ is the Killing form.

Restricting the definition Hashizume to our case (\cite{bib:Hashizume82} equation 6.4):
\index{$W^{Jacquet}\left(g: \nu, \psi \right)$: Jacquet's Whittaker function}
\begin{definition}[Jacquet's Whittaker function]
For $g \in G_0$, $\nu \in \afrak^* \otimes \C$ and $\psi$ a unitary non-degenerate character on $U_0$, define Jacquet's Whittaker function as:
$$ W^{Jacquet}\left(g: \nu, \psi \right) := \int_{U_0} h\left( \bar{w}_0^{-1} u g\right)^{\nu + \rho} \psi(u)^{-1} du$$
\end{definition}

The Harish-Chandra $c$ function is defined as :
$$ \forall \nu \in \afrak^*, c(\nu) = \int_{U_0} h\left( \bar{w}_0 u\right)^{\nu + \rho} du$$
An explicit expression for the $c$ function is in \cite{bib:Hashizume82} equation (6.12).

Both of integrals are convergent for $\nu$ belonging to the domain:
$$ D := \left\{ \nu \in \afrak^* \otimes \C \ | \ \Re\left( \langle \alpha, \nu \rangle \right) > 0 \right\}$$
Here $\Re$ denotes the real part of a complex number.

Using the Iwasawa decomposition of the element $g = u(g) h(g) k(g)$, one obtains that $W^{Jacquet}$ is entirely determined by a function on depending on $h(g)$ only:
\index{$\Psi_{\nu, \psi}$: Class one Whittaker function}
$$W^{Jacquet}\left(g: \nu, \psi \right) = \psi\left( u(g) \right) h(g)^{\rho} \Psi_{\nu, \psi}\left( \log h(g) \right)$$
The function $\Psi_{\nu, \psi}$ is called the class one Whittaker function. It has a convenient characterization of  Baudoin and O'Connell (\cite{bib:BOC09} proposition 4.1). We identify $\afrak$ and $\afrak^*$ using the Killing form.
\begin{thm}
For $\nu \in C$, the Weyl chamber, the class one Whittaker function solves the quantum Toda eigenfunction equation:
$$ \half \Delta \Psi_{\nu, \psi} - \sum_{\alpha \in \Delta} e^{2 \alpha(x)} \Psi_{\nu, \psi}
   = \half \langle \nu, \nu \rangle \Psi_{\nu, \psi}$$
with $e^{-\langle w_0 \nu, x \rangle} \Psi_{\nu, \psi}(x) $ being bounded and:
$$ \lim_{x \in -C, x \rightarrow \infty } e^{-\langle w_0 \nu, x \rangle} \Psi_{\nu, \psi}(x)  = c(\nu)$$
\end{thm}
\begin{proof}
Uniqueness comes from the martingale argument in proposition 2.3 in \cite{bib:BOC09}. We reproduce it in \ref{proposition:whittaker_characterization_boc}.

The partial differential equation is established in the proof of proposition 3.2 in \cite{bib:Hashizume82}. Notice that the difference of sign inside the exponential compared to the Toda potential.

Using the invariance property of the Haar measure:
$$W^{Jacquet}\left(g: \nu, \psi \right) =  \psi\left( u(g) \right)\int_{U_0} h\left( \bar{w}_0^{-1} u h(g) \right)^{\nu - \rho} \psi(u)^{-1} du$$
Then, because:
$$h\left( \bar{w}_0^{-1} u h(g) \right) = h(g)^{w_0} h\left( \bar{w}_0^{-1} h(g)^{-1} u h(g) \right)$$
and
$$ d\left( h(g) u h(g)^{-1} \right)= h(g)^{2\rho} du$$
we have:
$$W^{Jacquet}\left(g: \nu, \psi \right) =  \psi\left( u(g) \right) h(g)^{w_0 \nu + \rho} \int_{U_0} h\left( \bar{w}_0^{-1} u \right)^{\nu + \rho} \psi( h(g) u h(g)^{-1} )^{-1} du$$
Hence, an integral formula for the class one Whittaker function is:
$$ \Psi_{\nu, \psi}(x) = e^{ \langle w_0 \nu, x \rangle } \int_{U_0} h\left( \bar{w}_0^{-1} u \right)^{\nu + \rho} \psi( e^{x} u e^{-x} )^{-1} du$$
As $\psi$ is unitary, it is clear that $e^{-\langle w_0 \nu, x \rangle} \Psi_{\nu, \psi}(x) $ is bounded. Moreover, as:
$$ \forall n \in N, \lim_{x \in -C, x \rightarrow \infty} e^{x} n e^{-x} = id$$
we get the asymptotical behavior:
$$ \lim_{x \in -C, x \rightarrow \infty } e^{-\langle w_0 \nu, x \rangle} \Psi_{\nu, \psi}(x)  = c(\nu)$$
\end{proof}

Therefore, we can deduce:
\begin{corollary}
\label{corollary:link_to_class_one}
For $x \in \afrak$ and $\nu \in \afrak$:
$$ \psi_\nu\left( x \right) = \Psi_{2w_0 \nu, \psi}\left( -\half x - \sum_{\alpha \in \Delta} \log \frac{|\eta_\alpha|}{\sqrt{2 \langle \alpha, \alpha \rangle}} \omega_\alpha^\vee \right)
                              \frac{ b(\nu)}{c(-2 w_0 \nu)}
                              \prod_{\alpha \in \Delta} \left( \frac{|\eta_\alpha|}{\sqrt{2 \langle \alpha, \alpha \rangle }} \right)^{-\langle 2\nu, \omega_\alpha^\vee \rangle}$$
\end{corollary}
\begin{proof}
Let us prove the result for $\nu \in C$. The general case is obtained by meromorphic extension. For such a case, $-2w_0\nu$ is in domain of convergence for Jacquet's Whittaker function.

Using the previous theorem, $\Psi_{-2 w_0 \nu, \psi}\left( -\half x\right)$ solves:
$$ \half \Delta f(x) - \sum_{\alpha \in \Delta} \frac{1}{4}|\eta_\alpha|^2 e^{-\alpha(x)} f(x) 
= \half \langle \nu, \nu \rangle f(x)$$
Hence, adding the shift $s \in \afrak$ defined by:
$$ s:= \sum_{\alpha \in \Delta} \log \frac{|\eta_\alpha|}{\sqrt{2 \langle \alpha, \alpha \rangle}} \omega_\alpha^\vee $$
the function $\Psi_{-2 w_0 \nu, \psi}\left( -\half x - s \right)$ solves:
$$ \half \Delta f(x) - \sum_{\alpha \in \Delta} \half \langle \alpha, \alpha \rangle e^{-\alpha(x)} f(x) 
= \half \langle \nu, \nu \rangle f(x)$$
Therefore, both functions $\psi_\nu$ and $\Psi_{-2 w_0 \nu, \psi}\left( -\half x - s \right)$ solve the same eigenfunction equation. Both of them are bounded once multiplied by $e^{-\langle \nu, x \rangle}$. By uniqueness, they are proportionnal and their behavior at infinity inside the Weyl chamber allows us identify the right propotionnality constant:
$$ \lim_{x \in C, x \rightarrow \infty } e^{-\langle \nu, x \rangle} \psi_\nu(x)  = b(\nu)$$
$$ \lim_{x \in C, x \rightarrow \infty } e^{-\langle \nu, x \rangle} \Psi_{-2 w_0 \nu, \psi}( -\half x - s)  = c(-2 w_0 \nu) e^{\langle 2\nu, s \rangle}$$
Hence the result.
\end{proof}

\begin{rmk}
Instead of $G_0$ the split real subgroup, we could have considered the real group obtained by looking at $G$ as a real group. The classical Whittaker functions are again proportionnal to ours.
\end{rmk}

\subsection{The Whittaker Plancherel theorem}
Following Wallach (\cite{bib:Wallach92}, Chapter 15), Whittaker functions define an invertible integral transform. For the group $SL_2(\R)$, this recovers the well known Lebedev-Kontorovich transform. Because the result is stated in term of the class one Whittaker function $\Psi_{\nu, \psi}$, we take the time of reformulating it in terms of our Whittaker functions $\psi_{\nu}$.

Let $C^{\infty}_c\left( \afrak \right)$ be the space of infinitely differentiable functions with compact support and recall that $n = \dim \afrak$ is the rank of $G$. For $f \in C^{\infty}_c\left( \afrak \right)$, define its Whittaker transform as:
$$ \forall \nu \in \afrak, \hat{f}(\nu) := \int_{\afrak} f(x) \psi_{i\nu}(x) dx$$
The Sklyanin measure is the measure with density $s(\nu)$ defined for $\nu \in \afrak$ as:
\index{$s(\nu)$: Sklyanin measure's density}
\begin{align*}
s(\nu) & = \frac{1}{(2\pi)^n |W|} \prod_{\beta \in \Phi^+}\left( 
     \langle \beta^\vee, \nu \rangle \sinh\left( \pi \langle \beta^\vee, \nu \rangle \right)
     \sqrt{\frac{2}{\langle \beta, \beta \rangle}} \right)\\
& = \frac{1}{(2\pi)^n |W|} \prod_{\beta \in \Phi^+}\left( 
     \frac{\pi}{\Gamma(i \langle \beta^\vee, \nu \rangle) \Gamma(-i \langle \beta^\vee, \nu \rangle)}
     \sqrt{\frac{2}{\langle \beta, \beta \rangle}} \right)\\
\end{align*}
where $|W|$ stands for the cardinal of the Weyl group. Notice that the Sklyanin measure is invariant under the Weyl group action.
\begin{thm}
\label{thm:whittaker_plancherel}
The Whittaker transform defines an isometry from $L^2(\afrak, dx)$ to $L^2(\afrak, s(\nu) d\nu)$ with inverse, for $\hat{f} \in L^1(\afrak, s(\nu) d\nu)$:
$$ \forall x \in \afrak, f(x) = \int_{\afrak} \hat{f}(\nu) \psi_{-i\nu}(x) s(\nu) d\nu$$
\end{thm}
\begin{proof}
The isometry property follows from the above inversion formula, which we will now explain. It can be recovered from the similar transform in \cite{bib:Wallach92} 15.12.10 that uses class one Whittaker functions. It acts on $f \in C^\infty_c\left( \afrak \right)$ as:
\begin{align}
\label{eqn:formula_tilde}
\forall \nu \in \afrak, \tilde{f}(\nu) & = \int_\afrak f(x) \Psi_{i\nu, \psi}(x) dx 
\end{align}
that is inverted thanks to:
\begin{align}
\label{eqn:inversion_formula_tilde}
f(x) & = \frac{\gamma_A}{|W| c_A} \int_\afrak \tilde{f}(\nu) \Psi_{-i\nu, \psi}(x) \frac{d\nu}{c(i\nu)c(-i\nu)}
\end{align}
where, using the notations of the previous subsection, $c$ is Harish-Chandra $c$-function and $\gamma_A$ and $c_A$ are certains constants. Thanks to 12.5.3 and 13.8.2 in \cite{bib:Wallach92}, one sees that $\gamma_A = c(\rho)$ and therefore, using the explicit expression in \cite{bib:Hashizume82} (6.11):
$$ \gamma_A = \prod_{\beta \in \Phi^+}\left( \pi \sqrt{\frac{2}{\langle \beta, \beta \rangle}} \right)$$
The constant $c_A$ is given in \cite{bib:Wallach92} 13.3.2 and the usual Fourier inversion formula leads to:
$$ c_A = (2\pi)^n$$

Now, because of corollary \ref{corollary:link_to_class_one}, there is a fonction $h$ and a shift vector $s$ such that:
$$ \forall \nu \in \afrak, \forall x \in \afrak, \psi_{i\nu}(x) = \Psi_{2 w_0 i \nu, \psi}\left( - \half x + s \right) h(i\nu)$$
As a consequence of equation (\ref{eqn:inversion_formula_tilde}), after rearranging everything, the inversion formula for the transform $f \mapsto \hat{f}$ holds with the measure:
\begin{align*}
s(\nu) = & \frac{\gamma_A}{|W| c_A h(i\nu)h(-i\nu)c(2 w_0 i \nu)c(-2 w_0 i \nu)}\\
= & \frac{\gamma_A}{|W| c_A b(i\nu)b(-i\nu)}
\end{align*}
It is indeed the Sklyanin measure.
\end{proof}

\begin{rmk}
In the sense of distributions, theorem \ref{thm:whittaker_plancherel} leads to:
$$ \forall (x,y) \in \afrak^2, \delta_{x=y} = \int_{\afrak} \psi_{i\nu}(x) \psi_{-i\nu}(y) s(\nu) d\nu$$
Having in mind that Whittaker functions play the role of characters, this can be interpreted as an orthogonality of characters.
\end{rmk}

\section{Some proofs}
\label{section:some_proofs}
Here the proofs will use probabilistic results on Brownian motion that are proved in the next chapter.

\begin{proof}[Proof of theorem \ref{thm:omega_invariance}]
Invariance with respect to changes of parametrizations in both coordinate systems are equivalent, since if we write:
$$\left( x_{-i_1}(c_1) \dots x_{-i_j}(c_j) \right)^T = c_1^{-\alpha_{i_1}^\vee} \dots c_j^{-\alpha_{i_j}^\vee} x_{i_j}(t_j) \dots x_{i_1}(t_1)$$
Then thanks to lemma \ref{lemma:change_of_coordinates_UC}, $\left(\log t_j\right)_{1 \leq j \leq m}$ and $\left(\log c_j\right)_{1 \leq j \leq m}$ are related to each other by a linear transformation with matrix $M$. The matrix $M$ is upper triangular with unit diagonal, therefore the transformation has a jacobian equal to $1$.

Two reduced words ${\bf i}$ and ${\bf i'}$ can be obtained from each other by a sequence of braid moves. As such, it is sufficient to prove the statement for ${\bf i}$ and ${\bf i'}$ reduced words from a root system of rank $2$. This is exactly the computation made in \cite{bib:GLO1} lemma 3.1 for types $A_2$ and $B_2$ and in \cite{bib:Rietsch07} theorem 7.2 for all types. Another proof with no computation uses the following argument.

Thanks to \ref{thm:canonical_measure}, by writing for $\varphi$ a test function on $\Bc(\lambda)$:
$$ \E\left( \varphi\left( e^{-\theta} B_t\left( W \right) e^{\theta} \right) | \Fc^\Lambda_t, \Lambda_t = \lambda \right)
= \frac{1}{\psi_0(\lambda)} \int_{\R_{>0}^m} \varphi(x) e^{-f_B(x) } \omega(x)$$
we see that $\prod \frac{dt_j}{t_j}$ appears as the reference measure for the law of an intrinsic random variable on $\Bc(\lambda)$, when using the Lusztig parametrization for a specific reduced word ${\bf i}$. The law of the random variable $C_\mu(\lambda)$ is intrinsic in the sense that it should not depend on a choice of reduced word, hence the invariance.
\end{proof}

\begin{proof}[Proof of theorem \ref{thm:omega_invariance_on_crystal}]
Fix a reduced word ${\bf i} \in R(w_0)$ and an element $x \in \Bc(\lambda)$.

Using from properties \ref{properties:canonical_probability_measure}, the property $(ii)$ with $\delta = 0$, the reference toric measure on Lusztig parameters is transported by the Sch\"utzenberger involution to the toric measure on twisted Lusztig parameters as:
$$ \varrho^T\left( S(x) \right) = S\left( \varrho^L(x) \right)$$
giving the equality:
$$
\int_{\R_{>0}^m} \varphi \circ b_\lambda^{L} \circ x_{\bf i^*}\left( t_m, \dots, t_1 \right) \prod_{j=1}^m \frac{dt_j}{t_j}
=
\int_{\R_{>0}^m} \varphi \circ b_\lambda^{T} \circ x_{\bf i  }\left( t_1, \dots, t_m \right) \prod_{j=1}^m \frac{dt_j}{t_j}
$$
Notice that the change in order and reduced words. However, because of theorem \ref{thm:omega_invariance}, it does not matter.

Moreover, for an $x \in \Bc(\lambda)$, the parameters 
$$v = \varrho^K(x) = \circ x_{-\bf i^{op}}\left( c_m, \dots, c_1 \right)$$
$$u = \varrho^T(x) = \circ x_{ \bf i     }\left( t_1, \dots, t_m \right)$$
are linked by the simple transform (section \ref{section:geom_lifting})
$$ u = e^{-\lambda} v^T [v^T]_0^{-1} e^\lambda $$
which yields a monomial change of variable between $\left( c_1, \dots, c_m \right)$ and $\left( t_1, \dots, t_m \right)$ that preserves the toric measure. Hence:
$$
\int_{\R_{>0}^m} \varphi \circ b_\lambda^{K} \circ x_{-\bf i^{op}}\left( c_m, \dots, c_1 \right) \prod_{j=1}^m \frac{dc_j}{c_j}
=
\int_{\R_{>0}^m} \varphi \circ b_\lambda^{T} \circ x_{ \bf i}\left( t_1, \dots, t_m \right) \prod_{j=1}^m \frac{dt_j}{t_j}
$$

Invariance with respect to crystal actions comes from the fact that if $x \in \Bc(\lambda)$ has ${\bf i}$-Lusztig coordinates
$$ \left( t_1, t_2, \dots, t_m \right) $$
then $e^c_\alpha \cdot x$, with $\alpha = \alpha_{i_1}$, has Lusztig coordinates (proposition \ref{proposition:actions_in_coordinates}):
$$ \left( e^c_\alpha t_1, t_2, \dots, t_m \right) $$
\end{proof}

\begin{proof}[Proof of properties \ref{properties:canonical_probability_measure}]
Notice that
\begin{align}
 \label{eqn:theta_property}
 0 = & \theta + w_0 \theta 
\end{align}
 Indeed, the involution $*$ acts on simple roots as $\alpha^* = -w_0 \alpha$ for any $\alpha \in \Delta$. Then:
\begin{align*}
 e^{\alpha(\theta) } & = \frac{ \langle \alpha, \alpha \rangle }{2}\\
& = \frac{ \langle \alpha^*, \alpha^* \rangle }{2}\\
& = e^{\alpha^*(\theta)}\\
& = e^{\alpha( -w_0 \theta )}
\end{align*}

Now for fixed $T>0$, consider a Brownian motion $\left(W_{t}^{(\mu)} \right)_{ 0 \leq t \leq T}$ in $\afrak$ with drift $\mu$, on the time interval $[0,T]$ :

\begin{itemize}
 \item[(i)] The Sch\"utzenberger involution acts at the path level as (see subsection \ref{subsection:involution_S}):
$$ S\left( W^{(\mu)} \right)_t = - w_0 \left( W^{(\mu)}_T - W^{(\mu)}_{T-t} \right) ; 0 \leq t \leq T$$
which is also a Brownian motion, with drift $w_0 \mu$. Because the Sch\"utzenberger involution leaves the highest weights fixed (properties \ref{properties:involution_S}), we have equality for endpoints:
$$ \lambda = \Tc_{w_0}\left( W^{(\mu)} \right)_T = \Tc_{w_0} \circ S \left( W^{(\mu)} \right)_T$$
Hence, even if the filtrations generated by the paths $\Tc_{w_0}\left( W^{(\mu)} \right)$ and $\Tc_{w_0} \circ S \left( W^{(\mu)} \right)$ are different, we still have:
$$                   \left( B_T\left( W^{(\mu)}    \right) | \Tc_{w_0}\left( W^{(\mu)}    \right)_T = \lambda \right) 
   \stackrel{\Lc}{=} \left( B_T\left( S(W^{(\mu)}) \right) | \Tc_{w_0}\left( S(W^{(\mu)}) \right)_T = \lambda \right) $$
Using theorems \ref{thm:involution_S_automorphism} and \ref{thm:canonical_measure}, we have:
$$ C_\mu^\theta(\lambda) \stackrel{\Lc}{=} S\left( C_{w_0 \mu}^\theta(\lambda) \right)$$
Therefore:
$$ e^{-w_0 \theta } S\left( C_\mu(\lambda) \right) e^{w_0 \theta } \stackrel{\Lc}{=} e^{\theta} C_{w_0 \mu}(\lambda) e^{-\theta}$$
Evacuating $\theta$ using equation (\ref{eqn:theta_property}) finishes the proof.\\

 \item[(ii)] The proof is similar to (i) as the involution $\iota$ changes the drift $\mu$ to $-\mu$ when applied to a Brownian motion and $\iota\left( \Bc(\lambda) \right) = \Bc(-w_0 \lambda )$.\\

 \item[(iii)] It is a consequence of the invariance of $f_B$ w.r.t the action of $W$ on the crystal and the invariance of $\omega$ w.r.t. to crystal actions (theorem \ref{thm:omega_invariance_on_crystal}). The weight map is also equivariant. 
\end{itemize}

\end{proof}

\section{Geometric Littlewood-Richardson rule}

\subsection{Classical Littlewood-Richardson rule}
The simplest way to define Littlewood-Richardson coefficients $c_{\lambda, \mu}^\nu$ is to refer to the enumeration of Young tableaux. $c_{\lambda, \mu}^\nu$ is the number of skew-tableaux of shape $\nu / \lambda$ and weight $\mu$.

In the representation theory of $SL_n$, they are the linearization coefficients of Schur functions:
$$ s_\lambda s_\mu = \sum_{\nu \in P^+} c_{\lambda, \mu}^\nu s_\nu$$

As a natural extension to the representation theory for semi-simple groups, the Littlewood-Richardson rule is the way one can compute the multiplicities of irreductible highest weight representations $V(\nu)$ in a tensor product $V(\lambda) \otimes V(\mu)$. The generalized Littlewood-Richardson coefficients $c_{\lambda, \mu}(\nu)$ are thus defined so that:
$$V(\lambda) \otimes V(\mu) = \mathop{\bigoplus}_{\nu \in P^+} c_{\lambda, \mu}^\nu V(\nu) $$

\subsection{Generalized Littlewood-Richardson rule and probabilistic reinterpretation}
\label{subsection:reinterpretation_of_LR_rule}
In Littelmann's discrete path model (\cite{bib:Littelmann95, bib:Littelmann97}) for the the group $G^\vee$, one associates to every highest weight representation $V(\lambda), \lambda \in \left(P^+\right)^\vee$ a path crystal generated by a dominant path in $\afrak$ with endpoint $\lambda$. In this case, of course, the crystal actions are discrete. The irreductible components in $V(\lambda) \otimes V(\mu)$ are in bijection with the connected components of the crystal generated by $\eta_1 \star \eta_2$, where $\eta_1$ and $\eta_2$ are dominant paths with $\lambda$ and $\mu$ as endpoint. In order to count them, one has to count the dominant paths in $\langle \eta_1 \rangle \star \langle \eta_2 \rangle$, which are necessarily of the form $\eta_1 \star \pi_2$, $\pi_2 \in \langle \eta_2 \rangle$. Hence the statement:

\begin{thm}[Generalized Littlewood-Richardson rule(\cite{bib:Littelmann97} p.42)]
\label{thm:discrete_littlewoodrichardson}
  In the discrete Littelmann model, $c_{\lambda, \mu}^\nu$ is the number of paths $\pi_2$ in  
$\langle \eta_2 \rangle$ such that $\eta_1 \star \pi_2$ is dominant and has $\nu$ as endpoint.
\end{thm}

A probabilistic reinterpretation is possible. As usual endow both discrete highest weight crystals with the uniform probability measure. Then we can view (normalized) Littlewood-Richardson coefficients $c_{\lambda, \mu}^\nu \frac{\dim V(\nu)}{\dim V(\lambda) \ \dim V(\mu)} $ as the conditionnal distribution of a dominant path inside the crystal generated by $\eta_1 \star \eta_2$ knowing the dominant paths $\eta_1$ and $\eta_2$. The normalization is only there so that the coefficients sum up to 1. Such an idea still makes sense in a continuous setting provided that we consider canonical distributions.

\subsection{The geometry of connected components}
In the context of geometric crystals, connected components are continuous objects and therefore it is important to understand their geometry before proceeding to analyzing canonical measures on them.

Define the shift operator $\tau_s: C\left( [0,s+t], \afrak\right) \rightarrow C\left( [0,t], \afrak\right)$ as:
$$ \forall 0 \leq r \leq t, \tau_s\left( \pi \right)(r) := \pi(s+r) - \pi(s)$$
From now on, let $\pi_1 \in C_0\left( [0,s], \afrak \right)$, $\pi_2 \in C_0\left( [0,t], \afrak \right)$ and $\pi = \pi_1 \star \pi_2$. And consider the crystals they generate. Suppose that the highest weights are $\lambda = \Tc_{w_0}\left(\pi_1\right)(s)$ and $\mu = \Tc_{w_0}\left(\pi_2\right)(t)$. As such, we have isomorphisms of crystals (theorem \ref{thm:geometric_rsk_correspondence}):
\begin{align}
\label{eqn:isom_1}
\langle \pi_1 \rangle & \approx \Bc(\lambda)
\end{align}
\begin{align}
\label{eqn:isom_2}
\langle \pi_2 \rangle & \approx \Bc(\mu)
\end{align}
And (theorem \ref{thm:concatenation_is_isomorphism}):
\begin{align}
\label{eqn:isom_product}
\langle \pi_1 \rangle \otimes \langle \pi_2 \rangle & = \langle \pi_1 \rangle \star \langle \pi_2 \rangle 
                                                      \approx \Bc(\lambda) \otimes \Bc(\mu)
\end{align}

\subsubsection{Characterizing connected components in a tensor product}

\begin{proposition}
\label{proposition:concatenation_and_coordinates}
Consider $\pi_1 \in C_0\left( [0,s], \afrak \right)$, $\pi_2 \in C_0\left( [0,t], \afrak \right)$ and their concatenation $\pi = \pi_1 \star \pi_2$. If $u = \varrho^T\left( B_s(\pi_1) \right) \in U^{w_0}_{>0}$ is the twisted Lusztig parameter for the path $\pi_1 = \left( \pi(r); 0 \leq r \leq s \right)$, then:
$$ \left( \tau_s \circ \Tc_{w_0}(\pi)_r, 0 \leq r \leq t \right)
 = \left( T_u \circ \tau_s(\pi)_r, 0 \leq r \leq t \right)$$
\end{proposition}

\begin{corollary}
\label{corollary:characterization_of_connected_comp}
Consider $\eta_1 \in C_0\left( [0,s], \afrak \right)$, $\eta_2 \in C_0\left( [0,t], \afrak \right)$ and $\eta = \eta_1 \star \eta_2$. Let $\pi_1$, $\pi_1'$ in $\langle \eta_1 \rangle$ and $\pi_2$, $\pi_2'$ in $\langle \eta_2 \rangle$. The paths $\pi = \pi_1 \star \pi_2$ and $\pi' = \pi_1' \star \pi_2'$ belong to the same connected component in $\langle \eta_1 \rangle \star \langle \eta_2 \rangle$if and only if the following equality holds in $U^{w_0}_{>0}$:
$$ \varrho^T\left( B_s(\pi_1 ) \right) \varrho^L\left( B_t(\pi_2 ) \right)
 = \varrho^T\left( B_s(\pi_1') \right) \varrho^L\left( B_t(\pi_2') \right)$$
\end{corollary}
\begin{proof}
Set:
$$ u  = \varrho^T\left( B_s(\pi_1 ) \right), z  = \varrho^L\left( B_t(\pi_2 ) \right) $$
$$ u' = \varrho^T\left( B_s(\pi_1') \right), z' = \varrho^L\left( B_t(\pi_2') \right)$$
By theorem \ref{thm:connectedness_criterion}, $\pi$ and $\pi'$ belong to the same component if and only if the corresponding highest weight paths are equal:
$$ \left( \Tc_{w_0}(\pi )_u, 0 < u \leq s + t \right)
 = \left( \Tc_{w_0}(\pi')_u, 0 < u \leq s + t \right) $$
Equality already holds on the interval $[0, s]$, because we took $\pi_1$ and $\pi_1'$ in the same connected component. Therefore, by applying proposition \ref{proposition:concatenation_and_coordinates}, an equivalent condition is:
$$ \left( T_{u }(\pi_2 )_r, 0 \leq r \leq t \right)
 = \left( T_{u'}(\pi_2')_r, 0 \leq r \leq t \right)$$
Or by theorem \ref{thm:from_lowest_path_2_path}, writing $\kappa = e^{-\infty}_{w_0} \pi_2 = e^{-\infty}_{w_0} \pi_2'$:
$$ \left( T_{u z }(\kappa)_r, 0 \leq r \leq t \right)
 = \left( T_{u'z'}(\kappa)_r, 0 \leq r \leq t \right) $$
Hence the result. 
\end{proof}

Before the proof of proposition \ref{proposition:concatenation_and_coordinates}, we give a little lemma:
\begin{lemma}
$$ \tau_s \circ \Tc_{\alpha}\left( \pi \right) = T_{x_\alpha(\xi_\alpha)} \circ \tau_s \left( \pi \right)$$
where
$$\xi_\alpha = \frac{1}{\int_0^s e^{-\alpha\left( \pi^\iota \right)}}$$
\end{lemma}
\begin{proof}
For a time $0 \leq r \leq t$
\begin{align*}
\tau_s \circ \Tc_\alpha \left( \pi \right)(r) & = \pi(s+r) - \pi(s) + \log\left(\frac{ \int_0^{s+r} e^{-\alpha(\pi)} }{\int_0^s e^{-\alpha(\pi)}} \right) \alpha^\vee \\
& = \tau_s(\pi)(r) + \log \left( 1+ \frac{ \int_0^{r} e^{-\alpha(\pi(s+u))}du}{\int_0^s e^{-\alpha(\pi)}} \right)\alpha^\vee\\
& = \tau_s(\pi)(r) + \log \left( 1+ \frac{ \int_0^{r} e^{-\alpha(\pi(s+u)-\pi(s))}du}{\int_0^s e^{-\alpha(\pi - \pi(s))}} \right)\alpha^\vee\\
& = \tau_s(\pi)(r) + \log \left( 1+ \xi_\alpha \int_0^{r} e^{-\alpha(\tau_s \pi)} \right) \alpha^\vee 
\end{align*}
\end{proof}

\begin{proof}[Proof of proposition \ref{proposition:concatenation_and_coordinates}]
Applying repeatedly the previous lemma, we have by induction over the length $k$ of $w = s_{i_1} \dots s_{i_k} \in W$ that:
$$ \left( \tau_s \circ \Tc_w(\pi)_r, 0 \leq r \leq t \right)
 = \left( T_u \circ \tau_s(\pi)_r  , 0 \leq r \leq t \right)$$
where
$$ u = x_{\alpha_{i_k}}(t_k) \dots x_{\alpha_{i_1}}(t_1)$$
and for all $1 \leq j \leq k$:
\begin{align*}
t_j & = \frac{1}{\int_0^s e^{-\alpha_{i_j}\left( \iota \circ \Tc_{i_1 \dots i_{j-1}}(\pi) \right)}}\\
& = \frac{1}{\int_0^s e^{-\alpha_{i_j}\left( e^{-\infty}_{i_1 \dots i_{j-1}} \circ \iota(\pi) \right)}}
\end{align*}

Therefore, for $w = w_0$, the $t_j, 1 \leq j \leq m$ are the Lusztig parameters for the crystal element $B_s(\pi^\iota)$. The product giving $u$ uses the reversed order, and the following indentity concludes the proof, making $u$ the twisted Lusztig parameter associated to $B_s(\pi)$:
$$u = \iota \circ \varrho^L\left( B_s(\pi^\iota) \right)
    = \varrho^T\left( B_s(\pi) \right) $$
\end{proof}

\subsubsection{The tensor product as a crystal bundle}
Denote by $\Ccc$ the set of all connected components in $\Bc(\lambda) \otimes \Bc(\mu)$. Equivalently speaking, it is the quotient space for the equivalence relation $\sim$ of 'being connected':
\index{$\Ccc$: Set of connected components in a tensor product $\Bc(\lambda) \otimes \Bc(\mu)$}
$$ \Ccc := \Bc(\lambda) \otimes \Bc(\mu)/\sim$$
In general, a quotient space is a topological space at best, without natural smooth structure. Here, define:
\index{$\Pi$: Canonical surjection from a tensor product to the set of its connected components}
$$
\begin{array}{cccc}
\Pi : & \Bc(\lambda) \otimes \Bc(\mu) & \longrightarrow & U^{w_0}_{>0}\\
      & b_1 \otimes b_2               &   \mapsto       & \varrho^T(b_1) \varrho^L(b_2)
\end{array}$$
The product is simply the group product of totally positive elements which is totally positive. Clearly, thanks to corollary \ref{corollary:characterization_of_connected_comp}, after using the isomorphisms in equations (\ref{eqn:isom_1}), (\ref{eqn:isom_2}), (\ref{eqn:isom_product}), $b_1 \otimes b_2$ and $b_1' \otimes b_2'$ in $\Bc(\lambda) \otimes \Bc(\mu)$ belong to the same connected component if and only if $\Pi( b_1 \otimes b_2) = \Pi(b_1' \otimes b_2')$.

Therefore $\Pi$ is identified with the canonical surjection $\Bc(\lambda) \otimes \Bc(\mu) \rightarrow \Ccc$ that maps an element to its connected component. Notice that, $\Pi$ is smooth and, in terms of topological properties, it is a much nicer map than the universal quotient map.

The canonical surjection $\Pi$ defines a fiber bundle with base space $\Ccc \approx U^{w_0}_{>0}$. And for each $\Cc \in \Ccc$, the fiber $\Pi^{-1}(\{\Cc\})$ is a crystal isomorphic to $\Bc\left( hw(\Cc) \right)$. Hence, we can speak of crystal bundle. A ``trivialization'' is given by the map:
\begin{align}
\label{eqn:trivialization}
&
\begin{array}{cccc}
\phi: & \Bc(\lambda) \otimes \Bc(\mu) & \longrightarrow & \left\{ (\Cc, b) \in \Ccc \times \Bc | hw(\Cc) = hw(b) \right\}\\
      & b_1 \otimes b_2               &   \mapsto       & \left( \Pi( b_1 \otimes b_2 ), b_1 b_2\right)
\end{array}
\end{align}

\begin{rmk}
It is indeed almost a trivialization, since all highest weight crystals are diffeomorphic to $\R_{>0}^m$, $m = \ell(w_0)$, using for instance the Lusztig parametrization. Therefore, the target set for $\phi$ can be seen as $\Ccc \times \R_{>0}^m$. Hence, geometrically, $\left( \Bc(\lambda) \otimes \Bc(\mu), \Pi \right)$ is a trivial fiber bundle. However, this fact misses the point since the relevant information is in the crystal structure of the fibers.
\end{rmk}

\subsection{Geometric Littlewood-Richardson rule}
In the proofs of the geometric Littlewood-Richardson rule, we will use probabilistic results from the next chapter. Let $\pi = \left( W^{(\delta)}_r; 0 \leq r \leq s+t \right)$ a Brownian path on $[0, s+t]$ with drift $\delta$ in the Cartan subalgebra $\afrak$. We consider separately the crystals generated by the Brownian path $\pi_1 = \left( W^{(\delta)}_r; 0 \leq r \leq s \right)$ on $[0, s]$ and by its shift $\pi_2 = \left( \tau_s W^{(\delta)}_r; 0 \leq r \leq t \right)$ which is a path on $[0,t]$.

Following on the probabilistic interpretation in subsection \ref{subsection:reinterpretation_of_LR_rule}, the notion that embodies the idea of knowing the highest weight paths, $\Tc_{w_0}\left( W^{(\delta)} \right)_{0 < r \leq s}$ and $\Tc_{w_0}\left( \tau_s W^{(\delta)} \right)_{0 < r \leq t}$, in each of the crystals $\langle \pi_1 \rangle$ and $\langle \pi_2  \rangle$, is the $\sigma$-algebra:
$$ \Gc_{s, t} = \sigma\left\{ \Tc_{w_0}\left( W^{(\delta)} \right)_r, 0 < r \leq s ; \Tc_{w_0}\left( \tau_s W^{(\delta)} \right)_r , 0 < r \leq t \right\} $$
We will examine the measure induced by Brownian motion on the tensor product $\langle \pi_1 \rangle \otimes \langle \pi_2 \rangle$, knowing both highest weight paths. It is the perspective taken in \cite{bib:BBO2} for continuous path crystals and the same idea proves to be very fruitful in the geometric case too. Because of the semigroup structure, there is also a nice interpretation of the measure induced on fibers in term of measure convolution. Moreover, the central charge defined by Berenstein and Kazhdan in \cite{bib:BK06} naturally appears in the measure induced on the set of connected components, just like the superpotential showed up in the measure induced on geometric crystals.

For convenience, define also the filtration associated to the full highest weight path:
$$ \Fc_{s+t} = \sigma\left\{ \Tc_{w_0}\left( \pi \right)_r, 0 < r \leq s+t \right\} $$
And the $\sigma$-algebra generated by the connected component of $\pi = \pi_1 \star \pi_2 $ in $\langle \pi_1 \rangle \otimes \langle \pi_2 \rangle$ is:
$$ \sigma\left( \Pi\left( B_s(W^{(\delta)}) \otimes B_t(\tau_s W^{(\delta)}) \right) \right)$$

\subsubsection{Filtrations properties}
As a corollary of proposition \ref{proposition:concatenation_and_coordinates}, we find that the highest weight path in $\langle \pi_2 \rangle$ and the connected component of $\pi$ in $\langle \pi_1 \rangle \otimes \langle \pi_2 \rangle$ are functionals of the full highest weight path $\left( \Tc_{w_0}(W^{(\delta)})_r; 0 < r \leq s+t \right)$, which in terms of filtrations is expressed as:
\begin{proposition}
\label{proposition:filtration_properties}
$$ \Gc_{s,t} \subset \Fc_{t+s} $$
and
$$ \sigma\left( \Pi\left( B_s(W^{(\delta)}) \otimes B_t(\tau_s W^{(\delta)}) \right) \right) \subset \Fc_{t+s}$$
\end{proposition}
\begin{proof}
By applying the highest weight path transform to the relation given in proposition \ref{proposition:concatenation_and_coordinates}, we obtain:
 $$\left( \Tc_{w_0} \circ \tau_s \circ \Tc_{w_0} \right)(\pi)_r ; 0 \leq r \leq t = \left( \Tc_{w_0} \circ \tau_s \right)\left( \pi \right)_r ; 0 \leq r \leq t$$ 
which proves the first inclusion.

The second comes from the fact that $\Pi\left( B_s(W^{(\delta)}) \otimes B_t(\tau_s W^{(\delta)}) \right)$ is the Lusztig parameter of the path $\tau_s \circ \Tc_{w_0}(\pi)$. Indeed, as in the proof of corollary \ref{corollary:characterization_of_connected_comp}, set:
$$ u  = \varrho^T\left( B_s(W^{(\delta)} ) \right), z  = \varrho^L\left( B_t(\tau_s W^{(\delta)}) \right) $$
Then:
$$ \tau_s \circ \Tc_{w_0}(W^{(\delta)}) = T_u \circ \tau_s W^{(\delta)} = T_{uz}\left( e^{-\infty}_{w_0} \tau_s W^{(\delta)} \right)$$
Thanks to theorem \ref{thm:from_lowest_path_2_path}, $uz = \Pi\left( B_s(W^{(\delta)}) \otimes B_t(\tau_s W^{(\delta)}) \right)$ is definitely the Lusztig parameter of $\tau_s \circ \Tc_{w_0}(\pi)$.
\end{proof}

\subsubsection{Central charge}
The central charge on $\Bc(\lambda) \otimes \Bc(\mu)$ is defined (\cite{bib:BK06}) as:
$$
\begin{array}{cccc}
\Delta: & \Bc(\lambda) \otimes \Bc(\mu) & \longrightarrow & \R \\
        & b_1 \otimes b_2               &   \mapsto       & f_B(b_1) + f_B(b_2) - f_B(b_1 b_2)
\end{array}$$

As explained in \cite{bib:BK06}, one can easily prove it is invariant under the tensor product's crystal actions. Therefore, it can be lifted as a map on connected components. In coordinates, it is a rational substraction free expression. Therefore, it can  be tropicalized, giving a non-negative integer for very Kashiwara crystal element. Berenstein and Kazhdan used it to define a $q$-deformation of Littlewood-Richardson coefficients. We will prove that this function appears naturally in the geometric Littlewood-Richardson rule.

\subsubsection{Disintegration of canonical measures on a tensor product}

Now endow each of the highest weight crystals $\Bc(\lambda)$ and $\Bc(\mu)$ with canonical measures as in theorem \ref{thm:canonical_measure}. When forming the tensor product, we examine the induced measure and how it behaves on the fiber bundle. The following theorem tells us that the canonical measure on a tensor product separates nicely, respecting its structure as a crystal bundle.

First, notice that the natural measure on fibers is just a convolution measure.
\begin{thm}
On the Cartan subalgebra, let $W^{(\delta)}$ be a Brownian motion with drift $\delta$. Consider the 
$\lambda = \Tc_{w_0}\left( W^{(\delta)} \right)_s$ and $\mu = \Tc_{w_0} \circ \tau_s \left( W^{(\delta)} \right)_t$. 

The distribution induced on fibers within the tensor product on $\langle W_{0 \leq u \leq s}\rangle \otimes \langle \left(\tau_s W\right)_{0 \leq u \leq t}\rangle$ is a convolution measure that depends only on $\lambda$ and $\mu$:
$$B_{s+t}^\theta\left( W^{(\delta)} \right) | \Gc_{s, t} \stackrel{\Lc}{=} C_\delta\left( \lambda \right) C_\delta\left( \mu \right)$$
where $C_\delta\left( \lambda \right)$ and $C_\delta\left( \mu \right)$ are two independent canonical random variables with spectral parameter $\delta$.
\end{thm}
\begin{proof}
$$ B_{s+t}^\theta\left( W^{(\delta)} \right) = B_t\left( \tau_s W^{(\delta)} \right) B_s\left( W^{(\delta)} \right)$$
Then using theorem \ref{thm:canonical_measure}:
$$B_t^\theta\left( \tau_s W^{(\delta)} \right) | \Gc_{s, t} \stackrel{\Lc}{=} C_\delta\left( \lambda \right)$$
$$B_s^\theta\left( W^{(\delta)} \right) | \Gc_{s, t} \stackrel{\Lc}{=} C_\delta\left( \mu \right)$$
\end{proof}

Recall that the connected components of $\Bc(\lambda) \otimes \Bc(\mu)$ is denoted by:
$$ \Ccc = \Bc(\lambda) \otimes \Bc(\mu)/ \sim$$
We also use the ``trivialization'' map $\phi$ defined by equation \ref{eqn:trivialization}.
\begin{thm}[Geometric Littlewood-Richardson rule]
\label{thm:geom_littlewoodrichardson}
Consider the measure $\omega_{\Bc(\lambda)}(dx) \omega_{\Bc(\mu)}(dy)$ on $\Bc(\lambda) \otimes \Bc(\mu)$. Its image through the map $\phi$ is of the form:
$$ m(d\Cc) \omega_{\Bc\left(hw(\Cc)\right)}(db) $$
where $m$ is a Radon measure on $\Ccc$.

Moreover, the measure induced by canonical measures on a tensor product disintegrates as follows. For any positive measurable function $f$ on $\Ccc \times \Bc$, we have:
\begin{align*}
  & \int_{\Bc(\lambda) \times \Bc(\mu)} f \circ \phi( x \otimes y ) e^{-f_B(x)-f_B(y)} \omega_{\Bc(\lambda)}(dx) \omega_{\Bc(\mu)}(dy)\\
= & \int_{\Ccc} m(d\Cc) e^{-\Delta(\Cc)} \int_{\Bc(hw(\Cc))} f\left( \Cc, b \right) e^{-f_B(b) } \omega_{\Bc(hw(\Cc))}(db) 
\end{align*}
\end{thm}

Before giving a proof, let us a give a probabilistic version from which theorem \ref{thm:geom_littlewoodrichardson} will be easily deduced.
\begin{thm}[Probabilistic version]
\label{thm:geom_littlewoodrichardson_prob_version}
Let $\lambda$, $\mu$ and $\delta$ be elements in $\afrak$. 
If $C_\delta(\lambda)$ and $C_\delta(\mu)$ are two independent canonical random variables with spectral parameter $\delta$, then, for every positive measurable function $f$ on $\Ccc \otimes \Bc$, we have:
\begin{align*}
  & \E\left( f \circ \phi\left( C_\delta(\lambda) \otimes C_\delta(\mu) \right) \right)
= & \int_{\Ccc} \P\left( \Pi(C_\delta(\lambda) \otimes C_\delta(\mu)) \in d\Cc \right)
                \E\left( f\left( \Cc, C_\delta\left( hw(\Cc) \right) \right) \right)
\end{align*}
\end{thm}
\begin{proof}
We can suppose that for all $(\Cc, b) \in \Ccc \times \Bc$, $f(\Cc, b) = f_1(\Cc) f_2(b)$. Using theorem \ref{thm:canonical_measure}, we have:
$$ \E\left( f \circ \phi\left( C_\delta(\lambda) \otimes C_\delta(\mu) \right) \right)
 = \E\left( f \circ \phi\left( B^\theta_s(W^{(\delta)}) \otimes B^\theta_t(\tau_s W^{(\delta)}) \right) | \Gc_{s,t} \right)$$
Remember that thanks to proposition \ref{proposition:filtration_properties}, $\Gc_{s,t} \subset \Fc_{s+t}$. Then using the law of iterated expectations, we have:
\begin{align*}
  & \E\left( f \circ \phi\left( C_\delta(\lambda) \otimes C_\delta(\mu) \right) \right)\\
= & \E\left( \E\left( f \circ \phi\left( B^\theta_s(W^{(\delta)}) \otimes B^\theta_t(\tau_s W^{(\delta)}) \right) | \Fc_{s+t} \right) | \Gc_{s,t} \right)\\
= & \E\left( \E\left( f_1 \circ \Pi\left( B^\theta_s(W^{(\delta)}) \otimes B^\theta_t(\tau_s W^{(\delta)}) \right)
                      f_2 \left( B^\theta_{s+t}(W^{(\delta)}) \right) | \Fc_{s+t} \right) | \Gc_{s,t} \right)
\end{align*}
Again because of proposition \ref{proposition:filtration_properties}, the conjugation by $\theta$ playing no role:
\begin{align*}
  & \E\left( f \circ \phi\left( C_\delta(\lambda) \otimes C_\delta(\mu) \right) \right)\\
= & \E\left( f_1 \circ \Pi\left( B^\theta_s(W^{(\delta)}) \otimes B^\theta_t(\tau_s W^{(\delta)}) \right)
            \E\left( f_2 \left( B^\theta_{s+t}(W^{(\delta)}) \right) | \Fc_{s+t} \right) | \Gc_{s,t} \right)
\end{align*}
By writing $\Lambda_{t+s} = hw\left( B^\theta_{s+t}(W^{(\delta)}) \right)$, theorem \ref{thm:canonical_measure} tells us that the law of $B^\theta_{s+t}(W^{(\delta)})$ conditionally to $\Fc_{s+t}$ depends only on $\Lambda_{t+s}$. Moreover the highest weight $\Lambda_{t+s}$ is given by:
$$\Lambda_{t+s} = hw \circ \Pi\left( B^\theta_s(W^{(\delta)}) \otimes B^\theta_t(\tau_s W^{(\delta)}) \right)$$
Therefore:
\begin{align*}
  & \E\left( f \circ \phi\left( C_\delta(\lambda) \otimes C_\delta(\mu) \right) \right)\\
= & \E\left( f_1 \circ \Pi\left( B^\theta_s(W^{(\delta)}) \otimes B^\theta_t(\tau_s W^{(\delta)}) \right)
            \E\left( f_2 \left( C_\delta\left( \Lambda_{t+s} \right) \right) \right) | \Gc_{s,t} \right)\\
= & \int_{\Ccc} \P\left( \Pi(C_\delta(\lambda) \otimes C_\delta(\mu)) \in d\Cc \right) f_1(\Cc)
                \E\left( f_2 \left( C_\delta\left( hw(\Cc) \right) \right) \right)\\
= & \int_{\Ccc} \P\left( \Pi(C_\delta(\lambda) \otimes C_\delta(\mu)) \in d\Cc \right)
                \E\left( f\left( \Cc, C_\delta\left( hw(\Cc) \right) \right) \right)
\end{align*}
\end{proof}

\begin{proof}[Proof of theorem \ref{thm:geom_littlewoodrichardson}]
Recall that the distribution of the random variable $C_\delta(\lambda)$ is given by (definition \ref{def:canonical_probability_measure}):
$$ \P\left( C_\delta(\lambda) \in dx \right) = \frac{1}{\psi_\delta(\lambda)} e^{ \langle \delta, \gamma(x) \rangle - f_B(x) } \omega_{\Bc(\lambda)}(dx)$$
Then, the identity in theorem \ref{thm:geom_littlewoodrichardson_prob_version} becomes, for $f$ bounded measurable function on $\Ccc \otimes \Bc$:
\begin{align*}
  & \int_{\Bc(\lambda) \times \Bc(\mu)} f \circ \phi( x \otimes y ) e^{-f_B(x)-f_B(y)} \omega_{\Bc(\lambda)}(dx) \omega_{\Bc(\mu)}(dy)\\
= & \int_{\Ccc} \P\left( \Pi(C_\delta(\lambda) \otimes C_\delta(\mu)) \in d\Cc \right)
            \frac{\psi_\delta(\lambda) \psi_\delta(\mu)}{\psi_\delta\left(hw(\Cc) \right) }
            \int_{\Bc(hw(\Cc))} f\left( \Cc, b \right) e^{-f_B(b) } \omega_{\Bc(hw(\Cc))}(db)
\end{align*}
Now define a Radon measure on $\Ccc$ by:
$$ m(d\Cc) = \P\left( \Pi(C_\delta(\lambda) \otimes C_\delta(\mu)) \in d\Cc \right)
             \frac{\psi_\delta(\lambda) \psi_\delta(\mu)}{\psi_\delta\left(hw(\Cc) \right) } e^{\Delta(\Cc)}$$
Therefore:
\begin{align*}
  & \int_{\Bc(\lambda) \times \Bc(\mu)} f \circ \phi( x \otimes y ) e^{-f_B(x)-f_B(y)} \omega_{\Bc(\lambda)}(dx) \omega_{\Bc(\mu)}(dy)\\
= & \int_{\phi\left( \Bc(\lambda) \otimes \Bc(\mu) \right)} f\left( \Cc, b \right) e^{-f_B(b) - \Delta(\Cc)} m(d\Cc) \omega_{\Bc(hw(\Cc))}(db) 
\end{align*}
And since $f_B(x) + f_B(y) = \Delta\left( \Pi(x \otimes y) \right) + f_B(xy)$, the change of variable formula using the map $\phi$ tells us that $m(d\Cc) \omega_{\Bc(hw(\Cc))}(db)$ is indeed the image measure of $\omega_{\Bc(\lambda)}(dx) \omega_{\Bc(\mu)}(dy)$ through $\phi$.
\end{proof}

\begin{rmk}
 An interesting identity appearing in the previous proof is the law induced on connected components:
$$ \P\left( \Pi(C_\delta(\lambda) \otimes C_\delta(\mu)) \in d\Cc \right) 
 = \frac{\psi_\delta\left(hw(\Cc) \right) }{\psi_\delta(\lambda) \psi_\delta(\mu)} e^{-\Delta(\Cc)} m(d\Cc)$$
\end{rmk}

\subsection{Product formula for Whittaker functions}
Denote by $c_{\lambda,\mu}(d\nu)$ the image measure of $m(d\Cc) e^{-\Delta(\Cc) }$ through the map $hw: \Cc \rightarrow \afrak$. 
\begin{thm}
\label{thm:whittaker_product_formula}
For every $\delta \in \mathfrak{a}$:
\begin{align*}
  & \psi_{\delta}\left( \lambda \right) \psi_{\delta}\left( \mu \right)\\
= & \int_{\Ccc} m(d \Cc) e^{-\Delta(\Cc)} \psi_\delta\left( hw(\Cc) \right)\\
= & \int_{\afrak} c_{\lambda,\mu}(d\nu) \psi_\delta\left( \nu \right)
\end{align*}
\end{thm}
\begin{proof}
In theorem \ref{thm:geom_littlewoodrichardson}, take $f\left( \Cc, b \right) = e^{ \langle \delta, \gamma(b)\rangle }$, where $\gamma$ is the weight function on geometric crystal.
\end{proof}

\begin{rmk}
This identity is interpreted as the geometric counterpart of the linearization formula for characters evaluated at $\delta$:
$$ ch\left(V(\lambda)\right)(\delta) ch\left(V(\mu)\right)(\delta) 
 = \sum_{\nu \in P^+} c_{\lambda, \mu}^\nu ch\left(V(\nu)\right)(\delta)$$
\end{rmk}

\chapter{Random crystals and hypoelliptic Brownian motion on solvable group}

\label{chapter:random_crystals}
In the context of geometric crystals, we aim at identifying the previously described canonical measure in a natural way. A fruitful idea is to consider the measure induced on crystals by uniform paths in a path model. In the discrete Littelmann path model, taking the uniform probability measure on finite paths induces the measure \ref{eqn:canonical_measure_discrete} on the generated crystal (see RSK correspondence for the Littelmann path model). In the continuous setting, the uniform measure on finite paths is replaced by the Wiener measure: a 'uniformly chosen path' is nothing but Brownian motion. It is the essence of \cite{bib:BBO2}. And this idea proves to be very fruitful also in the geometric setting.\\

In this chapter, we consider a Brownian motion $W$ in $\afrak$ and study the path crystal it generates $\Bc_t = \langle \left(W_{s} \right)_{ 0 \leq s \leq t}\rangle$ as a random object. Recall that for a path $\pi \in \afrak$, $\langle \pi \rangle$ denotes the crystal generated by $\pi$. By analogy with the Young tableaux growth one obtains in the classical RSK correspondence (section \ref{section:classical_rsk}, one can think of $\Bc_t$ as a dynamical object growing with time. There are mainly two aspects in this description:
\begin{itemize}
 \item One can examine how the 'shape' of $\Bc_t$ evolves with time, or more precisely the dynamics of its highest weight $\Lambda_t = hw\left(\Bc_t\right)$ which encodes its isomorphism class as a geometric crystal. $\Lambda_t$ is a Markov process whose infinitesimal generator is given by a Doob transform of the quantum Toda Hamiltonian using one of its eigenfunctions, a Whittaker function.\\
 \item We will examine the law of $\langle \left(W_{s} \right)_{ 0 \leq s \leq t}\rangle$ conditionally to its highest weight being $\lambda$. This probability measure on $\Bc(\lambda)$ will be the canonical measure on a geometric crystal appropriately normalized into a probability measure.\\
\end{itemize}

We will start by stating precisely those two results in section \ref{section:random_crystals_results}.
\section{Main results}
\label{section:random_crystals_results}

The usual framework is the probability triplet $( \Omega, \Ac, \P)$ with $\Omega$ the sample space, $\Ac$ the set of events and $\P$ our working probability measure. Equality in law between random variables or processes will be denoted by $\stackrel{\mathcal{L}}{=}$.\\

\paragraph{Processes}: 
From now on, the abbreviation 'BM' will be short for 'Brownian motion'. For every process $X$ (or even a deterministic path) taking its values in a Euclidian space, we will use the notation $X^{(\mu)}_t := X_t + \mu t$. Also, we use $X^{x_0} := X + x_0$. Unless otherwise stated, the absence of superscript will indicate that the process is starting at zero.\\

The filtration generated by $X$ is denoted by:
$$ \Fc_t^X := \sigma\left( X_s ; s \leq t \right)$$

\paragraph{Laws}: 
We denote by $\gamma_{\mu}$ for $\mu>0$, a gamma random variable with parameter $\mu$:
$$ \P( \gamma_\mu \in dt ) = \frac{1}{\Gamma(\mu)} t^{\mu} e^{-t} \frac{dt}{t}$$
And ${ \bf e}_\mu$ stands for an exponential random variable with parameter $\mu$:
$$ \P( {\bf e}_\mu \in dt ) = \mu e^{-\mu t} dt$$

Having in mind the geometric RSK correspondence (theorem \ref{thm:dynamic_rsk_correspondence}), meaning the bijective map:
$$ \left( W_s; s \leq t \right) \mapsto \left( B_t(W^{(\mu)}), \left( \Tc_{w_0}( W )_s ; s \leq t \right) \right)$$
it is natural to look for a description of the highest weight process $\left( \Tc_{w_0}\left( W \right)_s ; s \leq t \right)$ and the distribution of the random crystal element $B_t(W^{(\mu)})$ conditionnally to the highest weight being fixed. Both aspects in this description have known analogues in the classical case of Young tableaux (see O'Connell \cite{bib:OC03}): the dynamic of the $Q$ tableau, or equivalently the shape, is Markovian and the distribution of the $P$ tableau conditionnally to the shape being fixed is the uniform measure on semi-standard tableaux. Our two main theorems are the geometric analogues.\\

Therefore, we will study the hypoelliptic Brownian motion $B_t(W^{(\mu)})$ (see equation \ref{lbl:process_B_explicit}) on the solvable group $B$, that is driven by the Euclidian Brownian motion $W^{(\mu)}$ with drift $\mu$ on $\afrak$. Only a little shift appears, morally because time is not flowing in the same way along all simple roots. Let us introduce the shift vector:
\index{$\theta$: Shift vector}
$$\theta = \sum_{ \alpha \in \Delta } \log\left( \frac{\langle \alpha, \alpha\rangle}{2} \right) \omega_\alpha^\vee$$
\begin{notation}
 \label{notation:theta}
 For the purpose of simpler notations, we write for any continuous path $\pi$ with values in $\afrak$:
\index{$B_t^\theta\left( \pi \right)$: Shift of $B_t\left( \pi \right)$ for a path $\pi$}
$$ B_t^\theta\left( \pi \right) = e^{-\theta} B_t\left( \pi \right) e^{\theta} $$
\end{notation}
\begin{rmk}
\label{rmk:theta_ADE}
In the simply-laced ($ADE$) cases, $\theta = 0$ because all roots can be chosen and are chosen to be of the same squared norm $2$. 
\end{rmk}
In the previous chapters, we chose to define the group picture of a path $\pi \in C_0\left( [0, T], \afrak \right)$ as $B_T(\pi)$ rather than $B_T^\theta(\pi)$ because, as long as we were not concerned with probability distributions, this shift would have made things needlessly more complicated. As such, now, we will be interested in the process $\left( B_t^\theta\left( W^{(\mu)} \right) ; t \geq 0\right)$. For every finite horizon $t>0$, it gives a random crystal element in $\Bc$ and dynamically, it can be seen as a random growing crystal.

\subsection{Markov property for highest weight}

As announced, an important aspect in the description of random growing crystals is that the highest weight is a Markov process.
\begin{thm}
\label{thm:highest_weight_is_markov}
Let $W^{(\mu)}$ be a Brownian motion with drift $\mu$ in the Cartan subalgebra $\afrak \approx \R^n$, then
$$\Lambda_t = hw\left( B_t^\theta\left( W^{(\mu)} \right) \right)  = \theta - w_0 \theta + \Tc_{w_0} W^{(\mu)}_t$$
is a diffusion process with infinitesimal generator
$$ \Lc = \psi_\mu^{-1}\left( \frac{1}{2} \Delta - \sum_{\alpha \in \Delta} \frac{1}{2} \langle \alpha, \alpha \rangle e^{-\alpha\left( x \right)}  - \frac{ \langle \mu, \mu \rangle }{2} \right) \psi_\mu = \frac{1}{2} \Delta + \nabla \log\left(\psi_\mu\right) \cdot \nabla $$
And, for fixed $t>0$, the law of $\Lambda_t$ is given by:
$$ dx \psi_\mu(x) e^{-\frac{t \langle \mu, \mu \rangle}{2}}
      \int_\afrak e^{-\frac{t \langle \nu, \nu \rangle}{2}} \psi_{-i\nu}(x) s(\nu) d\nu$$
where $s$ is the Sklyanin measure.
\end{thm}
\begin{proof}
In section \ref{section:whittaker_process}, we prove weaker versions and strenghten them in section \ref{section:intertwined_markov_kernels}.
\end{proof}

\begin{rmk}
A pedantic way of stating this theorem would be ``the isomorphism class of crystals generated by Brownian motion is Markovian''. 
\end{rmk}

\begin{rmk}
In the same fashion the classical Robinson-Schensted correspondence gives growing Young tableaux and a Markov property on the shape (or equivalently the $Q$ tableau) with dynamics given by Schur functions (see section \ref{section:classical_rsk}). Here, the 'shape' $\Lambda_t$ is Markov, with dynamics given by Whittaker functions, the geometric analogue of characters.
\end{rmk}

This theorem reduces to the Matsumoto-Yor theorem (\cite{bib:MY00-1, bib:MY00-2}) in the case of $SL_2$ and the theorem by O'Connell (\cite{bib:OConnell}) in the case of $GL_n$ (type $A_n$). Note that since O'Connell's contruction has an application to a semi-discrete polymer model, one expects other Lie types to be related to different geometries.\\

In chapter \ref{chapter:degenerations}, we will see that this theorem is a geometric lifting of theorem 5.6 in \cite{bib:BBO}, which represents the highest weight process in the continuous Littelmann path model as a Brownian motion in the Weyl chamber.

\subsection{The canonical probability measure as a conditionnal distribution}

Recall that we defined a canonical probability measure on $\Bc(\lambda)$ with spectral parameter $\mu \in \afrak$. It is the law of a random variable $C_\mu(\lambda)$ with density (definition \ref{def:canonical_probability_measure}):
$$ \P\left( C_\mu(\lambda) \in dx \right) = \frac{1}{\psi_\mu(\lambda)}  e^{ \langle \mu, \gamma(x) \rangle - f_B(x) } \omega(dx)$$
As the following theorem shows, this is indeed what appears when considering a random crystal element in $\Bc$ conditioned to have its highest weight being $\lambda$. This theorem is in fact its ``{\it raison d'\^etre}''.
\begin{thm}
\label{thm:canonical_measure}
Let $W^{(\mu)}$ be a BM with drift $\mu$ in $\afrak$ and fix $t>0$. The distribution of $B_t^\theta(W^{(\mu)})$ conditionally to the highest weight being $\lambda$ and the $\sigma$-algebra $\Fc_t^\Lambda$ depends only on $\lambda$ and given by:
\begin{align}
\label{eqn:canonical_measure_def}
\left( B_t^\theta\left( W^{(\mu)} \right) | \Fc^\Lambda_t, \Lambda_t = \lambda \right)  & \stackrel{\Lc}{=} C_\mu(\lambda)
\end{align}
\end{thm}
\begin{proof}
 See section \ref{section:intertwined_markov_kernels}.
\end{proof}

\subsection{Outline and strategy of proof}
We will start by definitions regarding what we mean by 'Brownian motion' and the spaces considered, before focusing on a hypoelliptic Brownian motion on the real solvable group and its harmonic functions. Our approach finds its source in the classical potential theory and its probabilistic counterparts, which revolves around a central idea: Harmonic functions on a domain are given by integrating functions on the boundary against the 'exiting law' of the underlying random walk, here the hypoelliptic Brownian motion.

Whittaker functions can be seen as harmonic functions for this process having a certain invariance property. This will be made precise in section \ref{section:poisson_boundary}. We are able to explicitly compute this boundary distribution in section \ref{section:conditionning_approach}, using inductively one dimensional results that were known to Matsumoto and Yor. Such results are reviewed beforehand in section \ref{section:review_matsumoto_yor}.
At this level, we will have to restrict our framework to $\mu \in C$, the Weyl chamber. For such $\mu$, Whittaker functions $\psi_\mu$ are defined simply as the integral of a character on $N$ over the exit law. The previously given formula (iii) in theorem \ref{thm:whittaker_functions_properties} is thus obtained.

Through conditionning, Whittaker functions define a different exit law that is absolutely continuous w.r.t the original one. The hypoelliptic BM conditioned to exit according to this law will give rise to the Whittaker process, a Markov process. For arbitrary finite starting points, it is an approximation of the highest weight process. Finally, we will have to deal with two technicalities in order to provide complete proofs of the two main theorems \ref{thm:highest_weight_is_markov} and \ref{thm:canonical_measure}.
\begin{itemize}
 \item Obtain the highest weight process and thus theorem \ref{thm:highest_weight_is_markov} by taking a starting point that is '$-\infty$' (section \ref{subsection:entrance_point}).\\
 \item Extend the framework to arbitrary $\mu$ (section \ref{section:intertwined_markov_kernels}). We will use an intertwining argument that will lead to a proof the theorem \ref{thm:canonical_measure}.
\end{itemize}

\section{The hypoelliptic Brownian motion on the solvable group and its Poisson boundary}
\label{section:poisson_boundary}

\subsection{Brownian motions}

Let us start with simple definitions.

Since $\afrak$ is an Euclidian space thanks to the Killing form $\langle ., . \rangle$, there is a natural notion of Brownian motion on $\afrak$. Plainly, fix an orthonormal basis on  $(X_1, \dots, X_n)$ on $\afrak$. Then a Brownian motion $W$ on $\afrak$ is written as:
$$ \forall t \geq 0, W_t := \sum_{r=1}^n \beta_t^r X_r$$
where $\left( \beta_t^{r} := \langle W_t, X_r \rangle ; t \geq 0 \right)_{1 \leq r \leq n} $ is a Brownian motion on $\R^n$.

For $\mu \in \afrak$, $W^{(\mu)}$, the Brownian motion with drift $\mu$ is given for $t \geq 0$ by:
$$ W^{(\mu)}_t = W_t + t \mu$$

Let $N_\R$ be the real part of the unipotent group $N$. It is the group generated by $y_\alpha(t), t \in \R, \alpha \in \Delta$. Then, the real part of the Borel group $B$, is given by $B_\R = N_\R A$. By using the explicit expression in equation (\ref{lbl:process_B_explicit}), we obtain after conjugation by $\theta$:
\begin{align}
\label{eqn:process_B_theta_explicit}
  & B_t^\theta(W^{(\mu)})\\
= & \left( \sum_{k \geq 0} \sum_{ i_1, \dots, i_k } \int_{ t \geq t_k \geq \dots \geq t_1 \geq 0} e^{ -\alpha_{i_1}(W^{(\mu)}_{t_1}) \dots -\alpha_{i_k}(W^{(\mu)}_{t_k}) } 
\frac{|| \alpha_{i_1} ||^2}{2} \dots \frac{|| \alpha_{i_k} ||^2}{2} f_{i_1} \dots f_{i_k} dt_1 \dots dt_k \right) e^{W^{(\mu)}_t}
\end{align} 

Considering the Stratonovitch integral, denoted by the symbol $\circ$, one has the same differentiation rules as in the usual case. Using equation (\ref{lbl:process_B_ode}), $\left( B_t^\theta\left( W^{(\mu)} \right) ; t \geq 0 \right)$ is a left-invariant process on $B_\R$ with independent increments, satisfying the left-invariant SDE on $B_\R$:
\begin{align}
\label{lbl:process_B_sde_stratonovich}
\left\{ \begin{array}{ll}
dB^\theta_t(W^{(\mu)}) = B_t^\theta(W^{(\mu)}) \circ \left( \sum_{\alpha \in \Delta} \half \langle \alpha, \alpha \rangle f_\alpha dt + dW^{(\mu)}_t\right) \\
 B^\theta_0(W^{(\mu)}) = id
\end{array} \right.
\end{align}

Let $\Cc^\infty(B_\R)$ be the space of continuously differentiable functions on the real solvable group $B_\R$. Recall that every $X \in \bfrak_\R$, can be seen as left-invariant derivation acting on $C^\infty\left( B_\R \right)$ as:
$$ \forall f \in \Cc^{\infty}(B_\R), \forall b \in B_\R, X f(b) = \frac{d}{dt}\left( f(g e^{t X}) \right)_{|t=0}$$
The Laplace operator on $A$ is defined using the orthonormal basis $\left(X_1, \dots, X_n\right)$ by:
$$ \Delta_\afrak :=  \sum_{i=1}^n X_i^2 $$

\begin{definition}
\label{def:hypoelliptic_operator}
\index{$\Dc^{(\mu)}$: Infinitesimal generator of a hypoelliptic Brownian motion on $B$}
For $\mu \in \afrak$, define $\Dc^{(\mu)}$ to be the left-invariant differential operator on the real solvable group $B_\R$ given by:
$$ \Dc^{(\mu)} := \half \Delta_\afrak + \mu + \half \sum_{\alpha \in \Delta} \langle \alpha, \alpha \rangle f_\alpha$$
\end{definition}

Such an operator is linked to the Casimir element in Kostant's Whittaker model. More details are given in the appendix \ref{appendix:whittaker_model}.

\begin{proposition}
\label{proposition:inf_generator}
The operator $\Dc^{(\mu)}$ is the infinitesimal generator of the hypoelliptic Brownian motion $\left( B_t := B_t^\theta(W^{(\mu)}), t\geq 0 \right)$ driven by $W^{(\mu)}$, a Euclidian Brownian motion on $\afrak$ with drift $\mu$.
\end{proposition}

This proposition follows easily from chapter 5, theorem 1.2 in Ikeda and Watanabe \cite{bib:IkedaWatanabe89}, which is a standard reference for stochastic analysis on manifolds. However, for the convenience of the reader, we explain how to proceed using only linear algebra, differential calculus and Euclidian Brownian motion.

\begin{rmk}
Let
$$ X_0 = \mu + \half \sum_{\alpha \in \Delta} \langle \alpha, \alpha \rangle f_\alpha$$
This allows us to write the infinitesimal generator in the usual form:
$$ \Dc^{(\mu)} = \half \sum_{i=1}^n X_i^2 + X_0$$
We call $B_t := B_t^\theta\left( W^{(\mu)} \right) ; t \geq 0$ the hypoelliptic Brownian motion on $B_\R$ driven by $W^{(\mu)}$, because its infinitesimal generator satisfies the (parabolic) H\"ormander condition: When taking the vector fields $\left( X_1, \dots, X_n \right)$ and their iterated Lie brackets with the family $\left\{ X_0, X_1, \dots, X_n \right\}$, one generates indeed all of $\bfrak_\R = T_e B_\R$.

Although we will not make use of this fact, it is reassuring to know that it has a smooth transition kernel.
\end{rmk}

Because the group $B_\R$ is a matrix group, we can consider that $B_\R \subset M_p(\R)$, for a certain $p \in \N$. Its Lie algebra $\bfrak_\R$ is also a subset of $M_p(\R)$. The left-invariant derivations can be expressed in coordinates using the usual differential calculus. Indeed, if we denote the first and second order differentials of $f \in \Cc^\infty(B_\R)$ at $b \in B_\R$ by:
$$ Df(b): M_p(\R) \longrightarrow \R$$
and
$$ D^2f(b): M_p(\R) \times M_p(\R) \longrightarrow \R$$
then, it is straightforward to check that for $f \in \Cc^{\infty}(B_\R)$, $b \in B_\R$ and $(X,Y) \in \bfrak_\R^2$:
\begin{align}
\label{eqn:link_lie_diff1}
X f(b) & = Df(b).(bX)
\end{align}

\begin{align}
\label{eqn:link_lie_diff2}
X Y f(b) & = Df(b).(bXY) + D^2 f(b).(bX).(bY)
\end{align}

Using the two conventions for stochastic integration: the Stratonovitch convention denoted by $\circ$ or the Ito convention denoted the usual way.

\begin{lemma}
The process $\left( B_t = B_t^\theta(W^{(\mu)}), t\geq 0 \right)$ solves the SDE written in matrix form:
\begin{align}
\label{lbl:process_B_sde}
dB_t & = B_t \circ \left( dW^{(\mu)}_t + \sum_{\alpha \in \Delta} \half \langle \alpha, \alpha \rangle f_\alpha dt \right) \\
     & = B_t \left( dW^{(\mu)}_t + \sum_{r=1}^n \half X_r^2 dt + \sum_{\alpha \in \Delta} \half \langle \alpha, \alpha \rangle f_\alpha dt \right)
\end{align}
with initial condition the identity element $id$.
\end{lemma}
\begin{proof}
If we assume the Stratonovitch differentiation convention, the chain rule is left unchanged and we can use equation \ref{lbl:process_B_ode} to obtain, after conjugation by $\theta$:
$$ dB_t = B_t \circ \left( dW^{(\mu)}_t + \sum_{\alpha \in \Delta} \half \langle \alpha, \alpha \rangle f_\alpha dt \right) $$
The SDE in Ito's convention needs to be established in coordinates using the fact that for two real semi-martingales $Y$ and $Z$:
$$ Y_t \circ Z_t = Y_t dZ_t + \half d\langle Y, Z \rangle_t$$
where $\langle Y, Z \rangle_t$ is their bracket. We fix two indices $1 \leq i,j \leq p$ and view the process $(B_t; t \geq 0)$ as taking its values in $M_p(\R)$. For shorter notation, introduce the $\bfrak_\R$-valued process given by:
$$ K_t = W^{(\mu)}_t + \sum_{\alpha \in \Delta} \half \langle \alpha, \alpha \rangle t f_\alpha$$
Therefore, the coefficient $(B_t)_{ij}$ satisfies:
\begin{align*}
  & d(B_t)_{ij}\\
= & \left( B_t \circ dK_t \right)_{ij}\\
= & \sum_{k=1}^p (B_t)_{ik} \circ (dK_t)_{kj}\\
= & \sum_{k=1}^p (B_t)_{ik} (dK_t)_{kj} + \half \sum_{k=1}^p d \langle (B)_{ik}, (K)_{kj} \rangle_t\\
= & (B_t dK_t)_{ij} + \half \sum_{k=1}^p d \langle (B)_{ik}, (K)_{kj} \rangle_t
\end{align*}
Now, recall that there is an Brownian motion $\left( \beta^1, \dots, \beta^n \right)$ such that:
$$\forall t\geq 0, W_t = \sum_{r=1}^n \beta_t^r X_r$$
Then, when considering the Doob-Meyer decomposition of the coefficient $(B_t)_{ik}$, the local martingale part is equal to:
$$(B_t dW_t )_{ik} = \sum_{r=1}^n (B_t X_r)_{ik} d\beta^r_t$$
The local martingale part in $(K_t)_{kj}$ is:
$$ (dW_t)_{kj} = \sum_{r=1}^n (X_r)_{kj} \beta^r_t$$
Because the bracket between two semi-martingales is the bracket of their local martingale parts, we have:
$$ d \langle (B)_{ik}, (K)_{kj} \rangle_t = \sum_{r=1}^n (B_t X_r)_{ik} (X_r)_{kj} dt$$
Therefore:
$$ d(B_t)_{ij} = (B_t dK_t)_{ij} + \half \sum_{i=1}^n (B_t X_r X_r)_{ij} dt $$
Hence the result.
\end{proof}

Now let us establish in the same fashion an Ito formula for the process $\left( B_t; t \geq 0 \right)$:
\begin{lemma}[Ito formula for $B_t$]
 For $f \in \Cc^{\infty}(B_\R)$:
$$ \forall t \geq 0, f(B_t) = f(id) + \int_0^t Df(B_s) dB_s + \half \int_0^t \sum_{i=1}^n D^2f(B_s).(B_s X_r).(B_s X_r) ds$$
\end{lemma}
\begin{proof}
Again, we view $b \in B_\R$ as a matrix $b = \left(b_{ij}\right)_{1 \leq i,j \leq p} \in M_p(\R)$. In coordinates, $f$ is a function of the $b_{ij}$ variables. Thanks to the chain rule, while using the Stratonovitch convention, we know that:
$$ df(B_t) = Df(B_t) \circ dB_t$$
Therefore:
\begin{align*}
  & df(B_t)\\
= & \sum_{i,j=1}^p \frac{\partial f}{\partial b_{ij}}(B_t) \circ d(B_t)_{ij}\\
= & \sum_{i,j=1}^p \frac{\partial f}{\partial b_{ij}}(B_t) d(B_t)_{ij} + \half \sum_{i,j,k,l=1}^p \frac{\partial^2 f}{\partial b_{ij}\partial b_{kl}} d\langle (B)_{ij}, (B)_{kl} \rangle_t
\end{align*}
Now, when considering the Doob-Meyer decomposition of the coefficient $(B_t)_{ij}$, the local martingale part is equal to:
$$(B_t dW_t )_{ij} = \sum_{r=1}^n (B_t X_r)_{ij} d\beta^r_t$$
Then:
$$ d\langle (B)_{ij}, (B)_{kl} \rangle_t = \sum_{r=1} (B_t X_r)_{ij} (B_t X_r)_{kl} dt$$
This yields the result.
\end{proof}

\begin{proof}[Proof of proposition \ref{proposition:inf_generator}]
Fix $b \in B_\R$. By the It\^o formula applied to $f \in \Cc^{\infty}(B_\R)$:
\begin{align*}
  & \lim_{s \rightarrow 0} \frac{ \E\left( f(B_{t+s}) -f(B_t) | \Fc_t^B, B_t = b \right)}{s}\\
= & Df(b).\left( b\mu + \half b \sum_{i=1}^n X_i^2 + \half b \sum_{\alpha \in \Delta} \langle \alpha, \alpha \rangle f_\alpha \right) + \half \sum_{r=1}^n D^2f(b).(b X_r).(b X_r)
\end{align*}
Moreover, because of equation \ref{eqn:link_lie_diff2}, we have:
$$ \Delta_\afrak f(b) = Df(b).\left( \sum_{r=1}^n X_r^2 \right) + \sum_{r=1}^n D^2f(b).(bX_r).(bX_r)$$
Hence, using equation \ref{eqn:link_lie_diff1}:
\begin{align*}
  & \lim_{s \rightarrow 0} \frac{ \E\left( f(B_{t+s}) -f(B_t) | \Fc_t^B, B_t = b \right)}{s}\\
= & \half \Delta_\afrak f(b) + \mu f(b) + \sum_{\alpha \in \Delta} \half \langle \alpha, \alpha \rangle f_\alpha f(b)
\end{align*}
This is the announced result.
\end{proof}

\subsection{Harmonic functions for an invariant process}

The classical approach to the Whittaker functions (cf. \cite{bib:Hashizume82} for instance) is to look at them as eigenfunctions of the quantum Toda Hamiltonian, a Schr\"odinger operator on $\afrak \approx \R^n$:
\index{$H$: Quantum Toda Hamiltonian}
\begin{align*}
 H & = \half \Delta - \half \sum_{\alpha} \langle \alpha, \alpha \rangle e^{-\alpha(x)}
\end{align*}
They satisfy:
$$ H \psi_\mu = \half \langle \mu, \mu \rangle \psi_\mu$$

A way to turn the problem into an invariant problem is to look at the Whittaker functions as harmonic functions for an invariant process. Let $\chi^-: N \rightarrow \C$ be the principal character on $N$ defined by:
$$ \forall t \in \C, \forall \alpha \in \Delta, \chi^-( e^{t f_\alpha} ) = t$$
Then:
\begin{lemma}
\label{lemma:equivalence_toda_harmonicity}
A function $\psi_\mu: \afrak \rightarrow \C$ on $\afrak$ solves:
$$ H \psi_\mu = \half \langle \mu, \mu \rangle \psi_\mu$$
if and only if the function $\Phi_\mu: B_\R \rightarrow \C$ defined by:
$$ \Phi_\mu(n e^x) := \exp\left( -\chi^-(n) \right) \psi_\mu(x) e^{-\langle \mu, x\rangle}$$
is harmonic for $\Dc^{(\mu)}$ (i.e $\Dc^{(\mu)} \Phi_\mu = 0$).
\end{lemma}
\begin{proof}
Notice that for $t \in \R$ and for $\alpha \in \Delta$:
\begin{align*}
\Phi_\mu(n e^x e^{t f_\alpha}) & = \Phi_\mu\left(n \exp(t e^{-\alpha(x)} f_\alpha) e^x \right)\\
                               & = \exp\left( - t e^{-\alpha(x)} \right) \Phi_\mu(n e^x)
\end{align*}
Hence:
\begin{align*}
f_\alpha \Phi_\mu(n e^x) & = \frac{d}{dt}\left( \Phi_\mu(n e^x e^{t f_\alpha}) \right)_{|t=0}\\
                         & = - e^{-\alpha(x)} \Phi_\mu(n e^x )
\end{align*}
Therefore, writing $\varphi_\mu = \psi_\mu e^{-\langle \mu, .\rangle}$, we have the following succession of equivalent statements:
\begin{align*}
                & \Dc^{(\mu)} \Phi_\mu = 0\\
\Leftrightarrow & \frac{1}{2} (\Delta \varphi_\mu) e^{-\chi^-} + \langle\mu, \nabla \varphi_\mu\rangle e^{-\chi^-} + \half \sum_{\alpha \in \Delta} \langle \alpha, \alpha \rangle f_\alpha \Phi_\mu = 0\\
\Leftrightarrow & \frac{1}{2} (\Delta \varphi_\mu) + \langle\mu, \nabla \varphi_\mu\rangle - \half \sum_{\alpha \in \Delta} \langle \alpha, \alpha \rangle e^{-\alpha(x)} \varphi_\mu = 0\\
\Leftrightarrow & H \psi_\mu = \half \langle \mu, \mu \rangle \psi_\mu
\end{align*}
\end{proof}

\subsection{Probabilistic integral representations on the boundary}
Because it is a harmonic function, we should be able to represent $\Phi_\mu$ as the integral of a function over the 'boundary' of $B$. Here however, the notion of boundary needed is not the topological one. Furstenberg developed such a notion and a very good account of the theory of boundaries on Lie groups is explained in \cite{bib:Bab02} in the case of random walks. The continuous case is, in a way, simpler.

For boundary, one has to consider a space with a $B_\R$-action and a natural invariant measure $\nu$. Restricting ourselves to the case where $\mu \in C$, we can see that the $N_\R$-part of the process $\left( B_t^\theta(W^{(\mu)}); t \geq 0 \right)$ (equation \ref{eqn:process_B_theta_explicit}) is:
\begin{align}
\label{eqn:process_N_theta_explicit}
  & N_t^\theta(W^{(\mu)})\\
= & \sum_{k \geq 0} \sum_{ i_1, \dots, i_k } \int_{ t \geq t_k \geq \dots \geq t_1 \geq 0} e^{ -\alpha_{i_1}(W^{(\mu)}_{t_1}) \dots -\alpha_{i_k}(W^{(\mu)}_{t_k}) } 
\frac{|| \alpha_{i_1} ||^2}{2} \dots \frac{|| \alpha_{i_k} ||^2}{2} f_{i_1} \dots f_{i_k} dt_1 \dots dt_k
\end{align} 
and converges in $N_\R$ when $t \rightarrow \infty$. Therefore, a natural choice for a boundary of $B_\R$ is simply $N_\R$ and the invariant measure $\nu$ is the law of $N_\infty^\theta(X^{(\mu)})$ when $\mu \in C$.

The Borel subgroup $B_\R$ acts on $N_\R$ as:
$$ \forall b = na \in B_\R, \forall n' \in N_\R, (na) \cdot n' = n a n' a^{-1}$$
Hence for any bounded function $\varphi$ on $N_\R$, we obtain a harmonic function for $\Dc^{(\mu)}$ simply by considering:
\begin{align}
\label{eqn:poisson_integral_phi}
n a & \mapsto \E\left( \varphi(n a N_\infty^\theta(W^{(\mu)}) a^{-1} ) \right) 
\end{align}
The previous subsection tells us to take $\varphi$ as a character of the unipotent subgroup $N_\R$, in order to obtain Whittaker functions. As we will see in theorem \ref{thm:N_infty_law}, $\nu$ has support in totally positive matrices, hence their importance.

In short, the key object that we need to understand is $\nu$, the law of $N_\infty^\theta(W^{(\mu)})$. This will be the subject of the two next sections, before resuming the study of random crystals.

\section{A review on a relationship proven by Matsumoto and Yor}
\label{section:review_matsumoto_yor}
This section contains a version of the Matsumoto and Yor relationship between Brownian motions with opposite drifts $\cite{bib:MY01}$, which itself is based on many previous works related to exponential functionals of BM.

\begin{thm}[ Matsumoto-Yor \cite{bib:MY01}, theorem 2.2 ]
\label{thm:my_bm_with_opposite_drifts_1}
Let $B^{(\mu)}$ be a Brownian motion on a Euclidian vector space $V \approx \R^n$ with drift $\mu$ and $\beta$ a linear form on $V$ such that $\beta(\mu)>0$. Denote by $s_\beta$ the hyperplane reflection with respect to $\ker \beta$ and by $\Q^y$ the measure of BM conditionally to its exponential functional being equal to $y>0$:
$$\Q^y := \P\left( \ | \int_0^\infty e^{-\beta( B_s^{(\mu)}) }ds = y \right)$$
and:
$$ \beta^\vee = \frac{2 \beta}{ \langle \beta, \beta \rangle }$$
Then:\\
$$ \hat{B}_t^{ (s_\beta \mu)} = B_t^{(\mu)} + \log\left(1-\frac{1}{y}\int_0^t e^{-\beta( B_s^{(\mu)}) }ds \right) \beta^{\vee} $$
is a $\Q^y$-BM with drift $s_\beta \mu = \mu - \beta(\mu) \beta^\vee$.
\end{thm}
This theorem has a dual version that characterises the reciprocal transform of Brownian motion as a Brownian motion conditioned with respect to its exponential functional:
\begin{thm}[ Matsumoto-Yor \cite{bib:MY01}, theorem 2.1]
\label{thm:my_bm_with_opposite_drifts_2}
Let $W^{(s_\beta \mu)}$ be a BM on $V \approx \R^n$ with drift $s_\beta \mu$, $\beta(\mu)>0$ and:
$$ X_t = W_t^{(s_\beta \mu)} + \log\left(1+\frac{1}{y}\int_0^t e^{ -\beta( W_s^{(s_\beta \mu)} ) }ds\right) \beta^{\vee} $$
Then $X$ is a BM with drift $\mu$, $B^{(\mu)}$, conditionned to $\int_0^\infty e^{- \beta(B_s^{(\mu)})} ds = y$.\\
If moreover, we pick $y$ as random with $ y \stackrel{\mathcal{L}}{=} \frac{2}{\langle \beta, \beta\rangle \gamma_{ \langle \beta^\vee, \mu \rangle } }$ independent from $W$ then $X$ is a Brownian motion with drift $\mu$.
\end{thm}

Notice that compared to the original formulation, we used a multidimensional setting. The change of sign is simply replaced by a hyperplane reflection. For completeness, we provide in this section a proof and an overview of the tools involved, while not using any group theory. In the following section, we start by explaining the link to our problem, interpreting the previous theorem as the $SL_2$ case of theorem \ref{thm:conditional_representation}.\\

\subsection{Exponential functionals of BM}
First let us start by proving theorem \ref{thm:my_bm_with_opposite_drifts_1} using known results on exponential functionals of BM. An important ingredient is Dufresne's identity in law:
\begin{proposition}[ Dufresne \cite{bib:Dufresne90} ]
 \label{lbl:dufresne_identity}
 If $W^{(\mu)}$ is a one dimensional Brownian motion with drift $\mu>0$, then:
$$ \int_0^\infty e^{-2 W_s^{(\mu)} }ds \stackrel{\mathcal{L}}{=} \frac{1}{2\gamma_\mu} $$
\end{proposition}
\begin{proof}[Quick proof]
By time inversion, for any fixed $t>0$, the random variable $\int_0^t e^{-2 W_s^{(\mu)} }ds$ has the same law as $e^{-2 W_t^{(\mu)}} \int_0^t e^{ 2 W_s^{(\mu)} } ds$. Let $\left(Z_t, t>0\right)$ be given by:
$$ e^{-Z_t} := e^{-2 W_t^{(\mu)}} \int_0^t e^{ 2 W_s^{(\mu)} } ds$$
And, by Ito's lemma, $Z_t$ can be easily checked to be a diffusion process since it satisfies for $t>0$ the SDE:
$$ d Z_t = 2 dW^{(\mu)}_t - e^{Z_t}dt$$
Hence it has as infinitesimal generator:
$$ \Lc = 2 \partial_z^2 + (2 \mu - e^z) \partial_z $$
The sequence $Z_t$ converges in law to a unique invariant measure because $e^{-Z_t}$ has the same distribution as $\int_0^t e^{-2 W_s^{(\mu)} }ds$, which converges almost surely. This invariant measure will be the law of $\int_0^\infty e^{-2 W_s^{(\mu)} }ds$. Therefore, all we need to do is to prove that the distribution of $\log 2 \gamma_\mu$ is an invariant measure for $Z_t$. This is done easily by checking that the adjoint of $\Lc$ annihilates the density $p(z)$ of the law $\log 2\gamma_\mu$.\\
We have:
$$ p(z) = \frac{1}{\Gamma(\mu) 2^\mu} \exp\left( \mu z - \half e^z \right)$$
Applying $\Lc^*$, the adjoint of $\Lc$:
$$ \Lc^* = 2 \partial_z^2 - \partial_z ( 2 \mu - e^z) $$
We get:
\begin{align*}
 \Lc^* p(z) & = 2 \partial_z^2 p(z) - 2 \partial_z \left( (\mu - \half e^z) p(z) \right)\\
& = 2 \partial_z^2 p(z) - 2 \partial_z^2 p(z)\\
& = 0
\end{align*}
\end{proof}

Now, let $B_t^{(\mu)}$ be an $n$-dimensional BM with drift $\mu$, $\Fc^B_t$ its natural filtration, $\beta$ a linear form such that $\beta(\mu)>0$ and:
$$N_t := \int_0^t \exp(-\beta( B_s^{(\mu)}) )ds $$
$$N_{\infty} = \lim_{t \rightarrow \infty} N_t $$
\subsubsection*{Law of $N_\infty$:}

The law of $N_\infty$ comes as a simple application:
\begin{corollary}
\label{corollary:N_infty_law_1d}
$$ N_\infty :=  \int_0^\infty \exp\left(-\beta( B_s^{(\mu)}) \right)ds \stackrel{\mathcal{L}}{=} \frac{2}{||\beta||^2 \gamma_{\langle \beta^\vee, \mu \rangle } }$$
Therefore, the density is:
$$ \P\left( N_\infty \in dn \right) = \frac{1}{\Gamma(\langle \beta^\vee, \mu \rangle)} n^{-\langle \beta^\vee, \mu \rangle} \exp\left( -\frac{2}{||\beta||^2 n} \right) \left(\frac{2}{||\beta||^2 }\right)^{\langle \beta^\vee, \mu \rangle} \frac{dn}{n}$$
\end{corollary}
\begin{proof}
Define the real Brownian motion $W$ by $\beta(B_t) = ||\beta|| W_t$ for $t \geq 0$. Then:
\begin{align*}
N_{\infty}
& = \int_0^\infty e^{- \beta( B_t^{(\mu)} ) }dt\\
& = \int_0^\infty e^{- ||\beta|| W_t - \beta(\mu)t }dt\\
& \stackrel{\mathcal{L}}{=} \int_0^\infty e^{- \frac{||\beta||}{\sqrt{c}} W_{tc} - \frac{\beta(\mu)tc}{c} }dt \textrm{ using Brownian scaling for } c>0\\
& = \frac{1}{c} \int_0^\infty e^{- \frac{||\beta||}{\sqrt{c}} W_{u} - \frac{\beta(\mu)u}{c} }du \textrm{ using change of variable } u = tc\\
& = \frac{4}{||\beta||^2} \int_0^\infty e^{- 2 W_{u} - 4\frac{\beta(\mu)u}{||\beta||^2} }du \textrm{ by choosing } c = \frac{ ||\beta||^2}{4}\\
& = \frac{4}{||\beta||^2} \int_0^\infty e^{- 2 (W_{u} + \langle \beta^\vee, \mu \rangle u ) }du
\end{align*}
The result holds using Dufresne's identity in law. As for the density, for all $f\geq0$ bounded measurable function, and while writing $\nu = \langle \beta^\vee, \mu \rangle$, we have:
\begin{align*}
\E\left( f(N_\infty) \right) & = \int_0^\infty f( \frac{1}{ ||\beta||^2 t/2}) \frac{ e^{-t} t^{\nu} }{\Gamma(\nu)} \frac{dt}{t}\\
& = \frac{1}{\Gamma(\nu)} \int_0^\infty f(n) \exp\left( -\frac{2}{||\beta||^2 n} \right) \left(\frac{2}{||\beta||^2 n} \right)^{\nu} \frac{2}{||\beta||^2} \frac{dn}{n} \textrm{ by letting } n = \frac{2}{||\beta||^2 t}\\
& = \frac{1}{\Gamma(\nu)} \int_0^\infty f(n) \exp\left( -\frac{2}{||\beta||^2 n} \right) \left(\frac{2}{||\beta||^2 } \right)^{\nu} n^{-\nu} \frac{dn}{n}
\end{align*}
\end{proof}

\subsubsection*{Initial enlargement of the filtration $\Fc^B$ using the random variable $N_\infty$}
In order to compute the law of $N_\infty$ conditionnally on $\Fc^B_t$, the following decomposition is essential:
\begin{align}
\label{eqn:N_infty_decomposition_1d}
N_{\infty} = \int_0^t e^{- \beta( B_s^{(\mu)} ) }ds + e^{- \beta( B_t^{(\mu)} ) } \tilde{N}_\infty 
\end{align}
with $\tilde{N}_\infty$ is a copy of $N_\infty$ independent from $\Fc^B_t$. Indeed:
\begin{align*}
N_{\infty}
& = \int_0^\infty e^{- \beta( B_s^{(\mu)} ) }ds\\
& = \int_0^t e^{- \beta( B_s^{(\mu)} ) }ds + e^{- \beta( B_t^{(\mu)} ) } \int_t^\infty e^{- \beta( B_s^{(\mu)} - B_t^{(\mu)}) }ds\\
& = \int_0^t e^{- \beta( B_s^{(\mu)} ) }ds + e^{- \beta( B_t^{(\mu)} ) } \tilde{N}_\infty
\end{align*}

For readability purposes, and because it is not necessary to invoke general filtration enlargement theorems, we will give a complete proof using the usual tools. Indeed, as proved before the law of $N_\infty$ has a (smooth) density $\frac{d\P( N_\infty \leq y)}{dy}$ with respect to the Lebesgue measure, making possible the following computations.

Let $\Q^y = \P( . | N_\infty = y)$ be a regular version of the conditionnal probability, $f \geq 0$ bounded measurable function, and $A \in \Fc^B_t$. We have:
\begin{eqnarray*}
 & & \int dy f(y) \Q^y(A) \frac{ d\P( N_\infty \leq y)}{dy} \\
& = & \mathbb{E}( \Q^{N_\infty}(A) f(N_\infty) )\\
& = & \mathbb{E}( \textbf{1}_A f(N_\infty) )\\
& = & \mathbb{E}( \textbf{1}_A \mathbb{E}( f(N_\infty) | \Fc^B_t) )\\
& = & \mathbb{E}( \textbf{1}_A \int f(n) d\P(N_\infty \in dy | \Fc^B_t) )\\
& = & \int dy f(y) \mathbb{E}( \textbf{1}_A \frac{ d\P(N_\infty \leq y | \Fc^B_t)}{dy} ) \textrm{ (Fubini) } 
\end{eqnarray*}
Then:
$$ \Q^y(A) = \E\left( \textbf{1}_A \frac{ d\P(N_\infty \leq y | \Fc^B_t)}{dy}/\frac{ d\P(N_\infty \leq y)}{dy} \right) $$
We conclude that $\Q^y$ is absolutely continuous with respect to $\P$ and that the likelihood/Radon-Nikodym derivative on $\Fc^B_t$ is given by the $\P$-martingale:
\begin{align*}
q( B_t^{(\mu)}, N_t, n) := & \frac{d\Q^n}{d\P}_{|\Fc^B_t} \\
                         = & \frac{ \frac{ d\P(N_\infty \leq n | \Fc^B_t)}{dn} }{\frac{ d\P(N_\infty \leq n)}{dn} }
\end{align*}
Using the expression for the density of $N_\infty$ from corollary \ref{corollary:N_infty_law_1d}, we get:
\begin{align*}
  & q( B_t^{(\mu)}, N_t, n)\\
= & \frac{ \frac{ d\P(\tilde{N}_\infty \leq (n-N_t) e^{\beta( B_t^{(\mu)})} | \Fc^B_t)}{dn} }
         { \frac{ d\P(N_\infty \leq dn)}{dn} } \\
= & e^{\beta( B_t^{(\mu)}) } \frac{ \exp\left( - \frac{2}{||\beta||^2(n-N_t)e^{\beta( B_t^{(\mu)}) } } \right) \left( (n-N_t)e^{\beta( B_t^{(\mu)})} \right)^{-(1+\langle \beta^\vee, \mu \rangle)} }
                                  { \exp\left( - \frac{2}{||\beta||^2 n} \right) n^{-(1+\langle \beta^\vee, \mu \rangle)} }
\end{align*}
Hence:
$$ \log q( B_t^{(\mu)}, N_t, n) = A_t - \langle \beta^\vee, \mu \rangle \beta(B_t^{(\mu)}) - \frac{2 e^{-\beta( B_t^{(\mu)}) }}{||\beta||^2(n-N_t)}$$
where $A_t$ has a zero quadratic variation. Therefore, the semimartingale bracket between $\beta(B^{(\mu)})$ and $\log q$ is:
\begin{align*}
  & \langle \beta(B^{(\mu)}), \log q \rangle_t\\
= & \langle \beta(B^{(\mu)}), - \langle \beta^\vee, \mu \rangle \beta(B_.^{(\mu)}) - \frac{2 e^{-\beta( B_.^{(\mu)}) }}{||\beta||^2(n-N_.)} \rangle_t\\
= & - \langle \beta^\vee, \mu \rangle||\beta||^2 - \frac{2}{||\beta||^2} \langle \beta(B), \frac{e^{-\beta( B_.^{(\mu)}) }}{n-N_.} \rangle_t\\
= & - 2\beta(\mu) + 2 \int_0^t \frac{e^{-\beta( B_s^{(\mu)}) }}{n-N_s} ds\\
= & - 2\beta(\mu) + 2 \int_0^t \frac{dN_s}{n-N_s}\\
= & - 2\beta(\mu) - 2 \log\left( 1 - \frac{N_t}{n} \right)
\end{align*}
In the end, using Girsanov theorem (\cite{bib:RevuzYor} Chapter VIII, theorem 1.4):
$$ \hat{B}_t = B_t - \frac{\beta}{||\beta||^2} \langle \beta(B^{(\mu)}), \log q \rangle_t $$
is a $\Q^n$ Brownian motion, hence the proof of theorem \ref{thm:my_bm_with_opposite_drifts_1}.

\subsection{Inversion}
A natural question is can we recover $\hat{B}$ or $B$ from the other. The answer is yes and the argument is again due to Matsumoto and Yor \cite{bib:MY01}. This can be restated as a disintegration formula for the Wiener measure that will give, later, a probabilistic interpretation of the group path transforms.\\

The proof of theorem \ref{thm:my_bm_with_opposite_drifts_2} follows from \ref{thm:my_bm_with_opposite_drifts_1} and an inversion lemma.
\begin{lemma}[Inversion lemma]
\label{lemma:inversion_lemma}
 Let $x$ and $y$ be $V$ valued paths i.e functions on $\R_+$. Then
$$ (1) \ x(t) = y(t) + \log( 1 + \frac{1}{n} \int_0^t e^{-\beta(y)}) \beta^{\vee}$$
if and only if
$$ (2) \left\{ \begin{array}{ll}
\forall t>0, \int_0^t e^{-\beta(x)} < n \\
y(t) = x(t) + \log( 1 - \frac{1}{n} \int_0^t e^{-\beta(x)}) \beta^{\vee}                
               \end{array} \right.
$$
Moreover, in any case:
$$ (3) \ ( 1 + \frac{1}{n} \int_0^t e^{-\beta(y)})( 1 - \frac{1}{n} \int_0^t e^{-\beta(x)}) = 1$$
and if $\int_0^\infty e^{ -\beta(y) } = \infty$ then $\int_0^\infty e^{-\beta(x)}=n$
\end{lemma}
\begin{proof}
 It is immediate to see that (1) and (2) are simultaniously true if and only if (3) is true. Then all we need to prove is $(1) \Rightarrow (3)$ and $(2) \Rightarrow (3)$
\begin{eqnarray*}
 (1) & \Rightarrow & e^{-\beta(x)} = \frac{ e^{-\beta(y)} }{ ( 1 + \frac{1}{n} \int_0^t e^{-\beta(y)})^2 }\\
& \Rightarrow & \frac{1}{n}\int_0^t e^{-\beta(x)} = [ \frac{-1}{ 1 + \frac{1}{n} \int_0^t e^{-\beta(y)} } ]_0^t\\
& \Rightarrow & \frac{1}{ 1 + \frac{1}{n} \int_0^t e^{-\beta(y)} } = 1-\frac{1}{n}\int_0^t e^{-\beta(x)}\\
& \Rightarrow & (3)
\end{eqnarray*}
and
\begin{eqnarray*}
 (2) & \Rightarrow & e^{-\beta(y)} = \frac{ e^{-\beta(x)} }{ ( 1 - \frac{1}{n} \int_0^t e^{-\beta(x)})^2 }\\
& \Rightarrow & \frac{1}{n}\int_0^t e^{-\beta(y)} = [ \frac{1}{ 1 - \frac{1}{n} \int_0^t e^{-\beta(x)} } ]_0^t\\
& \Rightarrow & \frac{1}{ 1 - \frac{1}{n} \int_0^t e^{-\beta(x)} } = 1+\frac{1}{n}\int_0^t e^{-\beta(y)}\\
& \Rightarrow & (3)
\end{eqnarray*}
Then (3) gives the convergence of $\int_0^t e^{-\beta(x)}$ to $n$, right-away.
\end{proof}
Now we are ready to prove theorem \ref{thm:my_bm_with_opposite_drifts_2}:
\begin{proof}[Proof of theorem \ref{thm:my_bm_with_opposite_drifts_2}:]
Consider a Brownian motion $B^{(\mu)}$ with drift $\mu$ conditionned to $\int_0^\infty e^{- \beta(B_s^{(\mu)})} ds = n$. By the previous filtration enlargement argument, there is $\hat{B}$ a BM in the enlarged filtration, such that:
$$ \hat{B}_t^{(s_\beta \mu)} = B^{(\mu)}_t  + \log\left(1-\frac{1}{n}\int_0^t e^{-\beta( B_s^{(\mu)} ) }ds \right) \beta^{\vee}$$
Using the inversion lemma:
$$ B^{(\mu)}_t = \hat{B}_t^{(s_\beta \mu)} + \log\left(1+\frac{1}{n}\int_0^t e^{ -\beta( \hat{B}_s^{(s_\beta \mu)} ) }ds\right) \beta^{\vee} $$
Then the following equalities in law between processes follow:\\
\begin{align*}
& \left( B^{(\mu)}_t; t \geq 0 | \int_0^\infty e^{- \beta(B_s^{(\mu)})} ds = n \right)\\
& = \ \hat{B}_t^{(s_\beta \mu)} + \log\left(1+\frac{1}{n}\int_0^t e^{-\beta( \hat{B}_s^{(s_\beta \mu)} ) }ds \right) \beta^{\vee} ; t \geq 0\\
& \stackrel{\mathcal{L}}{=} \ W_t^{(s_\beta \mu)} + \log\left(1+\frac{1}{n}\int_0^t e^{-\beta( W_s^{(s_\beta \mu)} ) }ds \right) \beta^{\vee} ; t \geq 0\\
& = \ X_t; t \geq 0
\end{align*}
This ends the proof of the first fact.

The second fact is just a consequence of knowing the law of $\int_0^\infty e^{- \beta(B_s^{(\mu)})} ds$ and usual disintegration formula given for $F$ continuous functional on the sample space by:
$$ \E\left( F( B^{(\mu)}_. ) \right) =  \E\left( \E\left( F( B^{(\mu)}_.  ) | \int_0^\infty e^{ -\beta( B_s^{(\mu)} ) }ds \right) \right)$$
\end{proof}

\section{Conditioned Brownian motion and invariant measure}
\label{section:conditionning_approach}

Now we will go back to our general group-theoretic setting and resume the study of the hypoelliptic Brownian motion on the solvable group $B$.

We prove a theorem that represents a certain transform of Brownian motion as having the same law as a Brownian motion $X^{(\mu)}$ conditioned to have $N_\infty^\theta(X^{(\mu)})$ fixed. As a consequence, we obtain an explicit expression for the invariant measure given by the law of $N_\infty^\theta(X^{(\mu)})$. Later, this will allow us to condition with respect to an apropriate 'exit' law, giving Whittaker functions. Finally, because this measure is defined in a coordinate chart indexed by a reduced word ${\bf i} \in R(w_0)$, it contains hidden identities we will discuss.

This approach owes a lot to Baudoin and O'Connell (\cite{bib:BOC09}) in spirit, but is quite different in essence. Indeed, that paper considered a conditionning of a Brownian motion $X^{(\mu)}$ with respect to simple integrals, whereas our representation theoretic approach makes it more natural to condition a Brownian path to have all its Lusztig parameters fixed. Thanks to theorem \ref{thm:inversion_lemma_lusztig}, we know that it is nothing more than conditionning with respect to $N_\infty^\theta(X^{(\mu)})$. The random variable $N_\infty^\theta(X^{(\mu)})$ not only contains the simple integrals $\int_0^\infty e^{-\alpha(X^{(\mu)})}$, but also interated ones. 

The subtlety at this level is that because Whittaker functions will be built out of an $N$-character applied to the random variable $N_\infty^\theta(X^{(\mu)})$, only the law of simple integrals will matter in the end.

In term of the group path transforms described in the first chapter, the result from Matsumoto and Yor ( theorem \ref{thm:my_bm_with_opposite_drifts_2} ) can be reformulated as:
\begin{thm}[$SL_2$ conditional representation]
\label{thm:sl2_conditional_representation}
If $g=x_\alpha(\xi)$, $\xi>0$, $W^{(s_{\alpha} \mu)}$ a BM on $\afrak$ with drift $s_{\alpha} \mu$ such that $\alpha(\mu)>0$ then $X := T_g(W^{(s_{\alpha} \mu)})$ is a BM with drift $\mu$ conditionned to $\int_0^\infty e^{ -\alpha( X_s ) }ds = \frac{1}{\xi}$.

If moreover we pick $\xi$ as random with $ \xi \stackrel{\mathcal{L}}{=} \frac{\langle \alpha, \alpha \rangle}{2} \gamma_{\langle \alpha^{\vee}, \mu \rangle }$ independent from $W$ then $X$ is a Brownian motion with drift $\mu$.
\end{thm}
It is very surprising and impressive to say that they fully worked out the $SL_2$ case without starting without any group-theoretic considerations. We will now state what seems like the natural extension of theorem \ref{thm:sl2_conditional_representation}.

\subsection{Conditional representation theorem}

From now on, fix a reduced word ${\bf i} \in R(w_0)$ of length $m = \ell(w_0)$ and call $(\beta_1, \dots, \beta_m)$ the associated positive roots enumeration.

\begin{thm}[Conditionnal representation of $T_g(W)$]
\label{thm:conditional_representation}
Let $g \in U_{>0}^{w_0}$, $\mu \in C$, $W$ a standard BM on $\afrak$.\\
Then $\Lambda^{x_0} := x_0 + T_{e^{\theta} g e^{-\theta}} \left( W^{(w_0 \mu)} \right)$ is distributed as a BM $X^{x_0, (\mu) }$
\begin{itemize}
 \item with drift $\mu$
 \item with initial position $x_0$
 \item conditioned to $N_\infty^\theta(X^{0, (\mu)}) = \Theta(g)$ where 
$$\Theta: U_{>0}^{w_0} \longrightarrow N_{>0}^{w_0}$$
 is the bijective function $\Theta(g) = [ g \bar{w}_0 ]_-$
\end{itemize}
Moreover, if we pick $g = x_{i_1}( t_1 ) \dots x_{i_m}( t_m ), m = \ell(w_0)$ being random with independent Lusztig parameters such that $t_j \stackrel{ \mathcal{L} }{=} \gamma_{ \langle\beta_{j}^\vee, \mu\rangle } $, then $\Lambda^{x_0}$ is a standard BM with drift $\mu$ starting at $x_0$.
\end{thm}
\begin{proof}
We can of course take $x_0 = 0$. Let $X = T_{e^{\theta} g e^{-\theta}}( W^{(w_0 \mu)})$, and we get:
\begin{eqnarray*}
X & = & T_{e^{\theta} g e^{-\theta}} W^{(w_0 \mu)}\\
& = & T_{ x_{i_1}( \frac{\langle \alpha_{i_1}, \alpha_{i_1} \rangle}{2} t_1 ) \dots 
          x_{i_m}( \frac{\langle \alpha_{i_m}, \alpha_{i_m} \rangle}{2} t_m ) } W^{(w_0 \mu)}\\
& = & T_{ x_{i_1}( \frac{\langle \alpha_{i_1}, \alpha_{i_1} \rangle}{2} t_1 )} \circ \dots \circ 
      T_{ x_{i_m}( \frac{\langle \alpha_{i_m}, \alpha_{i_m} \rangle}{2} t_m ) } W^{(w_0 \mu)}
\end{eqnarray*}
We apply inductively theorem \ref{thm:sl2_conditional_representation} with Lusztig parameters taken to follow the right laws, in order to get successive BM.

The end of proof follows from the deterministic formula proven in theorem \ref{thm:inversion_lemma_lusztig}, which we know to be also valid for an infinite time horizon (subsection \ref{subsection:infinite_T}): 
$$ N_\infty\left( T_{e^{\theta} g e^{-\theta}}\left( W^{(w_0 \mu)} \right) \right) = \Theta\left( e^{\theta} g e^{-\theta} \right)$$
Hence:
$$ N_\infty^\theta\left( X^{0, (\mu)} \right) = \Theta\left( g \right)$$
\end{proof}

\begin{rmk}
 Recall that in the case of $SL_2$, $\Theta\left( \begin{pmatrix} 1 & \xi\\ 0 & 1 \end{pmatrix} \right) = \begin{pmatrix} 1 & 0\\ \frac{1}{\xi} & 1 \end{pmatrix}$. This gives exactly theorem \ref{thm:sl2_conditional_representation}.
\end{rmk}

\subsection{Law of \texorpdfstring{$N_\infty(X^{(\mu)})$}{the N part}}

Dufresne identity in law (proposition \ref{lbl:dufresne_identity}) states that if $X^{(\mu)}$ is a one dimenstional BM with drift $\mu>0$ then $\int_0^\infty e^{-2 X^{(\mu)}_t} dt$ has the same law as $\frac{1}{2 \gamma_\mu}$. One can notice that in $A_1$ case, if we see this BM as living in $\mathfrak{a}$, computing the law of $N_\infty^\theta(X^{(\mu)})$ tantamounts to proving Drufesne's identity, since:
$$N_\infty^\theta(X^{(\mu)}) = \left( \begin{array}{cc} 1 & 0\\ 2\int_0^\infty  e^{-2 X^{(\mu)}_s}ds & 1 \end{array} \right)$$

\begin{rmk}
Only in the $A_1$ case, we take $\theta \neq 0$ and opt-out of the choice made in remark \ref{rmk:theta_ADE}. Indeed, the classical choice for the only root is $\alpha = 2$. Hence the factor $\frac{ \langle \alpha, \alpha \rangle }{2} = 2$ in the previous identity.
\end{rmk}

If however, we consider a general semi-simple group, with $X^{(\mu)}$ be an n-dimensional BM in $\mathfrak{a}$ and $\mu \in C$, the open Weyl chamber, an explicit formula for $N_\infty^\theta(X^{(\mu)})$ is a generalization of Dufresne identity and allows explicit representations of harmonic functions for the operator $\Dc^{(\mu)}$ (definition \ref{def:hypoelliptic_operator}). Thanks to the conditional representation theorem \ref{thm:conditional_representation}, we get it in fact with little effort.

\begin{thm}[Law of $N_\infty^\theta(X^{(\mu)})$]
\label{thm:N_infty_law}
If $X^{(\mu)}$ is a BM with drift $\mu \in C$, then $N_t^\theta(X^{(\mu)})$ converges almost surely inside the open cell $N_{>0}^{w_0}$ and 
$N_\infty^\theta(X^{(\mu)}) = \Theta\left( x_{i_1}( t_1 ) \dots x_{i_m}( t_m ) \right)$ where the Lusztig parameters $t_j$ are independent random variables with:
$$ t_j \stackrel{\mathcal{L}}{=} \gamma_{ \langle\beta_{j}^{\vee}, \mu\rangle } $$
\end{thm}
\begin{proof}
The condition $\mu \in C$ entails the convergence of iterated integrals in the explicit expression of $N_\infty(X)$. Moreover, in the same fashion as theorem \ref{thm:flow_B_total_positivity}, we have that:
$$ \forall w \in W, \forall \alpha \in \Delta, \Delta_{w \omega_\alpha, \omega_\alpha}\left( N_\infty(X^{(\mu)}) \right) > 0$$
Using the total positivity criterion given in \ref{thm:total_positivity_criterion}, we deduce that $N_\infty(X^{(\mu)}) \in N^{w_0}_{>0}$.

The law of $N_\infty(X^{(\mu)})$ comes directly from theorem \ref{thm:conditional_representation}. It is worth noting that the probability measure has a smooth density and charges the entire space $N^{w_0}_{>0}$ that is an open dense cell of $N_{\geq 0}$. Other cells are of smaller dimension and therefore of zero measure.
\end{proof}

Here, the law of $N_\infty(X^{(\mu)})$ can be seen as a meaningful generalization of Dufresne's identity which gives an inverse gamma shifted by a scalar factor $2$. In the group setting, we have then a natural notion of gamma law and inverse gamma law, where the map $\Theta$ plays the role of the inverse map.

\begin{definition}[Gamma law on $U$ and inverse gamma on $N$]
\label{def:group_gamma_law}
For $(\beta_1, \dots, \beta_m)$ being the positive roots enumeration associated to a reduced expression of $w_0 = s_{i_1} \dots s_{i_m}$, and $\mu \in C$ define $\Gamma_\mu$ to be the law of the positive (in the sense of total positivity) $U^{w_0}_{>0}$-valued random variable
$$ \Gamma_\mu \stackrel{\mathcal{L}}{=} x_{i_1}\left( \gamma_{ \langle\beta_{1}^{\vee}, \mu\rangle } \right) \dots x_{i_m}\left( \gamma_{ \langle\beta_{m}^{\vee}, \mu\rangle } \right) $$
Define the inverse gamma law on $N^{w_0}_{>0}$ as:
$$ D_\mu \stackrel{\mathcal{L}}{=} \Theta\left( \Gamma_\mu \right) $$
\end{definition}

Those laws are well defined, in the sense that the above expressions do not depend on the choice of a reduced expression. Indeed, theorem \ref{thm:N_infty_law} can be restated as:
\begin{align}
\label{eqn:N_infty_law_D_mu}
N_\infty^\theta(X^{(\mu)}) \stackrel{\mathcal{L}}{=} & \Theta\left( \Gamma_\mu \right)
			    \stackrel{\mathcal{L}}{=} D_\mu 
\end{align}
As such, it is obvious that $\Gamma_\mu$ and $D_\mu$ are unambiguously defined since the left-hand side does not depend on a choice of reduced expression for $w_0$. Furthermore, $D_\mu$ could be seen as a $N_{\geq 0}$-valued random variable, or even a $\left(G/B\right)_{\geq 0}$-valued random variable, since the lower dimensional cells have zero measure.

\begin{example}[$A_1$-type]
 For $G = SL_2$, $\afrak = \R$, $\alpha=2$, $\alpha^\vee=1$ and:
$$N_\infty^\theta(X^{(\mu)}) = \begin{pmatrix} 1 & 0\\ \int_0^\infty  2 e^{-2 X^{(\mu)}_s}ds & 1 \end{pmatrix} $$
We know that $\Theta\left( \begin{pmatrix} 1 & n \\ 0 & 1 \end{pmatrix} \right) = \begin{pmatrix} 1 & 0 \\ \frac{1}{n} & 1 \end{pmatrix}$, 
and as such:
$$ D_\mu = \begin{pmatrix} 1 & 0 \\ \frac{1}{\gamma_\mu} & 1 \end{pmatrix}$$
Exactly as announced, saying $D_\mu \stackrel{\mathcal{L}}{=} N_\infty^\theta(X^{(\mu)})$ recovers Dufresne identity in law.
\end{example}

\begin{example}[$A_2$-type]
 For $G = SL_3$, $\afrak = \left\{ x \in \R^3 \ | \ x_1 + x_2 + x_3 = 0 \right\}$. Consider a Brownian motion $X^{(\mu)}$ on $\afrak$ with drift $\mu$ in the Weyl chamber. The simple roots are $\alpha_1 = \begin{pmatrix} 1 \\ -1 \\ 0 \end{pmatrix} , \alpha_2 = \begin{pmatrix} 0 \\ 1 \\ -1 \end{pmatrix} $.
$$ N_\infty^\theta(X^{(\mu)}) = \begin{pmatrix}
 1                                                                                  & 0                                           & 0 \\
 \int_0^\infty e^{-\alpha_1(X_s^{(\mu)})}                                           & 1                                           & 0\\
 \int_0^\infty e^{-\alpha_1(X_s^{(\mu)}) }ds \int_0^s e^{-\alpha_2(X_u^{(\mu)}) }du & \int_0^\infty e^{-\alpha_2(X_s^{(\mu)}) }ds & 1 
 \end{pmatrix} $$
We choose the reduced word ${\bf i} = (1,2,1) \in R(w_0)$. If:
$$\left(t_1, t_2, t_3 \right) \stackrel{\Lc}{=} \left( 
\gamma_{\langle \alpha^\vee_1                , \mu \rangle},
\gamma_{\langle \alpha^\vee_1 + \alpha^\vee_2, \mu \rangle},
\gamma_{\langle \alpha^\vee_2                , \mu \rangle}
\right)$$
are independent gamma random variables with corresponding parameters, then equation (\ref{eqn:example_twist_A2}) and theorem \ref{thm:N_infty_law} tell us that:
\begin{align*}
N_\infty^\theta(X^{(\mu)}) & \stackrel{\Lc}{ = } y_1\left( \frac{1}{t_1 + t_3} \right) 
                                                 y_2\left( \frac{1 + \frac{t_1}{t_3} }{t_2} \right)
                                                 y_1\left( \frac{1}{t_1\left(1 + \frac{t_1}{t_3}\right)} \right)\\
& = \begin{pmatrix}
1                 & 0                             & 0 \\
\frac{1}{t_1}     & 1                             & 0\\
\frac{1}{t_1 t_2} & \frac{1+\frac{t_1}{t_3}}{t_2} & 1 
 \end{pmatrix} 
\end{align*}
\end{example}

\subsection{Beta-Gamma algebra identities}
The formula defining $\Gamma_\mu$ has more to it than it seems. Indeed, in order to have the law being the same for all reduced expressions of $w_0$, there has to be hidden non-trivial equalities in law. Those can be qualified as identities from the Beta-Gamma algebra as Dufresne, Letac, Yor and others call them (\cite{bib:Dufresne98} for instance). It is remarkable to think of them as a probabilistic manifestation of a group structure, and more precisely braid relationships.

Considering two reduced expressions of the same Weyl group element, one can obtain the other from successive braid moves:
$$ s_i s_j s_i \dots = s_j s_i s_j \dots $$
Using this fact, saying that $D_\mu$ is defined unambiguously is equivalent to saying that for any $\mu \in \mathfrak{a}$, such that $\alpha_i(\mu)>0$ and $\alpha_j(\mu)>0$, one has:
\begin{align*}
& x_i\left( \gamma_{\langle\alpha_i^\vee,\mu\rangle} \right) x_j\left( \gamma_{\langle s_i \alpha_j^\vee,\mu\rangle} \right) x_i\left( \gamma_{\langle s_i s_j \alpha_i^\vee,\mu\rangle} \right) \dots \\
& \stackrel{\mathcal{L}}{=} x_j\left( \gamma_{\langle\alpha_j^\vee,\mu\rangle} \right) x_i\left( \gamma_{\langle s_j \alpha_i^\vee,\mu\rangle} \right) x_j\left( \gamma_{\langle s_j s_i \alpha_j^\vee,\mu\rangle} \right) \dots 
\end{align*}
In the end, the rank 2 ($A_2$, $B_2$, $C_2$ and $G_2$) cases contain all the possible hidden identities. In the following $\gamma_.$ denote independent random variables. We write $\mu = a_1 \omega_i^\vee + a_2 \omega_j^\vee$, make use of the explicit formulas for the change of Lusztig parametrization in $U^{w_0}_{>0}$ and the root enumerations given in tables \ref{tab:roots_A2}, \ref{tab:roots_B2}, \ref{tab:roots_C2} and \ref{tab:roots_G2}. In the following list, we write $(p_1, p_2, \dots) = R_{\bf i, i'}\left( t_1, t_2, \dots \right)$ (cf equation \ref{eqn:geom_change_of_param_lusztig}) and give the corresponding equality in law between gamma variables.

\begin{itemize}

 \item Type $A_2$ (table \ref{tab:roots_A2}): $w_0 = s_1 s_2 s_1 = s_2 s_1 s_2$. ${\bf i} = (1,2,1)$ and ${\bf i'} = (2,1,2)$. Let $ p_1 = \frac{t_2 t_3}{t_1 + t_3}, p_ 2 = t_1 + t_3 , p_ 3 = \frac{t_1 t_2}{t_1 + t_3}$. Then:
$$(t_1, t_2, t_3) \stackrel{\mathcal{L}}{=}  \left( \gamma_{a_1}, \gamma_{a_1 + a_2}, \gamma_{a_2} \right)$$
\begin{center} if and only if \end{center}
$$(p_1, p_2, p_3) \stackrel{\mathcal{L}}{=}  \left( \gamma_{a_2}, \gamma_{a_1 + a_2}, \gamma_{a_1} \right) $$

 \item Type $B_2$ (table \ref{tab:roots_B2}): $w_0 = s_1 s_2 s_1 s_2 = s_2 s_1 s_2 s_1$. ${\bf i} = (1,2,1,2)$ and ${\bf i'} = (2,1,2,1)$. Let:
$$p_1 = \frac{t_2 t_3^2 t_4}{\pi_2}, p_2 = \frac{\pi_2}{\pi_1}, p_3 = \frac{\pi_1^2}{\pi_2}, p_4 = \frac{t_1 t_2 t_3}{\pi_1}$$
where $\pi_1 = t_1 t_2 + (t_1 + t_3)t_4, \pi_2 = t_1^2 t_2 + (t_1 + t_3)^2 t_4$. Then:
$$(t_1, t_2, t_3, t_4) \stackrel{\mathcal{L}}{=}  \left( \gamma_{a_1}, \gamma_{a_1 + a_2}, \gamma_{a_1 + 2 a_2}, \gamma_{a_2} \right)$$
\begin{center} if and only if \end{center}
$$(p_1, p_2, p_3, p_4) \stackrel{\mathcal{L}}{=}  \left( \gamma_{a_2}, \gamma_{a_1 + 2 a_2}, \gamma_{a_1 + a_2}, \gamma_{a_1} \right)$$

 \item Type $C_2$ (table \ref{tab:roots_C2}): $w_0 = s_1 s_2 s_1 s_2 = s_2 s_1 s_2 s_1$. ${\bf i} = (1,2,1,2)$ and ${\bf i'} = (2,1,2,1)$. Let:
$$p_1 = \frac{t_2 t_3 t_4}{\pi_1}, p_2 = \frac{\pi_1^2}{\pi_2}, p_3 = \frac{\pi_2}{\pi_1}, p_4 = \frac{t_1 t_2^3 t_3}{\pi_1}$$
where $\pi_1 = t_1 t_2 + (t_1 + t_3)t_4, \pi_2 = t_3 t_4^2 + (t_2 + t_4)^2 t_1$. Then:
$$(t_1, t_2, t_3, t_4) \stackrel{\mathcal{L}}{=}  \left( \gamma_{a_1}, \gamma_{2 a_1 + a_2}, \gamma_{a_1 + a_2}, \gamma_{a_2} \right)$$
\begin{center} if and only if \end{center}
$$(p_1, p_2, p_3, p_4) \stackrel{\mathcal{L}}{=}  \left( \gamma_{a_2}, \gamma_{a_1 + a_2}, \gamma_{2 a_1 + a_2}, \gamma_{a_1} \right)$$

 \item Type $G_2$ (table \ref{tab:roots_G2}): $w_0 = s_1 s_2 s_1 s_2 s_1 s_2 = s_2 s_1 s_2 s_1 s_2 s_1$. ${\bf i} = (1,2,1,2,1,2)$ and ${\bf i'} = (2,1,2,1,2,1)$. Let: 
$$ p_1 = \frac{t_2 t_3^3 t_4^2 t_5^3 t_6}{\pi_3}, p_2 = \frac{\pi_3}{\pi_2}, p_3 = \frac{\pi_2^3}{\pi_3 \pi_4} $$
$$ p_4 = \frac{\pi_3}{\pi_1 \pi_2}, p_5 = \frac{\pi^3_1}{\pi_4}, p_6 = \frac{t_1 t_2 t_3^2 t_4 t_5}{\pi_1}$$
where:
\begin{center}
 
\begin{align*}
 \pi_1 = & t_1 t_2 t_3^2 t_4 + t_1 t_2 (t_3 + t_5 )^2 t_6 + (t_1 + t_3 )t_4 t_5^2 t_6 
\end{align*}
\begin{align*}
 \pi_2 = & t_1^2 t_2^2 t_3^3 t_4 + t_1^2 t_2^2 (t_3 + t_5 )^3 t_6 + (t_1 + t_3 )^2 t_4^2 t_5^3 t_6 +\\
         & t_1 t_2 t_4 t_5^2 t_6 (3t_1 t_3 + 2t_2 + 2t_3 t_5 + 2t_1 t_5 )
\end{align*}
\begin{align*}
\pi_3 = & t_1^3 t_2^2 t_3^3 t_4 + t_1^3 t_2^2 (t_3 + t_5 )^3 t_6 + (t1 + t3 )^3 t_4^2 t_5^3 t_6 + \\
        & t_1^2 t_2 t_4 t_5^2 t_6 (3t_1 t_3 + 3t_2 + 3t_3 t_5 + 2t_1 t_5 )
\end{align*}
\begin{align*}
\pi_4 = & t_1^2 t_2^2 t_3^3 t_4 ( t_1 t_2 t_3^3 t_4 + 2t_1 t_2 (t_3 + t_5 )^3 t_6 + (3t_1 t_3 + 3t_3^2 + 3t_3 t_5 + 2t_1 t_5 )t_4 t_5^2 t_6) + \\
        & t_6^2 ( t_1 t_2 (t_3 + t_5 )^2 + (t_1 + t_3 )t_4 t_5^2)^3 
\end{align*}
\end{center}

Then:
$$(t_1, t_2, t_3, t_4, t_5, t_6) \stackrel{\mathcal{L}}{=}  \left( \gamma_{a_1}, \gamma_{3 a_1 + a_2}, \gamma_{2 a_1 + a_2}, \gamma_{ 3 a_1 + 2 a_2}, \gamma_{a_1 + a_2}, \gamma_{ a_2} \right)$$
\begin{center} if and only if \end{center}
$$(p_1, p_2, p_3, p_4, p_5, p_6) \stackrel{\mathcal{L}}{=}  \left( \gamma_{ a_2}, \gamma_{a_1 + a_2}, \gamma_{ 3 a_1 + 2 a_2}, \gamma_{2 a_1 + a_2}, \gamma_{3 a_1 + a_2}, \gamma_{a_1} \right) $$

\end{itemize}
Lukacs identity in law for gamma variables (\cite{bib:Luk55} or \cite{bib:Dufresne98}) is easy to retrieve from $A_2$ case. Indeed, by considering $(t_1, t_2, t_3) \stackrel{\mathcal{L}}{=}  \left( \gamma_{a_1}, \gamma_{a_1 + a_2}, \gamma_{a_2} \right)$ independent variables with the designated laws and $(p_1, p_2, p_3)$ algebraically defined as above, we know that $p_2$ and $p_3$ are independent. And since $p_2$ independent of $t_2$, we get that $p_2 = \gamma_{a_1} + \gamma_{a_2}$ is independent of $\frac{p_3}{t_2} = \frac{\gamma{a_1}}{\gamma_{a_1} + \gamma_{a_2}}$. \\
We can see that those identities in law are very rich, and one could try to retrieve other identities in law. This could be the object of future work. 

\subsubsection{Exponential identities}
One can recover a ``crystallized'' versions of these identities using Laplace method. They will involve exponential variables. The following easy lemma gives a hint on how gamma variables can degenerate to exponential variables.

\begin{lemma}
 As $\beta$ goes to zero, $-\beta \log \gamma_{\beta \mu}$ converges in law to ${\bf e}_{\mu}$, an exponential variable 
with parameter $\mu$.
\end{lemma}
\begin{proof}
 Considering $f$ a bounded measurable function, the following easily holds:
\begin{align*}
  & \E\left( f\left( -\beta \log \gamma_{\beta \mu} \right) \right)\\
= & \int_0^\infty \frac{dx}{\Gamma( \beta \mu)} f\left( -\beta \log x \right) e^{-x} x^{\beta \mu-1}\\
= & \int_\R \frac{dy}{\Gamma( \beta \mu)} f\left( -\beta y\right) e^{- e^{y} + \beta \mu y}\\
= & \int_\R \frac{dy}{ \beta \Gamma( \beta \mu)} f\left( y \right) e^{- e^{-\frac{y}{\beta} } - \mu y}\\ 
\stackrel{\beta \rightarrow 0}{\rightarrow} &  \int_0^\infty dy f\left( y \right) \mu e^{ - \mu y}\\
\end{align*}
\end{proof}

Therefore, we introduce the following component-wise crystallizing procedure for the rational substraction-free expressions $R_{\bf i, i'}$. The use of the logarithm function has to be understood as component-wise:
$$ \left[ R_{\bf i, i'} \right]_{trop} (u_1, u_2, \dots ) = \lim_{\beta \rightarrow 0 } - \beta \log R_{\bf i, i'}\left( e^\frac{-u_1}{\beta}, e^\frac{-u_2}{\beta}, \dots \right) $$
Thanks to the Beta-Gamma algebra identities for proven before, we know that before crystallization, for the parameter $\beta \mu$ and as soon as a braid relationship $s_i s_j \dots = s_j s_i \dots$ holds, we have:
\begin{align*}
& \left( \gamma_{\langle\alpha_i^\vee,\beta \mu\rangle} , \gamma_{\langle s_i \alpha_j^\vee,\beta \mu\rangle}, \dots \right) \\
& \stackrel{\mathcal{L}}{=} R_{i,i'} \left( \gamma_{\langle\alpha_j^\vee, \beta \mu\rangle}, \gamma_{\langle s_j \alpha_i^\vee,\beta \mu\rangle}, \dots \right)
\end{align*}
Then:
\begin{align*}
& \left( {\bf e}_{\langle\alpha_i^\vee,\mu\rangle} , {\bf e}_{\langle s_i \alpha_j^\vee,\mu\rangle}, \dots \right)\\
& \stackrel{\mathcal{L}}{=} \lim_{\beta \rightarrow 0} -\beta \log \left( \gamma_{\langle\alpha_i^\vee,\beta \mu\rangle} , \gamma_{\langle s_i \alpha_j^\vee,\beta \mu\rangle}, \dots \right) \\
& \stackrel{\mathcal{L}}{=} \lim_{\beta \rightarrow 0} -\beta \log R_{i,i'} \left( \gamma_{\langle\alpha_j^\vee, \beta \mu\rangle}, \gamma_{\langle s_j \alpha_i^\vee,\beta \mu\rangle}, \dots \right)\\
& =  \lim_{\beta \rightarrow 0} -\beta \log R_{i,i'} \left( e^\frac{- \beta \log\left( \gamma_{\langle\alpha_j^\vee, \beta \mu\rangle} \right) }{\beta}, e^\frac{- \beta \log\left( \gamma_{\langle s_j \alpha_i^\vee,\beta \mu\rangle} \right)}{\beta}, \dots \right)\\
& = \left[R_{i,i'}\right]_{trop} \left( {\bf e}_{\langle\alpha_j^\vee,\mu\rangle} , {\bf e}_{\langle s_j \alpha_i^\vee,\mu\rangle}, \dots \right)\\
\end{align*}

Again, as all the information is contained in the rank two case, there is a finite list of identities between exponential variables that sums up the results so far. Crystallizing rational expressions is easily computed with the rules:
$$ \overline{a + b} = \min( a, b) $$
$$ \overline{ a b } = a + b $$
$$ \overline{ a/b } = a - b $$

\begin{itemize}

 \item Type $A_2$ (table \ref{tab:roots_A2}): $w_0 = s_1 s_2 s_1 = s_2 s_1 s_2$. Let $ p_1 = t_2 + t_3 - \min(t_1, t_3), p_ 2 = \min( t_1, t_3) , p_ 3 = t_1 + t_2 - \min(t_1, t_3)$.Then:\\
$$(t_1, t_2, t_3) \stackrel{\mathcal{L}}{=}  \left( {\bf e}_{a_1}, {\bf e}_{a_1 + a_2}, {\bf e}_{a_2} \right)$$
\begin{center} if and only if \end{center}
$$(p_1, p_2, p_3) \stackrel{\mathcal{L}}{=}  \left( {\bf e}_{a_2}, {\bf e}_{a_1 + a_2}, {\bf e}_{a_1} \right) $$

 \item Type $B_2$ (table \ref{tab:roots_B2}): $w_0 = s_1 s_2 s_1 s_2 = s_2 s_1 s_2 s_1$. Let:
$$p_1 = t_2 + 2 t_3 + t_4 - \pi_2, p_2 = \pi_2 - \pi_1, p_3 = 2\pi_1 - \pi_2, p_4 = t_1 + t_2 + t_3 - \pi_1$$
where $\pi_1 = \min( t_1 + t_2, \min(t_1 , t_3) + t_4), \pi_2 = \min( 2 t_1 + t_2, 2\min(t_1 , t_3) + t_4 )$. Then:
$$(t_1, t_2, t_3, t_4) \stackrel{\mathcal{L}}{=}  \left( {\bf e}_{a_1}, {\bf e}_{a_1 + a_2}, {\bf e}_{a_1 + 2 a_2}, {\bf e}_{a_2} \right)$$
\begin{center} if and only if \end{center}
$$(p_1, p_2, p_3, p_4) \stackrel{\mathcal{L}}{=}  \left( {\bf e}_{a_2}, {\bf e}_{a_1 + 2 a_2}, {\bf e}_{a_1 + a_2}, {\bf e}_{a_1} \right) $$

 \item In the same fashion one can deduce the exponential identities for types $C_2$ and $G_2$.

\end{itemize}

At this point, it seems very important to mention that in \cite{bib:BBO2}, a path model for Coxeter groups was developped, and exponential laws play a key role as infinimums of a Brownian motion, with appropriate drift. One could define an exponential law on the Berenstein-Zelevinsky polytope. There, once again, hidden identities in law can be found, involving general Coxeter braid relations which goes beyond the crystallographic case we just considered.

\section{Whittaker process}
\label{section:whittaker_process}
Because of the conditional representation theorem, we know that:
 $$X^{x_0,(\mu)} := x_0 + T_{e^{\theta} g e^{-\theta}} \left( W^{(w_0 \mu)} \right)$$
is a BM starting at $x_0$ conditionned to $ N_\infty^\theta(X^{0,(\mu)}) = \Theta(g)$. The purpose of the next section is to take $g$ as random and independent of $W$ therefore conditionning $N_\infty^\theta(X^{0,(\mu)})$ to follow a certain specific law. The goal is to identify a remarkable law that will force $X^{x_0, (\mu)}$ into becoming a Markov process. These laws are deformations of the law $\Gamma_\mu$ introduced earlier (after theorem \ref{thm:N_infty_law}) and use the principal character $\chi^-$ on $N$.

Having in mind the construction in section \ref{section:poisson_boundary}, this can be seen as conditionning $B_t\left(X^{x_0,(\mu)}\right)$ to hit the Poisson boundary according to the distribution induced by a character of $N$.

In the rank one setting, it is the natural group-theoretic generalization of 'generalized inverse Gaussian' laws used by Matsumoto and Yor (\cite{bib:MY00-2}). The main ideas were in fact already in \cite{bib:BOC09} where the Whittaker process was built out of a Brownian motion conditionned with respect to its exponential functionals.

\subsection{Whittaker functions}
Recall that $b$ is the function:
$$ b(\mu) = \prod_{\beta \in \Phi^+} \Gamma\left( \langle \beta^\vee, \mu \rangle \right)$$
As starting point, we will consider a probabilistic definition as in theorem \ref{thm:whittaker_functions_properties} (iii), which makes sense only when $\mu \in C$:
$$ \forall x \in \afrak, \psi_\mu\left( x \right) := b(\mu) e^{ \langle \mu, x \rangle }\E\left( \exp\left( - \sum_{\alpha \in \Delta} \half \langle \alpha, \alpha \rangle \int_0^\infty ds \ e^{-\alpha(x + W_s^{(\mu)})} \right) \right)$$
where $W^{(\mu)}$ a Brownian motion on $\afrak$ with drift $\mu$. Later, in proposition \ref{proposition:link_with_canonical}, we will see that it coincides with the point of view adopted in definition \ref{def:whittaker_functions}. Therefore, this $\psi_\mu$ has an analytic extension to all $\mu \in \afrak$.

Baudoin and O'Connell give a nice characterization of this function:
\begin{proposition}[\cite{bib:BOC09} corollary 2.3]
\label{proposition:whittaker_characterization_boc}
The function $\psi_\mu$ is the unique solution to the quantum Toda eigenequation:
$$ \half \Delta \psi_\mu(x) - \sum_{\alpha \in \Delta } \half \langle \alpha, \alpha \rangle e^{-\alpha(x) } \psi_\mu(x) = \half \langle \mu, \mu \rangle \psi_\mu(x)$$
such that $\psi_\mu(x) e^{-\langle \mu, x \rangle }$ is bounded with growth condition $\psi_\mu(x) e^{-\langle \mu, x \rangle } \stackrel{ x \rightarrow \infty, x \in C }{\longrightarrow} b(\mu)$ 
\end{proposition}
\begin{proof}
Write $\varphi_\mu = \psi_\mu e^{-\langle \mu, .\rangle}$, which needs to be the unique bounded solution to:
$$ \half \Delta \varphi_\mu(x) + \langle \mu, \nabla \varphi_\mu \rangle - \sum_{\alpha \in \Delta } \half \langle \alpha, \alpha \rangle e^{-\alpha(x) } \varphi_\mu(x) = 0$$
with growth condition $\varphi_\mu(x) \stackrel{ x \rightarrow \infty, x \in C }{\longrightarrow} b(\mu)$.

As a consequence of the Feynman-Kac formula:
$$ \varphi_\mu(x) = b(\mu) \E\left( \exp\left( - \sum_{\alpha \in \Delta} \half \langle \alpha, \alpha \rangle \int_0^\infty ds \ e^{-\alpha(x + W_s^{(\mu)})} \right) \right)$$
solves the partial differential equation. For uniqueness, we use a martingale argument. If $\phi$ is a bounded solution such that:
$$\phi(x) \stackrel{ x \rightarrow \infty, x \in C }{\longrightarrow} 0$$
Then:
$$ \phi(x + W_t^{(\mu)})\exp\left( - \sum_{\alpha \in \Delta} \half \langle \alpha, \alpha \rangle \int_0^\infty ds \ e^{-\alpha(x + W_s^{(\mu)})} \right)$$
is bounded martingale going to zero as $t \rightarrow \infty$. Therefore, it must vanish identically. Hence uniqueness, by linearity.
\end{proof}

Notice that the Whittaker function $\psi_\mu$ can be written thanks to theorem \ref{thm:N_infty_law} as:
\begin{align}
\label{eqn:whittaker_function_D_mu}
\psi_\mu(x) & = b(\mu) e^{\langle \mu, x \rangle}\E\left( \exp\left( -\chi^-(e^{x} D_\mu e^{-x} ) \right) \right)
\end{align}

Hence the idea of introducing a deformation of the laws of $\Gamma_\mu$ and $D_\mu$:
\begin{definition}[Generalized gamma and inverse gamma]
 For $\mu \in C$ and $\lambda \in \afrak$, define $D_\mu\left( \lambda \right)$ as the $N$-valued random variable defined by:
$$ \forall \varphi \geq 0, \E\left( \varphi( D_\mu(\lambda) ) \right) = 
   \frac{b(\mu) e^{\langle \mu, x \rangle} }{\psi_\mu(\lambda)} \E\left( \varphi(D_\mu) \exp\left( -\chi^-(e^{\lambda} D_\mu e^{-\lambda} ) \right) \right) $$
and the $U$-valued random variable $\Gamma_\mu(\lambda)$ as:
$$ D_\mu(\lambda) = \Theta\left( \Gamma_\mu(\lambda) \right)$$
\end{definition}
Even though the definition assumes $\mu \in C$, those probability measures are in fact well-defined for all $\mu \in \afrak$ as we will see in proposition \ref{proposition:link_with_canonical}.

\begin{example}[$A_1$-type]
 In the $A_1$ type, if $D_\mu(\lambda) = \begin{pmatrix} 1 & 0\\ GIG_{\mu}(\lambda) & 1 \end{pmatrix}$ then $GIG_{\mu}(\lambda)$ is the generalized inverse gaussian law used in \cite{bib:MY00-2}:
$$\P\left( GIG_\mu(\lambda) \in dt \right) \sim t^{-\mu} e^{-\frac{1}{t} - t e^{-2\lambda} } \frac{dt}{t}$$
\end{example}

\subsection{Related Markov processes}

Such a deformation is remarkable because it is a key ingredient in the generalization of proposition 3.3 in \cite{bib:OConnell} to all Lie groups.
\begin{thm}
\label{thm:whittaker_process_g}
Let $W$ be a Brownian motion, a drift $\mu$ and an independent random variable $g \stackrel{\Lc}{=} \Gamma_{\mu}(x_0)$.

Then the process $X^{x_0,(\mu)}_t = x_0 + T_{e^{\theta} g e^{-\theta}}\left( W^{(w_0 \mu)} \right)_t, t\geq0$ is Markovian with infinitesimal generator
$$\frac{1}{2} \Delta + \nabla \log \psi_\mu .\nabla $$
\end{thm}
\begin{proof} 
The proof for $\mu \in C$ is given in \ref{proof:whittaker_process_g_mu_in_C}. The proof for all $\mu$ uses an intertwining argument in section \ref{section:intertwined_markov_kernels}.
\end{proof}

In order to prove theorem \ref{thm:highest_weight_is_markov}, we will need to take $x_0$ to '$-\infty$'. And that will be done is subsection \ref{subsection:entrance_point}. Nevertheless, we can already describe the end of the story to the impatient reader, at the expense of being a bit redundant later. There is a measure concentration result that tells us how the distribution of the random path transform $x_0 + T_{e^{\theta} \Gamma_{\mu}(x_0) e^{-\theta}}$ behaves as $x_0$ goes to infinity in the opposite Weyl chamber:

\begin{thm}
\label{thm:measure_concentration}
Consider the family of random path transforms $x_0 + T_{e^{\theta} \Gamma_{\mu}(x_0) e^{-\theta}}$ for $x_0 \in \afrak$. For $x_0 = -M \rho^\vee$, $M \rightarrow \infty$, we have the following convergence in probability for every continuous path $\pi \in C( \R^+, \afrak)$:
$$ \forall t >0, x_0 + T_{e^{\theta} \Gamma_{\mu}(x_0) e^{-\theta}}\pi(t) \stackrel{\P}{\longrightarrow} \theta - w_0 \theta + \Tc_{w_0}\pi(t)$$ 
\end{thm}
\begin{proof}
See subsection \ref{section:measure_concentration}.
\end{proof}

\begin{corollary}
\label{thm:whittaker_process}[Whittaker process]
Let $W^{(\mu)}$ be a Brownian motion in $\afrak$ with drift $\mu$. Then the process:
$$\left( \theta - w_0 \theta + \Tc_{w_0}\left( W^{(\mu)} \right)_t; t > 0 \right)$$
is Markovian with infinitesimal generator
$$\frac{1}{2} \Delta + \nabla \log \psi_\mu .\nabla $$
\end{corollary}
\begin{proof}
The Markov property is preserved after taking limits in probability. Moreover, changing $w_0 \mu$ for $\mu$ is allowed because Whittaker functions are invariant w.r.t to $W$.
\end{proof}

Now in order to prove theorem \ref{thm:whittaker_process_g} for $\mu \in C$, start by using theorem \ref{thm:conditional_representation}. It tells us that $X^{x_0, (\mu)} = x_0 + T_{e^{\theta} g e^{-\theta}} W^{(w_0 \mu)}$ with $g$ random is a Brownian motion having drift $\mu$ with $N_\infty^\theta(X^{0,(\mu)})$ conditioned to follow the law of $\Theta(g)$. If we take this law has having a density $v$ with respect to the original law, we need to be able to describe the process $X^{x_0, (\mu)}$ before going further. For now, let us do our computations for general $v$, even if we will end up taking $v$ proportional to a character of $N$. We simply ask for which $v$ the process $X^{x_0, (\mu)}$ ends up being Markovian. We will see that the announced choice is virtually the only interesting one.

\subsection{Conditionning \texorpdfstring{$N_\infty(X^{x_0, (\mu)})$}{the N part} to follow a certain law}
\label{subsection:condition_N_infty}
Let $\P$ be a probability under which $g \stackrel{\Lc}{=} \Gamma_\mu$. Under $\P$, $X^{x_0, (\mu)}$ is a Brownian motion in $\afrak$ with drift $\mu$. Denote its natural filtration by:
$$ \Fc_t := \sigma\left( X^{x_0, (\mu)}_s | 0 \leq s \leq t \right)$$
Let $v: N_{>0} \rightarrow \R_+^*$ be a smooth positive function such that:
\begin{align}
\label{eqn:v_normalization}
\E\left( v\left( N_\infty^\theta(X^{x_0, (\mu)}) \right) \right) & = 1 
\end{align}
The function $v$ is used to define a deformed probability measure $\P^v$. More precisely, define $\P^v$ thanks to its Radon-Nikodym derivative with respect to $\P$:
$$ \frac{d\P^v}{d\P} = v\left( N_\infty^\theta(X^{x_0, (\mu)}) \right)$$
The probability measures $\P$ and $\P^v$ are equivalent since by definition $\frac{d\P^v}{d\P} > 0$ almost surely. Moreover, $\P^v$ can be interpreted as a probability measure under which $N_\infty^\theta(X^{x_0, (\mu)})$ is conditioned to follow a certain law. This law has a density $v$ with respect to the original law of $N_\infty^\theta(X^{(x_0, \mu)}) \stackrel{\Lc}{=} e^{x_0} D_\mu e^{-x_0}$. We are interested in describing the process $X^{x_0,(\mu)}$ under $\P^v$. A first step is to explicit the likelihood process:
$$ L_t := \frac{d \P^v}{d \P}_{| \Fc_t }$$

In order to do so, a decomposition analogous to the one in equation (\ref{eqn:N_infty_decomposition_1d}) in the general group setting is needed. For $X = X^{0, (\mu)}$, let $B_t(X) = N_t(X) A_t(X)$ be the $NA$ decomposition of $B_.(X)$. Because of left-invariance, for all $t, s>0$:
\begin{eqnarray*}
& B_{t+s}( X) & = B_t( X ) B_s( X_{t+.} - X_{t}) \\
\Leftrightarrow & N_{t+s}(X) A_{t+s}(X) & = N_{t}(X) A_{t}(X) N_{s}( X_{t+.} - X_{t} ) A_{s}( X_{t+.} - X_{t} )\\
\Leftrightarrow & N_{t+s}(X)            & = N_{t}(X) e^{ X_t } N_{s}( X_{t+.} - X_{t} ) e^{ -X_t }
\end{eqnarray*}
Letting $s \rightarrow \infty$, it gives us a decomposition of $N_{\infty}(X)$ in terms of $\mathcal{F}_t^{X}$-measurable variables and an independent variable $N_{\infty}( X_{t+.} - X_{t} )$ with same law:
\begin{align}
\label{eqn:N_infty_decomposition_zero}
N_{\infty}(X) &= N_{t}(X) e^{ X_t } N_{\infty}( X_{t+.} - X_{t} ) e^{ -X_t }
\end{align}
The same decomposition holds easily for $N_\infty^\theta\left( X^{x_0, (\mu)} \right)$:
\begin{align}
\label{eqn:N_infty_decomposition}
N_{\infty}^\theta(X^{x_0, (\mu)}) & = N_t^\theta(X^{x_0, (\mu)}) e^{ X_t^{x_0, (\mu)} } N_{\infty}^\theta( X_{t+.}^{x_0, (\mu)} - X_{t}^{x_0, (\mu)} ) e^{ -X_t^{x_0, (\mu)} }
\end{align}

\begin{lemma}[Likelihood]
\label{lemma:likelihood}
$$ L_t := \frac{d \P^v}{d \P}_{| \Fc_t } = \Phi_v \left(N_t^\theta(X^{x_0, (\mu)}), X^{x_0, (\mu)}_t\right)$$
where:
$$\Phi_v(n, x) = \mathbb{E}\left( v\left( n e^{x} D_\mu e^{-x} \right) \right)$$ 
\end{lemma}
\begin{proof}
We drop the superscript $\mu$ for convenience. Using equation (\ref{eqn:N_infty_decomposition}) and the likelihood's definition:
\begin{align*}
L_t & := \frac{d \P^v}{d \P}_{| \Fc_t }\\
& = \P\left( \frac{d \P^v}{d \P} | \Fc_t \right)\\
& = \P\left( v\left( N_\infty^\theta(X^{x_0}) \right) | \Fc_t \right)\\
& = \P\left( v\left( N_{t}^\theta(X^{x_0}) e^{ X_t^{x_0} } N_{\infty}^\theta( X_{t+.}^{x_0} - X_{t}^{x_0} ) e^{ -X_t^{x_0} } \right) | \Fc_t \right)
\end{align*}
Because $N_{\infty}^\theta( X_{t+.}^{x_0} - X_{t}^{x_0} ) \stackrel{ \Lc }{=} D_\mu$ under $\P$ (equation \ref{eqn:N_infty_law_D_mu}) and is independent of $\Fc_t$, we have the required result:
$$ L_t = \E\left( v\left( n e^{ x } D_\mu  e^{-x} \right) | N_t^\theta(X^{x_0}) = n, X^{x_0}_t = x \right)$$
\end{proof}

\begin{corollary}
\label{corollary:girsanov}
$$ X^{x_0,(\mu)}_t - x_0 - \mu t - \int_0^t \nabla_x \log \Phi_v \left(N_s^\theta(X^{x_0,(\mu)}), X^{x_0,(\mu)}_s \right) ds$$
is a $\P^v$ Brownian motion.
\end{corollary}
\begin{proof}
 By Girsanov's theorem (\cite{bib:RevuzYor} Chapter VIII, theorem 1.4 or (\cite{bib:KaratzasShreve91} theorem 5.1)):
$$X^{x_0,(\mu)}_t - x_0 - \mu t - \int_0^t d \langle \log L, X \rangle_s$$
is a $\P^v$ Brownian motion. Thanks to the previous lemma and to the fact that $N_t^\theta(X^{x_0})$ has zero quadratic variation, the bracket $d \langle \log L, X \rangle_s$ is indeed:
$$\nabla_x \log \Phi_v \left(N_s^\theta(X^{x_0,(\mu)}), X^{x_0,(\mu)}_s \right) ds$$
\end{proof}

\subsection{The remarkable deformations \texorpdfstring{$D_\mu(\lambda)$ and $\Gamma_\mu(\lambda)$}{ } }
\label{subsection:remarkable_deformation}
We are aiming at identifying a function $v$ that forces $X^{x_0,(\mu)}$ to be a Markov process. As such, we want a term $\nabla_x \log \Phi_v (n, x)$ that only depends on $x$ for a certain $v$. Equivalently $\Phi_v( n, x)$ has to break down into the product of two functions $g_v(n) f_v(x)$.

\begin{proposition}
\label{proposition:choice_v}
If:
$$    v(n) = \frac{\exp\left( -\chi^-\left( n \right) \right)}
{ \mathbb{E}\left( \exp\left( -\chi^-\left( e^{x_0}    D_\mu e^{-x_0}          \right) \right) \right)} $$
Then:
\begin{align}
   \label{eqn:proposition:choice_v_1}
   \frac{d \P^v}{d \P} & = \exp\left( -\chi^-\left( N_\infty^\theta(X^{x_0, (\mu)}) \right) \right)
                            \frac{ b(\mu) }
                                 { \psi_\mu(x_0) e^{- \langle \mu, x_0 \rangle }}
\end{align}
The function $\Phi_v$ defining the likelihood $L_t$ splits into:
\begin{align}
   \label{eqn:proposition:choice_v_2}
   \Phi_v( n, x) & = \exp\left( -\chi^-\left( n \right) \right)
                     \frac{ \psi_\mu(x  ) e^{- \langle \mu, x   \rangle }}
                          { \psi_\mu(x_0) e^{- \langle \mu, x_0 \rangle }}
\end{align}
And under $\P^v$:
\begin{align}
   \label{eqn:proposition:choice_v_3}
   N_\infty^\theta( X^{x_0, (\mu)} ) \stackrel{\Lc}{=} e^{x_0} D_\mu(x_0) e^{-x_0}
\end{align}
\end{proposition}

Before diving into the proof, the following dicussion explains the reasons that lead to the choice of function $v$ made in proposition  \ref{proposition:choice_v}. Assume a general $v$. For starter, notice that for all $y \in N^{w_0}_{>0}$, $e^{x} y e^{-x} \longrightarrow id$ as $x \rightarrow \infty$ staying inside a sector of the open Weyl chamber. In order to see that, write $y = y_{\bf i}(t_1, \dots, t_m)$ hence:
$$ e^{x} y e^{-x} = y_{\bf i }\left( t_1 e^{-\alpha_{i_1}(x)}, \dots, t_m e^{-\alpha_{i_m}(x)} \right) \rightarrow id$$
For such asymptotical directions, $\Phi_v(x, n)$ actually converges to:
$$g_v(n)f_v(\infty) = \lim_{x \rightarrow \infty, x\in C} \mathbb{E}\left( v\left( n e^{x} D_\mu e^{-x} \right) \right) = v\left( n \right)$$
and then $\Phi_v$ has in fact to breaks down to $\Phi_v ( n, x) = v\left( n \right) \frac{f_v(x)}{f_v(\infty)}$.

Hence, the objective we want to achieve is:
$$ \exists v: N^{w_0}_{>0} \rightarrow \mathbb{R}, \Phi_v (x, n) = \mathbb{E}\left( v\left( n e^{x} D_\mu e^{-x}  \right) \right) \stackrel{?}{=} v\left( n \right) \frac{f_v(x)}{f_v(\infty)}$$
A sufficient (and probably necessary) condition to have this factorization is to make use of a multiplicative character of the lower nilpotent group $N$. Hence the idea to use the principal additive character $\chi^-$ defined by its action on the Chevalley generators:
$$ \forall t \in \C, \forall \alpha \in \Delta, \chi^-( e^{t f_\alpha} ) = t $$
And $v$ should be proportional to the multiplicative character:
$$ \exp\left( -\chi^-\left( n \right) \right) $$
Indeed, in order to avoid integrability issues, it is sufficient for $\exp\left(-\chi^-\right)$ to be bounded on $N^{w_0}_{>0}$, which is ascertained by the minus sign. There is no loss of generality in using the standard character. More general characters simply add a shift via conjugation by torus elements.

\begin{proof}[Proof of proposition \ref{proposition:choice_v}]
The quantity $\mathbb{E}\left( e^{-\chi^-}\left( e^{x_0} D_\mu e^{-x_0} \right) \right) = \frac{\psi_\mu(x_0) e^{-\langle \mu, x_0 \rangle} }{b(\mu)}$ is a normalization constant so that equation  (\ref{eqn:v_normalization}) is true.

With such a function $v$, equation (\ref{eqn:proposition:choice_v_1}) follows from:
\begin{align*}
  & \frac{d \P^v}{d \P}\\
= & v\left( N_\infty^\theta(X^{x_0, (\mu)}) \right)\\
= & \exp\left( -\chi^-\left( N_\infty^\theta(X^{x_0, (\mu)}) \right) \right)
    \frac{ b(\mu) }
    { \psi_\mu(x_0) e^{- \langle \mu, x_0 \rangle }}
\end{align*}
We also get the splitting we wanted for $\Phi_v$:
$$\Phi_v (x, n) 
= \exp\left( -\chi^-\left( n \right) \right) 
  \frac{ \E\left( \exp\left( -\chi^-\left( e^{x  } D_\mu e^{-x  } \right) \right) \right) }
       { \E\left( \exp\left( -\chi^-\left( e^{x_0} D_\mu e^{-x_0} \right) \right) \right) }$$
Then equation (\ref{eqn:whittaker_function_D_mu}) implies equation (\ref{eqn:proposition:choice_v_2}).

And finally, if $\varphi: N^{w_0}_{>0} \rightarrow \R_{>0}$ is a smooth test function. Equation (\ref{eqn:proposition:choice_v_3}) follows from:
\begin{align*}
  & \P^v\left( \varphi\left( N_\infty^\theta\left( X^{x_0, (\mu)} \right) \right) \right) \\
= & \P  \left( \varphi\left( N_\infty^\theta\left( X^{x_0, (\mu)} \right) \right) v\left( N_\infty^\theta\left( X^{x_0, (\mu)} \right) \right) \right) \\
= & \E  \left( \varphi\left( e^{x_0} D_\mu e^{-x_0} \right) v\left( e^{x_0} D_\mu e^{-x_0} \right) \right) \\
= & \frac{ \E  \left( \varphi\left( e^{x_0} D_\mu e^{-x_0} \right)
                      \exp\left( -\chi^-\left( e^{x_0} D_\mu e^{-x_0} \right) \right) \right)}
         { \E  \left( \exp\left( -\chi^-\left( e^{x_0} D_\mu e^{-x_0} \right) \right) \right)} \\
= & \E  \left( \varphi\left( e^{x_0} D_\mu(x_0) e^{-x_0} \right) \right)
\end{align*} 
\end{proof}

\begin{proof}[Proof of theorem \ref{thm:whittaker_process_g} for $\mu \in C$]
\label{proof:whittaker_process_g_mu_in_C}
If $g$ is taken to follow the law $\Gamma_\mu(x_0)$, we are in the situation where corollary \ref{corollary:girsanov} is applicable with choice for $v$ made in proposition \ref{proposition:choice_v}. The reference measure is already of the form $\P^v$, with $\P$ the probability under which $g \stackrel{\Lc}{=} \Gamma_\mu$. As such, there is a Brownian motion $B_t$ such that:
\begin{align*}
  & X^{x_0,(\mu)}_t\\
= & x_0 + \mu t + \int_0^t \nabla_x \log \left( \psi_\mu \left( X^{x_0,(\mu)}_s \right) e^{- \langle \mu, X^{x_0,(\mu)}_s \rangle } \right) ds + B_t\\
= & x_0 + \int_0^t \nabla_x \log \psi_\mu \left( X^{x_0,(\mu)}_s \right) ds + B_t
\end{align*}
\end{proof}

\subsection{Harmonicity: Quantum Toda equation}
\label{subsection:harmonicity_quantum_toda}
From now on, $v$ is chosen proportional to a character of $N$, like in the previous subsection. Writing:
$$ \forall n \in N, \forall x \in \afrak, \Phi_\mu(n e^x) = \Phi_v(n, x)$$
one defines a function of the same form as in lemma \ref{lemma:equivalence_toda_harmonicity}.

Clearly, $\Phi_\mu$ is harmonic for the operator $\Dc^{(\mu)}$ as:
$$ \Phi_v\left( N_t^\theta(X^{x_0,(\mu)}), X^{x_0,(\mu)}_t \right) = \frac{d\P^v}{d\P}_{| \Fc_t}$$
is a $\P$-martingale. Another way of seeing it, is the Poisson integral:
$$ \forall n \in N, \forall x \in \afrak, \Phi_\mu(n e^x) 
= \frac{ \E\left( \exp\left( -\chi^-\left( n e^{x  } N_\infty^\theta(W^{(\mu)}) e^{-x  } \right) \right) \right) }
       { \E\left( \exp\left( -\chi^-\left(   e^{x_0} N_\infty^\theta(W^{(\mu)}) e^{-x_0} \right) \right) \right) }$$
with $W^{(\mu)}$ Brownian motion in $\afrak$ with drift $\mu \in C$. Of course, this is directly hinting to equation  (\ref{eqn:poisson_integral_phi}) exactly as foretold.

Recall that $\Phi_\mu$ being harmonic is equivalent to saying that $\psi_\mu$ satisfies the quantum Toda eigenequation.

\subsection{Link between \texorpdfstring{$\Gamma_\mu\left( \lambda \right)$}{deformed Gamma} and canonical measure}
The following proposition shows that $\Gamma_\mu\left( \lambda \right)$ has the same law as the twisted Lusztig parameter corresponding to a canonical random variable on $\Bc(\lambda)$ (see definition \ref{def:canonical_probability_measure}). Moreover, we recover as a side product that $\psi_\mu$ is the one in definition \ref{def:whittaker_functions}, as announced before.

\begin{proposition}
\label{proposition:link_with_canonical}
For $\mu \in C$ and $\lambda \in \afrak$:
$$ \Gamma_\mu(\lambda) \stackrel{\Lc}{=} \varrho^{T}\left( C_{w_0 \mu}(\lambda) \right)$$
This extends the definition of $\Gamma_\mu(\lambda)$ to all $\mu \in \afrak$. Moreover, the Whittaker function has indeed the integral representation:
$$ \psi_\mu(\lambda) = \int_{\Bc(\lambda)} e^{ \langle \mu, \gamma(x) \rangle - f_B(x) } \omega(dx) $$
\end{proposition}
\begin{proof}
Let ${\bf i} = (i_1, \dots, i_m)$ be a reduced word for $w_0$ and $(\beta_1^\vee, \dots, \beta_m^\vee)$ a positive roots enumeration associated to it.

Let $\varphi$ be a positive measurable function on $U$ that will serve the purpose of test function, $\mu$ in the open Weyl chamber and $\lambda \in \afrak$. By definition, $\Gamma_\mu\left( \lambda \right)$ is a deformation of:
$$\Gamma_\mu = x_{ \bf i }\left( \gamma_{\beta_1^\vee(\mu)}, \dots, \gamma_{\beta_m^\vee(\mu)} \right)$$
using the unipotent character $e^{-\chi^-}$ on $D_\mu$ where $D_\mu = \Theta\left( \Gamma_\mu \right)$.  Hence:
\begin{align*}
  & \mathbb{E}\left( \varphi\left(\Gamma_\mu\left( \lambda \right) \right) \right)\\
= & \frac{b(\mu) e^{\langle \mu, \lambda \rangle} }{\psi_\mu(\lambda)} \E\left( \varphi(\Gamma_\mu) \exp\left( -\chi^-(e^{\lambda} [\Gamma_\mu  \bar{w_0}]_{-} e^{-\lambda} ) \right) \right)\\
= & \frac{ e^{\langle \mu, \lambda \rangle} }{\psi_\mu(\lambda)} \int_{(\R_{>0})^m} \prod_{j=1}^m \left( \frac{dt_j}{t_j} t_j^{\langle \beta_j^\vee, \mu \rangle} e^{-t_j} \right) \varphi(g) e^{-\chi^-(e^{\lambda} [g  \bar{w_0}]_{-} e^{-\lambda}) }
\end{align*}
where:
$$ g = x_{i_1}(t_1) \dots x_{i_m}(t_m)$$

Now let us reorganize the terms in the previous integral. On the one hand, thanks to the choice made in proposition \ref{proposition:crystal_param_maps}, we notice that the term in the exponential is:
\begin{align*}
  & \sum_{j=1}^m t_j + \chi^-\left( e^{\lambda} [g  \bar{w_0}]_{-} e^{-\lambda} \right)\\
= & \chi(g) + \chi\left( e^{-\lambda} [\bar{w_0}^{-1} g^T]_{+} e^{\lambda}) \right)\\
= & \chi \circ S \circ \iota \left( e^{-\lambda} [\bar{w_0}^{-1} g^T]_{+} e^{\lambda}) \right) + \chi(g)\\
= & f_B\left( S \circ \iota \left( e^{-\lambda} [ \bar{w}_0^{-1} g^T ]_{+} e^{\lambda} \right) \bar{w}_0 e^\lambda g \right)\\
= & f_B\left( b_\lambda^{T}(g) \right)
\end{align*}
On the other hand, using the expression for the weight map on $b_\lambda^{T}(g)$ given in theorem \ref{thm:geom_weight_map}:
\begin{align*}
  & e^{\langle \mu, \lambda \rangle} \prod_{j=1}^m \left( t_j^{\langle \beta_j^\vee, \mu \rangle} \right)\\
= & \exp\left(\langle \mu, \lambda \rangle + \sum_{j=1}^m \langle \beta_j^\vee, \mu \rangle \log t_j \right)\\
= & \exp\left(\langle \mu, \lambda + \sum_{j=1}^m \log(t_j) \beta_j^\vee \rangle \right)\\
= & \exp\left(\langle w_0 \mu, w_0\left( \lambda - \sum_{j=1}^m \log(\frac{1}{t_j}) \beta_j^\vee \right) \rangle \right)\\
= & \exp\left(\langle w_0 \mu,  \gamma( b_\lambda^{T}(g) ) \rangle \right)
\end{align*}
In the end, using the notations in section \ref{section:canonical_measure}:
\begin{align*}
  & \mathbb{E}\left( \varphi\left(\Gamma_\mu\left( \lambda \right) \right) \right)\\
= & \frac{1}{\psi_{\mu}(\lambda)} \int_{(\R_{>0})^m} \exp\left(\langle w_0 \mu,  \gamma( b_\lambda^{T}(g) ) \rangle - f_B\left( b_\lambda^{T}(g) \right) \right) \varphi(g) \prod_{j=1}^m \frac{dt_j}{t_j}\\
= & \frac{1}{\psi_{\mu}(\lambda)} \int_{\Bc(\lambda)} e^{ \langle w_0 \mu, \gamma(x) \rangle - f_B(x) } \varphi \circ \varrho^{T}(x) \omega(dx)\\
= & \E\left( \varphi \circ \varrho^{T}\left( C_{w_0 \mu}(\lambda) \right) \right)
\end{align*}
Notice that by taking $\varphi = 1$, we see that the Whittaker function defined using Brownian motion is the same as in  definition \ref{def:whittaker_functions}. It is known to be $W$-invariant (theorem \ref{thm:whittaker_functions_properties} (ii) ), allowing us to change $\psi_{w_0 \mu}$ to $\psi_\mu$.
\end{proof}

\section{Measure concentration}
\label{section:measure_concentration}
In this section, we investigate how the law of $C_\mu(\lambda)$ behaves as $\lambda$ goes to infinity in the opposite Weyl chamber. In subsection \ref{subsection:entrance_point}, this will be very important in order to complete the proofs of theorems \ref{thm:highest_weight_is_markov} and \ref{thm:canonical_measure}, by having the Whittaker process start at '$-\infty$'.

The following proposition is based on a weak version of the Laplace method. 
\begin{proposition}
\label{proposition:convergence_in_proba}
For given $\zeta \in \afrak$, $\mu \in \afrak$ and $M \rightarrow \infty$, we have convergence in probability for the $\Bc(\zeta)$-valued random variable:
$$ e^{-M \rho^\vee} C_\mu(\zeta - 2M \rho^\vee) e^{M \rho^\vee} \stackrel{\P}{\longrightarrow} m_\zeta$$
where $m_\zeta$ is the unique minimizer for the superpotential $f_B$ on $\Bc(\zeta)$.
\end{proposition}

\begin{corollary}
\label{corollary:convergence_in_proba}
With the same notations, as $M \rightarrow \infty$, the following limit holds in $B$:
$$ C_\mu(\zeta - 2M \rho^\vee) \stackrel{\P}{\longrightarrow} id$$
\end{corollary}
\begin{proof}
Recall that $\gamma(m_\zeta) = 0$ ( theorem \ref{thm:superpotential_minimum}) and therefore $m_\zeta \in N$. The result is obtained using proposition \ref{proposition:convergence_in_proba} and the following:
$$ \forall n \in N, e^{M \rho^\vee} n e^{-M \rho^\vee} \stackrel{\P}{\longrightarrow} id$$
\end{proof}

\begin{proof}[Proof of proposition \ref{proposition:convergence_in_proba}]
For easier notation write:
$$ X_M := e^{-M \rho^\vee} C_\mu(\zeta - 2M \rho^\vee) e^{M \rho^\vee}$$
Because the map $\kappa: y \mapsto e^{-M \rho^\vee} y e^{M \rho^\vee}$ maps $\Bc(\zeta - 2M \rho^\vee)$ to $\Bc(\zeta)$, this random variable is indeed in $\Bc(\zeta)$. That is a consequence of the straightforward computation:
\begin{align*}
  & hw( e^{-M \rho^\vee} y e^{M \rho^\vee} )\\
= & -w_0 M \rho^\vee + hw\left( y \right) + M \rho^\vee\\
= & -w_0 M \rho^\vee + \left( \zeta - 2M \rho^\vee \right) + M \rho^\vee\\
= & \zeta
\end{align*}

Let $\varphi$ be a positive test function on $\Bc(\zeta)$.  By the canonical measure's definition:
\begin{align*}
    \E\left( \varphi( X_M )\right)
& = \E\left( \varphi( e^{-M \rho^\vee} C_\mu(\zeta-2M \rho^\vee) e^{M \rho^\vee} ) \right)\\
& = \frac{1}{\psi_\mu(\zeta - 2M\rho^\vee)} \int_{\Bc(\zeta - 2M \rho^\vee)} e^{ \langle \mu, \gamma(y) \rangle - f_B(y)} \varphi\left( e^{-M \rho^\vee} y e^{M \rho^\vee} \right) \omega(dy) 
\end{align*}
Now, let us make the change of variables $x =  e^{-M \rho^\vee} y e^{M \rho^\vee} = \kappa(y)$. Since $\omega(dy)$ is by definition the toric measure on the twisted Lusztig parameters of $y$ and:
$$ \varrho^T(x) = e^{-M \rho^\vee} \varrho^T(y) e^{M \rho^\vee}$$
it is left unchanged. Hence the image measure for $\omega$ on $\Bc(\zeta-2M\rho^\vee)$ under this change of variable is again $\omega$ on $\Bc(\zeta)$:
\begin{align}
\label{eqn:convergence_in_proba_1}
\omega(dy) = \omega(dx) 
\end{align}
It is obvious that:
\begin{align}
\label{eqn:convergence_in_proba_2}
\gamma(y) &= \gamma(x)
\end{align}
And moreover:
\begin{align}
\label{eqn:convergence_in_proba_3}
f_B(y) &=  e^M f_B\left( x \right)
\end{align}
Indeed, by writing:
$$ y = u' \bar{w}_0 e^{\zeta - 2M \rho^\vee} u$$
where $u \in U^{w_0}_{>0}$, $u' \in U^{w_0}_{>0}$, we have:
$$ x = e^{-M \rho^\vee} u' e^{M \rho^\vee} \bar{w}_0 e^\zeta e^{-M \rho^\vee} u' e^{M \rho^\vee}$$
Therefore:
\begin{align*}
f_B(x) & = \chi\left( e^{-M \rho^\vee} u' e^{M \rho^\vee} \right) + \chi\left( e^{-M \rho^\vee} u' e^{M \rho^\vee} \right)\\
& = e^{-M} \chi\left( u \right) + e^{-M} \chi\left( u' \right)\\
& = e^{-M} f_B(y)
\end{align*}
In the end, putting equations \ref{eqn:convergence_in_proba_1}, \ref{eqn:convergence_in_proba_2} and \ref{eqn:convergence_in_proba_3} together yields the appropriate formula for the Laplace method:
\begin{align}
\label{eqn:convergence_in_proba_formula}
  \E\left( \varphi( X_M )\right)
& = \frac{ \int_{\Bc(\zeta)} e^{ \langle \mu, \gamma(x) \rangle - e^M f_B(x)} \varphi\left( x \right) \omega(dx) } 
         { \int_{\Bc(\zeta)} e^{ \langle \mu, \gamma(x) \rangle - e^M f_B(x)} \omega(dx) }
\end{align}

The following is quite standard. By theorem \ref{thm:superpotential_minimum}, $f_B$ has a unique minimizer on $\Bc(\zeta)$ denoted by $m_\zeta$. Consider $V$ a neighborhood of $m_\zeta$. Because $m_\zeta$ is a non-degenerate critical point, such a neighborhood contains a compact set of form:
$$ K_\delta := \left\{  x \in \Bc(\zeta) | f_B(x) \leq f_B(m_\zeta) + \delta \right\}$$
for $\delta$ small enough. We denote by $V^c$ the complement of $V$. The theorem is proved once the following holds:
$$ \P\left( X_M \in V^c \right) \stackrel{M \rightarrow \infty}{\longrightarrow} 0 $$
We have:
\begin{align*}
       \P\left( X_M \in V^c \right)
\leq & \P\left( X_M \in K_\delta^c \right)\\
=    & \frac{ \int_{K_\delta^c} e^{ \langle \mu, \gamma(x) \rangle - e^M f_B(x)} \omega(dx) }
            { \int_{\Bc(\zeta)} e^{ \langle \mu, \gamma(x) \rangle - e^M f_B(x)} \omega(dx) }\\
=    & \frac{1}{1+
       \frac{\int_{K_\delta  } e^{ \langle \mu, \gamma(x) \rangle - e^M (f_B(x)-f_B(m_\zeta)-\delta)} \omega(dx) }
            {\int_{K_\delta^c} e^{ \langle \mu, \gamma(x) \rangle - e^M (f_B(x)-f_B(m_\zeta)-\delta)} \omega(dx) } }
\end{align*}
In the ratio of two integrals, the numerator goes to infinity as $M \rightarrow \infty$ because for instance of the contribution of $K_{\frac{\delta}{2}} \subset K_{\delta}$:
\begin{align*}
& \int_{K_\delta  } e^{ \langle \mu, \gamma(x) \rangle - e^M (f_B(x)-f_B(m_\zeta)-\delta)} \omega(dx)\\
& \geq \int_{K_{\frac{\delta}{2}}} e^{ \langle \mu, \gamma(x) \rangle - e^M (f_B(x)-f_B(m_\zeta)-\delta)} \omega(dx)\\
& \geq \int_{K_{\frac{\delta}{2}}} e^{ \langle \mu, \gamma(x) \rangle - e^M \frac{\delta}{2} } \omega(dx)\\
& \rightarrow \infty
\end{align*}
The denominator decreases to zero as $M \rightarrow \infty$ using the dominated convergence theorem.
\end{proof}

As a consequence:

\begin{lemma}
Let $x_0 = \zeta - 2M \rho^\vee$ for any $\zeta \in \afrak$. Then as $M \rightarrow \infty$, in term of left $N$-orbits:
 $$ N e^{x_0} e^{\theta} \Gamma_{\mu}(x_0) e^{-\theta} \stackrel{ \P }{\longrightarrow} N e^{\theta - w_0 \theta} \bar{w}_0^{-1}$$
\end{lemma}
\begin{proof}
 Using proposition \ref{proposition:link_with_canonical}, we write:
$$ \Gamma_\mu(x_0) = \varrho^T\left( C_{w_0 \mu}(x_0) \right) = [\bar{w}_0^{-1} C_{w_0 \mu}(x_0)]_+$$
Let $X_M := e^{-M \rho^\vee } C_{w_0 \mu}(x_0) e^{M \rho^\vee }$. Thanks to proposition \ref{proposition:convergence_in_proba}, $X_M$ converges in probability to $m_\zeta$. Hence:
\begin{align*}
  & N e^{x_0} e^{\theta} \Gamma_{\mu}(x_0) e^{-\theta}\\
= & N e^{x_0 + \theta} [\bar{w}_0^{-1} e^{M \rho^\vee} X_M e^{-M \rho^\vee} ]_+ e^{-\theta}\\
= & N e^{x_0 + \theta + M \rho^\vee} [\bar{w}_0^{-1} X_M]_+ e^{-M \rho^\vee - \theta}\\
= & N e^{\zeta - M \rho^\vee + \theta } [\bar{w}_0^{-1} X_M]_{-0}^{-1} \bar{w}_0^{-1} X_M e^{-M \rho^\vee - \theta}\\
= & N e^{\zeta - M \rho^\vee + \theta } [\bar{w}_0^{-1} X_M]_{0}^{-1} \bar{w}_0^{-1} X_M e^{-M \rho^\vee - \theta}
\end{align*}
Moreover, $X_M \in \Bc(\zeta)$ and as such, $[\bar{w}_0^{-1} X_M]_{0}^{-1} = e^{-\zeta}$. Therefore:
\begin{align*}
  & N e^{x_0} \Gamma_\mu^\theta(x_0)\\
= & N e^{- M \rho^\vee + \theta } \bar{w}_0^{-1} X_M e^{-M \rho^\vee - \theta}\\
= & N e^{\theta - w_0 \theta} \bar{w}_0^{-1} e^{M \rho^\vee + \theta} X_M e^{-M \rho^\vee - \theta}
\end{align*}
The fact that $e^{M \rho^\vee + \theta} X_M e^{-M \rho^\vee - \theta} \rightarrow id$ concludes the proof.
\end{proof}

As a corollary, we can prove theorem \ref{thm:measure_concentration}.
\begin{proof}[Proof of theorem \ref{thm:measure_concentration}]
 Simply recall that the path transform $T_g \pi(t)$ on a path is defined for $t>0$ as:
$$ e^{T_g \pi(t)} = [ g B_t(\pi)]_0$$
Using the previous theorem, there is a sequence $n_M \in N$ such that as $M \rightarrow \infty$:
$$n_M e^{x_0} \Gamma_\mu^\theta(x_0) \stackrel{\P}{\rightarrow} e^{\theta - w_0 \theta} \bar{w}_0^{-1}$$
Hence:
\begin{align*}
  & \exp\left( x_0 + T_{e^{\theta} \Gamma_{\mu}(x_0) e^{-\theta}} \pi(t) \right)\\
= & [ e^{x_0} e^{\theta} \Gamma_{\mu}(x_0) e^{-\theta} B_t(\pi)]_0 \\
= & [ n_M e^{x_0} e^{\theta} \Gamma_{\mu}(x_0) e^{-\theta} B_t(\pi)]_0 \\
\rightarrow & \exp\left( \theta - w_0 \theta + \Tc_{w_0} \pi(t) \right)
\end{align*}

\end{proof}

\section{Intertwined Markov operators}
\label{section:intertwined_markov_kernels}
With theorem \ref{thm:whittaker_process_g}, we proved that for $\mu \in C$, if $W$ is a standard Brownian motion in $\afrak$ and $\Theta\left( g \right)$ independent following the law of $D_\mu\left( x_0 \right)$, then
$$ X^{x_0}_t = x_0 + T_{e^{\theta}g e^{-\theta}}\left( W^{ (w_0 \mu) } \right) ; t \geq 0$$
is Markovian, what we called the Whittaker process. The results of Rogers and Pitman in \cite{bib:RogersPitman} on Markov functions teach us that there should be an intertwining relation between the semi-groups of Brownian motion on the one hand, and the semi-group of the Whittaker process, using this remarkable law $D_\mu\left( x_0 \right)$. In fact, this is how the extensions of Pitman's theorem in \cite{bib:RogersPitman} and \cite{bib:OConnell} were proven. The only trick is that intertwining is easy to establish, once we know the answer.\\

What we did so far is identifying the right objects. Using intertwining Markov operators, we will strenghten the previous result to all possible drifts $\mu$ and not only for $\mu$ in the Weyl chamber. Then taking $x_0$ to '$-\infty$' will give us the highest weight process, finishing the proof of theorem \ref{thm:highest_weight_is_markov}.\\

Let us first quickly review the result of Pitman and Rogers on Markov function from \cite{bib:RogersPitman}.

\subsection{Markov functions }
Let $S$ and $S_0$ be topological spaces. Let $\phi: S \rightarrow S_0$ be a continuous function. Consider a Markov process $(X_t)_{t \geq 0}$ with state space $S$ and define the process $Y_t=\phi(X_t)$. We are interested in sufficient conditions that insure the Markov property for $Y$.

Of course one can suppose that $\phi$ is surjective by setting $S_0 = \phi(S)$. And clearly, in most cases of interest where $\phi$ is not injective, the inclusion between filtrations $\Fc^{Y} \subset \Fc^{X}$ is strict. Meaning that the observation of $Y$ contains only partial information on $X$. And in order to quantify this information, we need to ``filter'' $X$ through $\Fc^{Y}$.

In the sequel, we denote by $P_t$ the semi-group for $X$, and $Q_t$ the semi-group for $Y$, when it exists. 
$\Phi: \Cc(S_0) \rightarrow \Cc(S)$ is the Markov operator from $S$ to $S_0$ given by $\Phi(f)(x) = f \circ \phi(x)$. It just transports measures on $S$ to their image measure on $S_0$.

A first answer would be Dynkin's criterion, for cases where $Y$ is Markovian for all initial laws of $X$:
\begin{theorem}[Dynkin's criterion]
 If there exist a Markov operator $Q$ such that:
$$ \forall t \geq 0, P_t \circ \Phi = \Phi \circ Q_t $$
meaning, in terms of transporting measures, that the following diagram is commutative:
\begin{center}
\begin{tikzpicture}
\matrix(a)[matrix of math nodes, row sep=3em, column sep=3.5em,
text height=1.5ex, text depth=0.25ex]
{ S   & S  \\
  S_0 & S_0 \\ };
\path[->]
(a-1-1) edge node[auto] { $P_t$} (a-1-2)
(a-2-1) edge node[auto] { $Q_t$} (a-2-2)
(a-1-1) edge node[auto] { $\Phi$} (a-2-1)
(a-1-2) edge node[auto] { $\Phi$} (a-2-2);
\end{tikzpicture}
\end{center}
Then $Y$ is a Markov process and its semi-group is $Q_t$.
\end{theorem}

\begin{rmk}
In probabilistic terms, the condition that $P_t \circ \Phi(f)(x)$ only depends on $\phi(x)$ translates as saying that the law of $(\phi( X_t ) | X_0 = x)$ only depends of $\phi(x)$. The theorem seems then quite trivial. We wrote it that way to stress the intertwining.
\end{rmk}

Another solution has been formalized in \cite{bib:RogersPitman}. In some cases, if $Y$ starts at $y \in S_0$, it is Markovian only for specific entrance laws $\Kc(y, .)$ on $X$. In such a case, this initial law for $X$ is going to be the ``missing'' information from $\Fc^Y \subsetneq \Fc^X$.\\

Furthermore, at each time, we must ask the missing information to be stationnary in law, otherwise filtering $ X_t | \Fc_t^Y $ will give a fluctuating distribution and will not be able to extract the law of $X_t$ conditionnally to $\Fc^Y$, in such a way that it depends only on $Y_t$. One could speak of a ``Markovian stationary coupling'' or a ``Markovian filtering'' phenomenon, which brings the following equivalent definitions due to Rogers and Pitman:

\begin{thm}
\label{thm:intertwining}
Let $\Kc: \Cc(S) \rightarrow \Cc(S_0)$ be a Markov operator, $X$ a Markov process with semigroup $P_t$ and $Y = \phi(X)$. $Y$ is assumed to start at $y$. The following propositions are equivalent:
\begin{description}
 \item[$(i)$] ( Markovian filtering )
 $$\forall t \geq 0,  \forall f \in \Cc(S), \forall y \in S_0, \E_{ X_0 \sim \Kc(y,.) }\left( f( X_{t} ) | \Fc_{t}^Y \right) = \Kc(f)(Y_t) \quad a.s $$
where the subscript $X_0 \sim \Kc(y,.)$ indicates the initial law for $X$.
 \item[$(ii)$] ( Intertwining operators ) For all $t\geq 0$, $Q_t := \Kc \circ P_t \circ \Phi$ satisfies:
 $$ \Kc \circ \Phi = id_{S_0} $$
 $$ \Kc \circ P_t = Q_t \circ \Kc$$
Meaning, in terms of transporting measures, that the following diagram is commutative:
\begin{center}
\begin{tikzpicture}
\matrix(a)[matrix of math nodes, row sep=3em, column sep=3.5em,
text height=1.5ex, text depth=0.25ex]
{ S   & S   &  \\
  S_0 & S_0 & S_0 \\ };
\path[->]
(a-1-1) edge node[auto] { $P_t$} (a-1-2)
(a-2-1) edge node[auto] { $\Kc$} (a-1-1)
(a-2-2) edge node[auto] { $\Kc$} (a-1-2)
(a-2-1) edge node[auto] { $Q_t$} (a-2-2)
(a-2-2) edge node[auto] { $id$} (a-2-3)
(a-1-2) edge node[auto] { $\Phi$} (a-2-3);
\end{tikzpicture}
\end{center}

\end{description}
In both cases, $Q_t$ is a semi-group and is interpreted as:
$$ Q_t(f)(y) = \E_{ X_0 \sim \Kc(y,.)}\left( f( Y_t ) \right) $$
\end{thm}

\begin{proof}
The semi-group property of $Q_t$ is a consequence of $(ii)$:
\begin{align*}
  & Q_{t+s}\\
= & Q_{t+s} \circ \Kc \circ \Phi\\
= & \Kc \circ P_{t+s} \circ \Phi\\
= & \Kc \circ P_t \circ P_s \circ \Phi\\
= & Q_t \circ \Kc \circ P_s \circ \Phi\\
= & Q_t \circ Q_s \circ \Kc \circ \Phi\\
= & Q_t \circ Q_s
\end{align*}
$(i) \Rightarrow (ii):$ The first identity is easy. Indeed, let $g \in \Cc(S_0)$. By taking $f=g\circ \phi$ in $(i)$ and $t=0$, one gets:
$$ \Kc \circ \Phi (g)(y) = \E_{ X_0 \sim \Kc(y,.) }\left( \Phi (g)( X_{0} ) \right) = \E_{ X_0 \sim \Kc(y,.) }\left( g( Y_{0} ) \right)= g(y)$$
Concerning the second one:
 \begin{align*}
  & \Kc \circ P_t (f)(y)\\
= & \E_{ X_0 \sim \Kc(y, .) }\left( f( X_t ) \right)\\
= & \E_{ X_0 \sim \Kc(y, .) }\left( \E_{ X_0 \sim \Kc(y, .) }\left( f( X_t ) | \Fc_t^Y \right) \right)\\
= & \E_{ X_0 \sim \Kc(y, .) }\left( \Kc(f)(Y_t) \right)\\
= & \Kc \circ P_t \circ \Phi \circ \Kc (f)(y)\\
= & Q_t \circ \Kc (f)( y )
 \end{align*}

$(ii) \Rightarrow (i):$ Consider increasing times $0 \leq s_1 \leq s_2 \dots \leq s_n \leq t$ and test functions $g_1, \dots, g_n$ on $S_0$:
\begin{align*}
  & \E_{ X_0 \sim \Kc(y, .) }\left( g_1(Y_{s_1}) \dots g_n(Y_{s_n}) f(X_t) \right)\\
= & \Kc P_{s_1} \Phi(g_1) P_{s_2-s_1} \Phi(g_2) \dots P_{s_n-s_{n-1}} \Phi(g_n) P_{t-s_n} f (y)\\
= & Q_{s_1} \Kc \Phi(g_1) P_{s_2-s_1} \Phi(g_2) \dots P_{s_n-s_{n-1}} \Phi(g_n) P_{t-s_n} f (y)\\
= & Q_{s_1} g_1 Q_{s_2-s_1} g_2 \dots Q_{s_n-s_{n-1}} g_n \Kc P_{t-s_n} f (y)\\
= & Q_{s_1} g_1 Q_{s_2-s_1} g_2 \dots Q_{s_n-s_{n-1}} g_n Q_{t-s_n} \Kc f (y)\\
= & \E_{ X_0 \sim \Kc(y, .) }\left( g_1(Y_{s_1}) g_2( Y_{s_2} ) \dots g_n( Y_{s_n} ) \Kc(f)(Y_t) \right)
\end{align*}
This proves the Markovian filtering property.
\end{proof}

\begin{thm}[Pitman and Rogers criterion]
 \label{thm:RogersPitman}
 If the equivalent hypotheses previously cited are satisfied, take $X$ with initial law $\Kc(y, .)$ and $Y = \phi(X)$. Then $Y$ is a Markov process starting at $y$ and its semi-group is $Q_t$. 
\end{thm}
\begin{proof}
Let us prove the Markov property:
\begin{align*}
  & \E_{X_0 \sim \Kc(y, .)}\left( f( Y_{t+s} ) | \Fc_{s}^Y \right)\\
= & \E_{X_0 \sim \Kc(y, .)}\left( \E\left( f \circ \phi(X_{t+s}) | \Fc_{s}^X \right)  | \Fc_{s}^Y \right)\\
= & \E_{X_0 \sim \Kc(y, .)}\left( P_t\left( f \circ \phi \right)(X_s)  | \Fc_{s}^Y \right)\\
= & \E_{X_0 \sim \Kc(y, .)}\left( P_t \circ \Phi \left( f \right)(X_s)  | \Fc_{s}^Y \right)
\end{align*}
Using the hypothesis (i) from theorem \ref{thm:intertwining}, we have:
\begin{align*}
  & \E_{X_0 \sim \Kc(y, .)}\left( f( Y_{t+s} ) | \Fc_{s}^Y \right)\\
= & \Kc \circ P_t \circ \Phi \left( f \right)(Y_s)\\
= & Q_t \circ \Kc \circ \Phi \left( f \right)(Y_s)\\
= & Q_t (f)\left( Y_s \right) 
\end{align*}
\end{proof}

\subsection{The canonical measure intertwines the hypoelliptic BM and the highest weight process}

Now, let us specialize the previous framework to our case. The semi-group for the hypoelliptic Brownian motion $(B_t^\theta(W^{(\mu)}), t\geq 0)$ is denoted by $P_t$:
$$ P_t := \exp\left( t \Dc^{(\mu)} \right)$$
Recall that the highest weight process is:
$$\Lambda_t := hw( B_t^\theta(W^{(\mu)}) ) = \theta -  w_0 \theta + \Tc_{w_0} W^{(\mu)}_t $$
Finally, define the Markov kernel $\Kc_\mu: \Cc\left( \Bc \right) \rightarrow \Cc(\afrak)$ from $\afrak$ to $\Bc$ by:
$$\forall \varphi \in \Cc( \Bc ),  \Kc_\mu(\varphi)(\lambda) := \E\left( \varphi\left( C_\mu(\lambda) \right) \right)$$
Since the random variable $C_\mu(\lambda)$ is $\Bc(\lambda)$ valued, it is clear that:
$$ \Kc_\mu \circ hw = id_{\afrak}$$

The following 'Markovian filtering' holds:
\begin{thm}
\label{thm:markovian_filtering}
Let $x_0 \in \afrak$, $\mu \in \afrak$, $W^{(\mu)}$ a BM in the Cartan subalgebra $\afrak$ and $C_\mu(x_0)$ an independent random variable whose distribution follows the canonical probability measure on $\Bc(x_0)$, with spectral parameter $\mu$.
If:
$$ X^{x_0}_t := hw\left( C_\mu(x_0) B_t^\theta(W^{(\mu)}) \right)$$
and $f: \Bc \longrightarrow \R$ is a bounded function, then:
$$ \E\left( f\left( C_\mu(x_0) B_t^\theta(W^{(\mu)}) \right) | \Fc_t^{X^{x_0}}, X^{x_0}_t = x \right) = \E\left( f( C_\mu(x) ) \right)$$
\end{thm}
\begin{proof}
For notational reason, we write $\Fc_t$ instead of $\Fc_t^{X^{x_0}}$.

As a first step, let us prove that the theorem for general $\mu \in \afrak$ is a consequence of the case $\mu \in -C$ using a change of probability measure. Assume for now that the result is true for $\mu \in -C$. It is straightforward to check that for $\nu \in \afrak$:
$$ \E\left( f\left( C_\mu(x_0) \right) \right) = \frac{\psi_\nu(x_0)}{\psi_\mu(x_0)} \E\left( e^{\langle \gamma\left( C_\nu(x_0) \right), \mu - \nu \rangle } f\left( C_\nu(x_0) \right) \right)$$
and using the Girsanov-Cameron-Martin theorem for Brownian motion (\cite{bib:KaratzasShreve91} theorem 5.1), for any functional $F$:
$$ \E\left( F\left( W^{(\mu)}_s; s \leq t \right) \right) = 
   \E\left( e^{ \langle W^{(\nu)}, \mu - \nu \rangle - \half (||\mu||^2 - ||\nu||^2) t }
            F\left( W^{(\nu)}_s; s \leq t \right) \right)$$
Hence, because $C_\mu(x_0)$ and $W^{(\mu)}$ are independent, in the following change of probability, the density is the product of the two previous densities:
$$ \frac{\psi_\nu(x_0)}{\psi_\mu(x_0)} \exp\left( \langle \gamma\left( C_\nu(x_0) \right) + W^{(\nu)}, \mu - \nu \rangle - \half (||\mu||^2 - ||\nu||^2) t \right)$$
Using the Bayes formula, we have that, on the set $\{ X_t^{x_0} = x \}$:
\begin{align*}
  & \E\left( f\left( C_\mu(x_0) B_t^\theta(W^{(\mu)}) \right) | \Fc_t \right)\\
= & \frac{\E\left( e^{ \langle \gamma\left( C_\nu(x_0) \right) + W^{(\nu)}, \mu - \nu \rangle - \half (||\mu||^2 - ||\nu||^2) t } f\left( C_\nu(x_0) B_t^\theta(W^{(\nu)}) \right) | \Fc_t \right)}
         {\E\left( e^{ \langle \gamma\left( C_\nu(x_0) \right) + W^{(\nu)}, \mu - \nu \rangle - \half (||\mu||^2 - ||\nu||^2) t } | \Fc_t \right)}\\
= & \frac{\E\left( e^{ \langle \gamma\left( C_\nu(x_0) B_t^\theta(W^{(\nu)}) \right), \mu - \nu \rangle }
                   f\left( C_\nu(x_0) B_t^\theta(W^{(\nu)}) \right) | \Fc_t \right)}
         {\E\left( e^{ \langle \gamma\left( C_\nu(x_0) B_t^\theta(W^{(\nu)}) \right), \mu - \nu \rangle }
                   | \Fc_t \right)}
\end{align*}
Applying the result for $\nu \in -C$, one has:
\begin{align*}
  & \E\left( f\left( C_\mu(x_0) B_t^\theta(W^{(\mu)}) \right) | \Fc_t^{X^{x_0}}, X_t^{x_0} = x \right)\\
= & \frac{\E\left( e^{ \langle \gamma\left( C_\nu(x) \right), \mu - \nu \rangle } f\left( C_\nu(x) \right) \right)}
         {\E\left( e^{ \langle \gamma\left( C_\nu(x) \right), \mu - \nu \rangle } \right)}\\
= & \E\left( f\left( C_\mu(x) \right) \right)
\end{align*}
\linebreak
Now let us prove the theorem when $\mu \in -C = w_0 C$. Using proposition \ref{proposition:link_with_canonical}, we write:
$$C_\mu(x_0) = z \bar{w}_0 e^{x_0} g$$
where $g \stackrel{\Lc}{=} \Gamma_{w_0 \mu}(x_0)$. Now notice that:
\begin{align*}
X^{x_0}_t & = hw\left( C_\mu(x_0) B_t^\theta(W^{(\mu)}) \right)\\
& = \log [ \bar{w}_0^{-1} z \bar{w}_0 e^{x_0} g  B_t^\theta(W^{(\mu)}) ]_0\\
& = x_0 + \log [ g  B_t^\theta(W^{(\mu)}) ]_0\\
& = x_0 + \log [ e^{\theta}g e^{-\theta}  B_t(W^{(\mu)}) ]_0\\
& = x_0 + T_{e^{\theta}g e^{-\theta}} \left( W^{(\mu)} \right)
\end{align*}
For shorter notations introduce $n= N_t^\theta\left( X^{x_0} \right)$ and $x = X_t^{x_0}$. Thus we have:
$$ n e^x = \left[ e^{x_0} g B_t^\theta(W^{(\mu)}) \right]_{-0}$$
Then, using the properties of the Gauss decomposition:
\begin{align*}
  & B_t^\theta(W^{(\mu)}) \\
= & \left(e^{x_0} g\right)^{-1} e^{x_0} g B_t^\theta(W^{(\mu)})\\
= & \left(e^{x_0} g\right)^{-1} n e^x \left[e^{x_0} g B_t^\theta(W^{(\mu)})\right]_+
\end{align*}
But since $B_t^\theta(W^{(\mu)}) \in B$, we have:
$$ B_t^\theta(W^{(\mu)}) = \left[ \left(e^{x_0} g\right)^{-1} n e^x \right]_{-0}$$
And the following decomposition holds:
\begin{align}
e^{x_0} g B_t^\theta(W^{(\mu)}) & = e^{x_0} g \left[ \left(e^{x_0} g\right)^{-1} n e^x \right]_{-0}\\
& = n e^x \left[ \left(e^{x_0} g\right)^{-1} n e^x \right]_{+}^{-1}\\
& = n e^x \left[ \left(e^{x_0} g e^{-x_0}\right)^{-1} n e^x \right]_{+}^{-1}\\
& = n e^x \left[ \left( e^{-x} n^{-1} e^{x_0} g e^{-x_0} e^x \right)^{-1} \right]_{+}^{-1}
\end{align}
Therefore:
\begin{align*}
  & C_\mu(x_0) B_t^\theta( W^{(\mu)} )\\
= & b_x^T\left( \left[ \bar{w}_0^{-1} C_\mu(x_0) B_t^\theta( W^{(\mu)} ) \right]_+ \right)\\
= & b_x^T\left( \left[ e^{x_0} g B_t^\theta( W^{(\mu)} ) \right]_+ \right)\\
= & b_x^T\left( \left[ \left( e^{-x} n^{-1} e^{x_0} g e^{-x_0} e^x \right)^{-1} \right]_{+}^{-1} \right)
\end{align*}

Thanks to theorem \ref{thm:conditional_representation}, we know that:
$$ N_\infty^\theta(X^{x_0}) = \Theta( e^{x_0} g e^{-x_0} ) $$
Moreover, equation (\ref{eqn:N_infty_decomposition}) tells us:
$$ N_\infty^\theta(X^{x_0}) = n e^{x} N_\infty^\theta(X^{x_0}_{t+.} - X^{x_0}_t) e^{-x}$$
Hence, since $\Theta(ng) = n \Theta(g)$ for $n \in N$:
$$ N_\infty^\theta(X^{x_0}_{t+.} - X^{x_0}_t) = \Theta( e^{-x} n^{-1} e^{x_0} g e^{-x_0} e^{x} )$$
And:
$$ C_\mu(x_0) B_t^\theta( W^{(\mu)} ) = b_x^T\left( \left[ \Theta^{-1}\left( N_\infty^\theta(X^{x_0}_{t+.} - X^{x_0}_t) \right)^{-1} \right]_{+}^{-1} \right) $$
As $ \Theta^{-1}\left( N_\infty^\theta(X^{x_0}_{t+.} - X^{x_0}_t) \right) \in U$, we have in the end:
\begin{align}
\label{eqn:final_decomposition}
C_\mu(x_0) B_t^\theta( W^{(\mu)} ) & = b_x^T \circ \Theta^{-1}\left( N_\infty^\theta(X^{x_0}_{t+.} - X^{x_0}_t) \right)
\end{align}

Recall that under our working probability measure $g \stackrel{\Lc}{=} \Gamma_\mu(x_0)$ and $X^{x_0}$ follows the Whittaker process. In the context of proposition \ref{proposition:choice_v}, our working probability measure can be considered of the form $\P^v$. Under the equivalent probability measure $\P$, $g$ has the same law as $\Gamma_\mu$ and $X^{x_0}$ is distributed as a BM with drift $w_0 \mu \in C$.
$$  \frac{d \P^v}{d \P}            = \exp\left( -\chi_-\left( N_\infty^\theta(X^{x_0}) \right) \right)
                                     \frac{ b(w_0 \mu) }
                                          { \psi_{w_0 \mu}(x_0) e^{- \langle w_0 \mu, x_0 \rangle }} $$

$$ \frac{ d \P^v}{d \P}_{| \Fc_t } = \exp\left( -\chi_-\left( n \right) \right)
                                     \frac{ \psi_{w_0 \mu}(x  ) e^{- \langle w_0 \mu, x   \rangle }} 
                                          { \psi_{w_0 \mu}(x_0) e^{- \langle w_0 \mu, x_0 \rangle }} $$
Thus, we get the simplification:
\begin{align*}
  & \frac{d \P^v}{d \P} / \frac{ d \P^v}{d \P}_{| \Fc_t }\\
= & \frac{ b(w_0\mu) }{ \psi_{w_0\mu}(x) e^{- \langle w_0 \mu, x \rangle }}
    \frac{ \exp\left( -\chi_-\left( n e^x N_\infty^\theta(X^{x_0}_{t+.}-X^{x_0}) e^{-x} \right) \right) }
         { \exp\left( -\chi_-\left( n \right) \right) }\\
= & \frac{ b(w_0\mu) }{ \psi_{w_0\mu}(x) e^{- \langle w_0 \mu, x \rangle }}
    \exp\left( -\chi_-\left( e^x N_\infty^\theta(X^{x_0}_{t+.}-X^{x_0}) e^{-x} \right) \right)
\end{align*}
Therefore, on the set $\{ X_t^{x_0} = x \}$, by equation (\ref{eqn:final_decomposition}) and using the fact that $\E = \P^v$:
\begin{align*}
  & \E\left( f\left( C_\mu(x_0) B_t^\theta(W^{(\mu)}) \right) | \Fc_t\right)\\
= & \E\left( f \circ b_x^T \circ \Theta^{-1}\left( N_\infty^\theta(X^{x_0}_{t+.} - X^{x_0}_t) \right)
           | \Fc_t \right)\\ 
= & \P^v\left( f \circ b_x^T \circ \Theta^{-1}\left( N_\infty^\theta(X^{x_0}_{t+.} - X^{x_0}_t) \right)
             | \Fc_t \right)
\end{align*}
By the Bayes formula:
\begin{align*}
  & \E\left( f\left( C_\mu(x_0) B_t^\theta(W^{(\mu)}) \right) | \Fc_t \right)\\
= & \frac{ \P\left( \frac{d \P^v}{d \P}
                    f \circ b_x^T \circ \Theta^{-1}\left( N_\infty^\theta(X^{x_0}_{t+.} - X^{x_0}_t) \right)
                    | \Fc_t \right)}
         { \frac{ d \P^v}{d \P}_{| \Fc_t } }\\
= & \frac{ b(w_0\mu) }{ \psi_{w_0\mu}(x) e^{- \langle w_0 \mu, x \rangle }}
    \P\left( e^{-\chi_-\left( e^x N_\infty^\theta(X^{x_0}_{t+.}-X^{x_0}) e^{-x} \right)}
             f \circ b_x^T \circ \Theta^{-1}\left( N_\infty^\theta(X^{x_0}_{t+.} - X^{x_0}_t) \right)
             | \Fc_t \right)
\end{align*}

Since under $\P$, $X^{x_0}$ is a Brownian motion with drift $w_0 \mu$, we know that $N_\infty^\theta\left( X^{x_0}_{t+.} - X^{x_0}_{t} \right)$ is independent from $\Fc_t$ and has the same law as $D_{w_0 \mu}$. In the end:
\begin{align*}
  & \E\left( f\left( C_\mu(x_0) B_t^\theta(W^{(\mu)}) \right) | \Fc_t^{X^{x_0}}, X^{x_0}_t = x \right)\\
= & \frac{ b(w_0\mu) }{ \psi_{w_0\mu}(x) e^{- \langle w_0 \mu, x \rangle }}
    \E\left( \exp\left( -\chi_-\left( e^x D_{w_0 \mu} e^{-x} \right) \right)
             f \circ b_x^T \circ \Theta^{-1}\left( D_{w_0 \mu} \right) \right)\\
= & \E\left( f \circ b_x^T \circ \Theta^{-1}\left( D_{w_0 \mu}(x) \right) \right)\\
= & \E\left( f \circ b_x^T \left( \Gamma_{w_0 \mu}(x) \right) \right)
\end{align*}
Proposition \ref{proposition:link_with_canonical} yields the result by giving:
$$ C_\mu(x) \stackrel{\Lc}{=} b_x^T \left( \Gamma_{w_0 \mu}(x) \right)$$
\end{proof}

As a consequence, the condition (i) of theorem \ref{thm:intertwining} is valid with an initial law for the hypoelliptic Brownian motion being $C_\mu(x_0)$. Moreover
$$ Q_t := \Kc_\mu \circ P_t \circ hw$$
is a semi-group making the following diagram commutative.
\begin{center}
\begin{tikzpicture}
\matrix(a)[matrix of math nodes, row sep=3em, column sep=3.5em,
text height=1.5ex, text depth=0.25ex]
{ \Bc    & \Bc    &  \\
  \afrak & \afrak & \afrak \\ };
\path[->]
(a-1-1) edge node[auto] { $e^{t \Dc^{(\mu)}}$} (a-1-2)
(a-2-1) edge node[auto] { $\Kc_\mu$} (a-1-1)
(a-2-2) edge node[auto] { $\Kc_\mu$} (a-1-2)
(a-2-1) edge node[auto] { $Q_t$} (a-2-2)
(a-2-2) edge node[auto] { $id$} (a-2-3)
(a-1-2) edge node[auto] { $hw$} (a-2-3);
\end{tikzpicture}
\end{center}

The theorem \ref{thm:RogersPitman} is applicable and tells us that $X^{x_0}$ is Markov with semigroup $Q_t$. It can be easily identified:
\begin{proposition}
The semigroup $Q$ is generated by the Doob transform of the quantum Toda Hamiltonian:
$$ Q_t = \exp\left( t \Lc \right) $$
with:
$$ \Lc = \psi_\mu^{-1} (H - \half \langle \mu, \mu \rangle) \psi_\mu
       = \half \Delta_\afrak + \langle \nabla \log \psi_\mu, \nabla \rangle$$
\end{proposition}
\begin{proof}
When $\mu \in -C = w_0 C$, we are in the same situation as theorem \ref{thm:whittaker_process_g}, where we identified the infinitesimal generator as:
$$  \half \Delta_\afrak + \langle \nabla \log \psi_{w_0\mu}, \nabla \rangle$$
Hence the result as $\psi_{w_0\mu} = \psi_{\mu}$.\\
For general $\mu \in \afrak$, we use the fact that:
$$ \forall t \geq 0, Q_t = \Kc_\mu \circ P_t \circ hw$$
Therefore $Q_t$ has infinitesimal generator:
$$ \Lc_\mu = \Kc_\mu \circ \left( \half \Delta_\afrak + \langle \mu, \nabla_\afrak \rangle + \half \sum_{\alpha \Delta} \langle \alpha, \alpha \rangle f_\alpha \right) \circ hw$$
Against a smooth function $f: \afrak \rightarrow \R$, at a point $x$, $\Lc_\mu(f)(x)$ is analytic in the parameter $\mu$ and equal to
$$ \half \Delta_\afrak + \langle \nabla \log \psi_{\mu}, \nabla \rangle$$
for $\mu \in -C$. The result holds by analytic extension.
\end{proof}

\subsection{Entrance point at \texorpdfstring{'$-\infty$'}{minus infinity}}
\label{subsection:entrance_point}
Now it is quite easy to give a proof of theorems \ref{thm:highest_weight_is_markov} and \ref{thm:canonical_measure}.

\begin{proof}[Proof of theorems \ref{thm:highest_weight_is_markov} and \ref{thm:canonical_measure}]
Take in theorem \ref{thm:markovian_filtering} $x_0 = \zeta - 2M \rho^\vee$ and as in corollary \ref{corollary:convergence_in_proba} take $M \rightarrow \infty$, giving:
$$ C_\mu(\zeta - 2M \rho^\vee) \stackrel{\P}{\longrightarrow} id$$
The Markov process $X^{x_0}_t = hw\left( C_\mu(x_0) B_t^\theta(W^{(\mu)}) \right)$ will converge in probability to the highest weight process $\Lambda_t = hw\left( B_t^\theta(W^{(\mu)}) \right)$. The filtering equation in theorem \ref{thm:markovian_filtering} degenerates to the relation in theorem \ref{thm:canonical_measure}.

It also shows the Markov property in theorem \ref{thm:highest_weight_is_markov}, leaving only the fixed time marginal to prove. This fixed time marginal is the correct way of defining the entrance law. As we will see, it is a simple matter of diagonalizing the infinitesimal generator thanks to the Whittaker-Plancherel transform (theorem \ref{thm:whittaker_plancherel}). Let $g(t,.)$ be the density at time $t>0$ of the process $X^{x_0}_t$. It is obtained by solving the forward Kolmogorov equation ( or Fokker-Planck equation):
$$
\left\{ \begin{array}{ll}
\frac{\partial g}{\partial t} =  \Lc^* g \\
g(0,.) = \delta_{x_0}
\end{array} \right.
$$
where $\Lc^*$ is the dual of $\Lc$. Setting $g(t,x) = h(t,x) \psi_\mu(x) e^{-t\frac{\langle \mu, \mu \rangle}{2}}$, $h$ solves:
$$
\left\{ \begin{array}{ll}
\frac{\partial h}{\partial t} =  H h \\
h(0,.) = \psi_\mu(x_0)^{-1} \delta_{x_0}
\end{array} \right.
$$
where $H$ is the quantum Toda Hamiltonian. It is diagonalized by the the Whittaker-Plancherel transform. Indeed, since $H\psi_{i\nu} = -\half \langle \nu, \nu \rangle \psi_{i\nu}$, the function:
$$ \hat{h}(t,\nu) := \int_\afrak h(t, x) \psi_{i\nu}(x) dx$$
solves the simple PDE:
$$
\left\{ \begin{array}{ll}
\frac{\partial \hat{h}}{\partial t}(t, \nu) =  -\frac{\langle \nu, \nu \rangle}{2} \hat{h}(t,\nu) \\
\hat{h}(0,\nu) = \frac{\psi_{i\nu}(x_0)}{\psi_\mu(x_0)}
\end{array} \right.
$$
Hence:
$$ g(t,x) = \psi_\mu(x) e^{-\frac{t \langle \mu, \mu \rangle}{2}}
            \int_\afrak \tilde{h}(0, \nu) e^{-\frac{t \langle \nu, \nu \rangle}{2}} \psi_{-i\nu}(x) s(\nu) d\nu$$
Again, we use corollary \ref{corollary:convergence_in_proba} to finish the proof. It gives us as $M \rightarrow \infty$:
$$ \tilde{h}(0, \nu) = \E\left( e^{\langle i\nu - \mu, C_\mu(x_0) \rangle} \right) \rightarrow 1$$
\end{proof}

\subsection{Intertwining property at the torus level}
The geometric Duistermaat-Heckman measure intertwines Brownian motion and the quantum Toda Hamiltonian. Formally, introduce the Markov kernel $\hat{\Kc}_\mu: \Cc\left( \afrak \right) \rightarrow \Cc(\afrak)$ defined by:
\begin{align*}
 \forall \varphi \in \Cc( \afrak ),  \hat{\Kc}_\mu(\varphi)(\lambda) 
& := \E\left( \varphi \circ \gamma \left( C_\mu(\lambda) \right) \right)\\
& = \frac{1}{\psi_\mu(\lambda)} \int_\afrak e^{\langle \mu, \gamma(x) \rangle}f(y) DH^\lambda(dy)
\end{align*}
The following diagram is commutative.
\begin{center}
\begin{tikzpicture}
\matrix(a)[matrix of math nodes, row sep=3em, column sep=3.5em,
text height=1.5ex, text depth=0.25ex]
{ \afrak & \afrak \\
  \afrak & \afrak \\ };
\path[->]
(a-1-1) edge node[above] { $\exp\left( t\left( \half \Delta_\afrak + \langle \mu, \nabla_\afrak \rangle \right) \right)$} (a-1-2)
(a-2-1) edge node[left]  { $\hat{\Kc}_\mu$} (a-1-1)
(a-2-2) edge node[right] { $\hat{\Kc}_\mu$} (a-1-2)
(a-2-1) edge node[below] { $Q_t$} (a-2-2);
\end{tikzpicture}
\end{center}
This can be easily checked by applying the earlier intertwining:
$$\Kc_\mu \circ e^{t \Dc^{\mu}} = Q_t \circ \Kc_\mu$$
to functions depending on the weight only.

\chapter{Degenerations}
\label{chapter:degenerations}
As we have seen, there is a natural $q$-deformation of the geometric Littelmann path model that is given by rescaling paths and corresponding actions. We will interpret this deformation as a change of semi-fields and describe the deformed structures. The $q \rightarrow 0$ limit makes sense, and is exactly the free version of the continuous Littelmann model given in $\cite{bib:BBO2}$. A cutting procedure is needed in order to ``prune'' such a free Kashiwara crystal, and obtain a polytope. While Berenstein and Kazhdan have used the superpotential function $f_B$ (\cite{bib:BK00, bib:BK04}) to encode this cutting procedure, there is not a clear reason why it should be that way. In our point of view, the superpotential $f_B$ appeared naturally in the canonical measure on geometric crystals.

While describing deformations, we will see that the $q$-deformations of theorem \ref{thm:highest_weight_is_markov} uses the operator $q \Tc_{w_0} q^{-1} \stackrel{q \rightarrow 0}{\longrightarrow} \Pc_{w_0}$. This recovers the crystalline generalisation of Pitman's theorem proved in $\cite{bib:BBO}$ where $\Pc_{w_0} W$ is Brownian motion conditionned to never leave the Weyl chamber. In the $A_n$ type, this is a realization of Dyson's Brownian motion which gives a connection to Random Matrix theory.

Also, in this crystallization procedure, the canonical measure degenerates to the uniform measure on a polytope, which is nothing but the string polytope in the appropriate coordinates. This recovers previous results.\\

\section{Deformations}
\subsection{Semifields and Maslov quantification}
\label{subsection:semifields_and_maslov}
This subsection mainly follows the presentation of Itenberg in \cite{bib:Itenberg}. A semifield $\left( S, \oplus, \odot \right)$ is defined as the next best thing to a field, as we weaken the assumption of invertibility for the law $\oplus$. 

\begin{definition}
 A semifield is an algebraic structure $\left( S, \oplus, \odot \right)$ such that:
\begin{itemize}
 \item $\left( S, \oplus\right)$ is a commutative semigroup.
 \item $\left( S, \odot \right)$ is a commutative group with neutral element $e$.
 \item Distributivity of $\odot$ over $\oplus$:
	$$\forall a, b, c \in S, \left( a \oplus b \right) \odot c =  (a \odot c) \oplus (b \odot c)$$
\end{itemize}
\end{definition}

The universal semifield we have been working with so far is $\left(\R_{>0}, + , . \right)$. On this semi-field, the natural counterpart of rational functions with $n$ indeterminates is the set of \emph{rational and substraction free} expressions $\R_{>0}\left(x_1, \dots, x_n\right) $. Plainly, elements in $\R_{>0}\left(x_1, \dots, x_n\right)$ are rational functions with indeterminates $\left( x_1, \dots, x_n \right)$, real positive coefficients and using only operations $+$, $\times$ and $/$. For example $f(x_1, x_2) = \frac{x_1^3 + x_2^3}{x_1 + x_2} = x_1^2 - x_1 x_2 + x_2^2 \in \R_{>0}(x_1, x_2)$. It is easy to check that if endowed with the same operations, rational subtraction free expressions also form a semi-field.\\

Another classical example is the \emph{tropical} semifield $\left(\R, \min, + \right)$ as one easily checks that $+$ is distributive over $\min$. Its importance in representation theory is related to Kashiwara's crystal basis, as changes of coordinates are rational functions on $S_0$. The study of algebraic curves on this field has given rise to tropical geometry, now a field of its own, where $\max$ usually replaces $\min$. The name 'tropical' was coined by French computer scientists to honor their colleague Imre Simon for his work on the max-plus algebra. It has no intrinsic meaning aside from refering to the weather in Brazil.

In fact, we will see later that this semifield can be viewed like the zero temperature limit of family of semifields $S_q$. This suggests the name of ``crystallized'' semifield, that fits better in name to the crystal basis. However, it is too late to reverse the trend, already solidly established.\\

Tropicalization (or crystallization) is a procedure that takes as input objects on the semi-field $\left(\R_{>0}, +, . \right)$ and gives objects on $\left(\R_{>0}, \min, + \right) $. As such, if $f \in \R_{>0}(x_1, \dots, x_n)$, a substraction free rational function, one obtains $[f]_{trop}$ a function in the variables $\left(x_1, \dots, x_n\right)$ applying the morphism of semi-fields $[\ \ ]_{trop}$. If $a$ and $b$ are elements in $\R_{>0}(x_1, \dots, x_n)$ then:

$$ \begin{array}{cccc}
 [\ \ ]_{trop}: & \left(\R_{>0}, + , . \right) & \longrightarrow & \left(\R, \min, + \right) \\
                & a + b                        & \mapsto	 & \min( [\ a \ ]_{trop}, [\ b \ ]_{trop})\\
                & a . b                        & \mapsto	 & [\ a \ ]_{trop} + [\ b \ ]_{trop}\\
                & a / b                        & \mapsto 	 & [\ a \ ]_{trop} - [\ b \ ]_{trop}\\
                & a \in \R_{>0}                & \mapsto 	 & 0
    \end{array}
$$

A less algebraic definition could be used, using a limit that always exists:
\begin{proposition}
\label{proposition:analytic_tropicalization}
For $f$ a rational and substraction free expression in $k$ variables, we have for all $(x_1, \dots, x_k) \in \R^k$ and $q>0$:
\begin{align}
-q \log f\left( e^{-\frac{x_1}{q}}, \dots, e^{-\frac{x_k}{q}} \right) & = [f]_{trop}\left( x_1, \dots, x_k\right) + \Oc(q)
\end{align}
where $\Oc(q)$ is a quantity such that $\frac{\Oc(q)}{q}$ is bounded as $q \rightarrow 0$, uniformly in the variables $(x_1, \dots, x_k)$.
\end{proposition}
\begin{proof}
Let us prove the statement by induction on the size of the expression $f$, meaning the number of operations it uses (addition, multiplication and division). For the base case, notice that if $f$ is a monomial or a constant, then the statement is trivially true.\\
Now, for the inductive step, if $f$ is a product or ratio of two rational substraction free expressions, for which the statement is true, the statement carries on using the properties of the logarithm. If $f$ is a sum whose terms satisfies the induction hypothesis:
$$ f = f_1 + f_2$$
Then for ${\bf x} = \left( x_1, \dots, x_k \right) \in \R^k$:
\begin{align*}
  & -q \log f\left( e^{-\frac{x_1}{q}}, \dots, e^{-\frac{x_k}{q}} \right)\\
= & -q \log\left( f_1\left( e^{-\frac{x_1}{q}}, \dots, e^{-\frac{x_k}{q}} \right) 
                + f_2\left( e^{-\frac{x_k}{q}}, \dots, e^{-\frac{x_k}{q}} \right) \right)\\
= & -q \log\left( e^{ \Oc(1)-q^{-1} [f_1]_{trop}( {\bf x} )}
                + e^{ \Oc(1)-q^{-1} [f_2]_{trop}( {\bf x} )} \right)\\
= & \min\left( [f_1]_{trop}( {\bf x} ), [f_2]_{trop}( {\bf x} )\right) + \Oc(q)\\
  & \ \ -q \log\left( 1 + \exp\left(\Oc(1) - q^{-1} |[f_1]_{trop}-[f_2]_{trop}|( {\bf x} )\right) \right)\\
= & [f]_{trop}({\bf x}) + \Oc(q)
\end{align*}
\end{proof}

Such a limit suggests a continuous deformation from $\left(\R_{>0}, + , . \right)$ to $\left(\R, \min, + \right)$ called the Maslov quantification of real numbers. Define the continuous family of semifields 
$\left( S_q = \R, \oplus_q, \odot_q \right)$ for $q\geq 0$ with:
$$ a \oplus_q b = -q \log\left( e^{-\frac{a}{q}} + e^{-\frac{b}{q}}\right) $$
$$ a \odot_q b = a + b $$
At the limit, when $q$ goes to zero, we recover the previous example $\left(\R, \min, + \right)$. All the semifields $\left( S_q \right)_{q>0}$ are isomorphic to $\left(\R_{>0}, + , . \right)$ except for $q=0$. The isomorphism of semifields that transports structure is $\psi_q = -q \log: \left(\R_{>0}, + , . \right) \rightarrow \left( S_q = \R, \oplus_q, \odot_q \right)$\\
As such, $\psi_{q,q'} = \psi_{q'} \circ \psi_q^{-1}:\left( S_q = \R, \oplus_q, \odot_q \right) \rightarrow \left( S_{q'} = \R, \oplus_{q'}, \odot_{q'} \right)$ is a rescaling when identifying both semifields to $\R$: $\psi_{q,q'}\left( x \right) = \frac{q'}{q} x$

\begin{notation}
A tilde will refer to quantified variables when there is the possibility of confusing them with variables in $\R_{>0}$. In $c = e^{-\frac{\tilde{c}}{q}}$, $c$ is seen as a variable in the usual semi-field $\R_{>0}$ while $\tilde{c}$ is in $S_q$.
\end{notation}

\subsection{A remark on integrals of semifield valued functions}
Let $f: [0, T] \rightarrow S_q$ be a (smooth) function with values in the semifield $S_q$. For readability purposes, the subscript $q$ will be dropped when designating operations on $S_q$. The monoid of integers in $S_q$ is the monoid generated by the neutral element $0$. It is in fact given by all numbers $n_q = \mathop{\bigoplus}_{i=1}^n 0 = -q \log\left(n\right)$. 
As such, Riemann sums in $S_q$ take the form:
\begin{align*}
  & \left( 1 \oslash n_q \right) \odot \mathop{\bigoplus}_{i=1}^n  f(t_i)\\
= & -q \log\left( \frac{1}{n} \sum_{i=1}^n e^{-\frac{f(t_i)}{q}} \right)\\
\mathop{\longrightarrow}_{n \rightarrow +\infty} & -q \log\left( \int_0^T e^{-\frac{f}{q}} \right) 
\end{align*}
Therefore, the natural candidate for integrals on the semifield $S_q$ are exponential functionals and the $q=0$ limit gives 
$ \inf_{0 \leq s \leq T} f(s)$ using the Laplace method.

We will define a $q$-Littelmann model using exponential integrals over paths. It can be formulated in such a way that no minus sign appears. In the formalism of semi-fields, all actions become in fact rational, in the sense of the semi-field $S_q$. Exponential integrals are simply semifield integrals that degenerate to infimums.\\

\subsection{Deformed Lusztig and Kashiwara varieties}
Following the same idea as \cite{bib:BFZ96} section (2.2), one can define the Lusztig and Kashiwara varieties $U^{w_0}_{>0}$ and $C^{w_0}_{>0}$ by their parametrizations, by identifying $m$-tuples that give the same element. Since changes of parametrization $x_{\bf i'}^{-1} \circ x_{\bf i}$ are rational and substraction free, one can view them as rational for the semi-field $S_q$ and define:

\begin{definition}[ Lusztig and Kashiwara varieties on $S_q$ ]
  $$U_{>0}^{w_0}(S_q) := \left\{ ({\bf t}^{\bf i})_{{\bf i} \in R(w_0)} \in \left(S_q^m\right)^{R(w_0)} | \forall {\bf i}, {\bf i'} \in R(w_0), x_{\bf i'}^{-1} \circ x_{\bf i}( {\bf t^{\bf i}} ) = {\bf t^{\bf i'}} \right\}$$
  $$C_{>0}^{w_0}(S_q) := \left\{ ({\bf c}^{\bf i})_{{\bf i} \in R(w_0)} \in \left(S_q^m\right)^{R(w_0)} | \forall {\bf i}, {\bf i'} \in R(w_0), x_{\bf-i'}^{-1} \circ x_{\bf-i}( {\bf c^{\bf i}} ) = {\bf c^{\bf i'}} \right\}$$
\end{definition}

The $q \rightarrow 0$ limit gives the tropicalized version of the changes of parametrization. As such, by theorem 5.2 \cite{bib:BZ01}, the Lusztig variety $U_{>0}^{w_0}(S_0)$ really encodes the Lusztig parametrization of the $G^\vee$ canonical basis; while the tropical Kashiwara variety encodes the string parametrization.

\subsection{Deformed structure of Littelmann crystals}
We have seen that $q$-Littelmann models for different $q$ are equivalent, provided that we properly rescale the reals in the actions, and values taken by $\varepsilon_\alpha$ and $\varphi_\alpha$. In fact, the set of real numbers had to be considered as the semifield $S_q$, and this rescaling becomes natural as we also have to change the structure semifield. Now we are ready to list the $q$-deformation of our previous results.\\

In order to distinguish between structures at $q = 1$ and for $q$, let $L = \left( \gamma, \left( \varepsilon_\alpha, \varphi_\alpha, e^._\alpha \right)_{\alpha \in \Delta} \right)$ the path crystal structure for $q=1$ and $L_{q} = \left( \gamma', \left( \varepsilon^{'}_{\alpha}, \varphi^{'}_{\alpha}, e^{'.}_{\alpha} \right)_{\alpha \in \Delta} \right)$ for generic $q$. We will use the subscript $q$ in $\langle \pi \rangle_q$ to indicate the crystal generated by $\pi$ using the $q$-deformed structure.
 
\paragraph{Generated crystal:} Let $\langle \pi \rangle_q$ be the $q$-Littelmann crystal generated by $\pi$. After transporting the structure to $q=1$ by rescaling, we have to consider the geometric crystal generated by $q^{-1}\pi $. In the end, in term of the geometric structure (q=1), we have:
$$ \langle \pi \rangle_q = q \langle \frac{\pi}{q} \rangle$$

\paragraph{Highest weight:} The natural invariant under crystal action, which plays the role of highest weight, is then 
$$ \Tc^q_{w_0} := q \Tc_{w_0} q^{-1} $$
It is natural because varying $q$ interpolates between different path models, and gives for each $q$ the highest weight path. And it gives the rescaling considered in $\cite{bib:BBO}$ and $\cite{bib:BBO2}$ in order to recover the Pitman operator:
$$ \Pc_{w_0} = \lim_{q \rightarrow 0} q \Tc_{w_0} q^{-1} $$

Analogously, we have:
$$ \forall w \in W, \Tc_w^q := q \Tc_w q^{-1} \stackrel{q \rightarrow 0}{\longrightarrow} \Pc_{w}$$

\paragraph{Parametrizations:} Fix ${\bf i} \in R(w_0)$.\\
Transporting the semi-field structure from $\R_{>0}$ to $S_q$ and using the results from subsection \ref{subsection:string_params}, we define the $q$-deformed string parameter of a path as the $m$-tuple in  $\left(S_q\right)^m$ given by the map :
$$ \begin{array}{cccc}
\varrho_{\bf i}^{q, K}: & C_0\left( [0, T], \afrak \right) & \longrightarrow & \left(S_q\right)^m \\
                        &         \pi                      & \mapsto	     & \left( c_1, \dots, c_m \right)\\
    \end{array}
$$
For $\pi \in C_0\left( [0, T], \afrak \right)$, the $m$-tuple ${\bf c} = \left( c_1, \dots, c_m \right) = \varrho_{\bf i}^{q, K}(\pi)$ is defined recursively as:
$$ \forall 1 \leq k \leq m, c_k = q \log\int_0^T \exp\left( -q^{-1}\alpha_{i_k}\left( \Tc_{s_{i_1} \dots s_{i_{k-1}}}^q \pi \right) \right)$$
Clearly, as $q \rightarrow 0$, one recovers the definition of string parameters in the classical Littelmann path model (see \cite{bib:BBO2}):
$$ \forall 1 \leq k \leq m, q \log\int_0^T \exp\left( -q^{-1}\alpha_{i_k}\left( \Tc_{s_{i_1} \dots s_{i_{k-1}}}^q \pi \right) \right) \rightarrow - \inf_{0 \leq t \leq T} \alpha_{i_k}\left( \Pc_{s_{i_1} \dots s_{i_{k-1}}}(\pi) \right)$$
Finally, thanks to diagram \ref{fig:parametrizations_diagram} and the morphism of semi-fields $\psi_q = -q \log$  we have:
\begin{align}
\label{eqn:q_kashiwara_for_paths}
\forall \pi \in C_0\left( [0, T], \afrak \right), \varrho_{\bf i}^{q, K}(\pi) = \psi_q \circ x_{\bf -i}^{-1} \circ \varrho^K\left( B_T(q^{-1} \pi) \right)
\end{align}
where we applied the semi-field morphism $\psi_q = -q \log$ on $\R_{>0}^m$ point-wise.\\

Similarly, $q$-deformed Lusztig parameters are constructed. Indeed, all elements of $\langle \pi \rangle_q$ can be projected on the lowest path $\eta = q e^{-\infty}_\alpha q^{-1} \pi$ and every single path can be recovered via:
$$ \pi = q T_g q^{-1} \eta$$
where 
\begin{itemize}
 \item $ g = x_{i_1}\left( e^{-\frac{t_1}{q}} \right) \dots x_{i_m}\left( e^{-\frac{t_1}{q}} \right) \in U^{w_0}_{>0}$
 \item $\eta_j = q e^{-\infty}_{s_{i_1} \dots s_{i_j}} \cdot \frac{\pi}{q} = q e^{-\infty}_{\alpha_{i_j}} \cdot \frac{\eta_{j-1}}{q}$
 \item $t_j = q \log\int_0^T \exp\left( -q^{-1}\alpha_{i_k}\left( q e_{s_{i_1} \dots s_{i_{k-1}}}^{-\infty} q^{-1} \pi \right) \right)$
\end{itemize}

We are aiming at understanding the law $\varrho_{\bf i}^{q, K}\left( \pi \right)$ when $\pi$ is taken as a Brownian motion. This is the natural $q$-deformation of the previously studied canonical measure, viewed in Kashiwara coordinates.

\subsection{Brownian scaling and consequences}
In order to obtain the announced deformation of our probabilistic results, the tool we will use is simply Brownian scaling property. For $W$ a Brownian motion in $\afrak$, $\mu \in \afrak$ and $c>0$, it is the equality in law between processes:
\begin{align}
\label{eqn:brownian_scaling}
W_{t}^{(\mu)}; t \geq 0 \stackrel{\Lc}{=} c W_{t/c^2}^{(c\mu)}; t \geq 0
\end{align}

Let us first examine the effect of scaling on the flow $B_.(.)$:
\begin{lemma}[Effect of accelerating a path $X$ on $B_.\left(X\right)$]
\label{lemma:accelerating_paths}
Given a continuous path $X$ in $\afrak$:
 $$ B_t\left( X_{./c^2} \right) = c^{-2\rho^\vee} B_{t/c^2}\left( X \right) c^{2\rho^\vee}$$
\end{lemma}
\begin{proof}
Using the change of variable $u_j = t_j/c^2$, we have:
 \begin{align*}
 B_t\left( X_{\frac{.}{c^2}} \right) = & \left( \sum_{k \geq 0} \sum_{ i_1, \dots, i_k } \int_{ t \geq t_k \geq \dots \geq t_1 \geq 0} e^{ -\alpha_{i_1}(X_{t_1/c^2}) \dots -\alpha_{i_k}(X_{t_k/c^2}) } f_{i_1} \dots f_{i_k} dt_1 \dots dt_k \right) e^{X_{\frac{t}{c^2}}}\\
 = & \left( \sum_{k \geq 0} \sum_{ i_1, \dots, i_k } \int_{ \frac{t}{c^2} \geq t_k \geq \dots \geq t_1 \geq 0} e^{ -\alpha_{i_1}(X_{u_1}) \dots -\alpha_{i_k}(X_{u_k}) } f_{i_1} \dots f_{i_k} du_1 \dots du_k c^{2k} \right) e^{X_{\frac{t}{c^2}}}\\
 = & \left( \sum_{k \geq 0} \sum_{ i_1, \dots, i_k } \int_{ \frac{t}{c^2} \geq u_k \geq \dots \geq u_1 \geq 0} e^{ -\alpha_{i_1}(X_{u_1}) \dots -\alpha_{i_k}(X_{u_k}) } c^{-2\rho^\vee} f_{i_1} \dots f_{i_k} c^{2\rho^\vee} du_1 \dots du_k \right) e^{X_{\frac{t}{c^2}}}\\
 = & c^{-2\rho^\vee} B_{\frac{t}{c^2}}\left( X \right) c^{2\rho^\vee}
 \end{align*}
\end{proof}
Hence:
\begin{lemma}
\label{lemma:B_scaling}
$$ B_{t}\left( q^{-1} W^{(\mu)} \right) ; t \geq 0 \stackrel{\Lc}{=} q^{-2\rho^\vee}  B_{t/q^2}\left( W^{(q \mu)} \right) q^{2\rho^\vee}; t \geq 0$$
\end{lemma}
\begin{proof}
 Use the scaling equation (\ref{eqn:brownian_scaling}) with $c=q$ and lemma \ref{lemma:accelerating_paths}.
\end{proof}

Therefore, we can give a deformation of theorems \ref{thm:highest_weight_is_markov} and \ref{thm:canonical_measure}. Define the rescaled highest weight process as:
$$ \forall t>0, \Lambda_t^q := q \ hw\left( B_t^\theta( q^{-1} W^{(\mu)} ) \right) = q(\theta - w_0 \theta) + \Tc^q_{w_0}(W^{(\mu)})_t$$
The properly rescaled Whittaker function on $\afrak$ is, with $m = \ell(w_0)$:
$$\forall \lambda \in \afrak, \psi_{q, \mu}(\lambda) = q^{m}\psi_{q\mu}\left( \frac{\lambda-4 q \log(q) \rho^\vee}{q}\right)$$
Using theorem \ref{thm:whittaker_functions_properties}, it is immediate that when $\mu \in C$, $\psi_{q, \mu}$ solves the eigenfunction equation:
\begin{align}
\label{eqn:q_deformed_toda}
\frac{1}{2} \Delta \psi_{q, \mu} - \sum_{\alpha \in \Delta} \frac{1}{2} \langle \alpha, \alpha \rangle q^2 e^{-q^{-1} \alpha\left( . \right)}\psi_{q, \mu}  = & \frac{ \langle \mu, \mu \rangle }{2} \psi_{q, \mu} 
\end{align}
with $\psi_{q, \mu}(x) e^{-\langle \mu, x \rangle}$ being bounded and having growth condition:
$$ \lim_{x \rightarrow \infty, x \in C} \psi_{q, \mu}(x) e^{-\langle \mu, x \rangle} = q^{m} b(q \mu)$$

\begin{thm}[Markov property for rescaled highest weight]
\label{thm:q_highest_weight_is_markov}
The process $\Lambda^q$ is a diffusion with infinitesimal generator
$$ \psi_{q, \mu}^{-1}
   \left( \frac{1}{2} \Delta - \sum_{\alpha \in \Delta} \frac{1}{2} \langle \alpha, \alpha \rangle q^2 e^{-q^{-1} \alpha\left( x \right)} - \frac{ \langle \mu, \mu \rangle }{2} \right)
   \psi_{q, \mu}
 = \frac{1}{2} \Delta + \nabla \log\left(\psi_{q, \mu}\right) \cdot \nabla $$
\end{thm}
\begin{proof}
Thanks to lemma \ref{lemma:B_scaling} and properties \ref{properties:hw}:
$$ \Lambda_t^q ; t \geq 0 \stackrel{\Lc}{=} 4 q \log(q) \rho^\vee + q \ hw\left(B_{t/q^2}^\theta\left( W^{(q \mu)} \right) \right) ; t \geq 0$$
The result is a consequence of theorem \ref{thm:highest_weight_is_markov} and the following general fact applied to the highest weight process. Consider an Euclidian space $V$ and $a \in V$. If $(X_t; t \geq 0)$ is a diffusion on $V$ with generator $\Lc$ then $(q X_{t/q^2} + a ; t \geq 0)$ is a diffusion with generator $\Gc$. For a smooth function $f: V \rightarrow \R$, we have $\Gc(f): x \mapsto q^{-2} \Lc\left( f(q. + a)\right)(\frac{x-a}{q})$. Here, one needs to take $V = \afrak$, $X$ is the highest weight process and $a = 4 q \log(q) \rho^\vee$.\\
\end{proof}

The deformation of theorem \ref{thm:canonical_measure} is:
\begin{thm}[Rescaled canonical measure]
\label{thm:q_canonical_measure}
 $$\forall t>0, \left( B_t^\theta\left( q^{-1} W^{(\mu)} \right) | \Lambda_t^q = \lambda \right)
\stackrel{\Lc}{=} q^{-2\rho^\vee} C_{q\mu}\left( \frac{\lambda - 4 q\log(q) \rho^\vee}{q} \right) q^{2\rho^\vee}
$$
\end{thm}
\begin{proof}
Using lemma \ref{lemma:B_scaling}, we have for fixed $t>0$:
\begin{align*}
                  & \left( B_t^\theta\left( q^{-1} W^{(\mu)} \right) | \Lambda_t^q = \lambda \right)\\
\stackrel{\Lc}{=} & \left( q^{-2\rho^\vee} B_{t/q^2}^\theta\left( W^{(q \mu)} \right) q^{2\rho^\vee} |
                hw\left( q^{-2\rho^\vee} B_{t/q^2}^\theta\left( W^{(q \mu)} \right) q^{2\rho^\vee} \right)= q^{-1} \lambda \right)\\
               =  & \left( q^{-2\rho^\vee} B_{t/q^2}^\theta\left( W^{(q \mu)} \right) q^{2\rho^\vee} |
                hw\left( B_{t/q^2}^\theta\left( W^{(q \mu)} \right) \right)= q^{-1} \lambda - 4 \log q \rho^\vee\right)\\
\end{align*}
Combining this with theorem \ref{thm:canonical_measure} yields the result.
\end{proof}

\subsection{Explicit computation in string coordinates}

We are now able to give an integral formula for the law of $q$-deformed string parameters extracted from a finite Brownian path. We present it in a form that allows to compute the $q\rightarrow 0$ limit. It uses the map $\eta^{w_0, e}$ defined as:
$$ \forall v \in B \cap U \bar{w}_0 U, \eta^{w_0, e}\left( v \right) = [\left( \bar{w}_0 v^T\right)^{-1}]_+$$
Recall that $\eta^{w_0, e}$ restricts to a bijection from $C_{>0}^{w_0}$ to $U_{>0}^{w_0}$ (theorem \ref{thm:geom_from_lusztig_to_kashiwara}).

\begin{proposition}
\label{proposition:q_string_formula}
Let ${\bf i} \in R(w_0)$ and $t>0$. Consider in $\afrak$ a Brownian motion with drift $\mu$ up to time $t$, $\left( W^{(\mu)}_u; 0 \leq u \leq t \right)$. Then for any $\varphi: \R^m \rightarrow \R$ bounded measurable function:
\begin{align*}
  & \E\left( \varphi\left( \varrho_{\bf i}^{q, K}\left( W^{(\mu)}_u; 0 \leq u \leq t \right) \right) \ | \ \Tc_{w_0}^q W^{(\mu)}_t = \lambda \right)\\
= &  \frac{1}{\psi_{q, \mu}\left(\lambda+q(w_0 \theta - \theta)\right)}
     \int_{\R^m} {\bf dc} \varphi({\bf c}) \exp\left( \langle \mu, \lambda - \sum_{k=1}^m c_k \alpha_{i_k}^\vee \rangle- f_{B,q,\lambda}^{K, {\bf i}}({\bf c})\right)
\end{align*}
where ${\bf dc}$ is the Lebesgue measure on $\R^m$ and the deformed superpotential in string coordinates is given by:
\begin{align}
f_{B,q,\lambda}^{K, {\bf i}}({\bf c}) 
& := \sum_{\alpha \in \Delta} \frac{2 q^2}{\langle \alpha, \alpha \rangle} \chi_\alpha \circ \eta^{w_0, e} \circ x_{\bf-i}\left( e^{-q^{-1}c_1}, \dots, e^{-q^{-1}c_m} \right)\\
& \ \ + \sum_{j=1}^m \frac{2 q^2}{\langle \alpha_{i_j}, \alpha_{i_j} \rangle} \exp\left( -\frac{\lambda-\left(c_j+\sum_{k=j+1}^m c_k \alpha_{i_j}(\alpha_{i_k}^\vee)\right)}{q} \right)
\end{align}
\end{proposition}
\begin{proof}
Using equation (\ref{eqn:q_kashiwara_for_paths}), while writing $f = \varphi \circ \psi_q \circ x_{- \bf i}^{-1} \circ \varrho^K$:
\begin{align*}
  & \E\left( \varphi\left( \varrho_{\bf i}^{q, K}\left( W^{(\mu)}_u; 0 \leq u \leq t \right) \right) \ | \ \Tc_{w_0}^q W^{(\mu)}_t = \lambda \right)\\
= & \E\left( f\left( B_t(W^{(\mu)}) \right) \ | \ \Tc_{w_0}^q W^{(\mu)}_t = \lambda \right)
\end{align*}
As a consequence of theorem \ref{thm:q_canonical_measure} and then the integral formula from equation (\ref{eqn:canonical_measure_density}), we have:
\begin{align*}
  & \E\left( \varphi\left( \varrho_{\bf i}^{q, K}\left( W^{(\mu)}_u; 0 \leq u \leq t \right) \right) \ | \ \Tc_{w_0}^q W^{(\mu)}_t = \lambda \right)\\
= & \E\left( f\left( e^{\theta} \rho^{-2\rho^\vee} C_{q\mu }(\frac{\lambda - 4 q \log q \rho^\vee + q(\theta-w_0 \theta)}{q}) \rho^{2\rho^\vee} e^{-\theta} \right)  \right)\\
= & \frac{q^m}{\psi_{q, \mu}\left(\lambda + q (\theta-w_0\theta) \right)}
    \int_{\Bc\left( \frac{\lambda - 4 q \log q \rho^\vee + q(\theta-w_0 \theta)}{q}\right)}
    f( e^{\theta} \rho^{-2\rho^\vee} x \rho^{2\rho^\vee} e^{-\theta} ) e^{\langle q\mu, \gamma(x) - f_B(x)} \omega(dx)
\end{align*}
Making the change of variable $y = e^{\theta} \rho^{-2\rho^\vee} x \rho^{2\rho^\vee} e^{-\theta}$, which maps $\Bc\left( \frac{\lambda - 4 q \log q \rho^\vee + q(\theta-w_0 \theta)}{q}\right)$ to $\Bc\left( \frac{\lambda}{q}\right)$:
\begin{align*}
  & \E\left( \varphi\left( \varrho_{\bf i}^{q, K}\left( W^{(\mu)}_u; 0 \leq u \leq t \right) \right) \ | \ \Tc_{w_0}^q W^{(\mu)}_t = \lambda \right)\\
= & \frac{q^m}{\psi_{q, \mu}\left(\lambda + q (\theta-w_0\theta) \right)}
    \int_{\Bc\left( q^{-1} \lambda \right)}
    f( y ) e^{\langle q \mu, \gamma(y) - f_B(e^{-\theta} \rho^{2\rho^\vee} y \rho^{-2\rho^\vee} e^{\theta})} \omega(dy)
\end{align*}
And for:
$${\bf c} = \psi_q \circ x_{- \bf i}^{-1} \circ \varrho^K(y)$$
or equivalently
$$y = b_{q^{-1}\lambda}^K \circ x_{- \bf i}\left( e^{-\frac{c_1}{q}}, \dots, e^{-\frac{c_m}{q}}\right)$$
We explicit the previous integral in terms of the variable ${\bf c}$. Theorem \ref{thm:omega_invariance_on_crystal} leads to:
\begin{align}
\label{eqn:explicit_measure}
\omega(dy) & = q^m \prod_{k=1}^m d c_k = q^m {\bf dc}
\end{align}
Theorem \ref{thm:geom_weight_map} gives:
\begin{align}
\label{eqn:explicit_weight}
\gamma(y) & = \frac{1}{q}\left( \lambda - \sum_{k=1}^m c_k \alpha_{i_k}^\vee \right)
\end{align}
And as we will see:
\begin{align}
\label{eqn:explicit_superpotential}
f_B(e^{-\theta} \rho^{2\rho^\vee} y \rho^{-2\rho^\vee} e^{\theta}) & = f_{B, q, \lambda}^{K, {\bf i}}({\bf c})
\end{align}
Putting together equations (\ref{eqn:explicit_measure}), (\ref{eqn:explicit_weight}) and (\ref{eqn:explicit_superpotential}) yields the result. Now, we only need to prove the last equation. Recall that, by writing $v = x_{- \bf i}\left( e^{-\frac{c_1}{q}}, \dots, e^{-\frac{c_m}{q}}\right) \in C_{>0}^{w_0}$ and using proposition \ref{proposition:crystal_param_maps}:
\begin{align*}
y = & b_{q^{-1}\lambda}^K \circ x_{- \bf i}\left( e^{-\frac{c_1}{q}}, \dots, e^{-\frac{c_m}{q}}\right)\\
  = & b_{q^{-1}\lambda}^K (v)\\
  = & \eta^{w_0, e}\left( v \right) \bar{w}_0 v^T [v^T]_0^{-1} e^{q^{-1} \lambda}
\end{align*}
Therefore:
\begin{align*}
  & f_B(e^{-\theta} \rho^{2\rho^\vee} y \rho^{-2\rho^\vee} e^{\theta})\\
= & f_B(e^{-\theta} \rho^{2\rho^\vee} \eta^{w_0, e}\left( v \right) \bar{w}_0 v^T [v^T]_0^{-1} e^{q^{-1} \lambda} \rho^{-2\rho^\vee} e^{\theta})\\
= & \chi\left( e^{-\theta} \rho^{2\rho^\vee} \eta^{w_0, e}\left( v \right) \rho^{-2\rho^\vee} e^{\theta} \right)
  + \chi\left( e^{-\theta} \rho^{2\rho^\vee} e^{-q^{-1} \lambda} v^T [v^T]_0^{-1} e^{q^{-1} \lambda} \rho^{-2\rho^\vee} e^{\theta} \right) \\
= & \sum_{\alpha \in \Delta} \frac{2 q^2}{\langle \alpha, \alpha \rangle} \chi_\alpha \circ \eta^{w_0, e}(v)
  + \sum_{j=1}^m \frac{2 q^2}{\langle \alpha_{i_j}, \alpha_{i_j} \rangle} \exp\left( -\frac{\lambda-\left(c_j+\sum_{k=j+1}^m c_k \alpha_{i_j}(\alpha_{i_k}^\vee)\right)}{q} \right)
\end{align*}
\end{proof}

\section{Crystallization}
As the following proposition indicates, the superpotential degenerates to an indicator function of a polytope. It is the string polytope for the Langlands dual $G^\vee$, thereby recovering where the string parameters sit for the usual highest weight Kashiwara crystals $\Bfrak(\lambda)$.  

\begin{proposition}
For every ${\bf i} \in R(w_0)$, there is a cone $\Cc_{\bf i}^\vee$ such that:
$$ \forall {\bf c} \in \R,
   \lim_{q \rightarrow 0} e^{-f_{B, q, \lambda}^{K, {\bf i}}({\bf c})}
 = \mathds{1}_{\left\{ {\bf c} \in \Cc_{\bf i}^\vee\right\}} \mathds{1}_{\left\{ c_j \leq \lambda-\sum_{k=j+1}^m c_k \alpha_{i_j}(\alpha_{i_k}^\vee)\right\}}$$
It is the string cone for the group $G^\vee$ and it is given for any choice of ${\bf i'} \in R(w_0)$ by:
$$ \Cc_{\bf i}^\vee = \left\{ {\bf c} \in \R^m \ | \ [x_{\bf i'}^{-1} \circ \eta^{w_0, e} \circ x_{\bf-i}]_{trop}({\bf c}) \in \R_+^m \right\} $$
\end{proposition}
Notice the appearance of exactly the same cutting condition of the string cone $\Cc_{\bf i}^\vee$ as the condition given in \cite{bib:Littelmann} page 5, giving the string polytope associated to the highest weight $\lambda$.
\begin{proof}
Looking at proposition \ref{proposition:q_string_formula}, the deformed superpotential $e^{-f_{B, q, \lambda}^{K, {\bf i}}({\bf c})}$ is the product of two terms. Each one of them leads to an indicator function. The easier one to deal with is:
\begin{align*}
& \exp\left( - \sum_{j=1}^m \frac{2 q^2}{\langle \alpha_{i_j}, \alpha_{i_j} \rangle} \exp\left( -\frac{\lambda-\left(c_j+\sum_{k=j+1}^m c_k \alpha_{i_j}(\alpha_{i_k}^\vee)\right)}{q} \right) \right)\\
& \stackrel{q \rightarrow 0}{\rightarrow} \mathds{1}_{\left\{ c_j \leq \lambda-\sum_{k=j+1}^m c_k \alpha_{i_j}(\alpha_{i_k}^\vee)\right\}}
\end{align*}
For the other term:
$$ \exp\left( - \sum_{\alpha \in \Delta}^m \frac{2 q^2}{\langle \alpha, \alpha \rangle} \chi_\alpha \circ \eta^{w_0, e} \circ x_{\bf-i}\left( e^{-q^{-1}c_1}, \dots, e^{-q^{-1}c_m} \right) \right) $$
Start by choosing a reduced word ${\bf i'} \in R(w_0)$, independently of ${\bf i} \in R(w_0)$. This choice will not play any role. Let $f$ be the rational substraction free function given by (theorem \ref{thm:tropical_from_lusztig_to_kashiwara}):
$$ f = x_{\bf i'}^{-1} \circ \eta^{w_0, e} \circ x_{\bf-i}: \R_{>0}^m \rightarrow \R_{>0}^m $$
Component-wise, we write $f = \left( f_1, \dots, f_m \right)$. Then, after organizing that term and using the analytic  tropicalization procedure given in proposition (\ref{proposition:analytic_tropicalization}):
\begin{align*}
  & \exp\left( - \sum_{\alpha \in \Delta}^m \frac{2 q^2}{\langle \alpha, \alpha \rangle} \chi_\alpha \circ \eta^{w_0, e} \circ x_{\bf-i}\left( e^{-q^{-1}c_1}, \dots, e^{-q^{-1}c_m} \right) \right)\\
& = \exp\left( - \sum_{j=1}^m \frac{2 q^2}{\langle \alpha_{i_j'}, \alpha_{i_j'} \rangle} f_j\left( e^{-q^{-1}c_1}, \dots, e^{-q^{-1}c_m} \right) \right)\\
& = \exp\left( - \sum_{j=1}^m \frac{2 q^2}{\langle \alpha_{i_j'}, \alpha_{i_j'} \rangle} \exp\left( -\frac{-q\log f_j\left( e^{-q^{-1}c_1}, \dots, e^{-q^{-1}c_m} \right)}{q} \right) \right)\\
& = \exp\left( - \sum_{j=1}^m \frac{2 q^2}{\langle \alpha_{i_j'}, \alpha_{i_j'} \rangle} \exp\left( -\frac{[f_j]_{trop}({\bf c}) + \Oc(q)}{q} \right) \right)\\
& \stackrel{q \rightarrow 0}{\rightarrow} \mathds{1}_{\left\{ {\bf c} \in \R^m \ | \ \forall 1 \leq j\leq m, [f_j]_{trop}({\bf c}) \geq 0 \right\}}
\end{align*}
Finally:
$$ \left\{ {\bf c} \in \R^m \ | \ \forall 1 \leq j\leq m, f_j({\bf c}) \geq 0 \right\}
 = \left\{ {\bf c} \in \R^m \ | \ [x_{\bf i'}^{-1} \circ \eta^{w_0, e} \circ x_{\bf-i}]_{trop}({\bf c}) \in \R_+^m \right\}$$
is the set of ${\bf i}$-string parameters that are mapped to non-negative ${\bf i'}$-Lusztig parameters. It has to be exactly the string cone $\Cc_{{\bf i}}^\vee$ thanks to theorem \ref{thm:tropical_from_lusztig_to_kashiwara}. Clearly, changing ${\bf i'}$ tantamounts to changing charts for the Lusztig parameters, and these charts are bijections of the positive orthant $\R_+^m$.
\end{proof}

As a consequence, the geometric Duistermaat-Heckman measure degenerates to the classical one, and its Laplace transform degenerates to the 'asymptotic' Schur functions (see \cite{bib:BBO2} theorem 5.5):
$$ \forall (\lambda, \mu) \in \afrak^2, 
   h_\mu(\lambda) := \frac{\sum_{w \in W} (-1)^{\ell(w)} e^{\langle \mu, w \lambda \rangle} }
                          {\prod_{\beta \in \Phi^+} \langle \beta^\vee, \mu \rangle }$$
\begin{proposition}
For $\lambda \in \afrak$:
$$\lim_{q \rightarrow 0} \psi_{q, \mu}\left( \lambda \right) = h_\mu(\lambda)$$
where:
$$h_\mu\left(\lambda\right) = \int_{\Cc_{\bf i}^\vee} {\bf dc} \varphi({\bf c}) \exp\left( \langle \mu, \lambda - \sum_{k=1}^m c_k \alpha_{i_k}^\vee \rangle \right) \mathds{1}_{\left\{ c_j \leq \lambda-\sum_{k=j+1}^m c_k \alpha_{i_j}(\alpha_{i_k}^\vee)\right\}}$$
Moreover, for $\mu \in C$, $h_\mu$ is a harmonic function on $C$ the Weyl chamber with Dirichlet boundary conditions and growth condition:
$$ \lim_{\lambda \rightarrow \infty, \lambda \in C} h_\mu(\lambda) e^{-\langle \mu, \lambda \rangle} =
   \frac{1}{\prod_{\beta \in \Phi^+} \langle \beta^\vee, \mu \rangle}$$
\end{proposition}
\begin{proof}
The function $\psi_{q, \mu}\left( \lambda \right)$ plays the role of normalization constant in proposition \ref{proposition:q_string_formula}, hence:
$$\psi_{q, \mu}\left(\lambda\right) = \int_{\R^m} {\bf dc} \varphi({\bf c}) \exp\left( \langle \mu, \lambda -q(w_0 \theta - \theta) - \sum_{k=1}^m c_k \alpha_{i_k}^\vee \rangle- f_{B,q,\lambda-q(w_0 \theta - \theta)}^{K, {\bf i}}({\bf c})\right)$$
The previous proposition yields the convergence of $\psi_{q, \mu}$ to $h_\mu$.\\
Now consider $\mu \in C$. In order to see it is a harmonic function on the Weyl chamber with Dirichlet boundary conditions, one can look at equation (\ref{eqn:q_deformed_toda}) and notice that the potential goes to zero inside the Weyl chamber and $+\infty$ outside.\\
For the growth condition, since $\psi_{q, \mu}(\lambda)e^{-\langle \mu, \lambda \rangle}$ is monotonically increasing for any sequence $\lambda_n \in C, \lambda_n \rightarrow \infty$ along a ray, the convergence to $q^m b(q\mu)$ is uniform in $q$ by Dini's theorem. Therefore, we can obtain the limiting behaviour of $h_\mu(\lambda) e^{-\langle \mu, \lambda \rangle}$ by inspecting:
$$ \lim_{q \rightarrow 0} q^m b(q\mu) = \lim_{q \rightarrow 0} \prod_{\beta \in \Phi^+ } q \Gamma(q \langle \beta^\vee, \mu \rangle)$$
Recalling that for all $z>0$, $\lim_{q\rightarrow 0} q \Gamma(qz) = \frac{1}{z}$ finishes the proof.
\end{proof}

Also, now we can recover the following results, already known to \cite{bib:BBO} and \cite{bib:BBO2}, as degenerations of theorems \ref{thm:q_highest_weight_is_markov} and \ref{thm:q_canonical_measure}.
\begin{thm}
For $W^{(\mu)}$ a Brownian motion in $\afrak$ with drift $\mu$, $\Pc_{w_0}( W^{(\mu)})$ is Brownian motion conditioned to stay in the Weyl chamber by a Doob transform. It has the infinitesimal generator:
$$ \half \Delta + \langle \log \nabla h_\mu, \nabla \rangle$$
And for every $\varphi$ bounded measurable function on $\R^m$ and $t>0$:
\begin{align}
\label{eqn:zero_canonical_measure}
  & \E\left( \varphi\left( \varrho_{\bf i}^{q=0, K}\left( W^{(\mu)}_u; 0 \leq u \leq t \right) \right) \ | \ \Pc_{w_0}( W^{(\mu)})_t = \lambda \right)\\
= &  \frac{1}{h_\mu\left(\lambda\right)}
     \int_{\Cc_{\bf i}^\vee} {\bf dc} \varphi({\bf c}) \exp\left( \langle \mu, \lambda - \sum_{k=1}^m c_k \alpha_{i_k}^\vee \rangle \right) \mathds{1}_{\{ c_j \leq \lambda-\sum_{k=j+1}^m c_k \alpha_{i_j}(\alpha_{i_k}^\vee)\}}
\end{align}
\end{thm}
\begin{proof}
 Immediate.
\end{proof}

\appendix
\chapter{Reminder of geometric crystal's parametrizations}
\label{appendix:parametrizations_reminder}

\begin{figure}[htp!]
\centering
\begin{tikzpicture}[baseline=(current bounding box.center)]
\matrix(m)[matrix of math nodes, row sep=5em, column sep=6em, text height=3ex, text depth=1ex, scale=2]
{
                    & x \in \Bc(\lambda) &                    \\
 z \in U^{w_0}_{>0} & v \in C^{w_0}_{>0} & u \in U^{w_0}_{>0} \\
};
\path[->, font=\scriptsize] (m-1-2) edge node[above]{$\varrho^L$} (m-2-1);
\path[->, font=\scriptsize] (m-1-2) edge node[right]{$\varrho^K$} (m-2-2);
\path[->, font=\scriptsize] (m-1-2) edge node[above]{$\varrho^T$} (m-2-3);

\draw [->] (m-2-1) to [bend left=30]  node[above, swap]{$b^L_\lambda$} (m-1-2);
\draw [->] (m-2-2) to [bend left=45]  node[auto, swap]{$b^K_\lambda$} (m-1-2);
\draw [->] (m-2-3) to [bend right=30] node[above, swap]{$b^T_\lambda$} (m-1-2);

\draw [->] (m-2-1) to [bend left=0 ]  node[above, swap]{$\eta^{e, w_0}$} (m-2-2);
\draw [->] (m-2-2) to [bend left=30]  node[auto , swap]{$\eta^{w_0, e}$} (m-2-1);
\end{tikzpicture}
\caption{Reminder: Charts for the highest weight geometric crystal $\Bc(\lambda)$}
\label{fig:reminder_geom_parametrizations}
\end{figure}
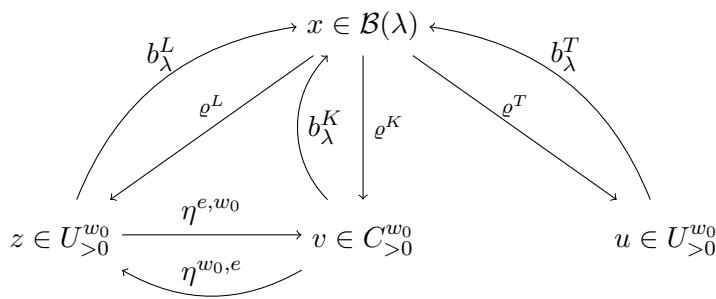

\begin{figure}[htp!]
\centering
\begin{tikzpicture}[baseline=(current bounding box.center)]
\matrix(m)[matrix of math nodes, row sep=3em, column sep=5em, text height=3ex, text depth=1ex, scale=1.2]
{ 
U_{>0}^{w_0} & \left( \R_{>0} \right)^m & \langle \pi \rangle    & \left( \R_{>0} \right)^m & C_{>0}^{w_0}\\
             &                          & \Bc(\lambda)\\
};
\path[->, font=\scriptsize] (m-1-3) edge node[above, scale=1.5]{$\varrho_{\bf i}^K$} (m-1-4);
\path[->, font=\scriptsize] (m-1-4) edge node[above, scale=1.5]{$x_{-\bf i}$} (m-1-5);
\path[->, font=\scriptsize] (m-1-3) edge node[above, scale=1.5]{$\varrho_{\bf i}^L$} (m-1-2);
\path[->, font=\scriptsize] (m-1-2) edge node[above, scale=1.5]{$x_{ \bf i}$} (m-1-1);

\path[->, font=\scriptsize] (m-1-3) edge node[left,  scale=1.5]{$p$} (m-2-3);
\draw[->] (m-2-3) to [bend right=30] node[right]{$p^{-1}$} (m-1-3);

\path[->, font=\scriptsize] (m-1-1) edge node[below, scale=1.5]{$b^L_\lambda$} (m-2-3);
\draw[->] (m-2-3) to [bend left=30] node[below]{$\varrho^L$} (m-1-1);
\path[->, font=\scriptsize] (m-1-5) edge node[below, scale=1.5]{$b^K_\lambda$} (m-2-3);
\draw[->] (m-2-3) to [bend right=30] node[below]{$\varrho^K$} (m-1-5);
\end{tikzpicture} 
\caption{Reminder: Parametrizations for a connected crystal $\langle \pi \rangle$, with $\pi \in C_0([0, T], \afrak)$ and $\lambda = \Tc_{w_0} \pi(T)$}
\label{fig:reminder_parametrizations_diagram}
\end{figure}

\chapter{Kostant's Whittaker model}
\label{appendix:whittaker_model}
For further details, we refer to the first section in \cite{bib:Sevostyanov00} as it gives a very good summary of Kostant's work on the Whittaker model and Whittaker modules. Here, we will mainly be interested in the image of the Casimir operator in the Whittaker model, seen as a right invariant differential operator on the lower Borel subgroup $B$.

\paragraph{Center:} 
$$\Zc(\mathfrak{g}) := \left\{ x \in \Uc(\gfrak) | \forall y \in \Uc(\gfrak) [x, y] = 0 \right\}$$
The center $\Zc(\gfrak)$ forms a commutative algebra. It is at the heart of both classical and quantum integrable systems. 

In Hamiltonian mechanics the Lie bracket is interpreted as a Poisson bracket and the center is an algebra of Poisson commuting functions. These functions are the observables that are integrals of motion. They are called the invariants. In quantum mechanics, the story is a bit different. Observables are differential operators acting on a Hilbert space of wave functions. Commuting observables give simultaniously measurable observables, which is a very desired property. The center, again called the set of invariants, is a commutative algebra of differential operators, and the Lie bracket is simply the commutator $[A, B] = AB - BA$. Having this in mind, it will be no surprise that an integrable quantum system will arise in our work.

The integrability property means that we have a maximal number of independent invariants. Chevalley's theorem tells us that the maximal number of independent central elements is $n$, the rank of Lie algebra.

\begin{thm}[Chevalley]
 $\Zc(\mathfrak{g})$ is a polynomial algebra with $n$ independent generators $I_1, I_2, \dots, I_n$. 
$\Zc(\mathfrak{g}) = \C[I_1, I_2, \dots, I_n]$
\end{thm}

\paragraph{Casimir element:} 
The only element of order 2 in the center is the Casimir element $C$. If $(X_1, \dots, X_n)$ is an orthonormal basis of $\hfrak$ with respect to the Killing form, then:
\begin{align}
\label{lbl:casimir}
\begin{array}{ll}
C & := \half \sum_{i=1}^n X_i^2 +  \half \sum_{ \beta \in \Phi^+}\left( f_\beta e_\beta + e_\beta f_\beta \right)\\
& = \half \sum_{i=1}^n X_i^2 +  \sum_{ \beta \in \Phi^+} f_\beta e_\beta + \rho^\vee\\
\end{array}
\end{align}
The second expression uses the Weyl co-vector $\rho^\vee$, which is the vector in $\afrak$ such that $\alpha\left( \rho^\vee \right) = 1$. $\rho^\vee$ is also the half sum of all positive coroots. In a way, $C$ is the simplest and most important element. In representation theory, it is used in order to prove the reducibility of certain classes of representations. In analysis, because it is of order 2, it can be considered as a heat kernel, when elliptic.

\paragraph{Reduction to $\Uc(\bfrak)$:} 
Let $\chi: U \longrightarrow \C$ be the standard (additive) character on the unipotent elements in $U$:
$$\forall \alpha \in \Delta, \chi\left( e^{t e_\alpha} \right) = t$$
Define the space of functions:
$$ \Cc^\infty_\chi(G) := \left\{ f \in \Cc^\infty\left(G\right) \ | \ \forall u \in U, f(g u) = f(g) e^{\chi(u)} \right\}$$
Because the subset $B B^+$, the cell where a Gauss decomposition holds, is dense in $G$, any function in $\Cc^\infty_\chi$ is entirely determined by its restriction to the lower Borel subgroup $B$. Moreover, differential operators in $\Uc\left( \mathfrak{g} \right)$ are reduced to elements of $\Uc\left( \bfrak \right)$ when acting on such functions. By simple differentiation and restriction to $\Cc^\infty(B)$, $C$ reduces to:
\begin{align}
\label{lbl:casimir_chi}
C_\chi & = \sum_{i=1}^n X_i^2 + \sum_{\alpha \in \Delta} f_\alpha + \rho^\vee
\end{align}

\paragraph{Whittaker model $W(\bfrak)$:} 
The algebraic construction by Kostant tantamounts to reducing central elements to elements in $\Uc(\bfrak)$. We reproduce the presentation of \cite{bib:Sevostyanov00} keeping the same notations. $\chi$ extends to the $\Uc(\nfrak^+)$ and gives a direct sum:
$$ \Uc(\nfrak^+) = \C \mathds{1} \oplus \ker \chi$$
Since $\Uc(\gfrak) = \Uc(\bfrak)\otimes \Uc(\nfrak^+)$ because of the PBW basis theorem, we have:
$$ \Uc(\gfrak) = \Uc(\bfrak) \oplus I_\chi$$
where $I_\chi = \Uc(\gfrak) \ker \chi$ is the left ideal generated by $\ker \chi$. Now let $\rho_\chi$ define the 
canonical projection:
$$ \rho_\chi: \Uc(\gfrak) \longrightarrow \Uc(\bfrak)$$

It defines the Whittaker model for the center thanks to:
\begin{thm}[Kostant, \cite{bib:Kostant78}, theorem 2.4.2]
Let $W(\bfrak) = \rho_\chi\left( \Zc(\gfrak) \right)$. The map:
$$ \rho_\chi: \Uc(\gfrak) \longrightarrow W(\bfrak)$$
is an isomorphism.
\end{thm}
One can easily compute the image of the Casimir element $C_\chi = \rho_\chi\left( C \right)$: if $\beta$ is a simple root, $e_\alpha$ acts like $1$ after reduction while if $\beta \in \Phi^+\\ \Delta$ acts like $0$. From the second line in (\ref{lbl:casimir}), we recover the same operator as in (\ref{lbl:casimir_chi}).

$C_\chi$ is then interpreted as an operator on the solvable group $B = NA$. The Laplacian $\half \sum_{i=1}^n  X_i^2$ is the infinitesimal generator of Brownian motion on $\afrak \approx \R^n$. Hence, $C_\chi$ as a whole is the infinitesimal generator of a Markov process driven by a simple Euclidian Brownian motion on $\afrak$.

Since, morally speaking, the Wiener measure charges all paths, we will need to study invariant ordinary differential equations driven by a deterministic path, giving us tools for a path-wise approach. This is another way of looking at the study of the flow $\left( B_t(.), t \geq 0 \right)$ defined by equation $\ref{lbl:process_B_ode}$.

\paragraph{Quantum Toda Hamiltonian:} 
A further reduction to a space of $\chi^-\chi$-binvariant functions:
$$ \Cc^\infty_\chi(G) := \left\{ f \in \Cc^\infty\left(G\right) \ | \ \forall u \in U, n \in n,  f(n g u) = e^{-\chi^-(n)} f(g) e^{\chi(u)} \right\}$$
Since a function in $\Cc^\infty_\chi(G)$ is entirely determined by its values on $A = \exp(\afrak)$, it can be viewed as a function on $\afrak$. This gives a Sch\"odinger operator on $\afrak \approx \R^n$ known as the quantum Toda Hamiltonian:
\begin{align}
\label{lbl:quantum_toda}
H & = \half \Delta + \rho^\vee - \sum_\alpha e^{-\alpha(x)}
\end{align}

\chapter{Enumeration of positive roots}
\label{appendix:positive_roots_enumeration}
There is a very simple yet very useful identity that can be found in the book by Kumar (corollary 1.3.22 \cite{bib:Kumar02}). We will use it several times. 

\begin{lemma}
\label{lbl:kumar} 
For $\lambda \in \mathfrak{a}$ and $w = s_{i_1} \dots s_{i_l}$ a reduced expression for 
the Weyl group element $w$ of length $l$, we have:
\begin{eqnarray}
\lambda - w \lambda & = & \sum_{k=1}^l \alpha_{i_k}(\lambda) \beta_k^\vee\\
\lambda - w^{-1} \lambda & = & \sum_{k=1}^l \beta_{k}(\lambda) \alpha_{i_k}^\vee
\end{eqnarray}
\end{lemma}

\begin{proof}
\begin{align*}
\lambda - w \lambda & = \sum_{k=1}^l s_{i_1} \dots s_{i_{k-1}} \lambda - s_{i_1} \dots s_{i_k} \lambda\\
& = \sum_{k=1}^l s_{i_1} \dots s_{i_{k-1}} \left( \lambda - s_{i_k} \lambda\right)\\
& = \sum_{k=1}^l s_{i_1} \dots s_{i_{k-1}} \left( \alpha_{i_k}(\lambda) \alpha_{i_k}^\vee \right)\\
& = \sum_{k=1}^l \alpha_{i_k}(\lambda) \beta_{k}^\vee
\end{align*}
\begin{align*}
\lambda - w^{-1} \lambda & = \sum_{k=1}^l s_{i_{k-1}} \dots s_{i_1} \lambda - s_{i_k} \dots s_{i_1} \lambda\\
& = \sum_{k=1}^l \left( id - s_{i_k} \right) \left( s_{i_{k-1}} \dots s_{i_1} \lambda\right)\\
& = \sum_{k=1}^l \alpha_{i_k}\left( s_{i_{k-1}} \dots s_{i_1} \lambda \right) \alpha_{i_k}^\vee\\
& = \sum_{k=1}^l \beta_{k}(\lambda) \alpha_{i_k}^\vee
\end{align*} 
\end{proof}

One can check on tables \ref{tab:roots_A2}, \ref{tab:roots_B2}, \ref{tab:roots_C2}, \ref{tab:roots_G2}, and \ref{tab:roots_A3} the previous facts.
\begin{table}
\centering
\caption{Positive roots enumerations for type $A_2$}
\label{tab:roots_A2}
$\alpha_1 = \begin{pmatrix} 1 \\ -1 \\ 0 \end{pmatrix} , \alpha_2 = \begin{pmatrix} 0 \\ 1 \\ -1 \end{pmatrix} $
  \begin{tabular}{|c|cc|}
  \hline
            &          121          &        212           \\
  \hline
  $\beta_1$ &      $\alpha_1$       &      $\alpha_2$      \\
  $\beta_2$ &   $\alpha_1+\alpha_2$ &  $\alpha_1+\alpha_2$ \\
  $\beta_3$ &      $\alpha_2$       &      $\alpha_1$      \\
  \hline
  \end{tabular}
\end{table} 

\begin{table}
\centering
\caption{Positive roots enumerations for type $B_2$}
\label{tab:roots_B2}
$\alpha_1 = \begin{pmatrix} 1 \\ -1 \end{pmatrix} , \alpha_2 = \begin{pmatrix} 0 \\ 1 \end{pmatrix} $
  \begin{tabular}{|c|cc|}
  \hline
            &          1212         &        2121           \\
  \hline
  $\beta_1$ &      $\alpha_1$       &      $\alpha_2$       \\
  $\beta_2$ &   $\alpha_1+\alpha_2$ &  $\alpha_1+2\alpha_2$ \\
  $\beta_3$ &  $\alpha_1+2\alpha_2$ &   $\alpha_1+\alpha_2$ \\
  $\beta_4$ &      $\alpha_2$       &      $\alpha_1$       \\
  \hline
  \end{tabular}
\end{table} 

\begin{table}
\centering
\caption{Positive roots enumerations for type $C_2$}
\label{tab:roots_C2}
$\alpha_1 = \begin{pmatrix} 1 \\ -1 \end{pmatrix} , \alpha_2 = \begin{pmatrix} 0 \\ 2 \end{pmatrix} $
  \begin{tabular}{|c|cc|}
  \hline
            &          1212         &        2121           \\
  \hline
  $\beta_1$ &      $\alpha_1$       &      $\alpha_2$       \\
  $\beta_2$ &  $2\alpha_1+\alpha_2$ &   $\alpha_1+\alpha_2$ \\
  $\beta_3$ &   $\alpha_1+\alpha_2$ &  $2\alpha_1+\alpha_2$ \\
  $\beta_4$ &      $\alpha_2$       &      $\alpha_1$       \\
  \hline
  \end{tabular}
\end{table} 

\begin{table}
\centering
\caption{Positive roots enumerations for type $G_2$}
\label{tab:roots_G2}
$$\alpha_1 = \begin{pmatrix} 0 \\ 1 \\ -1\end{pmatrix} ,
 \alpha_2 = \begin{pmatrix} 1 \\ -2 \\ 1\end{pmatrix},
 $$
  \begin{tabular}{|c|cc|}
  \hline
            &             121212           &          212121              \\
  \hline
  $\beta_1$ &          $\alpha_1$          &          $\alpha_2$          \\
  $\beta_2$ &    $3\alpha_1+\alpha_2$      &       $\alpha_1+\alpha_2$    \\
  $\beta_3$ &     $2 \alpha_1+\alpha_2 $   &      $3\alpha_1 + 2\alpha_2$ \\
  $\beta_4$ &     $3 \alpha_1+2\alpha_2 $  &      $2\alpha_1+\alpha_2$    \\
  $\beta_5$ &     $\alpha_1+ \alpha_2$     &      $3\alpha_1+\alpha_2$    \\
  $\beta_6$ &          $\alpha_2$          &          $\alpha_1$          \\
  \hline
  \end{tabular}
\end{table} 

\begin{table}
\centering
\caption{Some positive roots enumerations for type $A_3$}
\label{tab:roots_A3}
$$\alpha_1 = \begin{pmatrix} 1 \\ -1 \\ 0 \\ 0\end{pmatrix} ,
 \alpha_2 = \begin{pmatrix} 0 \\ 1 \\ -1 \\ 0\end{pmatrix},
 \alpha_3 = \begin{pmatrix} 0 \\ 0 \\ 1 \\ -1 \end{pmatrix}
 $$
  \begin{tabular}{|c|cc|}
  \hline
            &             123121           &          121321              \\
  \hline
  $\beta_1$ &          $\alpha_1$          &          $\alpha_1$          \\
  $\beta_2$ &     $\alpha_1+\alpha_2$      &      $\alpha_1+\alpha_2$     \\
  $\beta_3$ & $\alpha_1+\alpha_2+\alpha_3$ &          $\alpha_2$          \\
  $\beta_4$ &          $\alpha_2$          & $\alpha_1+\alpha_2+\alpha_3$ \\
  $\beta_5$ &     $\alpha_2+\alpha_3$      &      $\alpha_2+\alpha_3$     \\
  $\beta_6$ &          $\alpha_3$          &          $\alpha_3$          \\
  \hline
  \end{tabular}
\end{table}


\bibliographystyle{alpha} 



\listoftables
\listoffigures


\renewcommand{\indexname}{Notations index} 
\printindex 

\end{document}